\documentclass[twoside, 12pt]{amsart}
\usepackage[margin=3.5cm]{geometry}

\usepackage{subfiles} % Best loaded last in the preamble

\usepackage[T1]{fontenc}
\usepackage[lf]{Baskervaldx} % lining figures
\usepackage[cal=boondoxo]{mathalfa} % mathcal from STIX, unslanted a bit

\usepackage[utf8]{inputenc}

\usepackage[dvipsnames, table]{xcolor}
\usepackage{pdfpages}
\usepackage{latexsym}
\usepackage{amsthm,thmtools}
\usepackage{amscd}
\usepackage{amsmath}
\usepackage{amssymb}
\usepackage{graphicx}
\usepackage{mathrsfs}
\usepackage{mathtools}
\usepackage{pict2e}
\usepackage{stackrel}
\usepackage{faktor}
\usepackage{wasysym}
\usepackage{enumitem}
\usepackage{adjustbox}
\usepackage{perpage} 
\MakePerPage{footnote}
\usepackage{framed}
\usepackage{makecell}
\usepackage[all]{xy}
\usepackage[colorlinks,final,hyperindex]{hyperref}
\usepackage[noabbrev,capitalize]{cleveref}
\usepackage{tikz}
\usepackage{tikz-cd}
\usetikzlibrary{decorations.pathmorphing,decorations.markings,arrows,calc,shapes.geometric,arrows.meta,positioning}
\tikzset{math3d/.style=
	{x= {(-0.353cm,-0.353cm)}, z={(0cm,1cm)},y={(1cm,0cm)}}}
\tikzset{JLL3d/.style=
	{x= {(0.4cm,-0.2cm)}, z={(0cm,1cm)},y={(-1cm,0cm)}}}

\definecolor{Chocolat}{rgb}{0.36, 0.2, 0.09}
\definecolor{BleuTresFonce}{rgb}{0.215, 0.215, 0.36}
\definecolor{BleuMinuit}{RGB}{0, 51, 102}

\hypersetup{citecolor=BleuMinuit, linkcolor=Chocolat, filecolor=black, urlcolor=BleuMinuit}

\usepackage{xargs} 
\definecolor{armygreen}{rgb}{0.29, 0.33, 0.13}    

\usepackage[colorinlistoftodos,prependcaption,textsize=tiny]{todonotes}
\newcommandx{\unsure}[2][1=]{\todo[linecolor=red,backgroundcolor=red!25,bordercolor=red,#1]{#2}}
\newcommandx{\change}[2][1=]{\todo[linecolor=green,backgroundcolor=green!25,bordercolor=green,#1]{#2}}

\newcommandx{\info}[2][1=]{\todo[linecolor=yellow,backgroundcolor=yellow!25,bordercolor=yellow,#1]{#2}}
\newcommandx{\question}[2][1=]{\todo[linecolor=blue,backgroundcolor=blue!25,bordercolor=blue,#1]{#2}}
\newcommandx{\idea}[2][1=]{\todo[linecolor=orange,backgroundcolor=orange!25,bordercolor=orange,#1]{#2}}
\newcommandx{\link}[2][1=]{\todo[linecolor=black,backgroundcolor=white!25,bordercolor=black,#1]{#2}}

\allowdisplaybreaks

\newtheorem{Lem}{Lemma}[subsection]
\newtheorem{Th}[Lem]{Theorem}
\newtheorem*{Th*}{Theorem}
\newtheorem{Cor}[Lem]{Corollary}
\newtheorem{Prop}[Lem]{Proposition}
\newtheorem{Claim}[Lem]{Claim}
\newtheorem{Conv}[Lem]{Convention}

\newtheorem{assumption}[Lem]{\sc Assumption}

\theoremstyle{definition}
\newtheorem{Def}[Lem]{Definition}
\newtheorem{RQ}[Lem]{\sc Remark}
\newtheorem{Ex}[Lem]{\sc Example}

\newtheorem{Cons}[Lem]{\sc Construction}
\newtheorem{War}[Lem]{\sc Warning}
\newtheorem{notation}[Lem]{Notation}

\author{Albin Grataloup}
\title{Derived Symplectic Reduction and $\_L$-Equivariant Geometry}

\address{Albin Grataloup, IMAG, Univ. Montpellier, CNRS, Montpellier, France}

\email{\href{mailto:albin.grataloup@umontpellier.fr}{albin.grataloup@umontpellier.fr}}

\usepackage{scalerel,stackengine}
\stackMath
\newcommand\widecheck[1]{%
	\savestack{\tmpbox}{\stretchto{%
			\scaleto{%
				\scalerel*[\widthof{\ensuremath{#1}}]{\kern-.6pt\bigwedge\kern-.6pt}%
				{\rule[-\textheight/2]{1ex}{\textheight}}%WIDTH-LIMITED BIG WEDGE
			}{\textheight}% 
		}{0.5ex}}%
	\stackon[1pt]{#1}{\scalebox{-1}{\tmpbox}}%
}

\newcommand{\defi}[1]{\emph{#1}}
\newcommand{\Aa}{\mathbb{A}} 
\newcommand{\Nn}{\mathbb{N}} 
\newcommand{\Zz}{\mathbb{Z}} \newcommand{\Z}{\mathbb{Z}}
 
\newcommand{\Rr}{\mathbb{R}}

\newcommand{\Tt}{\mathbb{T}}
\newcommand{\Ll}{\mathbb{L}}

\newcommand{\Gg}{\mathbb{G}} 

\newcommand{\one}{\mathbf{1}}

\newcommand{\pre}[1]{#1_{\mathrm{pre}}}

\newcommand{\id}{\mathrm{id}} 

\newcommand{\tx}[1]{\mathrm{#1}}
\renewcommand{\bf}[1]{\mathbf{#1}}

\def\G_#1{\mathfrak{#1}} 
\def\t_#1{\widetilde{#1}}

\newcommand{\gmc}[1]{ #1^{\epsilon-\tx{gr}}}
\newcommand{\gmch}[1]{ #1^{h\epsilon-\tx{gr}}}
\newcommand{\gr}[1]{ #1^{\tx{gr}}}
\newcommand{\filt}[1]{ #1^{\tx{filt}}}
\newcommand{\cpl}[1]{ #1^{\tx{cpl}}}
\newcommand{\rel}[1]{ \left| #1 \right|}
\newcommand{\relg}[1]{ \left| #1 \right|_{\Delta}}
\renewcommand{\deg}[1]{ \left| #1 \right|}
\newcommand{\ndg}[1]{#1_\sharp}
\newcommand{\eq}[1]{ \left[ #1 \right] }
\newcommand{\red}[1]{ {#1}_{\tx{red}} }

\newcommand{\op}[1]{ {#1}^{\tx{op}} }

\newcommand{\comp}[1]{ \widehat{#1} }
\newcommand{\pund}[1]{ \underline{#1}_{\tx{pre}} }
\newcommand{\und}[1]{ \underline{#1}}

\newcommand{\ttl}[1]{ \emph{#1}}

\DeclareMathOperator{\Sym}{Sym}
\DeclareMathOperator{\cSym}{\widehat{Sym}}
\DeclareMathOperator{\Map}{Map}
\DeclareMathOperator{\Hom}{Hom}
\DeclareMathOperator{\iHom}{\underline{Hom}}
\DeclareMathOperator{\Der}{Der}

\DeclareMathOperator{\Mod}{Mod}
\DeclareMathOperator{\colim}{colim}
\newcommand{\colimsub}[1]{\underset{#1}{\colim}}

\newcommand{\Mapsub}[1]{\underset{#1}{\Map}}
\newcommand{\Homsub}[1]{\underset{#1}{\Hom}}
\newcommand{\iHomsub}[1]{\underset{#1}{\iHom}}

\DeclareMathOperator{\KT}{KT}
\DeclareMathOperator{\Stab}{\mathbf{Stab}}

\DeclareMathOperator{\dKT}{\bf{KT}}
\DeclareMathOperator{\BV}{BV}
\DeclareMathOperator{\FK}{FK}
\DeclareMathOperator{\iKT}{\mathbf{KT}}
\DeclareMathOperator{\iBV}{\mathbf{BV}}

\DeclareMathOperator{\FMP}{\mathbf{FMP}}

\DeclareMathOperator{\Symp}{\mathbf{Symp}}
\DeclareMathOperator{\Iso}{\mathbf{Iso}}
\DeclareMathOperator{\Lag}{\mathbf{Lag}}
\DeclareMathOperator{\IsoFib}{\mathbf{IsoFib}}
\DeclareMathOperator{\LagFib}{\mathbf{LagFib}}
\DeclareMathOperator{\LagCor}{\mathrm{Symp}}
\DeclareMathOperator{\Lagc}{\mathrm{Lag}_1}
\DeclareMathOperator{\Lagb}{\mathrm{Lag}_2}

\DeclareMathOperator{\MC}{\mathrm{MC}}
\DeclareMathOperator{\Sh}{\mathrm{Sh}}
\DeclareMathOperator{\tot}{\mathbf{Tot}}

\newcommand{\st}{\bf{St}}
\newcommand{\dst}{\bf{dSt}}
\newcommand{\sch}{\bf{Sch}}
\newcommand{\dsch}{\bf{dSch}}
\newcommand{\dfst}{\bf{dfSt}}
\newcommand{\dstfp}{\bf{dSt}^{\tx{afp}}}

\newcommand{\dpst}{\bf{dpSt}}
\newcommand{\dfpst}{\bf{dfpSt}}
\newcommand{\dpstfp}{\bf{dpSt}^{\tx{afp}}}

\newcommand{\CE}{Chevalley--Eilenberg }
\newcommand{\ce}{\mathbf{CE}}

\newcommand{\ceu}{\mathrm{CE}}
\newcommand{\algbd}{\mathrm{LieAlgd}}
\newcommand{\ialgbd}{\_L_\infty\mathrm{Algd}}
\newcommand{\lie}{\mathrm{Lie}}
\newcommand{\ilie}{L_\infty\mathrm{Alg}}

\newcommand{\cdga}{\mathrm{cdga}}
\newcommand{\cdgacon}{\mathrm{cdga}^{\tx{\leq 0}}}
\newcommand{\daff}{\mathbf{dAff}}
\newcommand{\aff}{\tx{Aff}}
\newcommand{\dafffp}{\mathbf{dAff}^{\tx{afp}}}

\newcommand{\igpd}{\mathbf{Gpd}_\infty}
\newcommand{\thick}{\mathbf{Thick}}
\newcommand{\thickp}{\mathbf{Thick}^{\tx{pre}}}

\newcommand{\QC}{\mathbf{QC}}
\newcommand{\perf}{\mathbf{Perf}}

\newcommand{\dr}{d_{\tx{dR}}}
\newcommand{\DR}{\mathbf{DR}}
\newcommand{\DRs}{\mathrm{DR}}
\newcommand{\Pol}{\mathrm{Pol}}

\newcommand{\Spec}{\mathbf{Spec}}
\newcommand{\Spf}{\mathbf{Spf}}
\newcommand{\Crit}{\tx{Crit}}
\newcommand{\RCrit}{\mathbf{Crit}}

\newcommand{\Llr}[2]{\Ll_{#1/#2}}
\newcommand{\Ttr}[2]{\Tt_{#1/#2}}
\newcommand{\Ttrl}[2]{\widetilde{\Tt}_{#1/#2}}
\newcommand{\QS}[2]{\left[ \faktor{#1}{#2} \right]}
\newcommand{\pQS}[2]{\left[ \faktor{#1}{#2} \right]_{\tx{pre}}}
\newcommand{\QSW}[2]{\widetilde{\left[ \faktor{#1}{#2} \right]}}

\newcommand{\dc}{\_A_{/A}}

\renewcommand{\small}{\mathbf{Art}_{/A}}

\newcommand{\ST}{ \mathrm{Stf} }

\newcommand\mlnode[1]{\fbox{\begin{tabular}{@{}c@{}}#1\end{tabular}}}

\date{\today}

\begin{document}

\maketitle

	\makeatletter
\def\@tocline#1#2#3#4#5#6#7{\relax
	\ifnum #1>\c@tocdepth % then omit
	\else
	\par \addpenalty\@secpenalty\addvspace{#2}%
	\begingroup \hyphenpenalty\@M
	\@ifempty{#4}{%
		\@tempdima\csname r@tocindent\number#1\endcsname\relax
	}{%
		\@tempdima#4\relax
	}%
	\parindent\z@ \leftskip#3\relax \advance\leftskip\@tempdima\relax
	\rightskip\@pnumwidth plus4em \parfillskip-\@pnumwidth
	#5\leavevmode\hskip-\@tempdima
	\ifcase #1
	\or\or \hskip 1em \or \hskip 2em \else \hskip 3em \fi%
	#6\nobreak\relax
	\hfill\hbox to\@pnumwidth{\@tocpagenum{#7}}\par% <---- \dotfill -> \hfill
	\nobreak
	\endgroup
	\fi}
\makeatother
\thispagestyle{empty}
\newpage
\thispagestyle{empty}
\pagebreak

\newpage
 \ 
 
 \newpage
\setcounter{tocdepth}{2}

	\newpage

\section*{Acknowledgment}

I would like to express my deepest gratitude to Damien Calaque, my PhD advisor. It has been a privilege to have had the opportunity to do my PhD under your direction. I am grateful for your support all along this PhD, with all its turns, advances and setbacks, for how available you always were to discuss and help me, and for your unlimited patience when answering all my questions (sometimes repeatedly).  \\

I would like to thank Giovanni Felder and Domenico Fiorenza for agreeing to referee this thesis. I am also grateful to Grégory Ginot, Benjamin Hennion, Pavel Safronov, Alexander Schenkel and  Chenchang Zhu for being part of my thesis committee. \\

If I would pinpoint where my journey into derived geometry and the use of homotopical methods in geometry and mathematical physics started, I would say that it was 5 years ago with a seminar talk of Alexander Schenkel on homotopical algebraic quantum field theory at the Institut Camille Jordon in Lyon. This opened my eyes to new mathematical possibilities and little did I know that not even a year later, I would go to the university of Hamburg to work for a few months under the supervision Marco Benini, a close collaborator of Alexander Schenkel.

I would like to express my gratitude to Marco Benini for everything he taught me during that time on his work homotopical algebraic quantum field theory and in particular on the foundations of Quillen's homotopy theory, which have proven to be an essential tool for my work during my PhD. \\

Although it would be difficult to try to acknowledge every single source of influence on my thesis, among my closest mathematical influence are of course Damien Calaque (everything symmetric algebras, Lagrangian correspondences and so much more), Joost Nuiten (for his expertise on Lie algebroids and derived geometry in general) and Pelle Steffens (on so many $\infty$-topics). 

When it comes to other mathematical influence, I would like to thank Ricardo Campos for everything he taught me on operads and deformation theory for our review on operadic deformation theory. 

I would also like to express my appreciation to Corina Keller for our short lived but promising project on blow-ups and groupoids.   \\

I could not express the depth of my gratitude to Ricardo, Joost, Pelle and Corina for welcoming me so openly in Montpellier. Never before did I feel so welcomed and I consider the time we have spent together as the highlight of my time in Montpellier. 
But beside our countless hours at the Broc or playing bang, I has been a real pleasure to have gotten to know, at least a little bit, each and everyone of you. 

To Ricardo, always bringing magic and cardos to the evenings. I truly enjoyed our time together, either while working together or while playing Bang, having drinks etc... I wish you all the best for your new journey into parenthood. 

To Joost, it has been a real pleasure to spend all this time with you. I have fond memories of all the times having drinks, going to concerts and our time together in general. 

To Pelle, I am thankful for your friendship during all these years. Our philosophical discussions on math, physics, philosophy of the mind and zombies were to me a rare and great pleasure. I am also deeply thankful to you to have introduced me to everything there is to know about Dutch cuisine. 

To Corina, I have enjoyed many great times in your company with our trips to the beach, or the concerts we have been to and much more. But most of all I am grateful for all the great conversations we have had. I also want to thank you for introducing me to the Art of popcorn, and trusted me to learn it properly despite my rather smoky and Covid rich debut.\\

I would also like to thank Montpellier for its unwavering beautiful weather, its challenging cycling roads, its nearby beaches that uplifted my summers, and for being a very enjoyable city overall.     \\

Et finalement, j'aimerais exprimer toute ma gratitude à mes parents pour leur soutient indéfectible tout au long de ma thèse. \\

This project has received funding from the European Research Council (ERC) under the
European Union’s Horizon 2020 research and innovation programme (grant agreement No
768679).

\newpage

\

\newpage

\tableofcontents

\newpage
\section*{Introduction}

The initial motivation of this thesis is the study of the classical BV (Batalin--Vilkovisky) formalism from the point of view of derived geometry. This construction, originally studied in \cite{BV81} and \cite{BV83}, was initially developed for the study of quantization of gauge theories. It is a tool to construct algebras of observables that would behave well with respect to quantization. It also plays an important role in modern quantum field theory and has been subject to much recent developments in this context (see for Example \cite{CG21, Re11, FR12, FR13}). In this thesis, we will only study the \emph{classical} BV formalism. 

In order to make sense of the geometric tools required to study the classical BV formalism, we will study (infinitesimal) equivariant derived symplectic geometry and define the notion of (shifted) moment map and (shifted) symplectic reduction in the setting of Lie algebroids and groupoids. \\

There are different variations of what a classical BV algebra associated to a functional should be, but all the existing constructions produce algebraic objects, even though as we will explain presently, these constructions seem to be designed to represent some geometric operations. We will focus on two types of constructions. For the first one, following  the ideas in \cite{FK14}, the construction goes as follows (see Construction \ref{cons:BV FK} for a more precise description of that construction):

\begin{Cons}\ \label{cons:construction bv fk}
	\begin{enumerate}
		\item We first consider the critical locus of an action functional (that is the space of solutions to the Euler--Lagrange equations). This is done by taking a cohomological resolution, called the Koszul--Tate resolution, of the critical locus of the functional. This resolution is obtained by successively adding anti-fields and anti-ghosts fields in negative degree.
		\item Then we construct a $(-1)$-shifted symplectic graded algebra (obtained as a shifted cotangent using the Koszul--Tate resolution from before). Essentially it adds \emph{ghost fields} in positive degrees to the Koszul--Tate resolution that are ``dual'' to the anti-ghosts fields via the symplectic structure. 
		\item We then build a Hamiltonian differential on this graded algebra by finding a solution to the classical master equation called the \defi{BV charge}: 
		\[ \lbrace Q, Q \rbrace\] 
	\end{enumerate} 
\end{Cons}

From a geometric point of view, this construction amounts to first take a resolution of the critical locus of the action functional (by adding anti-fields and anti-ghost fields) and then take its quotient by ``maximal symmetries'' (by adding the ghost fields) such that the quotient is symplectic.\\ 

The other approach to the BV construction we are going to consider is similar to the one described in \cite[Part 1 Section 4]{CG21}. The idea is to take the derived critical locus of the induced ``quotient maps''.  

\begin{Cons}\label{cons:construction bv CG}\
	
	\begin{enumerate}
		\item First, we chose a formal neighborhood of a point in the space of solutions of the Euler--Lagrange equation given by a $\_L_\infty$-algebra. This amounts to take the ghost fields and corresponds to looking at an infinitesimal quotient by infinitesimal symmetries of the chosen point. Later on, we will instead consider the formal neighborhood of \emph{all} the ``solutions'' and consider a \emph{Lie algebroid} of symmetries.
		
		\item Then the BV algebra produced is constructed by as some kind of \emph{derived critical locus} together with its canonical $(-1)$-shifted symplectic structure. This second step amounts to creating the anti-fields and anti-ghosts fields, dual to the ghosts fields.
	\end{enumerate}
\end{Cons} 

In both of these constructions, we gave a description based on geometric operations like ``taking a quotient'' and taking a ``zero locus''. However in practice these operations are only handled algebraically, and the dictionary between the algebraic constructions and their geometric counterparts is not very precise, a problem which this thesis is trying to address.\\ 

As both the BV constructions we described involve cohomological technics, we will naturally need to turn to \emph{derived geometry} in order to make sense of those geometric objects that are defined only ``up to homotopy'', since only the cohomology of the commutative differential graded algebra of observables has a physical meaning. It is all the more compelling to use derived geometry for this problem since derived geometry is well suited to handle intersections and quotients, which are the building blocks of the BV construction.    \\ 

Derived geometry is in a nutshell the merging of homotopy theory and geometry. It was developed, to mention only a few references, in \cite{TV05,TV08,Lu04,Lu09}, and has provided powerful tools to handle problems ranging from deformation theory, intersection theory, equivariant geometry and others. Our goal will be to explain how to think about the algebraic constructions of BV algebras in geometric terms.

We are setting ourselves in the context of \emph{derived algebraic geometry}, where most of the geometric tools we need, such as derived symplectic geometry and formal geometry, have already been well developed. Moreover for the BV construction, we will restrict ourselves to the situation where we take a functional of the forms $f: X \to \Aa_k^1$ with $X$ a smooth affine algebraic variety.  When it comes to concrete examples used in physics, these are very restrictive assumptions since the typical spaces $X$ arising from physics would be infinite dimensional differential manifolds. Our restrictions allow us to bypass the difficulties coming from geometry and analysis in infinite dimensions, and provides a blueprint for a more geometrical approach to the BV construction. 

We hope that these ideas would adapt to the context of derived differential geometry and later on to the infinite dimensional setting.  \\

Our starting point in order to try to understand the classical BV construction is the following: 

\begin{RQ}\label{rq:trivial bv}
	The derived critical locus is a classical BV algebra. It is the simplest classical BV algebra as there are no ghost fields or anti-ghost fields. In other words, it is not a quotient by symmetries. It is a well known fact, going back to \cite{PTVV}, that the derived critical locus is canonically $(-1)$-shifted symplectic.  
\end{RQ}

In fact, following the idea of Construction \ref{cons:construction bv CG}, classical BV algebras are derived critical loci of ``equivariant maps associated to $f$''. But $f$ itself is equivariant for the trivial action and its associated BV construction is the derived critical locus as mentioned in Remark \ref{rq:trivial bv}. We will come back to this idea, but for now, it turns out that we need a more general framework to encompass Construction \ref{cons:construction bv fk}. Indeed, that construction will not be a derived critical locus in general\footnote{Essentially because taking a \emph{resolution} of the critical locus will produce a differential too complicated to be coming from a derived critical locus.}. 

Following and generalizing the construction of \cite{FK14}, we can start our construction by taking the derived critical locus $\RCrit(f)$ and then ``add anti-ghost fields''. Adding enough anti-ghost fields leads to the Koszul--Tate resolution used in \cite{FK14} in order to make Construction \ref{cons:construction bv fk}. But then Construction \ref{cons:construction bv CG} does not fit in this framework. Therefore we take the freedom to take as many (or few) anti-ghost fields as we like, essentially choosing which (higher) symmetries of the system we want to consider or not. This leads to the notion of ``almost derived critical loci'' (Definition \ref{def:almost derived critical loci}) which are objects sitting ``in between'' the strict critical locus given by the Koszul--Tate construction (which gives in some sens a ``maximal'' BV construction) and the derived critical locus (the ``trivial'' BV construction). 

\begin{RQ}
	The classical BV construction (in the algebraic setting) is given by the \CE algebra associated to a Lie algebroid of ``symmetries'' on an almost derived critical locus. 
\end{RQ}

The reason the BV construction is build from this \CE algebra is that it represents an algebra of ``derived infinitesimal invariants'' with respect to an \emph{infinitesimal action} of the Lie algebroid. In other word, it represents a notion of \emph{derived infinitesimal quotient}. However, this heuristic is difficult to justify from a geometric point of view. The dictionary between the geometric and algebraic concepts is really unclear and we will come back to these notions in a moment.    

Until then, we need to point out that not all Lie algebroid of ``symmetries'' would do the trick. Indeed, one of the most important feature of the classical BV algebra is that it is $(-1)$-shifted symplectic. Therefore, in order to get a symplectic structure on the (infinitesimal) quotient, we need to quotient by ``maximal symmetries'', which amounts to require a non-degenerate duality between the ghost and anti ghosts fields. We will geometrically interpret this procedure of taking the ``quotient by maximal symmetries'' as a taking the \emph{derived symplectic reduction}. \\

We can summarize the concepts involved in the following \emph{incorrect} dictionary:\

\vspace{6mm} 

\begin{tabular}{|c|c|c|}
	\hline
	Concept	& Geometric side & Algebraic side \\
	\hline
	\makecell{Solutions to the \\ E.L equation }	&	\makecell{(Almost) derived \\ critical locus } &  	\makecell{Koszul--Tate resolution \\ anti-ghost fields }\\
	\hline
	\makecell{Reduction by \\ maximal symmetries}	& \makecell{Lie algebroids \\ Infinitesimal actions/quotients} & \makecell{$\ce$ algebras \\ ghost fields}  \\
	\hline
	\makecell{Algebra of \\ observables} & \makecell{$(-1)$-shifted sympletic \\ geometry and reduction}  & $P_0$-algebras \\
	\hline
\end{tabular} \\

\vspace{6mm}

In this thesis, we aims at giving a sens to the geometric column and to the classical BV construction using only the geometric operations. We will then try to compare them to what we get on the algebraic side. This geometric construction will also turn out to naturally generalize to groupoid actions leading to new constructions of classical BV geometric object. \\

In order to make sense of the geometric objects we need, we will study in details the notions of (infinitesimal) actions and quotients as well as how they behave with respect to \emph{derived symplectic geometry}. \\

Derived symplectic geometry was developed, to mention only a few, in \cite{PTVV,CPTVV,Ca15,Ca19,Ca21,AC21}. It turns out that derived symplectic geometry is in fact much better behaved than non-derived geometry when it comes to constructing new (shifted) symplectic structures. In fact Section \ref{sec:new-constructions-from-old-ones} is devoted to the study of such constructions and in particular, it explains why the derived critical locus (who is defined as a derived intersection) is canonically $(-1)$-shifted symplectic. This is a key fact from the BV construction of \cite{CG21}. We also develop some a new construction of Lagrangian fibrations (see \cite{Gr22} and also \cite{Sa20}), which turns out to be relevant for quantizing shifted symplectic structures (see \cite{Sa20}). We can also produce new Lagrangian correspondences from some kind of Lagrangian intersections (see Section \ref{sec:the-higher-categories-of-lagrangians} and in particular Theorem \ref{th:lagragian correspondence pullback}).  \\

To relate derived symplectic geometry to the BV construction, let us summarize what we want for a BV construction:

\begin{itemize}
	\item Take an almost derived critical locus $S$. It is endowed with the pullback pre-symplectic structure along the natural map $S \to \RCrit(f)$.
	\item Take a ``quotient'' by a Lie algebroid $\_L$ such that the quotient is $(-1)$-shifted symplectic.
\end{itemize}

We can make a parallel between that construction and the notion of \defi{symplectic reduction} which in classical terms is constructed out of a moment map $\mu: X\to \G_g^*$ associated to an action of a Lie group $G$ on $X$ as follows: 
\begin{itemize}
	\item Take the kernel of a moment map $Z(\mu) := \ker(\mu: X \to \G_g^*)$. It is endowed with the pre-symplectic structure pulled-back from $X$. 
	\item Take a quotient $\faktor{X}{G}$ by a Lie group $G$ integrating $\G_g$ such that the quotient is symplectic.
\end{itemize}

This motivates us to study the notion of \emph{derived symplectic reduction}. It is a known fact that although classical symplectic reduction only makes sense under some conditions on $X$ and the action so that the quotient exists. Derived geometry has no such restriction as long as we are willing to work with \defi{quotient stacks}. In fact, the symplectic reduction is the derived setting is obtained by the \emph{derived pullback}:
\[ \begin{tikzcd}
	\QS{Z(\mu)}{G} \arrow[r] \arrow[d] &  \textbf{B}G := \QS{\star}{G} \arrow[d] \\
	\QS{X}{G}  \arrow[r, "\eq{\mu}"] & \QS{\G_g^*}{G} 
\end{tikzcd} \]

We will see that this pullback is naturally symplectic as a derived intersection of Lagrangian morphisms in a $1$-shifted symplectic derived stack given by the quotient stack by the coadjoint action. 

The study of moment maps from the derived point of view is done in \cite{AC21} which we extensively recall in Section \ref{sec:for-groups}. The main intakes from that section are the following: 
\begin{itemize}
	\item A moment can be defined as a $G$-equivariant map $\mu: X \to \G_g^*$  such that we have an equivalence of derived symplectic stacks (Definition \ref{def:moment map structure}):
	\[X \simeq  \QS{X}{G} \times_{\QS{\G_g^*}{G}} \G_g^* \]
	
	\item The reduced space of $X$ (its symplectic reduction) is defined as the pullback (Definition \ref{def:moment map structure}):
	\[ \begin{tikzcd}
		\red{X} \arrow[r] \arrow[d] &  \textbf{B}G := \QS{\star}{G} \arrow[d] \\
		\QS{X}{G}  \arrow[r, "\eq{\mu}"] & \QS{\G_g^*}{G} 
	\end{tikzcd} \]
	\item There is a Lagrangian correspondence (Proposition \ref{prop:symplectic reduction lagrangian correspondence}): 
	\[\begin{tikzcd}
		& Z(\mu)  \arrow[dr] \arrow[dl]&\\
		\red{X}& &  X 
	\end{tikzcd} \] 
\end{itemize}  

It also turns out that moment maps have the nice property to be stable under ``Lagrangian intersection'' for an appropriate notion of Lagrangian having a ``Lagrangian reduction'' (Theorem \ref{th:symplectic reduction commutes with lagrangian intersection groups}). In particular, we get a $(-1)$-shifted moment map: \[\mu_{-1}: \RCrit(f) \to \G_g^*[-1]\]
whose symplectic reduction gives the equivariant derived critical locus of the map: \[\eq{f}: \QS{X}{G} \to \Aa_k^1\]

The symplectic reduction of this moment map is a version of the BV construction based on a group of symmetries. It was studied in \cite{BSS21} and is further discussed in Section \ref{sec:with-moment-maps}. We will see in that section that this construction fits both the heuristics of Constructions \ref{cons:construction bv fk} and \ref{cons:construction bv CG}. 

However, in both heuristics, the usual classical BV construction does not involve a group action but rather an ``infinitesimal action''. In the group case, it would correspond to the action of the associated Lie algebra of the group, but in general we need to come up with a good notion of infinitesimal quotient and action. \\ 

We will explain in a moment that we can make sense of such notions. Then we can adapt the previous story with groups to the setting of Lie algebroid with ``infinitesimal actions and quotients''. It is however a more complicated story and requires a slightly different notion of moment map. Essentially, we consider a map of nice enough \defi{derived stacks}, $\mu: Y \to L^*$ with $\_L$ a Lie algebroid over a derived stack $X$, and $Y$ a symplectic derived stack. Then the structure of moment map is given by (see Definition \ref{def:moment map lie algebroid}): 
\begin{itemize}
	\item An \defi{action} of the Lie algebroid $\_L$ on $Z(\mu):= \ker(\mu)$. 
	\item A $n$-shifted symplectic structure on the infinitesimal quotient\footnote{The notion of infinitesimal quotient is a bit complicated if $Y$ is not nice enough. In full generality, we will need a notion of ``weak infinitesimal quotient''.}, $\QS{Z(\mu)}{\_L}$, together with a Lagrangian correspondence: 
	\[\begin{tikzcd}
		& Z(\mu)  \arrow[dr] \arrow[dl]&\\
		\QS{Z(\mu)}{\_L} & &  Y
	\end{tikzcd} \] 
\end{itemize}

Then $\red{Y}:=\QS{Z(\mu)}{\_L}$ is called the \defi{infinitesimal symplectic reduction} of $Y$ along $\mu$. \\

The typical example of such moment maps is the dual of the anchor $T^*X \to L^*$. Moreover, this notion of moment map is again well behaved with respect to derived intersections (see Theorem \ref{th:lagrangian intersection of moment map lie for lie algebroids}). In particular, there is a $(-1)$-shifted moment map: \[\RCrit(f) \to L^*[-1]\] for $\_L$ a Lie algebroid of ``symmetry'' of $f$ over $X$ (when $X$ satisfies some technical assumptions). Moreover the symplectic reduction of this moment map is exactly $\RCrit(\eq{f})$ (Example \ref{ex:shifted moment map algebroid derived critical locus}).\\

Then the analogy with the BV formalism of Construction \ref{cons:construction bv CG} is that given $\_L$, a Lie algebroid of symmetries, $\RCrit(\eq{f})$ is by definition the BV construction we are interested in, and it is also the symplectic reduction of the $(-1)$-shifted moment map $\mu_{-1}: \RCrit(f) \to \_L^*[-1]$. It can therefore also be seen as a quotient of the almost derived critical locus given by $Z(\mu_{-1})$, and therefore fitting is a generalized version of Construction \ref{cons:construction bv fk}. \\

We have now seen that Construction \ref{cons:construction bv CG} seems to be the symplectic reduction of a $(-1)$-shifted moment map on $\RCrit(f)$. The problem is that Construction \ref{cons:construction bv fk} does not exactly fit in this context as there is in general no moment maps whose kernel will recover the full Koszul--Tate resolution. Therefore we propose a generalization of the previous ideas and define a BV construction as some kind of ``generalized symplectic reduction'' of $S \to \RCrit(f)$ (Definitions \ref{def:generalized symplectic reduction groupoid} and \ref{def:generalized symplectic reduction}), with $S$ an ``almost derived critical locus'' of $f$. A derived stack $Y$ will be called an (infinitesimal) BV construction of $f$ with respect to a given almost derived critical locus $S$ if (see Definitions \ref{def:BV complex} and \ref{def:BV complex infinitesimal}):

\begin{itemize}
	\item $Y$ is an ``(infinitesimal) quotient'' of $S$. 
	\item $Y$ is $(-1)$-shifted symplectic and is part of a Lagrangian correspondence:  \[\begin{tikzcd}
		& S  \arrow[dr] \arrow[dl]&\\
		Y & &  \RCrit(f) 
	\end{tikzcd} \]
\end{itemize}

This is a rather wide generalization of both BV constructions as the symplectic structure could have complicated homotopy theory (see Theorem \ref{th:strictification symplectic structure}) and needs \emph{not} have a BV charge.

Nevertheless, we will discuss why such objects should still ``look like'' classical BV-algebras as a generalization of Construction \ref{cons:construction bv fk} (see Section \ref{sec:strictification}). \\

Finally, all this story generalizes very nicely when replacing infinitesimal actions, quotients, moment maps etc... by the similar notions for \emph{Segal groupoids}. In fact, the groupoid picture works in even greater generality and is much better behaved than its infinitesimal counterpart\footnote{There is in particular no need for a notion of ``weak'' quotient in this setting.}. The groupoid picture simultaneously generalizes the analogous stories for groups and Lie algebroids we previously discussed. This leads to a new notion of BV construction that recalls the global action of symmetries instead of only the associated infinitesimal action.\\

The main issue of our definition of BV is that for a given an almost derived critical locus $S$ (which amounts to fixing the symmetries we want to consider), we have little control on the existence and uniqueness of a BV structure on $S \to \RCrit(f)$. Indeed, there might be a full (and possibly empty) topological space of such structures. In Section \ref{sec:examples-of-bv-constructions}, we show that we have in fact a rather large class of examples of BV constructions, including some recovering Constructions \ref{cons:construction bv fk} and \ref{cons:construction bv CG}. \\

Finally, much of what we discussed so far relies on a good theory of (infinitesimal) \emph{equivariant geometry}, in other words, the study of the geometry of (infinitesimal) quotients. In Section \ref{sec:g-equivariant-geometry}, we discuss the main results on $\_G$-equivariant geometry that we need to do symplectic geometry and symplectic reduction. This behaves very well, and we want to mimic this for infinitesimal equivariant geometry, that is, geometry equivariant with respect to the action of Lie algebroids. \\

One of the main problem is to come up with a good notion of infinitesimal quotient of Lie algebroids. In Section \ref{sec:on-the-derived-geometry-of-lie-algebroid}, we define such a notion and study its main properties. Unfortunately, this is only valid in a very restricted setting (namely when the base satisfies Assumptions \ref{ass:very good stack}). In particular, many examples of interest do not fit in this framework. \\

Yet, this approach is the ``correct'' notion of infinitesimal quotient from a geometric point of view, in a sens made precise in Section \ref{sec:derivation-and-integration-of-lie-algebroids}. In that section, we explain the heurestic saying that Lie algebroids are the infinitesimal versions of groupoids. We describe a derivation and integration procedure showing that under some mild conditions, the infinitesimal quotient of a derived stack by a Lie algebroid is the formal completion of the quotient of this derived stack by a groupoid ``integrating'' it, thus explaining the terminology ``infinitesimal quotient''.\\  

In order to palliate to the issue of only having a notion of infinitesimal quotients in a very restricted setting, and motivated by some results of Section \ref{sec:g-equivariant-geometry}, we come up with a notion of ``weak infinitesimal quotients'' in order to encompass the missing examples that Assumptions \ref{ass:very good stack} did not permit. The only draw back is that it loses sight of an actual action (in the sens given in Section \ref{sec:infinitesimal-action-of-lie-algebroids}), and the infinitesimal quotients become part of the data. As such they need neither to exist or to be unique. However, most of our examples of interests are weak infinitesimal quotients and the desired properties analogous to the groupoid case hold.  \\ 

Equipped with these tools, we can talk about $\_L$-equivariant geometry in an almost identical way as for groupoids (although in a more restricted setting). Through all of our discussion, we will make a point to show that all the infinitesimal construction (BV, symplectic reduction, actions and quotient...) are exactly formal completions of their global counterpart for groupoids (again under some mild conditions).

\subsection*{Notations and Convention}
\begin{itemize}
	\item All along $k$ will denote a field of characteristic $0$. 
	\item We use cohomological conventions, that is, the differential have degree $+1$. We denote by $\Mod_k$ the category of cochain complexes over $k$ together with its projective model structure.
	\item Everything will be differential graded unless stated otherwise. For example when we say ``modules'' and ``algebras'' we mean ``differential graded modules'' and ``differential graded commutative algebras''.  The degree of an object $a$, an homogeneous element of a differential graded object, will be denoted by $\deg{a}$.
	\item All algebras and monoids will be unital commutative algebras and monoids. In fact an algebras are defined as a commutative monoid in a given symmetric monoidal category (by default $\Mod_k$). The category of commutative differential graded algebras (in $\Mod_k$) will be denoted by $\cdga$ and is endowed with its projective model structure induced by the one on $\Mod_k$. Variations and properties of these model structures when replacing $\Mod_k$ by a ``good'' model category are discussed in Appendix \ref{sec:good-model-structures}.   
	\item We will make no difference between the notations for a model category and its associated $\infty$-category. Everything can be assumed to come from model categorical constructions unless we specify otherwise. For example, by default, all functor are Quillen unless they are called ``$\infty$-functors''. 
	\item All the $\infty$-categories we consider will be $(\infty, 1)$-categories.
	\item Everything will be ``up to homotopy'' unless stated otherwise (usually by using the adjective ``strict'', ``classical'' or ``underived''). In particular, all diagrams commute only up to homotopy\footnote{In other words, all diagrams come together with the data of the homotopies making the diagram commute.} and all functors are by default derived including for example: 
	\begin{itemize}
		\item[$\bullet$] Limits and colimits, including pullbacks and pushouts. 
		\item[$\bullet$] The (enriched) $\Hom$ functors and mappings spaces. 
	\end{itemize}  
	\item All along, our geometric objects (derived schemes, stacks etc...) will be augmented over the point $\star := \Spec(k)$ making them derived $k$-schemes, derived $k$-stacks and so on... Keeping this in mind, $k$ be omitted from the notations. In a dual way, all algebras will be $k$-algebras and we will also omit $k$ from the notations. 	
	
	\item If $A$ is an algebra, then $\Mod_A$ is the category of module over $A$. We will sometimes consider algebras and modules in a more general \emph{good model} categories. We refer to Appendix \ref{sec:good-model-structures} for more about these categories and modules. 
	
	\item $\_S$ is the $\infty$-category of topological spaces which is equivalent to the $\infty$-category of $\infty$-groupoids, denoted by $\igpd$. This equivalence is given by the geometric realization functor: 
	\[ \relg{-} : \igpd \to \_S \]

	\item The $\Hom$-set of a category $\_C$ is denoted by: \[\Homsub{\_C} \ \  \text{or} \ \ \Hom_\_C\] 
	\item Given an $\infty$-category $\_C$, then we denote by:  \[\Map_\_C(-,-) \in \_S\] the \defi{mapping space} of $\_C$.  
	\item In general for a $\_V$-enriched category $\_C$ we denote by: \[\Hom_{\_C}^\_V(-,-)\] the $\_V$-enriched Hom functor. If $\_V = \Mod_k$, we will write: 
	\[ \iHom(-,-) := \Hom^{\Mod_k}(-,-) \]
	
	\item To simplify some notations, we will also write: 
	\[\Hom_B(-,-) := \Hom_{\Mod_B} (-,-) \]
	and similarly for mapping spaces, or enriched Hom functors.

	\item The zero section of a linear stack (see Section \ref{sec:definition-and-examples}) is denoted by $s_0$. 
	
	\item If $V$ and $W$ are cochain complexes, and $f: V \to W$ a morphism between them, then we denote by: \[ V[1] \oplus^f W \]
	the cochain complex whose underlying module is $V[1] \oplus W$ with differential the sum of the differential on $V$, the differential on $W$ and $f$ viewed as a map of degree $1$. This is a model for the homotopy fiber of $f$: 
	\[\tx{fiber}\left( M \overset{f}{\to} N\right)\]
	
	\item Let $f: F \to E$ be a linear morphism of linear stack over $X$ (see Section \ref{sec:definition-and-examples}). Then we denote: 
	\[Z(f) := \tx{fiber}_{s_0}\left( M \overset{f}{\to} N\right)\]
	The fiber of $f$ at the zero section of $G$. We will see in Section \ref{sec:definition-and-examples} that this is the linear stack associated to $\_F[1] \oplus^f \_E$ where $\_F$ and $\_E$ are the quasi-coherent sheaves on $X$ associated to  $F$ and $E$.

	\item We denote by $\_C_{/A}$ the slice category over $A$ and $\_C_{A/}$ the slice category under $A$. 
	\item Given a morphism $f: A \to B$, we denote by $\_C_{A//B}$ the whose objects are elements $C \in \_C$ fitting is a commutative diagram:
	\[ \begin{tikzcd}
		& C \arrow[dr] & \\
		A \arrow[ur] \arrow[rr, "f"] & & B
	\end{tikzcd}\]
	Morphisms in that category are morphisms in $\_C$, $C \to C'$ commuting with all maps from $A$ and to $B$.

\end{itemize}

\newpage

Cette thèse a pour objectif l'étude du formalisme BV du point de vue de la géométrie dérivé. Nous développons le formalisme nécessaire pour parler de \emph{quotient infinitésimaux} et de \emph{réduction symplectique infinitésimal} qui apparaissent comme version géométrique de la construction BV.

\begin{itemize}
	\item[Section \ref{sec:derived-algebraic-geometry}] Cette première section est une introduction aux outils de géométrie dérivé dont nous aurons besoin pour la suite. \\
	
	\begin{itemize}
		\item[Section \ref{sec:introduction-to-derived-algebraic-geometry} --] Cette section vise à introduire les notions et notations pour les champs supérieurs dérivés.\\
		
		\begin{itemize}
			\item[Section \ref{sec:quotients-and-higher-stacks} --] Les schémas classiques n'ont généralement pas de quotient bien définit pour une action générale. Pour y remédier, on regarde leur foncteurs de points et remplace les ensembles de points par des \emph{ensemble simpliciaux}. Cela définit la notion de champ supérieur.
			\item[Section \ref{sec:derived-higher-stacks} --]  Les schémas classique n'ont également pas toujours une ``bonne théorie d'intersection''. Une des motivation initiale la géométrie dérivé est de résoudre le problème des intersections non transverses. Nous définissons donc ici les schémas dérivés. 
			\item[Section \ref{sec:derived-higher-stacks} --] En mélangeants les deux sections précédentes, nous obtenons la notions de champ dérivé, avec lesquels nous pouvons à la fois prendre des intersections et des quotients sans problèmes. 
			\item[Section \ref{sec:relative-spectrum-functor-and-algebra-of-functions} --] Nous finissons par parler de champs dérivé affine et relativement affine, nous décrivons les algèbres différentielles gradués décrivant les algèbres de fonctions de nos champs dérivés.
		\end{itemize} \
		
		\item[Section \ref{sec:linear-and-semi-linear-stacks} --] Nous discutons ici la notion de champ linéaire sur une base, par analogie avec la notion d'espace vectoriel. \\
		
		\begin{itemize}
			\item[Section \ref{sec:some-sheaves-of-module} --] Nous rappelons les notions de faisceaux quasi-cohérent et parfait au dessus d'une base $X$. Ils vont jouer le rôle de ``module de sections'' d'un fibré vectoriel. 
			\item[Section \ref{sec:relative-cotangent-complex-on-stacks} --] Nous définissons la notion complexe cotangent et tangent comme les faisceaux de module représentant respectivement les formes différentielles (formes différentielles de Kähler mais au sens dérivé) et les dérivations. 
			\item[Section \ref{sec:definition-and-examples} --] Nous définissons la notion de champs linéaire au dessus d'une base associé à un faisceau quasi-cohérent sur cette base. Ceci nous permet de définir les champs tangents et cotangents. Ensuite nous montrons quelques propriétés de tiré en arrière analogue avec ce à quoi on s'attend avec des fibrés vectoriels. 
			\item[Section \ref{sec:relative-cotangent-complex} --] Nous montrons que le complexe tangent relatif d'un champ linéaire, $\Ttr{\Aa_X(\_F)}{X}$, est naturellement équivalent à $\pi^*\_F$. C'est l'espace vertical associé à un champ linéaire. 
			\item[Section \ref{sec:connections-on-semi-linear-stacks} --] Nous introduisons la notion de connection sur un champ linéaire (encore une fois par analogie avec une connection sur un fibré vectoriel) lorsque la base est affine. Ceci nous permet de calculer le tangent d'un champ linéaire. 
		\end{itemize}\
		
		\item[Section \ref{sec:formal-geometry} --] Nous introduisons des notions de géométrie formel, complétions formel et champ formel.\\
		
		\begin{itemize}
			\item[Section \ref{sec:formal-derived-stacks-and-formal-neighborhood} --] Dans cette section, nous définissons les notions de champs formel, champs de de Rham et complétion formel. Ceci nous donne le langage pour parler de voisinage infinitésimal que nous utiliserons à profusion dans notre étude de la géométrie infinitésimale équivariant. 
			
			\item[Section \ref{sec:formal-stack-from-formal-moduli-problems} --] Nous exposons un résultat clé qui identifie (avec de bonne hypothèse) les \emph{épaississement formels} d'une base $X$ avec les \emph{problèmes de modules formels} en dessous de $X$. Nous parlons ensuite du spectre formel, et du calcul du tangent relative de ces objets. 
		\end{itemize} \
		
		\item[Section \ref{sec:de-rham-complex-and-closed-p-forms} --] Cette section introduit les objets permettant le calcul différentiel dans le cadre dérivé .\\
		
		\begin{itemize}
			\item[Section \ref{sec:de-rham-complex} --] Nous introduisons le complexe de de Rham dans le cadre dérivé avec sa structure de complexe mixe gradué. Il est définit par une propriété universelle (comme adjoint à gauche) et nous expliquons qu'il est en fait naturellement équivalent à l'algèbre symétrique: \[ \Sym_A \Ll_A[-1]\]
			\item[Section \ref{sec:shifted-closed-p-forms} --] Nous définissons la notion de \emph{formes différentielles décalées} qui seront très importantes pour parler de géométrie symplectique décalée. 
		\end{itemize}\
		
	\end{itemize}
	\item[Section \ref{sec:new-constructions-in-derived-symplectic-geometry} --] Nous rappelons les définitions de base de géométrie symplectique dérivé et étudions en détail les constructions de telles structures via des procédures ``d'intersections Lagrangiennes''. \\
	
	\begin{itemize}
		\item[Section \ref{sec:derived-symplectic-geometry} --] Nous donnons une introduction à la géométrie symplectique dérivée. \\
		
		\begin{itemize}
			\item[Section \ref{sec:shifted-symplectic-structures}--] Nous rappelons la définition de structure symplectique $n$-décalées. Nous décrivons également la structure symplectique canonique sur le cotangent décalé. 
			\item[Section \ref{sec:lagrangian-structure} --] Nous définissons la notion de structure Lagrangienne sur un morphism $L \to X$. Nous décrivons quelque exemples classique et introduisons la notion de correspondance Lagrangienne qui va s'avérer crucial dans le reste de cette thèse. 
			\item[Section \ref{sec:lagrangian-fibrations} --]  De façon similaire à la section précédente, nous introduisons la notion de fibration Lagrangienne et donnons quelques exemples utiles de ces structures.
		\end{itemize}\
		
		\item[Section \ref{sec:new-constructions-from-old-ones} --] Nous montrons dans cette section différentes constructions par ``intersection Lagrangienne''. \\
		
		\begin{itemize}
			\item[Section \ref{sec:derived-intersection-of-lagrangians-morphisms} --] Nous commençons par rappeler le résultat disant qu'une intersection Lagrangienne dans un champ symplectique $n$-décalé donne un champ symplectic $(n-1)$-décalé. En particulier, le lieu critique dérivé est toujours symplectique $(-1)$-décalé.
			
			\item[Section \ref{sec:derived-intersection-and-lagrangian-fibrations} --] Nous étendons ce résultat pour des intersections Lagrangiennes au dessus de fibrations Lagrangiennes. En particulier, le morphism naturel: 
			\[ \RCrit(f) \to X\]
			est une fibration Lagrangienne. 
			\item[Section \ref{sec:the-higher-categories-of-lagrangians} --] Ces procédures d'intersections Lagrangiennes se généralises et induise une structure de catégorie (et même catégorie supérieure) sur l'espace des Lagrangiens. En particulier, ce formalisme nous permet de démontrer une nouvelles construction de \emph{correspondance Lagrangienne} par intersection Lagrangienne. 
			
		\end{itemize}\
		
		\item[Section \ref{sec:example-of-constructions-of-lagrangian-fibrations} --] Nous discutons quelques exemples des constructions précédentes.\\
		
		\begin{itemize}
			\item[Section \ref{sec:derived-critical-locus} --] Nous rappelons en détail la géométrie symplectique des lieux critiques dérivés. 
			\item[Section \ref{sec:non-degenerate-functionals} --] Nous donnons une description précise de la structure de fibration Lagrangienne sur un lieu critique dérivé dans le cas ou la fonctionnelle a un lieu critique ``non-dégénéré''. 
			\item[Section \ref{sec:g-equivariant-twisted-cotangent-bundles} --] De façons similaire  au lieu critique dérivé, nous décrivons la géométrie symplectique obtenue par intersection Lagrangienne sur les cotangents tordus.
		\end{itemize}
	\end{itemize}\
	
	\item[Section \ref{sec:on-the-derived-geometry-of-lie-algebroid} --] Nous étudions la géométrie des algébroïdes de Lie et leur quotient infinitésimaux. \\
	
	\begin{itemize}
		\item[Section \ref{sec:lie-and-linfty-algebroids} --] Nous introduisons la notions d'algébroïde de Lie. \\
		
		\begin{itemize}
			\item[Section \ref{sec:definitions-and-basic-properties-of-lie-algebroids} --] Nous donnons les définitions et exemples de base d'algebroid de Lie. 
			\item[Section \ref{sec:homotopy-theory-for-lie-algebroid} --] Nous rappelons les éléments de bases de théorie de l'homotopie pour les algébroïdes de Lie. 
		\end{itemize}\
		\item[Section \ref{sec:quotient-of-lie-algebroid-and-geometric-properties} --] Nous introduisons et étudions la notions de quotient infinitésimal d'un algébroïde de Lie. \\
		\begin{itemize}
			\item[Section \ref{sec:quotient-stack-of-a-lie-algebroid} --] Nous définissons un champ dérivé associé à un algébroïde de Lie que nous appelons ``quotient infinitésimal''. Nous montrons quelques propriété de cette construction et calculons son tangent relative. \item[Section \ref{sec:pullback-and-base-change-of-lie-algebroids} --] Cette section veut définir les notions de tiré en arrière et changements de bases d'algébroïdes de Lie, et montre que ces constructions ont un analogue sur leurs algèbres de \CE non-dérivés. 
		\end{itemize}\
		
		\item[Section \ref{sec:infinitesimal-action-of-lie-algebroids} --] Action d'un algébroïde de Lie. \\
		\begin{itemize}
			\item[Section \ref{sec:representation-and-action-of-lie-algebroids} --] Nous définissons la notion de représentation à homotopie près et montrons que les complexes tangent et cotangent de $\gmc{\ce}(\_L)$ définissent des représentations à homotopie près que nous pouvons voir comme des représentations adjointes et coadjointes. 
			\item[Section \ref{sec:action-of-lie-algebroids} --] Nous définissons la notion général d'action d'un algébroïdes et montrons que les représentations à homotopie près définissent de telles actions. 
		\end{itemize}\
		
		\item[Section \ref{sec:infty-morphism-and-homotopy-transfer-theorem} --] Nous donnons une preuve du théorème de transfert homotopique dans le cadre algébrique. \\
		
	\end{itemize}\ 
	
	\item[Section \ref{sec:equivariant-symplectic-geometry} --] Nous introduisons les notions de géométries équivariante, dans le cadre infinitésimal et global, puis nous étudions leur géométrie symplectique. \\
	
	\begin{itemize}
		\item[Section \ref{sec:g-equivariant-geometry} --] Nous commençons par l'étude de la géométrie équivariante pour l'action d'un groupoïde. \\
		\begin{itemize}
			\item[Section \ref{sec:g-equivariant-maps-and-quotient} --] Nous définissons les notions de groupoïdes et application équivariantes. Nous montrons également quelques propriétés importantes de tirés en arrière des quotients par des groupoïdes, et par l'action de groupoïdes. 
			\item[Section \ref{sec:tangent-and-cotangent-of-quotient-stack} --] Nous décrivons les tangents et cotangents des champs quotients.
			\item[Section \ref{sec:some-symplectic-structure-on-quotient-by-a-group-action} --] Nous rappelons la géométrie symplectique sur les champs quotients (et leurs cotangents).  
		\end{itemize} \
		
		\item[Section \ref{sec:l-equivariant-symplectic-geometry} --] Nous passons à la version infinitésimal de la géométrie équivariante. \\
		
		\begin{itemize}
			\item[Section \ref{sec:derivation-and-integration-of-lie-algebroids} --] Nous montrons que les algébroïdes de Lie sont les versions infinitésimales des groupoïdes. En particulier nous montrons comment intégrer les algébroïdes de Lie et dériver les groupoïdes ainsi que leur actions. 
			
			\item[Section \ref{sec:l-equivariant-maps-and-quotient} --] Nous imitons les définitions et propriétés décrites en Section \ref{sec:g-equivariant-maps-and-quotient} sur la géométrie équivariante pour les groupoïdes. 
			
			\item[Section \ref{sec:tangent-and-cotangent-of-infinitesimal-quotient-stack} --] De même, nous décrivons autant que possible les tangents et cotangent des quotients infinitésimaux.
		\end{itemize}\
		
		\item[Section \ref{sec:shifted-moment-maps-and-derived-symplectic-reduction} --] Nous étudions les notions d'applications moment et de réduction symplectique dans le cadre dérivé. \\
		
		\begin{itemize}
			\item[Section \ref{sec:for-groups} --] Nous commençons par rappeler la notion d'application moment pour une action de groupe. Nous suivons principalement les résultats de \cite{AC21} dans le but de motiver les généralisations qui suivrons. Crucialement, nous expliquons que la propriété essentielle d'une application moment est lié à une correspondance Lagrangienne.
			
			\item[Section \ref{sec:for-lie-algebroid} --] Motivé par le case des groupe, nous pouvons donner une généralisation des notions d'applications moment et réduction symplectique pour des actions infinitésimales. De plus nous avons également un théorème d'intersection Lagrangienne d'application moment permettant d'obtenir de nouvelles applications moment décalées intéressantes. 
			
			\item[Section \ref{sec:for-groupoid} --] Les notions de la section précédentes se généralises verbatim au cas des groupoïdes.  
		\end{itemize}
	\end{itemize}\
	
	\item[Section \ref{sec:derived-perspective-of-the-bv-complex} --] Nous discutons les applications de nos résultats sur la géométrie équivariante au formalisme BV. \\
	
	\begin{itemize} 
		\item[Section \ref{sec:bv-as-a-generalized-derived-symplectic-reduction} --] Nous motivons une approche géométrique au formalisme BV. \\
		
		\begin{itemize}
			\item[Section \ref{sec:context-and-construction-for-infinitesimal-actions} --] Nous partons des constructions algébriques pour les expliquer et généraliser via une interprétation géométrique commune. Nous expliquons que la construction BV est moralement une version généralisé de la réduction symplectique. 
			
			\item[Section \ref{sec:generalized-bv-construction} --] Nous définissons une construction BV comme une réduction symplectique généralisé. 
			
			\item[Section \ref{sec:symmetries} --] De nombreux exemples de construction BV sont obtenus via des symétries ``off-shell''. Nous décrivons de tels symétries dans le cas global comme infinitésimal. 
		\end{itemize}\
		\item[Section \ref{sec:bv-charge} --] Le but de cette section est de comparer notre définition géométrique avec les constructions algébrique usuelles. \\
		
		\begin{itemize}
			\item[Section \ref{sec:strictification} --] En se mettant dans le contexte de \cite{PS20}, nous montrons que notre notion géométrique ressemble a la construction BV usuelle, mais reste plus générale, avec l'apparition de termes ``homotopique'' non triviaux. 
			
			\item[Section \ref{sec:hamiltonian-vector-field-and-bv-charge} --] Nous montrons que en se restreignant à un cadre non dérivé, nous pouvons trouver une charge BV. 
			\item[Section \ref{sec:deformation-retract-of-de-rham-algebras} --] C'est une section technique ou l'on montre que l'algèbre de de Rham d'un complexe de \CE a un rétracte par déformation avec l'algèbre de de Rham de la base. 
		\end{itemize} \
		
		\item[Section \ref{sec:examples-of-bv-constructions} --] Nous donnons des exemples importants de construction BV.\\
		
		\begin{itemize}
			\item[Section \ref{sec:for-the-koszul--tate-resolution} --] Nous montrons que la construction BV de \cite{FK14} est une construction BV au sens géométrique. 
			\item[Section \ref{sec:with-moment-maps} --] Nous montrons que l'action d'un groupe de symétries off-shell induit une construction BV retrouvant la discussion dans \cite{BSS21}. 
			\item[Section \ref{sec:lie-algebroid-moment-maps} --] De la même manière, chaque action infinitésimal off-shell induit une construction BV donné par le lieu critique dérivé équivariant. 
			\item[Section \ref{sec:bv-construction-for-groupoid-action} --] De la même manière, chaque action d'un groupoïde off-shell induit une construction BV donnée par le lieu critique dérivé équivariant. 
		\end{itemize}
	\end{itemize}
	
\end{itemize}    

\newpage
\section{Derived Algebraic Geometry}\label{sec:derived-algebraic-geometry}\

In this first section, we focus on setting up the geometric context in which we are going to work, including derived stacks, tangent complexes, linear stack, connections, formal geometry, de Rham algebra and shifted (closed) differential forms.  \\

We start in Section \ref{sec:introduction-to-derived-algebraic-geometry} by motivating and introducing derived schemes and derived stack, their tangent and cotangent complexes, and introduce the main basic notations and constructions we will use. 

Then Section \ref{sec:linear-and-semi-linear-stacks} discusses linear (and semi-linearly represented) stacks over a base $X$. We describe  morphisms and pullbacks of such objects and explain how to compute their cotangent and relative cotangent complexes by using connections.    

We then turn, in Section \ref{sec:formal-geometry}, to the study of derived formal geometry. In particular, we will be interested in the derived formal stacks arising from \emph{formal moduli problems}, as they will be the main object of interest in the study of Lie algebroids and their infinitesimal quotients in Section \ref{sec:on-the-derived-geometry-of-lie-algebroid}.

We finish in Section \ref{sec:de-rham-complex-and-closed-p-forms} by recalling the construction of the de Rham algebra and the notion of $n$-shifted (closed) $p$-form over a derived stack, providing the prerequisites to speak about derived symplectic geometry in Section \ref{sec:new-constructions-in-derived-symplectic-geometry}.

\subsection{Introduction to Derived Algebraic Geometry} \label{sec:introduction-to-derived-algebraic-geometry} \

\medskip

The idea behind derived geometry is to mix homotopy theory and geometry to get ``better behaved'' geometric objects. There are many reasons we can be interested in such objects. We are going to be mainly interested in the following two operations that are inherently well behaved in derived geometry: 
\begin{itemize}
	\item We can define \defi{homotopy quotients} (which from now on we will simply call \emph{quotients}). In higher geometry, there are well defined colimits\footnote{Recall that by colimit (respectively limit) we always means \emph{homotopy} colimit (respectively \emph{homotopy} limit).} which we can use to define quotients and study $G$-equivariant geometry.   
	\item We have good intersection theory\footnote{By ``good intersection theory'' we mean that (homotopy) intersections exist and have good behavior with respect to the intersection number and recover Serre's formula (see the introduction of \cite{Lu04}).} and even all limits. We will see in particular in Section \ref{sec:new-constructions-from-old-ones} how well behaved derived intersections are with respect to derived symplectic geometry. 
\end{itemize}

Further the tools of derived geometry are also well suited to study \emph{formal} geometry and infinitesimal neighborhoods, which we will also extensively use when working with \emph{infinitesimal quotients} (see Section \ref{sec:on-the-derived-geometry-of-lie-algebroid}).

To mention a few foundational work on derived geometry, we refer to \cite{TV05} and \cite{TV08} in the model categorical setting and to \cite{Lu04}, \cite{Lu07} and \cite{Lu09} in the $\infty$-categorical setting.

\subsubsection{Quotients and higher stacks}\
\label{sec:quotients-and-higher-stacks}

\medskip

In order to encode quotients ``up to homotopy'' of a geometric space $X$, we add to it a \emph{space} of $1$-simplices such that there exists a simplex from $x$ to $y$ if and only if $x$ and $y$ are equivalent under the relation along which we want to quotient. This idea can be encoded in the notion of groupoid:
\[ \begin{tikzcd}
	\_G \arrow[r, shift left, "s"] \arrow[r, shift right, "t"'] & X 
\end{tikzcd}\] 

where we view $\gamma \in \_G$ as an equivalence from $x:= s(\gamma)$ to $y:=t(\gamma)$, and the strict quotient by this equivalence relation corresponds to the strict equalizer of this groupoid. \\

Now such strict colimit may not always exist in the category of ``classical'' geometric object we are working with (for example if $X$ and $\_G$ are smooth manifolds or even schemes). To fix that issue, we need to consider a larger category of ``higher''\footnote{Higher geometry should not be confused with derived geometry, which is an other kind of generalization as we will see.} geometric objects.\\ 

First we consider sheaves on the site of affine schemes, $\op{\aff}\to \tx{Set}$, motivated by the notion of functor of points of a scheme (see Example \ref{ex:functor of point derived scheme}). In this category of sheaves, colimits do exist but are still pathological\footnote{In the sens that the ``sets of points'' are not sets of points of a geometric object.}, and therefore we will consider such sheaves valued not in the category $\tx{Set}$ but in the $\infty$-category of $\infty$-groupoids, $\igpd$ (see \cite[Section 1.2.5]{Lu09HTT}). \\

The idea behind this comes from the fact that the set of points of a singular quotient can be represented, \emph{up to homotopy}, by a ``simplicial set of points'' whose $\pi_0$ recovers the set of points of the quotient, and such that each set of simplices is ``nice''. Doing this remember not only the quotient, but also the way in which two points are equivalent, the way in which two such equivalences are equivalent, and so on... This turns out to have many technical advantages, at the price of having to work up to homotopy. 
Since we want to work up to homotopy, the natural objects we are interested in are ``$\infty$-sheaves'', a.k.a. higher stacks:

\begin{Def}
	\label{def:higher stacks}
	A \defi{higher stack} over the site of classical affine spaces, denoted by $\aff := \op{\left(\tx{CAlg}\right)}$ (the opposite of the category of commutative algebras concentrated in degree zero), is a functor: 
	\[ X : \aff^{\tx{op}} \to \igpd \]
	
	satisfying a descent condition (see \cite[Definition 3.4.9]{TV05}) making it an $\infty$-sheaf valued in the $\infty$-category of $\infty$-groupoids, $\igpd$. 
	
	Note that this requires a choice of Grothendieck topology\footnote{The notion of Grothendieck topology is the notion that generalizes the idea of open cover on smooth manifolds. We refer to \cite{MM94} and \cite{AGV72} for the classical (underived) notion of site and Grothendieck topology.} on $\aff$. By default, we will always consider the \emph{étale} topology (see \cite{Mi80}).
\end{Def}

\begin{Ex} \ 	
	\begin{itemize}
		\item \ttl{Functor of points:} The functor of points of a scheme induces a fully-faithful functor $\sch \to \st$ that sends a scheme $X$ to the higher stack:
		\[\begin{tikzcd}[row sep= 1mm]
			X: \op{\aff} \arrow[r] & \tx{Set} \subset \igpd \\
			\qquad A \arrow[r, mapsto] & \Homsub{\sch}(\Spec(A), X)
		\end{tikzcd}\]
		where the set of points is viewed as a simplicial set via the fully-faithful functor $\tx{Set} \hookrightarrow \igpd$ sending a set $X$ to the constant simplicial set with $X$ in each simplices\footnote{This fully-faithful functor is the right adjoint to the connected component functor: \[\pi_0: \igpd \to \tx{Set}\]}. 
		
		\item \ttl{Quotient stacks:} If $G$ is a group object in higher stacks that acts on a higher stack $X$ via an action $\rho$, we can define the higher quotient stack as the colimit:  
		\[\QS{X}{G} := \colim\left( N\left(\begin{tikzcd}
			X \times G \arrow[r, shift left, "\tx{pr_X}"] \arrow[r, shift right, "\rho"'] & X 
		\end{tikzcd}\right)\right)\]
		Where $N$ is the nerve of the groupoid determined by the action.
		This example will be detailed and expended upon in Example \ref{ex:derived stacks}.
		
	\end{itemize}
\end{Ex}

\subsubsection{Intersection and derived schemes}\label{sec:intersection-and-derived-schemes}\

\medskip

For the problem of intersections, the problem lies in the algebra of functions that does not have enough homotopy theory. For a more detail explanation of the relationship between intersection theory and derived geometry, we refer to the introduction in \cite{Lu04} and to \cite[Lecture 1.1]{Ca14}. In short, we have to replace the affine space of commutative algebras by the category of simplicial algebras, which comes equipped with a non trivial homotopy theory we can use to do intersections up to homotopy.\\

As we are always going to work with $k$-algebras, with $k$ a field of characteristic $0$, we can use the Dold--Kan correspondence to work with connective commutative differential graded $k$-algebras, $\daff := \left( \cdgacon \right)^{\tx{op}}$ instead of simplicial algebras.   

\newpage

\begin{Def}
	\label{def:derived scheme}
	
	A \defi{derived scheme} is a derived locally ringed space\footnote{Derived locally ringed spaces are topological spaces $X$ with an $\infty$-sheaf of connective commutative differential graded algebras, see \cite[Section 2.2]{To14} for more details.} $(X, \_O_X)$ such that: 
	
	\begin{itemize}
		\item $\_O_X$ is an $\infty$-sheaf of \emph{connective} cdgas over $X$, that is $H^n(\_O_X) =0$ for $n \geq 0$. 
		\item The underlying locally ringed space, $(X, H^0(\_O_X))$, is a scheme.
		\item For all $n \geq 0$ the sheaf $H^{-n}(\_O_X)$ is a quasi-coherent $H^0(\_O_X)$-modules.
	\end{itemize}
	
	We will denote the $\infty$-category of derived schemes by $\dsch$. Of course any ordinary scheme is also a derived scheme, and this defines a fully-faithful $\infty$-functor, $\iota : \sch \to \dsch$, that has a left adjoint: 
	\[ \begin{tikzcd}
		t_0 : \dsch \arrow[r, shift left] & \arrow[l, shift left] \sch : \iota
	\end{tikzcd} \]
	defined by the \defi{underlying scheme}, $t_0 (X, \_O_X) = (X, H^0(\_O_X))$. 
\end{Def}

As for ordinary scheme, there is a spectrum functor, $\Spec$, which is the right adjoint to the global algebra of functions functor: 
\[ \begin{tikzcd}
	\Gamma : \dsch \arrow[r, shift left] & \arrow[l, shift left] \left(\cdgacon\right)^{\tx{op}} =: \daff : \Spec
\end{tikzcd} \]

\begin{RQ}
	In \cite{To14}, it is shown that $\daff$ is equivalent to the category of derived schemes whose truncation $t_0(X)$ are affine. For a derived scheme $(X, \_O_X)$, its underlying scheme, $t_0(X)$, is locally affine. Therefore, this shows that $(X, \_O_X)$ is locally equivalent to derived affine scheme. 
\end{RQ}

\begin{Def}
	\label{def:derived scheme almost finitely presented}
	
	A derived scheme $X$ is called \defi{locally almost finitely presented} if $X$ is locally equivalent to the spectrum of almost finitely presented cdgas\footnote{$A$ is an almost finitely presented cdga if $H^0(A)$ is a finitely generated algebra and each $H^i(A)$ is a finitely presented $H^0(A)$-module.}. In particular, $t_0(X)$ is a finitely generated scheme. In this situation, $H^{-n}(\_O_X)$ are \emph{coherent} $H^0(\_O_X)$-modules. 
\end{Def}

\begin{Ex}	
	The pullback of affine derived schemes, 
	\[ \begin{tikzcd}
		\Spec(A) \times_{\Spec(C)} \Spec(B) \arrow[r] \arrow[d] & \Spec(B) \arrow[d] \\
		\Spec(A) \arrow[r] & \Spec(C)
	\end{tikzcd}\]
	
	can be computed by the spectrum of their \emph{derived} tensor product: 
	\[ \Spec(A) \times_{\Spec(C)} \Spec(B) \simeq \Spec( A \otimes_C^\Ll B)\]

	The main gain from using this derived setting is that the tensor product is now \emph{derived}. As explained in \cite{Lu04}, this enables to recover Serre's formula to compute the correct \emph{intersection number}. 
	
	From now on, we will omit the notation $\otimes^\Ll$ and denote the derived tensor product by $\otimes$. Recall that by convention all our functors are derived, and our constructions and commutative diagrams are ``up to homotopy'' unless specified otherwise.	
\end{Ex}

\begin{Ex}
	\label{ex:derived critical locus of smooth derived scheme}
	
	The derived critical locus of a functional $f: X \to \Aa^1$, with $X$ a smooth scheme, is defined as the following pullback: 
	\[ \begin{tikzcd}
		\RCrit(f) \arrow[r] \arrow[d] & X \arrow[d, "df"]\\
		X \arrow[r, "s_0"] & T^*X
	\end{tikzcd} \]
	
	As a derived scheme, $\RCrit(f)$ can be described as the locally ringed space: \[(S, (i^{-1}(\Sym_{\_O_X} \Tt_X [1]), \iota_{df}))\] where $i :S \to X$ is the inclusion of the topological strict critical locus\footnote{The strict critical locus is the scheme given by the \emph{strict} pullback instead.}, $\Tt_X$ is the tangent complex\footnote{Since we assumed $X$ smooth, $\Tt_X$ is in fact concentrated in degree $0$ and has no differential.} (see Section \ref{sec:relative-cotangent-complex-on-stacks}) and the sheaf of cdgas is the restriction\footnote{Given by the restriction functor $i^{-1}: \cdgacon_{\QC(X)} \to \cdgacon_{\QC(Y)}$.} to $S$ of the sheaf of connective\footnote{This symmetric algebra is connective as soon as $\Tt_X[1]$ is concentrated in non-positive degree which is the case when $X$ is smooth.} cdgas, $\Sym_{\_O_X} \Tt_X [1]$, together with the differential given by the contraction $\iota_{df}$.
\end{Ex}

\begin{RQ}
	The derived critical locus has many good properties. In many regard, it behaves almost like the shifted cotangent $T^*[-1]X$ (Definition \ref{def:tangent and cotangent stack}) and is in fact equivalent to the $(-1)$-shifted cotangent if $df=0$. In particular, we will see in Section \ref{sec:new-constructions-from-old-ones} that much of the symplectic geometry of $T^*[-1]X$ can be extended to $\RCrit(f)$. 
	
	We will study in more details the (shifted) symplectic geometry of derived intersections in Section \ref{sec:derived-critical-locus}.
\end{RQ}

\subsubsection{Derived higher stacks} \label{sec:derived-higher-stacks}\

\medskip

Derived higher stacks appear when we want to take both homotopy quotients and derived intersections. The successive generalizations of the notion of geometric space in order to arrive to derived higher stacks can be summarize by the following diagram:

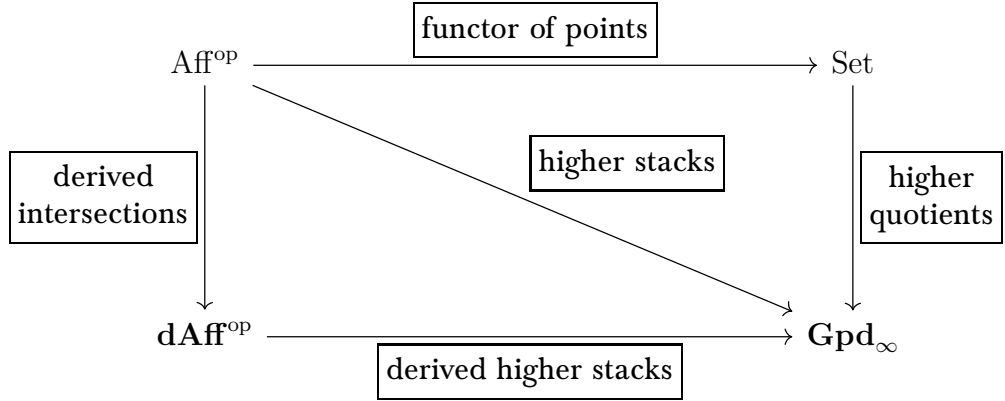
\[
\begin{tikzcd}[column sep = 7cm, row sep = 3cm]
	\op{\aff} \arrow[r, "\mlnode{functor of points}"] \arrow[dr, "\mlnode{higher stacks}"]  \arrow[d,"\mlnode{derived\\ intersections}"']  & \tx{Set} \arrow[d,"\mlnode{higher \\ quotients}"] \\
	\op{\daff} \arrow[r, "\mlnode{derived higher stacks}"'] &  \igpd
\end{tikzcd}\]

One way to view this is that derived higher stacks are to derived schemes what higher stacks are to schemes. In particular, the functor of points of a derived scheme (valued in $\igpd$) is a derived stack, and induces a fully-faithful $\infty$-functor $\dsch \to \dst$. In particular, similarly to higher stacks, they are better behaved than derived schemes when it comes to handling homotopy quotients and colimits. 

\begin{Def}
	\label{def:higher serived stacks}
	
	A \defi{derived higher stack} over an $\infty$-site (see \cite[Section 3.1]{TV05}) $\daff$ of \emph{derived} affine objects, is a functor:
	\[ F : \daff^{\tx{op}} \to \igpd \]
	
	satisfying a descent condition (see \cite[Definition 3.4.9]{TV05}) making it an $\infty$-sheaf valued in $\infty$-groupoids. All the stacks we will consider will be viewed as derived higher stacks, therefore we will omit from now on the terms ``higher'' and ``derived'' and call them either \emph{derived stacks} or just \emph{stacks}. 
	
	We will work with connective cdgas\footnote{In fact, the natural generalization from $\aff$ would be to consider simplicial algebras. However, this we are working over a fields, of characteristic zero, we can use the Dold--Kan correspondence to work with \emph{connective} cdgas instead.} as affine $\daff := \op{\left(\cdgacon\right)}$ together with the étale topology. This gives us the notion of $D^{-}$-stack in the terminology of \cite{TV08}. With these choices, we denote the $\infty$-category of all derived stacks by $\dst$. Given an other choice of site and Grothendieck topology $\tau$, we will denote derived stacks for that topology by $\dst^\tau$. 
\end{Def}

\begin{notation}\ \label{not:prestacks}
	Functors $F: \op{\daff} \to \igpd$ that do not satisfy descent are called \defi{pre-stacks} and the category of such functor is denoted by $\dpst$. We will also use variations of the other notations for stacks by writing for example $\dfpst$ instead of $\dfst$ to denote the associated category of formal derived pre-stacks (which we will define in Section \ref{sec:formal-derived-stacks-and-formal-neighborhood}).  
\end{notation}

We will also need to restrict ourselves to almost finitely presented stacks. 

\begin{Def}
	\label{def:derived stacks almost finitely presented}
	An \defi{almost finitely presented derived stack} is a derived stack over the $\infty$-site of almost finitely presented algebras, $\dafffp$, together with the étale topology. In other words, they are functors:
	
	\[ F : \left(\dafffp\right)^{\tx{op}} \to \igpd \]
	
	satisfying a descent condition (see \cite[Definition 3.4.9]{TV05}) making it an $\infty$-sheaf valued in $\infty$-groupoids. We denote the category of almost finitely presented derived stacks by $\dstfp$.
\end{Def}

\begin{Ex}
	\label{ex:functor of point derived scheme}
	
	Since the category $\dsch$ is an $\infty$-category, it has a mapping space and we can define the functor of point of a given derived scheme $X$ by: 
	\[\begin{tikzcd}[row sep= 1mm]
		X: \op{\daff} \arrow[r] & \igpd \\
		\quad \qquad \Spec(A) \arrow[r, mapsto] & \Mapsub{\dsch}(\Spec(A), X)
	\end{tikzcd}\]
	
	This defines a fully-faithful functor $\dsch \to \dst$. If $\Spec(A) \in \daff$, then its associated stack, which we also denote by $\Spec(A)$, is the stack given by: 
	\[\begin{tikzcd}[row sep= 1mm]
		\Spec(A): \cdgacon \arrow[r] & \igpd \\
		\quad \qquad B \arrow[r, mapsto] & \Mapsub{\cdgacon}(A, B) 
	\end{tikzcd}\]

	This is called the derived stack \emph{corepresented} by $A$.  
\end{Ex}

\begin{RQ}
	\label{rq:artin and geometric derived stacks}
	
	The objects of $\dst$ can be very wild and do not always behave as nicely as we would like for geometric purposes. In order to have a better control on derived stacks, we will often restrict ourselves to certain sub-classes of derived stacks. \\
	
	One type of restriction is to consider derived stacks which can be obtained inductively, starting with affine derived stacks, and by gluing our stacks together successively along a certain type of morphisms. This is the idea behind the definition of a \emph{geometric stack} (see \cite[Chapter 1.3]{TV08} and in particular \cite[Proposition 1.3.4.2]{TV08}). 
	
	Different choices of the ``type of morphism'' along which we are gluing give different notions of geometric stacks. We will only consider the class of smooth morphisms. The category of geometric stacks for smooth morphisms is called the category of \defi{Artin stacks} (see \cite[Definition 2.1.1.3]{TV08}). 
\end{RQ}

Unfortunately, these type of stacks will not be general quite enough for us. When we will study infinitesimal quotients of Lie algebroids in Section \ref{sec:on-the-derived-geometry-of-lie-algebroid}, we will not obtain Artin stacks, but rather \emph{derived formal stacks} (see Section \ref{sec:formal-geometry}). \\

\begin{Ex} \
	\label{ex:derived stacks}
	
	\begin{itemize}
		
		\item We will denote by $\Aa^1$ the affine line as the derived stack corepresented by $k[x]$, $\Aa^1 := \Spec(k[x])$.
		
		\item Any pullback or pushout of derived stacks is also a derived stack. 
		
		\item In Section \ref{sec:definition-and-examples} we will define a class of stacks called \emph{linear stacks} corresponding to a generalization of the notion of vector bundle over $X$. In the situation where $X$ admits a cotangent complex (see Section \ref{sec:relative-cotangent-complex-on-stacks}), we will be able to define the $n$-shifted tangent and cotangent stacks, $T[n]X$ and $T^*[n]X$ (Definition \ref{def:tangent and cotangent stack}).
		
		\item In Section \ref{sec:formal-geometry} we will see a construction of ``formal derived stacks'' (Definition \ref{def:formal derived stack}) obtained from a formal moduli problem (Definition \ref{def:FMP} and Proposition \ref{th:fmp are formal thickenings}). This will enable us to define the derived \emph{infinitesimal} quotient stack of a Lie algebroid as the formal stack associated with the formal moduli problem induced by the given Lie algebroid (see Section \ref{sec:quotient-stack-of-a-lie-algebroid}).

		\item \defi{Classifying stack:} Given a group object in derived stacks, $G$, we define its \defi{classifying stack} as the colimit of the following simplicial stack\footnote{Each of the simplices are stacks given by $G^{\times n}$. The maps $G^{\times n} \to G^{\times n-1}$ are given by the projection on the $(n-1)$-first (and $n-1$-last) component and for each $i=1\cdots n-1$ sends $(g_1, \cdots, g_n)$ to $(g_1, \cdots g_{i-1}, g_i g_{i+1}, g_{i+2}, \cdots, g_n )$.}:
		\[ 
		\bf{B}G := \colim  \left(\begin{tikzcd} \star & \arrow[l, shift left = 1] \arrow[l, shift left = -1]  G &   \arrow[l, shift left = 2] \arrow[l, shift left = -2] \arrow[l, shift left = 0] G^2 \cdots \end{tikzcd} \right)
		\]

		\item \defi{Quotient Stack:} Similarly to the previous example, consider $G$ a group object in derived stacks acting on a derived stack $X$. Then we can define the \defi{quotient stack} of $X$ by $G$ as the following colimit: 
		\[ 
		\QS{X}{G} := \colim  \left(\begin{tikzcd} X & \arrow[l, shift left = 1] \arrow[l, shift left = -1] X \times G &    \arrow[l, shift left = 2] \arrow[l, shift left = -2] \arrow[l, shift left = 0] X \times G^2 \cdots \end{tikzcd} \right)
		\]
		
		\item More generally if $\_G$ is a groupoid over $X$ (Definition \ref{def:groupoid in derived stacks}): 
		\[ \begin{tikzcd}
			\_G \arrow[r, shift left, "s"] \arrow[r, shift right, "t"'] & X
		\end{tikzcd}\]
		then we can define its quotient stack as the colimit of its nerve\footnote{The three maps of the nerve $\_G \times_X \to \_G$ are given by the natural projections and the groupoid multiplication, and similarly for higher simplices.}: 
		\[ 
		\QS{X}{\_G} := \colim  \left(\begin{tikzcd} X & \arrow[l, shift left = 1, "t"] \arrow[l, shift left = -1, "s"']  \_G &   \arrow[l, shift left = 2] \arrow[l, shift left = -2] \arrow[l, shift left = 0] \_G \times_X \_G \cdots \end{tikzcd} \right)
		\]
		
		In particular if $\_G := G \times X$, $s:= \tx{pr}_X :G \times X \to X$ and $t:= \rho$ is an action of $G$, then we recover the quotient stack from the previous example.
		
		Moreover, if $X$ and $\_G$ are Artin stacks and the source and target maps are ``smooth'', then the quotient is also Artin (in fact Artin stacks are such successive quotients of \emph{smooth Segal groupoids}, see \cite[Definition 1.3.4.1.]{TV08} and \cite[Proposition 1.3.4.2.]{TV08}). We will discuss the notion of Segal groupoids in Section \ref{sec:g-equivariant-maps-and-quotient}.
	\end{itemize} 
\end{Ex}

\begin{Prop}[{\cite[Theorem 4.2.9]{TV05}}]\label{prop:girauds theorem} 
	
	The category of derived stack forms a model topos (see \cite{TV05}). In particular it satisfies Giraud's theorem for model topoi: 
	\begin{enumerate}
		\item We have disjoint homotopy coproducts\footnote{A coproduct $A\coprod B$ is a disjoint coproduct if it is equivalent to the colimit $A \coprod\limits_{\emptyset} B$ where $\emptyset$ is the initial object.}.
		\item Colimits are preserved under base change: 
		\[ \colim \left(A_i \times_B C\right) \simeq \colim \left(A_i \times_B C\right)\]
		\item Segal equivalence relations are homotopy effective. We refer to \cite[Definition 4.9.1]{TV05} for the precise meaning of that condition. In particular the natural projection: 
		\[ X \to \QS{X}{\_G}\]
		is an effective epimorphism in the sens that there is an equivalence: 
		\[  \colim\left( N\left( \begin{tikzcd}
			X \times_{\QS{X}{\_G}} X \arrow[r,shift left] \arrow[r, shift right] & X 
		\end{tikzcd} \right)\right) \overset{\sim}{\to} \QS{X}{\_G} \]
		
		and moreover the natural morphisms: \[ \underbrace{\_G \times_X \cdots \times_X \_G}_{\tx{n-times}} \to \underbrace{X \times_{\QS{X}{\_G}} \cdots \times_{\QS{X}{\_G}} X}_{\tx{(n+1)-times}}\]
		are equivalences for all $n \geq 0$. 
	\end{enumerate} 
\end{Prop}

\subsubsection{Relative spectrum functor and algebra of functions.}\label{sec:relative-spectrum-functor-and-algebra-of-functions}\

\medskip

In this section we set up notations for the spectrum functor and its relative version, and describe the sheaf of functions of a derived stacks.
In the context of derived schemes, we have a functor $\daff \to \dsch$ which is part of the adjunction: 
\[ \begin{tikzcd}
	\Gamma : \dsch \arrow[r, shift left] & \arrow[l, shift left]  \daff : \Spec
\end{tikzcd} \]

where $\Gamma$ is the global algebra of functions functor, 
$\Gamma \left(X\right) := \_O_X(X)$.
The goal is to give a similar adjunction for derived stacks and give a relative version of it. To do that, we need to discuss the notion of sheaf of functions $\_O_X$ of a stack $X$.\\ 

First, recall that through the fully-faithful functor $\dsch \to \dst$, an affine derived scheme $\Spec(B)$ is sent to the functor of points ``corepresented'' by $B$ which leads to the following definition:

\begin{Def}
	\label{def:corepresentable functor of point and spectrum}
	
	A derived stack called \defi{affine} or \defi{corepresented} by $B \in \cdgacon$ if it is of the form\footnote{Recall that the derived stack $\Spec(B)$ is the functor of points associated to the derived scheme also denoted $\Spec(B)$.}:
	\[\begin{tikzcd}[row sep= 1mm]
		\Spec(B): \op{\daff} \arrow[r] & \igpd \\
		\qquad \qquad A \arrow[r, mapsto] & \Mapsub{\cdgacon}(B, A)
	\end{tikzcd}\]
	
	The functor $\Spec$ defines a fully-faithful functor $\daff \to \dst$. 
\end{Def}

This definition can be easily be extended to unbounded algebras. 

\begin{Def}
	\label{def:pro-corepresentable derived stacks}
	
	For $B \in \cdga$, a possibly unbounded algebra. We define the \defi{pro-affine} derived stack \defi{pro-corepresented by $B$} to be: 
	\[\begin{tikzcd}[row sep= 1mm]
		\Spec(B): \op{\daff} \arrow[r] & \igpd \\
		\qquad \qquad A \arrow[r, mapsto] & \Mapsub{\cdga}(B, A)
	\end{tikzcd}\]
\end{Def}

\begin{War}\
	\label{war:non-affine prorepresentable}
	Even though we make an abuse of notation by denoting $\Spec(B)$ the pro-corepresentable derived stack associated to a possibly unbounded $\cdga$, $B$, this derived stack \emph{is not affine} in general, and is not even necessarily a derived scheme. We will keep this convenient notation but the reader should be aware that we say that $X$ is an \defi{affine derived stack} only if $X = \Spec(B)$ for $B$ a \emph{connective} cdga. 
\end{War}

We also want to consider schemes that are ``affine relative to another scheme''. For example vector bundles, or as we will see in Section \ref{sec:linear-and-semi-linear-stacks}, linear stacks, are affine (or just pro-affine) relative to their base. We want to be able to take the ``relative spectrum'' of such objects. \\

\begin{Def}
	\label{def:relative spectrum functor derived schemes}
	
	There is a functor\footnote{The category $\cdgacon_{\QC(X)}$ is the category of non-positively graded commutative algebras in $\QC(X)$, the category of quasi-coherent sheaves on $X$ (Definition \ref{def:quasi-coherent sheaves}).}:
	\[\dsch_{/X} \to \op{\left(\cdgacon_{\QC(X)}\right)}\] 
	that sends $p:Y \to X$ to $p_* \_O_Y$.  
	
	This functor has a right adjoint called the \defi{relative spectrum}: \[\Spec_X: \op{\left(\cdgacon_{\QC(X)}\right)} \to \dsch_{/X} \]  
	
	In particular, for any $\_A \in  \cdgacon_{\QC(X)} $ and $p: Y \to X$, we have: 
	\[ \Homsub{\cdgacon_{\QC(X)}} \left( \_A, p_* \_O_Y \right) \simeq \Homsub{\dsch_{/X}}\left( Y, \Spec_X(\_A) \right)  \]
	
	Note that for $X = \star$, this is the usual adjunction $\Gamma \dashv \Spec$ since for $p: Y \to \star$, we have $p_* \_O_X \simeq \_O_X(X)$.  
\end{Def}

The idea is that the objects of $ \cdgacon_{\QC(X)}$ are the ``affine relative to $X$'' and the relative spectrum produces a scheme (over $X$) out of it.

In order to do the same thing for derived stacks, we need to make sense of the sheaf of functions of a derived stack.   

\begin{Def}\
	\label{def:sheaf of function on derived stacks}
	Given a derived stacks $X$, the \defi{sheaf of algebras of functions} of $X$ is defined as the following stack (valued in $\cdga$) on the site\footnote{The over category $\daff_{/X}$ admits a natural étale topology coming from the étale topology on $\daff$.} $\daff_{/X}$ by: 
	\[ \_O_X(\Spec(A)\to X) = A \]
	
	This can be left Kan extended to a functor from $\dst_{/X}$, the category of all derived stack over $X$. In particular for $\tx{id} : Y=X \to X$ we get a definition of the global functions functor: 
	\[ \_O_X (X) = \lim_{\Spec(A) \to X} A \]
	
	Moreover, if $f : X \to Z$ is a map of stacks, then there is a morphism of $\infty$-sheaves on $\dst_{/Z}$, $\_O_Z \to f_* \_O_X$, extending the functor on $\daff_{/Z}$ defined by: \[f_* \_O_X (\Spec(A) \to Z) := \_O_X\left(\Spec(A)\times_Z X \to X\right)\] 
	
	This map:
	\[\_O_Z(\Spec(A)\to Z) := A \mapsto \lim_{\Spec(B) \to \Spec(A)\times_Z X} B  \] 
	
	is defined by the natural morphism obtain from the composition of the maps $\Spec(B) \to \Spec(A)\times_Z X \to \Spec(A)$, inducing maps $A \to B$. Therefore we get a unique morphism from $A$ to the limit. 
\end{Def}

\begin{Lem}\ 
	\label{lem:equivalent description of global functions}
	Given a derived stack $X$, the following descriptions of the global algebra of functions are equivalent: 
	\begin{enumerate}
		\item \[ \_O_X(X) = \lim_{\Spec(A) \to X} A\]
		\item As the mapping ``commutative monoid in spectra\footnote{The mapping spectra is obtain by considering the spectrum object given by taking the free stabilization of the mapping space. This is also a commutative monoid in spectra because of the commutative monoid structure on $\Aa^1$ (not all mapping spectra of derived stacks can be naturally viewed as commutative monoids).}'' into $\Aa^1$:
		\[ \_O_X(X)= \Map_{\dst}^{\bf{Com}(\bf{Sp})}\left(X, \Aa^1\right)\] 
		
		This commutative monoid in spectra can be identify with a possibly unbounded cdga using the \emph{stable Dold--Kan correspondence}\footnote{An equivalence between $k$-algebra in spectra ($k$ of characteristic zero) and unbounded cdga, $\bf{Com}(\bf{Sp}) \simeq \cdga$ extending the usual Dold--Kan correspondence.}. 
	\end{enumerate}
\end{Lem}
\begin{proof}
	We have that: \[X = \colimsub{\Spec(A) \to X} \Spec(A)   \]
	therefore, when using the structure of $\Aa^1$ to enrich this mapping stack into simplicial abelian groups, we get:
	\[ \begin{split}
		\Map_{\dst}^{\bf{Com}(\bf{Sp})}\left(X, \Aa^1\right) & \simeq \Map_{\dst}^{\bf{Com}(\bf{Sp})}\left(\colimsub{\Spec(A) \to X} \Spec(A), \Aa^1\right) \\ 
		& \simeq \lim_{\Spec(A) \to X}\Map_{\dst}^{\bf{Com}(\bf{Sp})}\left( \Spec(A), \Aa^1\right) \\ 
		& \simeq \lim_{\Spec(A) \to X}\iHomsub{\cdga}\left( k[x], A \right) \\
		& \simeq  \lim_{\Spec(A) \to X} A
	\end{split}\]
\end{proof}

\begin{Prop}\label{prop:spec is adjoint to global section derived stacks}
	We have again an adjunction: 
	\[ \begin{tikzcd}
		\Gamma : \dst \arrow[r, shift left] & \arrow[l, shift left]  \daff : \Spec
	\end{tikzcd} \]
\end{Prop}
\begin{proof}
	\[\begin{split}
		\Mapsub{\left(\cdgacon\right)^{\tx{op}}} \left( \Gamma(X), B \right) & \simeq \Mapsub{\cdgacon} \left(B, \lim_{\Spec(A)\to X} A \right) \\
		& \simeq \lim_{\Spec(A) \to X} \Map(B, A) \\
		& \simeq \Mapsub{\dst}\left(\colimsub{\Spec(A)\to X} \Spec(A), \Spec(B)\right)\\
		& \simeq \Mapsub{\dst}(X, \Spec(B))
	\end{split}\]
\end{proof}

\begin{RQ}\label{rq:algebra of function of pro-corepresentable stacks}
	Let $B \in \cdga$ and consider $\Spec(B)$ the derived stack pro-corepresented by $B$ (Definition \ref{def:pro-corepresentable derived stacks}). Then we have: 
	\[  \Mapsub{\cdga} \left( k[x], B \right) \simeq B^{\tx{con}} \]
	where $B^{\tx{con}}$ is the connective truncation of $B$. In particular, the mapping space, $\Map(X, \Aa^1)$, forgets the coconnective part of $B$. 
\end{RQ}

We will now describe a relative version of the spectrum functor for stacks.  By analogy with the derived scheme situation (Definition \ref{def:relative spectrum functor derived schemes}) we can give the following definition:

\vspace{1cm}

\begin{Def} 
	\label{def:relative spectrum over X}
	
	Give $X$ a derived stack, a derived stack is called \defi{corepresentable over $X$} by $\_B \in \cdgacon_{\QC(X)}$ if it is of the form: 
	
	\[ \begin{tikzcd}[row sep= 1mm]
		\Spec_X(\_B) : \op{(\daff_{/X})} \arrow[r] & \igpd \\
		\qquad (\Spec(A) \to X) \arrow[r, mapsto] &  \Mapsub{\cdgacon_{\QC(X)}}\left( \_B, f_* \_O_{\Spec(A)}\right)
	\end{tikzcd}\]
	
	$\Spec_X$ defines a fully-faithful functor $\op{\left(\cdgacon_{\QC(X)}\right)} \to \dst^{\daff_{/X}}$ (where the right hand side denotes stacks on $\daff_{/X}$ with respect to the étale topology).
	
	We can also define \defi{pro-corepresentable} stacks over $X$ by taking $\_B$ in $\cdga_{\QC(X)}$ instead.    
\end{Def}

\begin{RQ}\label{rq:extending stack on affines over X to stacks over X}
	The relative spectrum gives us a stack on $\daff_{/X}$, but what we would like is a stack over $X$. The way to obtain such a stack is to take the left Kan extension along the forgetful functor $i:\daff_{/X} \to \daff$. The restriction gives a functor $\dst \to \dst^{\daff_{/X}}$ whose left adjoint is the left Kan extension functor $\dst^{\daff_{/X}} \to \dst$. We can show that this functor factors through stacks over $X$: 
	\[ \dst^{\daff_{/X}} \to \dst_{/X} \to \dst\]
	
	Note that $\dst^{\daff_{/X}}$ has a terminal object $\star$ sending everything to the point $\star$. Therefore any object in the essential image of the left Kan extension has a canonical morphism to $\tx{Lan}_i (\star)$. We only have to show that  $\tx{Lan}_i (\star) \simeq X$: 
	\[ \begin{split}
		\tx{Lan}_i (\star) \simeq & \int^{\Spec(A) \to X} \Mapsub{\op{\daff}_{/X}} \left(\Spec(A), -\right) \times \star(\Spec(A) \to X) \\
		\simeq & \int^{\Spec(A) \to X} \Mapsub{\op{\daff}_{/X}} \left(\Spec(A), -\right)  \\
		\simeq & \int^{\Spec(A) \to X} \Spec(A) \\
		\simeq & \colimsub{\Spec(A)\to X} \Spec(A) \simeq X
	\end{split}\]
	
	The functor $\dst^{\daff_{/X}} \to \dst_{/X}$ is an equivalence with inverse: 
	\[ \begin{tikzcd}[row sep= 1mm]
		\dst_{/X} \arrow[r] & \dst^{\daff_{/X}} \\
		(Y \to X) \arrow[r, mapsto] &  \left( (\Spec(A)\to X) \mapsto \Mapsub{\dst_{/X}} \left(\Spec(A), Y \right) \right)
	\end{tikzcd}\]
\end{RQ}

\newpage

\begin{RQ}
	\label{eq:semi-free over an affine base}
	If $X = \Spec(A)$ is affine, then the derived stack given by the relative spectrum of a semi-free algebra $\Spec_X(\Sym_A \_F^\vee)$ (where $\_F \in \Mod_A$) is pro-corepresented by $\Sym_A \_F^\vee$. 
	Indeed, we have for all $B \in \daff$: 
	\[ 	\Homsub{\dst_{/X}}(\Spec(B), \Spec(\Sym_A \_F^\vee))\\
	\simeq  \Homsub{\cdga_{A/}}(\Sym_A \_F^\vee, B) 
	\]
	
	Where the equivalence comes from the fact that $\Spec$ is a fully-faithful functor. Moreover, since $X$ is affine, taking global sections between affine stacks is also fully-faithful and we get:
	\[ \begin{split}
		\Spec_X(\Sym_{\_O_X} \_F^\vee) (\Spec(B) \to X) \simeq & \Homsub{\Sh_X(\cdga)_{\_O_X/}}(\Sym_{\_O_X} \_F^\vee, f_* \_O_{\Spec(B)})\\
		\simeq & \Homsub{\cdga_{A/}}(\Sym_A \_F^\vee, B)
	\end{split}\]
\end{RQ}

\subsection{Linear and Semi-Linear Stacks}\ \label{sec:linear-and-semi-linear-stacks}

\medskip

The goal of this section is to introduce the analogue of vector bundles in derived algebraic geometry and compute their (relative) tangent complexes. We will use the result due to the Serre-Swan theorem, saying that vector bundles are equivalent to projective finitely generated sheaves of (non differential graded) modules, to motivate the construction of derived stacks out of \emph{quasi-coherent} sheaves of modules. These will define the notion of ``linear stacks'' out of sheaves of modules over a given base $X$. When working in the context of derived geometry, we can drop the condition of being projective as we can always take a \emph{projective resolution}, and the condition of being locally finitely generated will be replaced by the notion of \emph{perfect} sheaf of module when we need to use dualization properties of quasi-coherent sheaves. \\

We discuss, in Section \ref{sec:some-sheaves-of-module}, the notions of \emph{quasi-coherent} and \emph{perfect} sheaves of modules. The reason we look at quasi-coherent sheaves is that they behave like the gluing of ``local modules'' over the base $X$. This is a notion required when working in algebraic geometry ensuring that a sheaf of modules on an affine scheme $X = \Spec(A)$ is the same as an $A$-module. Perfect sheaves of modules are quasi-coherent sheaves of modules that are \emph{dualizable}. 

Then as a notable and important example of such sheaves, we will discuss the tangent and cotangent complexes. The construction of the cotangent complex in a general context is discussed in Appendix \ref{sec:cotangent-complex-in-m}. This is a construction only on affine stacks and in Section \ref{sec:relative-cotangent-complex-on-stacks} we discuss how to extend this definition to derived schemes and stacks. \\

Then Section \ref{sec:relative-cotangent-complex} computes the relative tangent complex of a linear stack, and \ref{sec:connections-on-semi-linear-stacks} the full tangent complex of a linear stack (over an affine base) by making use of connections.

\subsubsection{Some sheaves of modules}\
\label{sec:some-sheaves-of-module}

\medskip

For any type of geometric space $X$ obtained from gluing affine objects, we are interested by the sheaves of modules over $X$ that are obtained as some kind of gluing of modules over those affine objects. These are precisely the quasi-coherent sheaves on $X$. 

\begin{Def}\label{def:quasi-coherent sheaves}
	
	All along $A$ denotes an algebra, $X$ is a derived stack.
	\begin{itemize}
		\item We denote by $\QC(X)$ the category of quasi-coherent sheaves. The objects of this category can be defined as the data of a $A$-module for each map $f:\Spec(A) \to X$, called $M_A$ and for all commutative diagram: 
		
		\[ \begin{tikzcd}
			\Spec(A) \arrow[rr, "f"] \arrow[dr]&& \Spec(B) \arrow[dl]\\
			& X&   
		\end{tikzcd}\]
		there is a weak equivalence $\alpha_f : M_B \otimes_B A \to M_A$ together with some coherence conditions (see \cite[Section 1.3.7]{TV08}). 
		
		\item We can also describe the category of quasi-coherent sheaves on a stack $X$ as the following limit: 
		\[ \QC(X) = \lim\limits_{\Spec(A)\to X} \Mod_A \]
		This formula is what we obtain by left Kan extending the functor $\QC$ along $\daff \to \dst$.
	\end{itemize}
\end{Def}

\begin{Def}
	\label{def:dualizable and reflexive}
	
	Take $\_M$ a closed symmetric monoidal model category with unit denoted by $1$. We can define the \defi{dual} of an object $V \in \_M$ as\footnote{Recall that $\Hom_\_M^\_M$ denotes the $\_M$ enriched Hom functor in $\_M$.}: 
	\[ V^\vee := \Hom_\_M^\_M(V,1)\]
	
	\begin{itemize}
		\item An object $V \in \_M$ is called \defi{reflexive} if the natural morphism\footnote{obtained as the image of the identity under the composition: \[\Hom_\_M^\_M(1,1) \to \Hom_\_M^\_M(V \otimes V^\vee, 1) \to \Hom_\_M^\_M(V, \Hom_\_M^\_M(V^\vee, 1)) \simeq \Hom_\_M^\_M(V, (V^\vee)^\vee)\]}: 
		\[V \to (V^\vee)^\vee\]
		is an equivalence. 
		
		\item There is a stack $\perf$ of \defi{perfect modules} (see \cite[Corollary 1.3.7.4]{TV08}) such that $\perf(A)$ is the category of perfect $A$-modules (see \cite[Definition 1.2.3.6]{TV08}) where $V$ is in $\perf(A)$ if the natural morphism: 
		\[ V \otimes V^\vee \to \Hom_\_M^\_M(V,V)\] 
		is an equivalence.   
	\end{itemize}
\end{Def}

For now we will go back to the situation where $\_M := \Mod_A$ with $1= A$. 

\begin{Prop} 
	\label{prop:properties of perfect and reflexive sheaves}
	
	For $\_F \in \perf(A)$ and $M, N \in \Mod_A$, we have: 
	\begin{itemize}
		\item Perfect module are the strongly dualizable objects in $\Mod_A$ and we have (\cite[Proposition 1.2.3.7]{TV08}): 
		\[ \iHom_{A} (M, \_F \otimes_A N) \simeq \iHom_{A} (M \otimes \_F^\vee, N)\]
		\[ \iHom_{A} (\_F, M) \simeq \_F^\vee \otimes M \]
		\item An $A$-module $\_F$ is perfect if an only if it is quasi-isomorphic to a finitely presented projective $A$-module (\cite[Lemma 2.2.2.2]{TV08}).
		\item A perfect $A$-module $\_F$ is reflexive since we have the natural equivalences: 
		\[ \_F \to \iHom_{A} (A, \_F) \simeq \iHom_{A} (\_F^\vee, A) := (\_F^\vee)^\vee\]
	\end{itemize}
\end{Prop}

\begin{Prop}
	Any map of algebras $f: B \to A$ induces Quillen adjunctions\footnote{We consider $\Mod_A$ with the standard projective model structure. This is an instance of \emph{good model structure} as defined in Appendix \ref{sec:good-model-structures}.}: 
	\[ \begin{tikzcd}
		f^* : \Mod_B \arrow[r, shift left] & \arrow[l, shift left] \Mod_A : f_*
	\end{tikzcd}\] 
	where $f^* M := M \otimes_B A$ is the extension of scalars and $f_*$ is the restriction of scalars. 
	This result extends to quasi-coherent sheaves; given a morphism $f: X \to Y$, there is an adjunction: 
	\[ \begin{tikzcd}
		f^{-1} : \Mod_{\_O_Y} \arrow[r, shift left] & \arrow[l, shift left] \Mod_{f^{-1}\_O_Y} : f_* 
	\end{tikzcd}\]
	
	where $f_*$ is the direct image functor between sheaves and $f^{-1}$ is the sheaf restriction functor\footnote{This comes from the adjunction between $\Sh_X(\cdgacon)$ and $\Sh_Y(\cdgacon)$ given by $f^{-1} \dashv f_*$.}.	Moreover, there is an adjunction between $\_O_X$ and $f^{-1} \_O_Y$-modules given by the scalar extension-restriction. This together with the previous adjunction, induces an other adjunction between quasi-coherent sheaves: 
	\[ \begin{tikzcd}
		f^* : \QC(Y) \arrow[r, shift left] & \arrow[l, shift left] \QC(X) : f_* 
	\end{tikzcd}\]
	with $f^*\_F = f^{-1}\_F \otimes_{f^{-1}\_O_Y} \_O_X \in \QC(X)$. 
\end{Prop}

\subsubsection{Relative (co)tangent complex on stacks}\ \label{sec:relative-cotangent-complex-on-stacks}

\medskip

Recall from Appendix \ref{sec:cotangent-complex-in-m} that the cotangent complex of an $A$-algebra $B$ is by definition the $B$-module representing $A$-linear derivation in the sens that for all $B$-module $N$, there is a natural equivalence (see Definition \ref{def:cotagent complex in M}):
\[ \Hom_{B}(\Llr{B}{A}, N) \simeq \Der_A(B, N)\]

In the Appendix \ref{sec:cotangent-complex-in-m} the construction was made in an arbitrary \emph{good model category} (Definition \ref{def:good model categories}). For now we only need to restrict to $\_M = \Mod_k$. 
Before we move to the general complex, let us recall the classical (underived) construction of the Kähler differential:

\begin{Def} \label{def:kahler differential}
	
	Given a morphism of algebras $A \to B$, the $B$-module, $\Omega_{B/A}^1$, of $A$-linear Kähler differentials is defined as the module representing the (underived) $A$-linear derivations: 
	\[ \Hom_B \left( \Omega_{B/A}^1, M \right) \cong \Der_A(B, M) \cong  \Homsub{\cdga_{A//B}} \left( B, B \boxplus M \right)\] 
	Note that in this adjunction, the $\Hom$ functor is \emph{not} derived and by definition $\Omega$ is not a derived functor.
\end{Def}

\begin{Cons}
	
	In concrete terms, $\Omega_{B/A}^1$ is the $B$-module freely generated by terms denoted by $db$ for $b \in B$ subject to the relations: 
	\begin{itemize}
		\item For all $b, b' \in B$: \[d(b.b') = b.db' + (-1)^{\vert b \vert}db . b'\]
		\[ d(b+b') = db + db'\]
		\item $A$-linearity, for all $a \in A$: 
		\[ da = 0\] 
	\end{itemize}
	We can show that this satisfies the condition of Definition \ref{def:kahler differential} and this explains why this construction gives the module of ``differential forms''. Moreover, the cotangent complex only differs from the module of Kähler differentials by the fact that the adjunction and functors in the definition are \emph{derived}. As such, up to picking a  ``good replacement'', we can obtain a similar presentation for the cotangent complex. 
\end{Cons}

In Appendix \ref{sec:cotangent-complex-in-m}, we give a general construction of the cotangent complex in an arbitrary ``good'' model category $\_M$ (in the sens of  Definition \ref{def:good model categories}). The definition is again essentially the same as Definition \ref{def:kahler differential} where the adjunction is derived. \\ 

Using Appendix \ref{sec:cotangent-complex-in-m}, we have a definition of the cotangent and tangent complexes for affine objects. It is however unclear whether the cotangent complex of affines glues to a quasi-coherent sheaf on a derived stack $X$. The answer in full generality is that it does not glue well and we will need to have restriction on the stacks we will work with.

\begin{RQ}\label{rq:artin stacks have a cotangent complex}
	By definition every affine derived stack has a global cotangent complex given by $\Ll_{\Spec(A)} := \Ll_A$ (see \cite[Corollary 2.2.3.3]{TV08} and Remark \ref{rq:artin and geometric derived stacks}). Moreover \cite[Theorem 7.4.3.18]{Lu17} tells us that if $H^0(A)$ is finitely generated, then $\Ll_A$ is perfect if and only if $A$ is finitely presented.
	
	More generally, any derived Artin stack admits a (relative) cotangent complex\footnote{We call this a cotangent complex although it is a \emph{quasi-coherent sheaf} of complexes over $X$.} and in particular derived schemes always admit a cotangent complex since they are in particular derived Artin stacks. Moreover, a locally finitely presented derived Artin stack has a \emph{perfect} cotangent complex. 
\end{RQ}

\begin{RQ}\label{rq:cotangent complex for almost finitely presented stacks}
	If we restrict to stacks of local almost finite presentation (Definition \ref{def:derived stacks almost finitely presented}), then all stacks having a global cotangent complex have a coherent and eventually coconnective cotangent complex (see \cite[Introduction of Section 2]{CPTVV}).
\end{RQ}

\subsubsection{Definition and examples of linear and semi-linear stacks}\ \label{sec:definition-and-examples}\

\medskip

In this section, we will take a quasi-coherent sheaf $\_F \in \QC(X)$ over a derived stack $X$ and define the linear stack over $X$ associated to $\_F$. We show that these stacks are in fact pro-corepresentable (see Proposition \ref{prop:representability theorem linear stacks}). Then we extend this to the notion of \emph{semi-linear representation} of stacks, which are essentially stacks over $X$ pro-corepresented (under $\_O_X$) by a sheaf of semi-free algebras $\_O_X \to \Sym_{\_O_X} \_F^\vee$ with $\_F \in \QC(X)$. Semi-linear representation of stacks are going to be important for us with the examples of the derived critical locus and variations of it called \emph{almost derived critical loci} (Definition \ref{def:almost derived critical loci}).

\begin{Def}
	\label{def:linear stacks}
	
	Given $\_F \in \QC(X)$ a quasi-coherent sheaf over a derived stack, we can construct a \defi{linear stack}, denoted $\Aa (\_F)$, the stack on $\daff_{/X}$, defined (as a stack on $\daff_{/X}$)\footnote{We have seen in Remark \ref{rq:extending stack on affines over X to stacks over X} that stacks on $\daff_{/X}$ naturally extend to stacks \emph{over} $X$. We will not make a distinction between stacks on $\daff_{/X}$ and their extensions to stacks over $X$ since both notions are equivalent.}, by:
	\[ \Aa_X(\_F) \left( f: \Spec(A) \rightarrow X \right) := \Map_{A} \left( A, f^* \_F \right)\]
	
	Any linear stack comes equipped with a natural projection $\pi : \Aa_X(\_F) \to X$. This defines a functor: 
	\[ \Aa_X : \QC(X) \to \dst_{/X} \]
	
	If $\_F$ is perfect, then $\Aa_X(\_F)$ is called a \defi{perfect linear stack}. 
\end{Def}

\begin{notation}\label{not:linear stack and their sheaves}
	
	The quasi-coherent sheaves will be denoted by curved calligraphic letter such as $\_F$, $\_E$, $\_L$ and their associated linear stacks will often be denoted by the corresponding straight capital letter $F$, $E$, $L$ to simplify the notations.  
\end{notation}

Earlier we defined the notions of tangent and cotangent complexes, which are the quasi-coherent sheaves of derivation and differential forms respectively. Naturally, their associated linear stacks define the notion of tangent and cotangent stacks respectively. 

\begin{Def}
	\label{def:tangent and cotangent stack}
	
	Assume that $X$ admits a global cotangent complex. Then we define the \defi{$n$-shifted cotangent stack} as the linear stack: \[T^*[n]X := \Aa_X(\Ll_X[n])\]
	
	Similarly, we define the \defi{$n$-shift tangent stack} over $X$ as the linear stack: 
	\[T[n]X := \Aa_X(\Tt_X[n])\]
\end{Def}

\begin{Prop}\label{prop:linear stack of connective modules}
	If $\_F \in \perf(X)$ is connective, then $\Aa_X (\_F)$ is affine relatively to $X$, meaning that is it is given by a relative spectrum of a sheaf of connective $\_O_X$-algebras, and we have:
	\[ \Aa_X (\_F) \simeq \Spec_X(\Sym_{\_O_X} \_F^\vee )\] 
\end{Prop}

For a general $\_F$, the linear stack will not be an affine stack for reasons as discussed in Warning \ref{war:non-affine prorepresentable}. However, they are still ``pro-corepresented'' by a possibly non-connective free cdga (under $\_O_X$). Therefore Proposition \ref{prop:linear stack of connective modules} is a consequence of the following:

\begin{Prop}
	\label{prop:representability theorem linear stacks}
	Any perfect linear stack $\Aa_X(\_F)$ can be pro-corepresented relatively to $X$ by $\Sym_{\_O_X} \_F^\vee$. In other words, for all $\Spec(A) \to X$ we have:
	\[ \Aa_X(\_F) \simeq \Spec_X \left( \Sym_{\_O_X} \_F^\vee \right) \] 
\end{Prop}

\begin{Lem}\label{lem:sheaf of function technical lemma}
	Let $f: Y \to X$ be a map of derived stacks. Then we have the following: 
	\[ \_O_{Y} \simeq f^*\_O_X \]  
	\[ \Sym_{\_O_Y} f^*\_F^\vee \simeq f^{-1}\left( \Sym_{\_O_X} \_F^\vee \right)\otimes_{f^{-1}\_O_X} \_O_{Y} \]
	\[ \Sym_{\_O_Y} f^*\_F^\vee \simeq f^*\left(\Sym_{\_O_X}\_F^\vee\right) \]
	
\end{Lem}
\begin{proof}[Proof of Proposition \ref{prop:representability theorem linear stacks}]
	First because $\_F$ is perfect we have the equivalences (using Lemma \ref{lem:sheaf of function technical lemma}): 
	\[ \begin{split}
		\Aa_X(\_F)(f: \Spec(A)\to X) & \simeq \Map_A \left(A, f^* \_F\right)\\ 
		& \simeq \Map_A \left(f^*\_F^\vee, A\right) \\
		& \simeq \Mapsub{\cdga_{A/}} \left(\Sym_A f^*\_F^\vee, A\right)\\
		&\simeq \Mapsub{\cdga_{\_O_{\Spec(A)}/}} \left(\Sym_{\_O_{\Spec(A)}} f^*\_F^\vee, \_O_{\Spec(A)}\right)\\
		&\simeq \Mapsub{\cdga_{\_O_X/}}( \Sym_{\_O_X} \_F^\vee, f_* \_O_{\Spec(A)}) 
	\end{split}\]

\end{proof}

\begin{Prop}\label{prop:map to linear stack}
	A map $f : Y \to \Aa_X (\_F)$ into a perfect linear stack is determined by: 
	\begin{itemize}
		\item A map \[g = \pi \circ f : Y \to X\] 
		\item A section: 
		\[ s \in \Mapsub{\QC(Y)} \left( \_O_Y, g^*\_F\right)\]  
	\end{itemize}
	
	More precisely we have a pullback diagram: 
	\[ \begin{tikzcd}
		\Map_{\dst}\left(Y, \Aa_X(\_F) \right) \arrow[r] \arrow[d] & \left(\QC(Y)_{\_O_Y/}\right)^\simeq \arrow[d] \\
		\Map(Y, X) \arrow[r, "(-)^* \_F"] & \QC(Y)^\simeq
	\end{tikzcd}\]
\end{Prop}
\begin{proof}
	The map $Y \to X$ is given by the composition $\pi \circ f$. Then $f$ becomes a map in $\dst_{/X}$ and we have: 
	\[ \begin{split}
		\Mapsub{\dst_{/X}}\left(Y, \Spec_X\left(\Sym_{\_O_X} \_F^\vee\right) \right)   \simeq & \Mapsub{\cdga_{\_O_X/}}  \left(\Sym_{\_O_X} \_F^\vee, g_* \_O_Y \right)\\
		\simeq & \Mapsub{\cdga_{\_O_Y/}} \left(g^*\left(\Sym_{\_O_X} \_F^\vee\right), \_O_Y \right) \\
		\simeq& \Mapsub{\cdga^{\QC(Y)}} \left(\Sym_{\_O_Y} g^* \_F^\vee, \_O_Y \right) \\
		\simeq& \Mapsub{\QC(Y)} \left(g^* \_F^\vee, \_O_Y \right) \\
		\simeq & \Mapsub{\QC(Y)} \left(\_O_Y, g^* \_F \right) \\
	\end{split}\] 
	
	Where the first and third equivalences follow from Definition \ref{def:relative spectrum over X} and Lemma \ref{lem:sheaf of function technical lemma} respectively.
\end{proof}

\begin{Cor}\label{cor:linear stack section and zero section}
	The space of sections of $\pi : \Aa_X(\_F) \to X$ is equivalent to the mapping space $\Map_{\QC(X)}\left( \_O_X, \_F \right)$ in $\QC(X)$. In particular, there is a \defi{zero section} determined by the zero element in this mapping space: \[s_0 : X \to \Aa_X(\_F)\]
\end{Cor}

\begin{Prop}\
	\label{prop:map of linear stacks}
	Let $\_F \in \QC(X)$ and $\_G \in \QC(Y)$ be quasi-coherent sheaves on the derived stacks $X$ and $Y$. 
	\begin{itemize}
		\item The definition of linear stack is natural in $\_F$. Any map $f: \_F \to \_G$ in $\QC(X)$ naturally induces a map\footnote{As an abuse of notation, we will not distinguish the notation between the map between quasi-coherent sheaves and the induced map between the associated linear stacks.} of derived stacks: 
		\[ f : \Aa_X(\_F) \to \Aa_X(\_G) \]
		we will called such maps \defi{linear maps of linear stack}.
		\item Considering the following pullback: 
		\[\begin{tikzcd}
			f^* \Aa_Y (\_F) \arrow[r] \arrow[d]& \Aa_Y (\_F) \arrow[d] \\ 
			X \arrow[r, "f"] & Y
		\end{tikzcd}\]
		we have a natural equivalence $\Aa_Y(f^*\_F) \simeq f^*\Aa_X ( \_F)$.
		\item A map $\phi : \Aa_X(\_F) \to \Aa_Y (\_G)$ fitting in the commutative diagram: 
		\[\begin{tikzcd}
			\Aa_X (\_F) \arrow[r] \arrow[d]& \Aa_Y (\_G) \arrow[d] \\ 
			X \arrow[r, "f"] & Y
		\end{tikzcd}\]
		is equivalent to the data of a map $\phi_f : \Aa_X(\_F) \to \Aa_X (f^*\_G)$ over $X$. 
		
		Moreover $\phi$ is called \defi{linear} if $\phi_f$ comes from a map of quasi-coherent sheaves $\phi_f : \_F \to f^* \_G$, which is exactly the data of a map $\_F \to \_G$ of quasi-coherent sheaves over $f$.
	\end{itemize}
\end{Prop}
\begin{proof}\
	
	\begin{itemize}
		\item Since both stacks are over the same base, the map between the linear stacks is given for each $A$-point $g : \Spec(A) \to X$ by a map:
		\[ \Map_A (A, g^*\_F) \to \Map_A (A, g^*\_G) \]
		This map is given by $g^* f : g^*\_F \to g^* \_G$.
		
		\item Take $Z$ a cone over the pullback diagram. The cone is determined by: 
		\[ \begin{tikzcd}
			Z \arrow[r, "g"] & X \arrow[r, "h"] & Y
		\end{tikzcd}\] 
		\[ s \in \Mapsub{\QC(Z)} \left(\_O_Z, g^* h^* \_F \right) \]

		From this we can get a unique map $Z \to \Aa_X(h^* \_F)$ making the cone diagram commute and therefore exhibiting $\Aa_X(h^* \_F)$ as the pullback. This map is given by:
		\begin{itemize}
			\item $g: Z \to X$
			\item the image of $s$ under the map\footnote{Using the unit $\id \to g_*g^*$ of the extension-restriction of scalars adjunction.}: 
			\[ \Mapsub{\QC(Z)} \left( \_O_Z, g^*h^* \_F \right) \simeq \Mapsub{\QC(X)} \left( g_* g^* \_O_X, h^* \_F \right) \to \Mapsub{\QC(X)} \left(\_O_X, h^*\_F\right) \] 
		\end{itemize}  
		\item The third point is a direct consequence of the first two. 
	\end{itemize}
\end{proof}

\begin{Lem}
	\label{lem:linear pullback of linear stacks}
	
	Given $\_F, \_G, \_H \in \QC(X)$, then the pullback of linear maps between the linear stacks (over $X$) is given by the linear stack associated to the pullback of modules: 
	\[ \begin{tikzcd}
		\Aa_X(\_H \times_{\_G} \_F) \arrow[r]  \arrow[d] & \Aa_X(\_F) \arrow[d]\\
		\Aa_X(\_H) \arrow[r] & \Aa_X(\_G)
	\end{tikzcd}\]
\end{Lem}
\begin{proof}
	Let $Y$ be a cone over this pullback and $f: Y \to X$ the natural map. Then the map from $Y$ into the pullback is completly characterized by a morphism in $\Hom(\_O_Y, f^* \_F)$ (and similarly for $\_H$ and $\_G$). These morphisms are compatible with the maps between $\_F$, $\_G$ and $\_H$. In other words, the map to the pullback from the cone induces a unique element in: 
	\[ \Hom(\_O_Y, f^* \_H) \times_{\Hom(\_O_Y, f^* \_G)} \Hom(\_O_Y, f^* \_F) \simeq \Hom(\_O_Y, f^*(\_H \oplus_\_G \_F))\]
	
	This is exactly the data of a map $f: Y \to \Aa_X(\_H \times_\_G \_F)$, proving that $\Aa_X(\_H \times_\_G \_F)$ is indeed the pullback.	 
\end{proof}

From Proposition \ref{prop:representability theorem linear stacks}, every perfect linear stack is pro-corepresented by a sheaf of free algebras on $\_O_X$. In what follows we will generalize this to \emph{semi-free} algebras to encodes stacks that behave almost like linear stacks even though they are not linear.

\begin{Def}
	\label{def:semi linear stacks}
	
	A \defi{semi-linear presentation of a derived stack} over a derived stack $X$ an equivalence: 
	\[ Y \simeq \Spec_X\left(\left(\Sym_{\_O_X} \_F^\vee, \delta\right) \right)\]
	with $\left(\left(\Sym_{\_O_X} \_F^\vee, \delta\right) \right)$ a semi-free algebra over $\_O_X$ in $\cdga_{\QC(X)}$. Notice that if $\_F$ is not coconnective, then $Y$ is not affine but only pro-represented by a semi-free algebra. If $\_F$ is perfect, we called it a \defi{perfect presentation}.
\end{Def}

\begin{Def}\
	\label{def:derived critical locus}
	Let $X$ be a derived stack with a cotangent complex, and $f: X \to \Aa^1$. Then the derived critical locus of $f$ is defined as the following pullback\footnote{This pullback can be taken either in the category of derived stacks, derived stack over $X$ or derived stacks on $\daff_{/X}$ since the functor $\dst^{\daff_{/X}} \to \dst_{/X}$ is an equivalence and the functor $\dst_{/X} \to \dst$ preserves and reflects connected limits (including pullbacks).}: 
	\[ \begin{tikzcd}
		\RCrit(f) \arrow[r] \arrow[d] & X \arrow[d, "df"] \\
		X \arrow[r, "s_0"] & T^*X
	\end{tikzcd} \]
\end{Def}

\begin{Prop}
	\label{prop:derived critical locus corepresentability}
	The derived critical locus $\RCrit(f)$ is has a semi-linear presentation given by the semi-free algebra $\Sym_{\_O_X} \Tt_X[1]$ with differential $\iota_{df}$ plus the differentials on $\Tt_X[1]$ and $\_O_X$. 
\end{Prop}
\begin{proof}
	We need to take a cofibrant resolution of the zero map $\Sym_{\_O_X}\Tt_X \to \_O_X$. We take the inclusion $\Sym_{\_O_X} \Tt_X \to \Sym_{\_O_X} (\Tt_X \oplus \Tt_X[1])$ with differential induced by the identity $\Tt_X[1] \to \Tt_X$ and the differentials on $\_O_X$ and $\Tt_X$ on the right hand side. Then, using that, $\RCrit(f)$ is the semi-linear stack pro-corepresented by: 
	\[  \begin{split}
		\_O_X \otimes_{\Sym_{\_O_X} \Tt_X} \Sym_{\_O_X}(\Tt_X \oplus \Tt_X[1]) \simeq &\Sym_{\_O_X} \left(0 \oplus_{\Tt_X} (\Tt_X \oplus \Tt_X[1])\right)\\
		\simeq & \Sym_{\_O_X} \left( \Tt_X[1] \right)
	\end{split}\]
	
	where the differential (obtained on the tensor product of the semi-free algebras) is exactly given by the differentials on $\Tt_X[1]$, $\_O_X$ and $\iota_{df}$ viewed as a map degree $1$: \[\Tt_X[1] \to \_O_X\]
\end{proof}

In Section \ref{sec:bv-as-a-generalized-derived-symplectic-reduction}, we will be interested in a mild generalization of the derived critical locus given by some particular type of semi-linear stacks. We will describe these objects here as they are typical examples of semi-linear stacks. We start by giving the construction of one of the most important semi-linear stack, namely the Koszul--Tate resolution of a global function $f: X \to \Aa^1$. 

\begin{Cons}
	\label{cons:koszul tate resolution}
	
	If $X:= \Spec(B)$ is a smooth affine scheme, then the strict pullback of $df: X \to T^*X$ by the zero section is called the \defi{strict critical locus} and is denoted by $\Crit(f)$.
	
	If we view $\Crit(f)$ as a derived scheme, the map $i:\Crit(f) \to X$ corresponds to the projection (of sheaves over the critical locus), $B \to B/I$, with $I$ the ideal of $B$ generated by the elements $df.X$ for $X \in \Tt_B$.  
	
	The \defi{Koszul--Tate} resolution of $f$, denoted $\KT(f)$, is a derived scheme obtained from a specific cofibrant resolution of this map, and is constructed as follows: 
	\begin{enumerate}
		\item Consider $\Sym_B \Tt_B[1]$ together with the differential given by $\iota_{df}$. This is the algebra of functions of the derived critical locus thanks to Proposition \ref{prop:derived critical locus corepresentability}, but it is not in general a resolution of the quotient. However, its cohomology in degree $0$ is isomorphic to $B/I$. 
		
		\item We add free generators in negative degree (in degree $\leq -2$) to ``kill'' the cohomology in negative degrees of the derived critical locus. This can be done inductively (on the cohomological degree) and we refer to \cite{Ta57} for the detailed description of this procedure. We get a semi-free algebra over $B$ of the form: 
		\[ \KT(f) := \Sym_B \left( \Tt_B[1] \oplus \_L_{\KT}[2]\right) \]
		
		where $\_L_{\KT} \in \Mod_B$ is a connective projective $B$-module ($\_L_{\KT}$ can even be chosen to be free). The differential restricted to $\Tt_B[1]$ is still $\iota_{df}$ and the complex is acyclic in negative degree. 
	\end{enumerate} 
	Then we define $\dKT(f):= \Spec(\KT(f))$. Since this is a resolution of the strict critical locus, there is a natural weak equivalence of derived schemes:\[\Crit(f) \overset{\sim}{\to} \dKT(f) \]
\end{Cons}

\begin{RQ}\label{rq:ghost and ghosts of ghosts}
	This example of semi-linear stack in an instance of the procedure of adding ``anti-fields'' and ``anti-ghosts fields''. Essentially, in our model, generators of $\Tt_X[1]$ are called \defi{anti-fields} and the generators of $\_L_{\KT}[2]$ are called \defi{anti-ghost fields}. This idea will be discussed further in Section \ref{sec:context-and-construction-for-infinitesimal-actions}.
\end{RQ}

We will be interested, in Section \ref{sec:derived-perspective-of-the-bv-complex}, by objects where we add anti-ghost fields to the derived critical locus without necessarily getting a resolution of the strict critical locus.  
This is a generalization of the Koszul--Tate construction that does not kill all the cohomology of the derived critical locus. 

\begin{Def} 
	\label{def:almost derived critical loci}
	
	An affine derived scheme $S$ is said to be an \defi{almost derived critical locus of $f$} if it is equivalent to a derived scheme of the following form: 
	\[ S \simeq \Spec( \Sym_B \left( \Tt_B[1] \oplus \_L[2]\right))\]
	
	where $\Sym_B \left( \Tt_B[1] \oplus \_L[2]\right)$ is a semi-free $B$-algebra such that $B$ is a finitely generated smooth algebra,  the differential restricted to $\Tt_B[1]$ is given by $\iota_{df}$ and $\_L$ is a connective $B$-module. 
\end{Def}

\begin{RQ}\label{rq:maps almost derived critical locus}
	Definition \ref{def:almost derived critical loci} is very closed to the definition of a Koszul--Tate resolution of $B$ except for the fact that this might not be a resolution of $B/I$. 
	However, it is equipped with a map to the quotient inducing a map: \[\Crit(f) \to S\]
	Moreover, up to adding further elements to kill the cohomology, there is a Koszul--Tate construction that naturally projects to the semi-free algebra defining our almost derived critical locus. This induces a sequence of maps: 
	\[ \Crit(f) \overset{\sim}{\to} \iKT(f) \to S \to \RCrit(f) \to X\]
	Dual to the maps: 
	\[ B \to \Sym_{B} \Tt_B[1] \to \Sym_{B}\left( \Tt_B[1] \oplus \_L[2]\right) \to \Sym_{B}\left( \Tt_B[1] \oplus \_L_{\KT}[2]\right) \overset{\sim}{\to} \faktor{B}{I} \]
	As such, $S$ is an almost derived critical locus of $f$ in the sens that this is an object sitting ``in between'' the strict and derived critical loci. 	
\end{RQ}

\begin{notation}\label{not:maps almost derived critical locus}
	Through this text, all ``natural projection'' will be denoted by $\pi$ (or variations of it such as $\pi_S$ or $\pi_\_F$), for example:
	\[ \pi : \RCrit(f) \to X \qquad \pi_X: TX \to X \qquad \pi_\_F : \Aa_X(\_F) \to X \]
	
	The natural map from an almost derived critical locus to the derived critical locus will be denoted by:
	\[ i : S \to \RCrit(f)\]
\end{notation}

\begin{RQ}\label{rq:pullbak of almost derived critical loci}
	We consider morphisms of almost derived critical loci as morphisms of derived schemes over $\RCrit(f)$. Then if $S$ and $S'$ are almost derived critical loci, then $S \times_{\RCrit(f)} S'$ is also an almost derived critical locus.
\end{RQ}

\subsubsection{Relative cotangent complex of a (semi-)linear stack}
\label{sec:relative-cotangent-complex}\

\medskip

This section is devoted to the study of relative cotangent complex of the projection $\pi:\Aa_X(\_F) \to X$ from a linear stack. Given $\_F \in \perf(X)$ with $X$ a derived stack that admits a tangent complex, the goal of this section is to show that there is an equivalence: \[\Llr{\Aa_X (\_F)}{X} \simeq \pi^* \_F^\vee\]
and study the functoriality of this equivalence in $\_F$ and $X$.

\begin{Prop} 
	\label{prop:relative cotangent complex for linear stacks}
	
	Let $X$ be a derived stack admitting a tangent complex and $\_F \in \perf(X)$ a perfect quasi-coherent sheaf on $X$. We denote by $\pi : \Aa_X (\_F) \rightarrow X$ the natural projection. Then we have:
	
	\[ \Ll_{\pi_X} \simeq \Llr{\Aa_X(\_F)}{X} \simeq \pi_X^* \_F^\vee \]  
\end{Prop}
\begin{proof}
	We will show the result for any $B$-point $y : \Spec(B) \rightarrow \Aa_X (\_F )$ and we write $x = \pi \circ y : \Spec(B) \rightarrow X$. We can show that for all $M \in \Mod_B$ connective, we have: 
	\[ \Hom_{B} \left( y^*\Llr{\Aa_X (\_F)}{X} , M \right) \simeq \Hom_{B} \left(x^* \_F^\vee, M \right) \] 
	
	First we observe that $\Hom_{B} \left( y^*\Llr{\Aa_X (\_F)}{X} , M \right) $ is equivalent, using the universal property of the cotangent complex, to the following homotopy fiber at $y$:\\
	
	\adjustbox{scale=0.93,center}{$\tx{hofiber}_{y} \left( \Homsub{\faktor{\dst}{X}} \left( \Spec(B \boxplus M), \Aa_X (\_F ) \right)\rightarrow \Homsub{\faktor{\dst}{X}} \left( \Spec(B), \Aa_X (\_F ) \right) \right)$} \\
	
	with $\Spec(B \boxplus M) \rightarrow X$ being the composition: 
	\[ \begin{tikzcd}
		\Spec(B \boxplus M) \arrow[r, "p"] & \Spec(B)\arrow[r, "x"] & X 
	\end{tikzcd}
	\]
	
	Thus a map in $ \Hom_{B} \left( y^*\Llr{\Aa_X (\_F)}{X} , M \right)$ is completely determined by a map 
	\[ \Phi : \Spec(B \boxplus M) \rightarrow \Aa_X (\_F)\] making the following diagram commute: 
	
	\[ \begin{tikzcd}
		\Spec (B) \arrow[d, "i"] \arrow[dr, "y"] & \\
		\Spec(B \boxplus M) \arrow[r, "\Phi"] \arrow[d, "p"] &  \Aa_X( \_F) \arrow[d, "\pi_X"] \\
		\Spec(B) \arrow[r, "x"] & X
	\end{tikzcd}\]
	
	Thus, we obtain that $ \Hom_{B} \left( y^*\Llr{\Aa_X (\_F)}{X} , M \right)$ is equivalent to 
	\[ \tx{hofiber}_{s_y} \left(  \Map_{B \boxplus M } \left( B \boxplus M, p^* x^* \_F \right) \rightarrow \Map_{B} \left( B, x^* \_F \right)\right)\]

	where $s_y \in \Map_{B} \left( B, x^* \_F \right)$ is the section associated to $y : \Spec(B) \rightarrow \Aa_X (\_F )$ from Proposition \ref{prop:map to linear stack}. The map is then given by pre-composition with $i^*$. We can now observe that $p^* x^* \_F = x^* \_F \oplus x^* \_F \otimes_B M $ and that:
	\[ \Map_{B \boxplus M } \left( B \boxplus M, p^* x^* \_F \right) \simeq \Map_{B } \left( B , x^* \_F \oplus x^* \_F \otimes_B M \right) \].
	
	We obtain: 
	\[\begin{split}
		\Hom_{B} \left( y^*\Llr{\Aa_X (\_F)}{X} , M \right)  \simeq&  \tx{hofiber} \left(  \Map_{B} \left( B , x^* \_F \oplus x^* \_F \otimes_B M \right) \right. \\
		&\qquad \qquad \left. \rightarrow \Mapsub{B } \left( B, x^* \_F \right)\right)\\
		\simeq& \Map_{B } \left( B, x^* \_F \otimes_B M \right) \\
		\simeq&  \Map_{B} \left( x^* \_F^\vee, M \right) 
	\end{split} \] 
	
	Now the result follows from the fact that the functor		
	\[ \begin{tikzcd}[row sep=1mm, column sep= small]
		\Mod_B \arrow[r] & \tx{Fun}\left( \Mod_B^{\leq 0}, \tx{sSet} \right) \\
		\quad \quad 	N \arrow[r, mapsto] & \Map_{B} \left( N, \bullet \right)
	\end{tikzcd} \]
	
	is fully-faithful and the fact that everything we did is natural in $B$. 
\end{proof}

\begin{Lem}
	\label{lem:relative cotangent base change}
	Let $f: X \rightarrow Y$ be a morphism of derived Artin stacks. We consider $\_F \in \perf(Y)$. Then there is a commutative square: 
	
	\[ \begin{tikzcd}
		\Phi^* \Llr{\Aa_X ( \_F )}{X} \arrow[r] \arrow[d, "\simeq"] &  \Llr{\Aa_X ( f^*\_F)}{Y} \arrow[d, "\simeq"] \\
		\Phi^* \pi_Y^*\_F^\vee \arrow[r, "\simeq"] & \pi_X^* f^* \_F^\vee \end{tikzcd} \]

	with $\Phi$ the natural morphism in the following pullback:
	
	\[ \begin{tikzcd}
		\Aa_X (f^* \_F) \arrow[r, "\Phi"] \arrow[d, "\pi_X"] & \Aa_X (\_F) \arrow[d, "\pi_Y"] \\
		X \arrow[r, "f"] & Y
	\end{tikzcd} \]
	
	and the lower horizontal equivalence $\Phi^* \pi_Y^* \_F^\vee \rightarrow \pi_X^* f^* \_F^\vee$ being the equivalence coming from the fact that $\pi_Y \circ \Phi \simeq f \circ \pi_X$.  
\end{Lem} 
\begin{proof}
	The first thing we observe is that $\Aa_X (f^* \_F) \simeq f^* \Aa_X(\_F)$. We consider as before $B$-points: 
	
	\[ \begin{tikzcd}
		\Spec(B) \arrow[r, "y"] \arrow[rr, bend left, "\tilde{y}"] \arrow[dr, "x"] \arrow[drr, bend right, "\tilde{x}"' ] & f^*\Aa_X( \_F) \arrow[r, "\Phi"] \arrow[d, "\pi_X"] & \Aa_X (\_F) \arrow[d, "\pi_Y"] \\
		& X \arrow[r, "f"] & Y
	\end{tikzcd} \]

	We want to show that the following diagram is commutative: 
	
	\begin{equation}
		\label{dia:Dia_1}
		\begin{tikzcd}
			\iHom_{B} \left( y^* \Llr{ \Aa_X (f^*\_F)}{X}, M \right) \arrow[r] \arrow[d, "\simeq"] & \iHom_{B} \left( \tilde{y}^* \Llr{ \Aa_X (\_F)}{Y}, M \right) \arrow[d, "\simeq"] \\
			\iHom_{B} \left( y^* \pi_X^* f^* \_F^\vee, M \right)  \arrow[r] & \iHom_{B} \left( \tilde{y}^* \pi_Y^* \_F^\vee, M \right)
		\end{tikzcd}
	\end{equation}
	
	Using the universal property of the cotangent complex, the top horizontal arrow is naturally equivalent to the map
	
	\[ \begin{tikzcd}
		\tx{hofiber}_{y} \left( \iHomsub{\faktor{\dst}{X}} \left( \Spec(B \boxplus M), \Aa_X (f^* \_F) \right) \rightarrow  \iHomsub{\faktor{\dst}{X}} \left( \Spec(B), \Aa_X (f^* \_F) \right) \right) \arrow[d] \\
		\tx{hofiber}_{\tilde{y}} \left( \iHomsub{\faktor{\dst}{Y}} \left( \Spec(B \boxplus M), \Aa_X (\_F) \right) \rightarrow  \iHomsub{\faktor{\dst}{Y}} \left( \Spec(B), \Aa_X (\_F) \right) \right) 
	\end{tikzcd} \]
	
	induced by $\Homsub{\dst}\left( -, \Phi \right)$. A map $\psi : \Spec(B \boxplus M) \rightarrow \Aa_X(f^* \_F)$ in this homotopy fiber fits in the following commutative diagram:
	
	\[ \begin{tikzcd}
		\Spec(B) \arrow[rd, "y"'] \arrow[rrd, "\tilde{y}"] \arrow[d, "i"] & & \\
		\Spec(B \boxplus M) \arrow[d, "p"] \arrow[r, "\psi"'] & \Aa_X (f^* \_F ) \arrow[d, " \pi_X" ] \arrow[r, "\Phi"'] & \Aa_X (\_F) \arrow[d, "\pi_Y"] \\
		\Spec(B) \arrow[r, "x"] & X \arrow[r, "f"] & Y
	\end{tikzcd} \] 
	
	and the map between the homotopy fiber sends $\psi$ to $\Phi \circ \psi$. Since the underlying map of $\psi$ is $\pi_X \circ \psi : \Spec(B \boxplus M) \rightarrow X$ is $x \circ p$ and the underlying map of $\Phi \circ \psi$ is $\pi_Y \circ \Phi \circ \psi : \Spec(B \boxplus M) \rightarrow Y$ is $f\circ x \circ p = \tilde{x} \circ p$, this map between the homotopy fiber of derived stacks is therefore naturally equivalent to the map: 
	
	\[ \begin{tikzcd}
		\tx{hofiber}_{s_y} \left( \iHom_{B \boxplus M} \left( B \boxplus M, p^*x^* f^* \_F) \right) \rightarrow  \iHom_{B} \left( B,  p^*x^* f^* \_F \right) \right) \arrow[d] \\
		\tx{hofiber}_{s_{\tilde{y}}} \left( \iHom_{B \boxplus M} \left( B \boxplus M, p^*\tilde{x}^* \_F) \right) \rightarrow  \iHom_{B} \left( B,  p^* \tilde{x}^* \_F \right) \right)
	\end{tikzcd} \]
	
	where $s_y $ and $s_{\tilde{y}}$ are the sections associated to $y$ and $\tilde{y}$ respectively. This map is in fact induced by the natural identification $ p^* \tilde{x}^* \_F \simeq  p^*x^* f^* \_F$ (since $\tilde{x} = f\circ x$). But following the steps of the proof of Proposition \ref{prop:relative cotangent complex for linear stacks}, this map is naturally equivalent to the map 
	
	\[ \Hom_{B} \left( y^* \pi_X^* f^*  \_F^\vee , M \right) \rightarrow \Hom_{B} \left( \tilde{y}^* \pi_Y^* \_F^\vee, M \right)\]

	The natural equivalences we used are the natural equivalences used in the proof of Proposition \ref{prop:relative cotangent complex for linear stacks} which proves that the Diagram \eqref{dia:Dia_1} is commutative. Now the result follows once again from the fact that the functor 
	
	\[ \begin{tikzcd}[row sep=tiny]
		\Mod_B \arrow[r] & \tx{Fun}\left( \Mod_B^{\leq 0}, \tx{sSet} \right) \\
		N \arrow[r, mapsto] & \Map_{B} \left( N, \bullet \right)
	\end{tikzcd} \]\\
	
	is fully-faithful and the fact that everything we did is natural in $B$. 
\end{proof}

\begin{Lem}
	\label{lem:relative cotangent naturality wrt fiber morphisms}
	Let $X$ be a derived Artin stacks. We consider $\_F, \_G \in \QC(X)$ dualisable and $h : \_F \rightarrow \_G$. Then there is a commutative square: 
	
	\[ \begin{tikzcd}
		\hat{h}^*\Llr{\Aa_X (\_G )}{X} \arrow[r] \arrow[d, "\simeq"] &  \Llr{\Aa_X ( \_F)}{X} \arrow[d, "\simeq"] \\
		\pi_X^* \_G^\vee \arrow[r, "\pi_X^* h^\vee"] &  \pi_X^* \_F^\vee \end{tikzcd} \]
	
	with $\hat{h} : \Aa_X ( \_G) \rightarrow \Aa_X (\_F)$ the map induced by $h$. 
\end{Lem} 

\begin{proof}
	Every step of the proof of Proposition \ref{prop:relative cotangent complex for linear stacks} is functorial in $\_F \in \perf(X)$.
\end{proof}

\begin{Prop}
	
	\label{prop:relative cotangent complex functoriality for linear stacks} 
	Let $f: X \rightarrow Y$ be a morphism of derived Artin stacks. We consider $\_F \in \QC(X)$ and $\_G \in \QC(Y)$ both dualizable and a morphism $h : f^* \_F \rightarrow \_G$. Then there is a commutative square: 
	
	\[ \begin{tikzcd}
		\Llr{\Aa_X ( \_F )}{X} \arrow[r] \arrow[d, "\simeq"] & \hat{f}^* \Llr{\Aa_X ( \_G )}{Y}\arrow[d, "\simeq"] \\
		\pi_X^* \_F^\vee \arrow[r, "\pi_X^* h^\vee"] & \pi_X^* f^* \_G^\vee = \hat{f}^* \pi_Y^* \_G^\vee
	\end{tikzcd} \]
\end{Prop}
\begin{proof}
	It follows from Lemma \ref{lem:relative cotangent base change} and Lemma \ref{lem:relative cotangent naturality wrt fiber morphisms}
\end{proof}

\begin{RQ}\
	\label{rq:relative cotangent for semi-linear stacks}
	If two semi-free cofibrant algebras over $X$ have the same underlying cofibrant graded free algebra (forgetting the differential), then their cotangent complexes have (for the model provided by the given choice of cofibrant semi-free algebra) the same underlying graded $\_O_X$-module. In particular, if $Y := \Spec_X(\Sym_{\_O_X} \_F^\vee)$ is a semi-linear stack with differential $\delta$. Up to taking a cofibrant replacement we can assume that $\_F^\vee$ is projective and we get: 
	\[ \Llr{Y}{X} \simeq (\pi^* \_F^\sharp, \delta_{\tx{lin}})\]
	where $\pi^* \_F^\sharp$ is endowed with a differential $\delta_{\tx{lin}}$ induced by the differential $\delta$ on the semi-free algebra. In particular, $\delta_{\tx{lin}}$ contains the differential on $\_F$ but also extra terms. It is of the form: 
	\[\delta_{\tx{lin}} : \_F \to \_F \otimes_{\_O_X} \Sym_{\_O_X} \_F^\vee\]
	therefore, this differential can be decomposed along the natural grading of the symmetric algebra (see Notation \ref{not:graded symmetric algebra}) on the right hand side with the weight zero part $\_F \to \_F$ recovering the differential on $\_F$.
\end{RQ}

\subsubsection{Connections and tangent complexes of (semi-)linear stacks} 
\label{sec:connections-on-semi-linear-stacks}\

\medskip

The goal of this section is to have a way to compute the tangent complex of a linear (or semi-linear) stack. To do so requires the use of a connection, enabling us to split the tangent in the fiber and the base parts of the linear stack. Unfortunately, connections may not always exist in algebraic geometry. We will see that they exist if the base is affine and the module of sections is projective (see Proposition \ref{prop:connection existence}).   
In this section, $X$ will be a derived stack admitting a cotangent complex.

\begin{Def} 
	\label{def:ehresmann connection}

	Let $\_F \in \QC(X)$ with $X$ a derived stack. An \defi{Ehresmann connection} on the linear stack\footnote{Recall from Notation \ref{not:linear stack and their sheaves} that $F$ denotes the linear stack $\Aa_X(\_F)$.} $\pi :  F \rightarrow X$ is given by a splitting on the following exact sequence of quasi-coherent sheaves on $F$:
	
	\[ \begin{tikzcd}
		V  \arrow[r] &  \Tt_{F} \arrow[r, "\pi_*", shift left]&  \arrow[l, shift left, "s"] \pi^* \Tt_X 
	\end{tikzcd}\] 
	
	where $V$ is the fiber of $\ker (\pi_*)$ and is called the \defi{vertical space}. Note that by we have $V \simeq \Ttr{F}{X}$ and using Proposition \ref{prop:relative cotangent complex for linear stacks}, such a connection is a splitting of the fiber sequence:
	\[ \begin{tikzcd}
		\pi^*\_F  \arrow[r] &  \Tt_{F}\arrow[r, "\pi_*", shift left]&  \arrow[l, shift left, "s"] \pi^* \Tt_X 
	\end{tikzcd}\]  
\end{Def}

\begin{Def}
	\label{def:covariant derivative}
	
	A \defi{covariant derivative} of the linear stack $F \rightarrow X$ is a map of quasi-coherent sheaves over $X$: 
	\[ \nabla :  \Tt_X \rightarrow  \bf{End}_k(\_F) \]
	
	such that for all $X \in \Tt_X$, $\nabla_X$ is in $\bf{End}_k(\_F))$ and satisfies the Leibniz rule making it a degree $\vert X \vert$ derivation: 
	
	\[ \nabla_X (fs) = X(f).s + (-1)^{\vert X \vert \dot \vert f \vert }f \nabla_X s \]
\end{Def}

\begin{notation}
	\label{not:non-dg map}
	
	Let $A$ and $B$ be differential graded objects (such as dg-module, cdga etc...). We denote by $A^\sharp$ and $B^\sharp$ their underlying graded objects where we ``forget'' the differential. 
	
	We denote by $\ndg{f} : A \to B$ a map that does not necessarily respects the differential. 
\end{notation}

\begin{Def} 
	\label{def:connection semi linear stack}
	
	A connection on a semi-linear stack $Y$ over $X = \Spec(B)$ is again given by a splitting: 
	\[ \begin{tikzcd}
		V \arrow[r] & \Tt_Y \arrow[r, shift left, "\pi_*"] & \arrow[l, shift left, "\ndg{s}"] \pi^* \Tt_X
	\end{tikzcd} \]
	where this time $ \ndg{s}$ is not required to respect the differential (Notation \ref{not:non-dg map}). In other word, a connection is the data of a section of the underlying graded modules.    
\end{Def}

\begin{Cons}
	\label{cons:connection from covariant derivative}
	
	We will briefly recall how covariant derivatives induce Erhesmann connections. 
	Take a covariant derivative $\nabla$. If we pullback the map $\Tt_X \rightarrow \bf{End}_k(\_F)$ by $\pi$, we associate to each $X$ in $\pi^* \Tt_X$ an element $\nabla_X$, which can be extended by the Leibniz rule to a derivation in $\Tt_{F}$ via the following formula:
	
	\[ \nabla_X (fv_1 \cdots v_n) = X(f) v_1 \cdots v_n + \sum_{i=1}^n \pm f v_1 \cdots ( \nabla_X v_i ) \cdots v_n\]  
	
	It remains to see that this defines a section. Since $\pi_*$ annihilates $\_F$ and $\nabla_X$ sends elements in $\_F$ to elements in $\_F$, we get that $\pi_* \nabla_X v = 0$ for all $v$ in $\_F$. Moreover, $\pi_* \nabla_X f  = X(f) $ therefore $\pi_* (\nabla_X)_{\vert \_O_X}=X$ in $\pi^* \Tt_X$ and therefore this defines a section. 
\end{Cons}

Using Construction \ref{cons:connection from covariant derivative} and the splitting induced by a connection, we can \emph{almost} identify the tangent complex of linear and semi-linear stacks with a direct sum of $\_F$ and $\Tt_X$.

\begin{Lem}\label{lem:connection splitting linear stacks}
	Given a perfect linear stack $F$, we have an isomorphism: 
	\[ \Tt_{F} \cong \pi^* (\Tt_X \oplus^\nabla \_F) \]
	where the underlying graded, $\pi^* (\Tt_X \oplus^\nabla \_F)^\sharp$, is given by the usual direct sum $\pi^* (\Tt_X \oplus \_F)^\sharp$. This isomorphism sends: \begin{itemize}
		\item $\phi \in \Tt_F$ to $(\phi_{\vert X}, \phi - \phi_{\vert X} - \nabla_X)$. 
		\item $(X, v)$ to the derivation $X + \nabla_X + \iota_v$ with $\nabla_X \in \pi^*\_F$ where $\nabla_X \in \_F \otimes \_F^\vee$ is viewed a $\Sym_{\_O_X}^{\geq 1} \_F^\vee$-valued derivation given by the contraction along the terms in $\_F$.
	\end{itemize} 
	
	Moreover, the differential on $\pi^* (\Tt_X \oplus^\nabla \_F)$ is given by the differential on $\Tt_X$, $\_F$ plus a term of connection $\delta_\nabla$  sending $\Tt_X$ to $\_F \otimes \_F^\vee \subset \pi^* \_F$ (in particular, $\delta_{\nabla}$ is valued in $\_F \otimes \Sym_{\_O_X}^{\geq 1} \_F^\vee$).   
\end{Lem}

\begin{Lem}
	\label{lem:tangent complex semi-linear stacks}
	For $Y$ a perfect semi-linear stack, the splitting: 
	\[ \begin{tikzcd}
		V \arrow[r] & \Tt_Y \arrow[r, shift left, "\pi_*"] & \arrow[l, shift left, "\ndg{s}"] \pi^* \Tt_X
	\end{tikzcd} \]
	
	Induces an isomorphism:
	
	\[ \begin{tikzcd}
		\Tt_Y \arrow[r, shift left] & \arrow[l, shift left] \pi^* (\_F \oplus^\nabla \Tt_X)
	\end{tikzcd}\]
	where the underlying graded, $\pi^* (\Tt_X \oplus^\nabla \_F)^\sharp$, is given by the usual direct sum $\pi^* (\Tt_X \oplus \_F)^\sharp$ (where we forgot the differential). This isomorphism sends: 
	\begin{itemize}
		\item $\phi \in \Tt_Y$ to $(\phi_{\vert X}, \phi - \phi_{\vert X} - \nabla_X)$. 
		\item $(X, v)$ to the derivation $X + \nabla_X + \iota_v$ with $\nabla_X \in \pi^*\_F$ with $\nabla_X \in \_F \otimes \_F^\vee$ is viewed a $\Sym_{\_O_X}^{\geq 1} \_F^\vee$-valued derivation given by the contraction along the terms in $\_F$.
	\end{itemize} 
	The differential on the right hand side is obtained from the differential on $\Tt_Y$ by transfer along the isomorphism of the underlying graded complexes. We obtain a non-canonical differential on $\left(\pi^* (\_F \oplus \Tt_X)\right)^\sharp$ extending the natural one. It will in particular recover the differential on $\Tt_X$, on $\_F$, the term of connection $\delta_\nabla$ and extra terms coming from the ``non-linear'' parts of the semi-free differential. 
\end{Lem}

\begin{Prop} \label{prop:connection existence}
	
	Take $X = \bf{Spec}(A)$ affine, then any projective $A$-module $M$ admits a connection. 
\end{Prop}
\begin{proof}
	This is a differential graded adaptation of \cite[Proposition 2.3 and 2.5]{RM17}. Indeed, we can clearly define connections on free $A$-modules, and projective modules are also direct summand of free ones (in particular of their \emph{free resolution}) in the differential graded case. 
\end{proof}

In particular, any semi-linear stacks over an affine base $X= \Spec(A)$ can be resolved as a semi-free algebra on $A$ with projective module which therefore admits a connection.

\begin{Ex}
	\label{ex:tangent complex of the derived critical locus}
	
	Lemma \ref{lem:tangent complex semi-linear stacks} ensures that the tangent of the derived critical locus is: \[\Tt_{\RCrit(f)} \simeq \ \pi^*(\Tt_X \oplus^\nabla \Ll_X[-1])\] and therefore this is given by $\pi^*(\Tt_X \oplus \Ll_X[-1])^\sharp$ together a differential depending on the chosen connection. We could try to use this Lemma to work out the differential but there is a more conceptual approach to that question using the easier description of the differential obtained from a connection on a \emph{linear stacks} (Lemma \ref{lem:connection splitting linear stacks}).
	
	The tangent complex of a pullback is the pullback of the tangent complexes (see Proposition \ref{prop:pushout cotangent complexes in M}). Therefore have the pullback: 
	\[\begin{tikzcd}
		\Tt_{\RCrit(f)} \arrow[r] \arrow[d] & i^*\Tt_X \arrow[d] \\
		i^*\Tt_X \arrow[r] &  i^* s_0^*\Tt_{T^*X} \simeq i^* df^* \Tt_{T^*X} 
	\end{tikzcd} \]
	
	where $i: \RCrit(f) \to X$ and $s_0, df: X \to T^*X$.	
	Using a connection on the linear stack $T^*X$, we have an isomorphism 
	\[\Tt_{T^*X} \cong \pi^*(\Tt_X \oplus^\nabla \Ll_X)\] and the differential is given by the differentials on $\Tt_X$, $\Ll_X$ and a term coming from the connection, that we will denote by $\delta_\nabla$, sending $X \in \Tt_X$ to $\nabla_X$ a $\Sym_{\_O_X}^{\geq 1} \Tt_X$-valued $1$-form on $X$, that is, an element in $\pi^*\Ll_X$.  
	
	Computing this pullback, along the maps: 
	\[i^*\Tt_X \to  i^*(\Tt_X \oplus \Ll_X)\simeq i^* s^* \pi^*(\Tt_X \oplus \Ll_X)\] (with $s = s_0$ or equivalently $s=df$) given by the inclusion on one hand and the map $X \mapsto (X, \dr(df.X))$ on the other hand. We get:
	\[ \Tt_{\RCrit(f)} \cong \pi^*(\Tt_X \oplus^\nabla \Ll_X[-1])\]
	and the differential is given by the differential on $\Tt_X$, on $\Ll_X[-1]$, $\delta_\nabla$ and the extra map $\Tt_X \to \Ll_X[-1]$ given by $\dr \circ \iota_{df}$.  
\end{Ex}

\subsection{Formal Derived Geometry}\label{sec:formal-geometry}\

\medskip

We are interested in formal geometry as formal stacks appear naturally in the study of ``infinitesimal quotients by Lie algebroids'' in Section \ref{sec:lie-and-linfty-algebroids}. \\

We start in Section \ref{sec:formal-derived-stacks-and-formal-neighborhood} by recalling the important definitions and properties of formal stacks, formal completions and formal thickenings.   

In Section \ref{sec:formal-stack-from-formal-moduli-problems}, we will be interested in the type of formal stacks that arise from \emph{formal moduli problems}, as these will be the ``infinitesimal quotient'' stacks associated to a Lie algebroid. In particular, we will be interested in the formal stacks associated to formal spectra (Definition \ref{def:formal spectrum}) as we will be interested in the case of  formal spectra (see Section \ref{sec:quotient-stack-of-a-lie-algebroid}).\\

In this section all derived stacks and algebras will be locally of almost finite presentation.  

\newpage

\subsubsection{Formal derived stacks and formal completions}\ \label{sec:formal-derived-stacks-and-formal-neighborhood}

\medskip

\begin{Def}[{\cite[Definition 2.1.1]{CPTVV}}]  
	\label{def:formal derived stack}
	
	A \defi{formal derived stack $X$} is an almost finitely presented derived stack $X \in \dstfp$ satisfying the following conditions: 
	\begin{itemize}
		\item $X$ is \defi{nilcomplete}, that is for all $B \in \cdgacon$, we have a weak equivalence of spaces: 
		\[ F(B) \to \lim\limits_{k} F(B_{\leq k}) \]
		where $B_{\leq k}$ is the $k$-th Postnikov\footnote{A Postnikov tower for $B \in \cdgacon$ is a sequence: \[B \to \cdots \to B_{\leq n} \to B_{\leq n-1} \to \cdots \to B_0 := \pi_0(B)\] such that $H^{-i}(B_{\leq n}) = 0$ for all $i > n$ and the morphism $B \to B_{\leq n}$ induces isomorphism on $H^{-i}$ for $i \leq n$. $B_{\leq n}$ is called the $n$-th Postnikov truncation of $B$.} truncation of $B$. 
		\item $X$ is \defi{infinitesimally cohesive}, that is for all Cartesian square of almost finitely generated cdgas in non-positive degree, 
		\[ \begin{tikzcd}
			B \arrow[r] \arrow[d] & B_1 \arrow[d] \\
			B_2 \arrow[r] & B_0
		\end{tikzcd}\] 
		such that $H^0(B_i) \to H^0 (B_0)$ are surjective with nilpotent kernel, the following induced diagram in spaces is Cartesian: 
		
		\[ \begin{tikzcd}
			X(B) \arrow[r] \arrow[d] &X(B_1)\arrow[d] \\
			X(B_2) \arrow[r] & X(B_0)
		\end{tikzcd}\]   
	\end{itemize}
\end{Def}

\begin{RQ}\label{rq:formal prestacks and cotangent}
	This definition also makes sense for derived \emph{pre-stacks}, in which case, they are called \defi{formal derived pre-stacks}. Moreover, in \cite[Section]{CG18}, the derived pre-stacks are assumed to admit a pro-cotangent complex. In most situation we are interested in, we are even going to assume that we have a cotangent complex. 
\end{RQ}

\begin{RQ} 
	If $B_i$ are also Artin algebras (in $\cdgacon$), then from \cite[Proposition 1.1.11 and Lemma 1.1.20]{Lu11}, saying that $H^0(B_i) \to H^0(B_0)$ are surjective is exactly the condition for a morphism to be \emph{small} according to Definition \ref{def:small objects and morphisms} (the nilpotent kernel condition being automatic for Artin algebras).  
\end{RQ}

\begin{RQ} \label{rq:infinitesimal cohesinveness from formal stack to fmp}
	From \cite[Remark 2.1.2]{CPTVV} and \cite[Proposition 2.1.13]{Lu12}, if we assume that a derived (pre)-stack $X$ has a cotangent complex, then $X$ is formal if and only if it is nilcomplete and satisfies the infinitesmal cohesiveness, where only one of the maps $B_i \to B_0$ is surjective in $H^0$ with nilpotent kernel. 
	
	In particular, the restriction of $X$ to Artin algebras satisfies the pullback condition of a formal moduli problem (Definition \ref{def:FMP}).  
\end{RQ}

\begin{Ex}\
	\label{ex:formal derived stacks}
	
	\begin{itemize}
		\item Any Artin stack is formal (see \cite[Section 2.1]{CPTVV}).
		\item All small limits of formal derived stacks are formal. 
		\item The de Rham stack, $F_{\tx{DR}}$ (Definition \ref{def:de rham and reduced stacks adjunction}) is formal.
		\item The formal completion of a map of derived stacks $X \to Y$ (Definition \ref{def:formal completion of stacks}), with $Y$ formal, is also formal (see Proposition \ref{prop:formal stack and de rham stack and formal completions}). 
		\item The quotient stack of $X$ by a Lie algebroid is by definition formal (see Section \ref{sec:quotient-stack-of-a-lie-algebroid}). 
	\end{itemize}
\end{Ex}

In formal geometry we care about infinitesimal neighborhoods, which we will call formal thickenings. Many example of formal thickenings are obtained via the formal completion of a map (Definition \ref{def:formal completion of stacks}). \\

To manipulate these objects, we need the construction of the ``de Rham'' and ``reduced'' stacks. Essentially, the de Rham stack of $X$ is the formal completion of the map $X \to \star$ (thanks to Proposition \ref{prop:formal completion and de rham stacks}) and the reduced functor is the left adjoint to the de Rham functor. \\

The idea is that we want to make sense of ``infinitesimally close'' points, which translates into saying that the ``difference'' of infinitesimally close points is nilpotent.  

\begin{Def}
	\label{def:reduced commutative algebra}
	
	An algebra $A \in \cdgacon$ is called \defi{reduced} if it is discrete\footnote{A cdga is called discrete if it is concentrated in degree $0$} and has no non-zero nilpotent element. In other words, $\tx{Nil}(A) =0$ where $\tx{Nil}(A)$ denotes the ideal of nilpotent elements in $A$. We denote by $(\cdgacon)^{\tx{red}}$ the full sub-category of $\cdgacon$ of reduced algebras.   
\end{Def}

\begin{Lem} 
	\label{lem:reduced algebra functor}
	
	The inclusion $i:(\cdgacon)^{\tx{red}} \to \cdgacon$ admits a left adjoint: \[(-)^{\tx{red}} : \cdgacon \to (\cdgacon)^{\tx{red}}\]
	defined by:
	\[A^{\tx{red}} := \faktor{H^0(A)}{\tx{Nil}(H^0(A))}\]
	
	We drop the $i$ from the notations so that the unit of the adjunction induces a map: 
	\[ A \to A^{\tx{red}} \]
\end{Lem}

One of the important notion in formal geometry is the formal neighborhood of a map $f: X \to Y$ which essentially is the stack whose points are points of $Y$ which are infinitesimally close to points of $X$. 

\begin{Def} \label{def:formal completion of stacks}
	The formal completion of a map $f:X \to Y$, denoted by $\comp{Y_X} \in \dstfp$, is defined by sending an almost finitely presented algebra $A \in \cdgacon$ to $\comp{Y_X}(A)$, defined as the space of all commutative diagrams: 
	\[  \begin{tikzcd}
		\Spec(A^{\tx{red}}) \arrow[r] \arrow[d] & X \arrow[d] \\
		\Spec(A) \arrow[r] & Y
	\end{tikzcd}  \]
	In other words, it is defined as the space: 
	\[ \comp{Y_X}(A) := \Map\left(\Spec(A), Y\right)\times_{\Map\left(\Spec(A^{\tx{red}}), Y\right)} \Map\left(\Spec(A^{\tx{red}}), X\right) \]
	These are the points of $Y$ that are ``infinitesimally closed'' to a point of $X$ where two points $x, y \in Y(A)$ are infinitesimally close\footnote{For example if $A := k[\epsilon]$, with $\epsilon^2 =0$, then a $k[\epsilon]$-point is a tangent vector of $Y$. Two such points are then infinitesimally close if they have the same underlying $k$-point in $Y$.} if they have the same image under $Y(A) \to Y(A^{\tx{red}})$.
\end{Def}

Following the idea of the definition of the formal completion, we want $Y(A^{\tx{red}})$ to represent the equivalence classes of points for the equivalence relation ``being infinitesimally close''. Precomposing by the reduction functor, we get a new functor (see \cite[Section 2.1]{CPTVV}): 
\[ i^* : \dstfp \to (\dstfp)^{\tx{red}}\]
where $(\dstfp)^{\tx{red}}$ denotes the $\infty$-category of derived stacks over the category of connective reduced almost finitely presented algebras. \[((\cdga^{\leq 0, \tx{afp}})^{\tx{red}})^{\tx{op}}\]
We have that $i^*$ possesses both a right adjoint $i_*$ and a left adjoint $i_!$ satisfying: 
\begin{itemize}
	\item $i_* \simeq ((-)^{\tx{red}})^*$, so that $i_*(\Spec(A)) \simeq \Spec(A^{\tx{red}})$. 
	\item $i_*$ and $i_!$ are fully-faithful. 
	\item $i^* \simeq ((-)^{\tx{red}})_!$. 
\end{itemize}

\begin{Def}[{\cite[Definition 2.1.3]{CPTVV}}] 
	\label{def:de rham and reduced stacks adjunction}
	
	Let $X \in \dstfp$ be a derived stack. 
	\begin{itemize}
		\item The \defi{de Rham stack} of $X$ denoted $X_{\tx{DR}}$ is defined as: 
		\[ (-)_{\tx{DR}} := i_* i^* : \dstfp \to \dstfp\]
		
		The unit of the adjunction induces natural maps $X \to X_{\tx{DR}}$. Moreover we have $X_{\tx{DR}}(A) = X ( A^{\tx{red}})$. 
		\item The \defi{reduced stack} of $X$ denoted $X_{\tx{red}}$ is defined as: 
		\[ (-)_{\tx{red}} := i_! i^* : \dstfp \to \dstfp\]
		
		The counit of the adjunction induces natural maps $X_{\tx{red}} \to X$. Moreover we have $\Spec(A)_{\tx{red}} = \Spec(A^\tx{red})$.
	\end{itemize}
\end{Def}

\begin{Prop}
	\label{prop:de rham and reduced stacks}
	
	We have that $i_* $ and $i_!$ are fully-faithful and therefore the unit $\id \to i^* i_!$ and the counit $i^* i_* \to \id$ are object wise equivalences. This implies that $i_! i^*$ is left adjoint to $i_* i^*$ and in particular, this shows that $(-)_{\tx{DR}}$ is right adjoint to $(-)_{\tx{red}}$. 
\end{Prop}

\begin{Prop}\label{prop:formal completion and de rham stacks}
	
	The formal completion of $f : X \to Y$ is part of the pullback square: 
	\[ \begin{tikzcd}
		\comp{Y_X} \arrow[r] \arrow[d] & X_{\tx{DR}} \arrow[d] \\
		Y \arrow[r] & Y_{\tx{DR}}
	\end{tikzcd}\]
\end{Prop}
\begin{proof}
	By definition the space of $A$-points in $\comp{Y_X}$ is:
	\[\begin{split}
		\comp{Y_X}(A) \simeq &	\Map\left(\Spec(A), Y\right)\times_{\Map\left(\Spec(A^{\tx{red}}), Y\right)} \Map\left(\Spec(A^{\tx{red}}), X\right) \\
		\simeq & \Map\left(\Spec(A), Y\right)\times_{\Map\left(\Spec(A)_{\tx{red}}, Y\right)} \Map\left(\Spec(A)_{\tx{red}}, X\right) \\
		\simeq & \Map\left(\Spec(A), Y\right)\times_{\Map\left(\Spec(A), Y_{\tx{DR}}\right)} \Map\left(\Spec(A), X_{\tx{DR}}\right) \\
		\simeq & \Map\left(\Spec(A), Y \times_{Y_{\tx{DR}}} X_{\tx{DR}}\right)\\
		\simeq &\left( Y \times_{Y_{\tx{DR}}} X_{\tx{DR}}\right)(A)
	\end{split}\]
\end{proof}

\begin{Cor}\label{cor:de rham stack as formal completion}
	$X_{\tx{DR}}$ is the formal completion of the terminal morphism $X \to \star$: 
	\[ X_{\tx{DR}} \simeq \comp{\star_X}\]
\end{Cor}

\begin{Cor}\label{cor:formal completion and pullacks}
	Consider a cone $X_1 \times_X X_2$ over a pullback $Y_1 \times_Y \times Y_2$ induced by compatible maps $X_1 \to Y_1$, $X\to Y$ and $X_2 \to Y_2$. Then we have an equivalence: \[ \comp{\left(Y_1 \times_Y Y_2\right)_{X_1 \times_X X_2}} \simeq \comp{(Y_1)_{X_1}} \times_{\comp{Y_X}} \comp{(Y_2)_{X_2}} \]
\end{Cor}
\begin{proof}
	\[\begin{split}
		\comp{\left(Y_1 \times_Y Y_2\right)_{X_1 \times_X X_2}}\simeq & \left(Y_1 \times_Y Y_2\right) \times_{\left(Y_1 \times_Y Y_2\right)_{\tx{DR}}} \left(X_1 \times_X X_2\right)_{\tx{DR}} \\
		\simeq &  \left(Y_1 \times_Y Y_2\right) \times_{(Y_1)_{\tx{DR}} \times_{Y_{\tx{DR}}} (Y_2)_{\tx{DR}}} \left((X_1)_{\tx{DR}} \times_{X_{\tx{DR}}} (X_2)_{\tx{DR}} \right)\\
		\simeq &  \left(Y_1 \times_{(Y_1)_{\tx{DR}}} (X_1)_{\tx{DR}} \right) \times_{Y \times_{Y_{\tx{DR}}} X_{\tx{DR}}} \left(Y_2 \times_{(Y_2)_{\tx{DR}}} (X_2)_{\tx{DR}} \right)\\
		\simeq & \comp{(Y_1)_{X_1}} \times_{\comp{Y_X}} \comp{(Y_2)_{X_2}} 
	\end{split}\]
\end{proof}

\begin{Prop}[{\cite[Proposition 2.1.4]{CPTVV}}]
	\label{prop:formal stack and de rham stack and formal completions}
	
	We have the following:
	\begin{enumerate}
		\item  $X_{\tx{DR}}$ is a formal derived stack. 
		\item If $Y$ is a formal derived stack, then for any map $f: X \to Y$ the formal completion $\comp{Y_f}$ is a formal derived stack. 
	\end{enumerate}
\end{Prop}

Taking the de Rham stack forget about infinitesimally closed points and therefore, forgets about the ``tangent directions'' of the stack. 
\begin{Lem}[{\cite[Lemma 2.1.10]{CPTVV}}] 
	\label{lem:cotangent complex de rham stack}
	
	The cotangent complex of a de Rham stack exists and we have: 
	\[ p^*\Ll_{X_{\tx{DR}}} \simeq 0 \]
	where $p: X \to X_{\tx{DR}}$ is the natural projection. 
\end{Lem}

Therefore the tangent complex, and the de Rham stack are complementary informations. We will see in the rest of this section that among formal stacks, $X_{\tx{DR}}$ and $\Tt_X$ \emph{characterize} $X$ in the sens that if $X$ and $Y$ are formal stacks such that $X$ and $Y$ have the ``same'' de Rham stacks and tangent complexes, then they are equivalent as stacks. In the rest of this section, we are going to assume that $X$ and $Y$ have a tangent complex. 

\begin{Def}
	\label{def:nilisomorphism}
	A map $f :X \to Y$ of derived stacks is called a \defi{nil-equivalence} if it induces an equivalence $f_{\tx{red}} : X_{\tx{red}} \to Y_{\tx{red}}$ (or equivalently an equivalence $f_{\tx{DR}} : X_{\tx{DR}} \to Y_{\tx{DR}}$).  
\end{Def}

\begin{Def}\label{def:formally etale maps}
	A morphism $f: X \to Y$ is called \defi{formally étale} if it induces an equivalence: 
	\[ \Tt_X \overset{\sim}{\to} f^*\Tt_Y \]
\end{Def}

\begin{Def}\label{def:formal thickening}
	We define a \defi{formal thickening} of a formal stack $X$ as either: 
	\begin{itemize}
		\item A \emph{formal pre-stack}, $Y \in \dpstfp$, together with a nil-equivalence $f: X \to Y$. The category of such thickenings will be denoted $\thickp(X)$. 
		\item A \emph{formal stack}, $Y \in \dstfp$, together with a nil-equivalence $f: X \to Y$. The category of such thickenings will be denoted $\thick(X)$.
	\end{itemize}
	The stackification $\dpstfp \to \dstfp$ might not preserve the property of being formal. As such, we will often consider stacks given by the stackification of a formal thickening in pre-stack.
\end{Def}

\begin{RQ}
	If $f: X \to Y$ is a nil-equivalence, then there is an equivalence: \[\comp{Y_X} \simeq Y\]
	
	This a direct consequence of Proposition \ref{prop:formal completion and de rham stacks} and the fact that the map $X_{\tx{DR}} \to Y_{\tx{DR}}$ is an equivalence.  
\end{RQ}

\begin{Lem}\label{lem:formal completion factorization and morphims properties}
	Given $f: X \to Y$, the formal completion induces a factorization: 
	\[ X \to \comp{Y_X} \to Y \]
	such that: 
	
	\begin{itemize}
		\item $X \to \comp{Y_X}$ is a nil-equivalence.  
		\item $\comp{Y_X} \to Y $ is formally étale. 
	\end{itemize} 
\end{Lem}
\begin{proof}The factorization is clear from the pullback in Proposition \ref{prop:formal completion and de rham stacks} ($X$ is clearly a cone over this pullback). Then we have that:
	
	\begin{itemize}
		\item  Since $X$ is a cone over the pullback defining $\comp{Y_X}$, applying $(-)_{\tx{DR}}$, we get a retract\footnote{using the fact that $(X_{\tx{DR}})_{\tx{DR}} \simeq X_{\tx{DR}}$.}:
		\[ X_{\tx{DR}} \to \left(\comp{Y_X}\right)_{\tx{DR}} \to X_{\tx{DR}}\]
		such that the second morphism is an equivalence (as pullback of the identity of $Y_{\tx{DR}}$). Therefore the map $X_{\tx{DR}} \to \left(\comp{Y_X}\right)_{\tx{DR}}$ is also an equivalence.
		\item The fact that $\comp{Y_X} \to Y$ is formally étale follows from Lemma \ref{lem:cotangent complex de rham stack} and the fact that, thanks the pullback of Proposition \ref{prop:formal completion and de rham stacks}, we have an equivalence: 
		\[\Ttr{\comp{Y_X}}{Y}  \overset{\sim}{\to}  p^*\Ttr{X_{\tx{DR}}}{Y_{\tx{DR}}} \simeq 0 \]
	\end{itemize}
\end{proof}

\begin{Lem}
	\label{lem:niliso + etal = isom}
	
	Let $f: X \to Y$ a map of formal pre-stacks such that:
	
	\begin{itemize}
		\item $f$ is a nil-equivalence. 
		\item $f$ is formally étale. 
	\end{itemize}
	then $f$ is a weak equivalence. 
\end{Lem}
\begin{proof}
	The idea of the proof is to show that for each $\Spec(A) \to Y$ there is a unique lift:
	\[ \begin{tikzcd}
		& X\arrow[d,"f"] \\
		\Spec(A) \arrow[ur, dashed] \arrow[r] & Y
	\end{tikzcd}\]
	
	First observe that from nil-completeness of our formal stacks, such a lift is in fact equivalent to the data of lifts: 
	\[ \begin{tikzcd}
		& X\arrow[d,"f"] \\
		\Spec(A_{\leq k}) \arrow[ur, dashed, "l_k"] \arrow[r] & Y
	\end{tikzcd}\]
	
	such that the following diagrams commute: 
	\[ \begin{tikzcd}
		\Spec(A_{\leq k})\arrow[r, dashed, "l_k"] \arrow[d] & X \arrow[d] \\
		\Spec(A_{\leq k+1}) \arrow[r] \arrow[ur,dashed,"l_{k+1}"] & Y
	\end{tikzcd}\] 
	
	Assume for now that $H^0(A) := A_{\leq 0}$ admits a lift. Then we need to produce the lift $l_k$ for $k \geq 1$ compatible with the previous lift. 
	
	To do that, we observe that the Postnikov truncations are successive square zero extensions of each other (see \cite[Proposition 4.1]{PV13}). The condition of being formally étale implies that there is a lift to the following diagram for any $A$-module $M$: 
	\[ \begin{tikzcd}
		\Spec(A) \arrow[r] \arrow[d] & X \arrow[d, "f"] \\
		\Spec(A \boxplus M) \arrow[r] \arrow[ur, dashed] & Y 
	\end{tikzcd} \] 
	
	Indeed, we can think of this diagram as a lift of $\Spec(A \boxplus M) \to Y$ is the category of stacks under $\Spec(A)$. Using the cotangent complex adjunction this can be rephrase by asking for a unique lift: 
	\[ \begin{tikzcd}
		(\Ll_Y)_{\vert \Spec(A)} \arrow[r] \arrow[d]& M  \\
		(\Ll_X)_{\vert \Spec(A)} \arrow[ur,dashed] & 
	\end{tikzcd} \]
	
	which exactly the condition that the cotangent complexes (restricted to $\Spec(A)$) are equivalent. In other words, the requirement that $f$ is formally étale. \\

	We have reduced the problem to finding a lift:
	\[ \begin{tikzcd}
		& X\arrow[d,"f"] \\
		\Spec(H^0(A)) \arrow[ur, dashed] \arrow[r] & Y
	\end{tikzcd}\]

	What we know, from the fact that $f$ is a nil-equivalence, is that there is a lift: 
	
	\[ \begin{tikzcd}
		& X\arrow[d,"f"] \\
		\Spec(A^{\tx{red}}) \arrow[ur, dashed] \arrow[r] & Y
	\end{tikzcd}\]
	
	Using again the lifting property against square zero extensions, we only have to prove that $\Spec(A^{\tx{red}}) \to \Spec(A)$ can be obtained as a finite composition of square zero extensions. 
	
	First recall that $A^{\tx{red}} := H^0(A)/\G_M$ with $\G_M := \tx{Nil}(H^0(A))$. Since $A$ is of almost finite presentation, $H^0(A)$ is of finite presentation and there exists $N \in \Nn$ such that $\G_M^N = 0$. Then we observe that $A$ is the square zero extension of $A/\G_M^{N-1}$ by $\G_M^{N-1}$ and $A/\G_M^{N-1}$ is an algebra with nilpotent ideal $\G_M$ such that $\G_M^{N-1} =0$. By induction, we can write $H^0(A)$ as finitely many sucessive square zero extensions of $A^{\tx{red}} = A/\G_M$. *
	
	This proves that there is a lift: 
	
	\[ \begin{tikzcd}
		\Spec(A^{\tx{red}}) \arrow[r] \arrow[d] & X \arrow[d, "f"] \\
		\Spec(H^0(A)) \arrow[r] \arrow[ur, dashed] & Y 
	\end{tikzcd} \]
	
	and concludes the proof. 
\end{proof}

\subsubsection{Formal stacks from formal moduli problems}\label{sec:formal-stack-from-formal-moduli-problems}\

\medskip

As we have seen in Remark \ref{rq:infinitesimal cohesinveness from formal stack to fmp}, the infinitesimal cohesiveness of formal derived stacks is very similar to the pullback condition for formal moduli problems. The goal of this section is to study the relationship between formal moduli problems and formal derived stacks.
In particular, we will study in details the formal moduli problems given by formal spectra and their associated formal stack. \\

We start by explaining the procedure that extends formal moduli problems $F$ over $A$, a connective cdga of almost finite presentation, to formal thickenings $\und{F}$ under $X = \Spec(A)$.

\begin{Th}
	\label{th:fmp are formal thickenings}
	
	For $A$ an almost finitely presented connective cdga, there is an equivalence (see \cite[Proposition 4.1]{CG18} or \cite[Chapter 5 Proposition 1.4.2]{GR20}): 
	\[ \begin{tikzcd}
		\bf{FMP}_{A} \arrow[r, shift left, "\pund{(-)}"] &\arrow[l, shift left, "R"] \thickp(\Spec(A))
	\end{tikzcd}\]
	where $\FMP_A$ denotes formal moduli problem over $A$ (see Definition \ref{def:FMP}).
	
	Moreover, composing with the stackification functor (which is also a left adjoint), we get an adjunction: 
	\[ \begin{tikzcd}
		\bf{FMP}_{A} \arrow[r, shift left, "\und{(-)}"] &\arrow[l, shift left, "R"] \dst_{\Spec(A)/}
	\end{tikzcd}\]
	such that the completion functor $\und{(-)}$ preserves colimits and finite limits. 
\end{Th}
\begin{proof}[Sketch of proof]
	The first equivalence is due to \cite[Chapter 5 Proposition 1.4.2]{GR20} (see also \cite[Proposition 4.1]{CG18}). Essentially the completion functor is the left Kan extension along the inclusion $\small \to \daff$ which can be shown to be valued in $\thickp(\Spec(A))$. 
	
	For the second statement, we compose this adjunction with the adjunction: 
	\[ \begin{tikzcd}
		\ST : \dpstfp\arrow[r, shift left] & \arrow[l, shift left, hookrightarrow] \dstfp  : j
	\end{tikzcd} \]   
	
	This adjunction is a \emph{topological localization} according to \cite[Lemma 6.2.2.7]{Lu09HTT} and such localizations are \emph{left exact} thanks to \cite[Corollary 6.2.1.6]{Lu09HTT} and therefore preserve finite limits. 
\end{proof}

\begin{RQ} \label{rq:stackification issues}
	The stackification functor, $\ST$, does not behave quite as nicely as we would like. In particular is does not preserve the cotangent complex. We will find ourselves dealing with the following two types of situations: 
	\begin{itemize}
		\item Either we start with formal moduli problems. Then we need to do everything on the underlying pre-stacks and only then use the stackification. 
		\item We start with a formal thickening in stack, in which case working in stacks or pre-stacks is equivalent (because $j$ is fully-faithful and preserves the tangent complexes).  
	\end{itemize} 
\end{RQ}

We give a few results on the behavior of the stackification with respect to formal geometry. 

\begin{Prop}\label{prop:stackification and formal geometry}
	We have the following properties of the stackification functor: 
	\begin{enumerate}
		\item The functor $j$ is fully-faithful and for any map of stack $h:X \to Y$, with $X = \Spec(A)$, we have an equivalence of $A$-modules: 
		\[ h^* \Tt_{Y} \simeq h^* \Tt_{j(Y)} \] 
		\item $j$ commutes with the de Rham functor\footnote{All our discussion in Section \ref{sec:formal-derived-stacks-and-formal-neighborhood} adapts well to pre-stacks.}: 
		\[j \circ (-)_{\tx{DR}} = (-)_{\tx{DR}} \circ j\]
		\item $\ST$ commutes with the reduced functor $\red{(-)}$: 
		\[\ST \circ \red{(-)} = \red{(-)}\circ \ST\]
		\item $\ST$ preserves nil-equivalences.
	\end{enumerate}
\end{Prop}
\begin{proof}\
	
	\begin{enumerate}
		\item Since $j$ is fully-faithful, we have the equivalences: 
		\[ \begin{split}
			h^* \Tt_Y \simeq & \Map_{\dst}\left(\Spec(A \boxplus A[n]), Y \right) \\
			\simeq &\Map_{\dpst}\left(j\left(\Spec(A \boxplus A[n])\right), j(Y) \right) \\
			\simeq & \Map_{\dpst}\left(\Spec(A \boxplus A[n]), j(Y) \right)\\
			\simeq & h^* \Tt_{j(Y)}
		\end{split}  \]
		\item For all $X \in \dst$ and $A \in \cdgacon$, we have: 
		\[ j\left(X_{\tx{DR}}\right)(A) = j(X)(A^{\tx{red}}) = X(A^{\tx{red}}) = j(X_{\tx{DR}})(A)  \]
		\item The previous result shows the commutativity of the right adjoint functors. This implies that their left adjoints also commute. 
		\item Since $\ST$ commutes with the reduced functor and preserves weak equivalences, then is preserves nil-equivalences.   
	\end{enumerate}
\end{proof}

Now that we have a way to see a formal moduli problem as a formal derived stack, we will be interested in the particular example of the formal moduli problems given by formal spectra (Definition \ref{def:formal spectrum}). 

Let us recall that the formal spectrum (relative to $A$) of $B \in \cdga_{/A}$ is defined as the formal moduli problem:
\[  \begin{tikzcd}[row sep = 1mm, column sep = tiny]
	\Spf_A (B) :  \small \arrow[r] & \igpd \\
	\qquad  \qquad C \arrow[r, mapsto] & \Mapsub{\cdga_{/A}} \left( B,C \right) 
\end{tikzcd} \]

\begin{RQ}\label{rq:extended adjunction}
	It turns out that the restriction functor $R$ does not only send formal thickenings to formal moduli problems, but it also sends any formal stack under $X$ having a cotangent complex to a formal moduli problem. Indeed, from Remark \ref{rq:infinitesimal cohesinveness from formal stack to fmp}, the pullback condition of formal moduli problems is satisfied, and because the terminal object in $\FMP_A$ is $\Spf_A(A)$ itself, we have: \[ \Mapsub{\FMP_A}(\Spf_A(A), F) \simeq \star\]   
	Therefore $R(F)$ is a formal moduli problem. 
\end{RQ}

\begin{RQ}\label{rq:restriction of the spectrum functor is the formal spectrum}
	We have, by definition of the formal spectrum, that for all $B \in \cdga_{/A}$: 
	\[ R(\Spec(B)) \simeq \Spf_A(B)\]
\end{RQ}

In Section \ref{sec:tangent-complex-of-a-formal-moduli-problem}, we define the tangent complex of a formal moduli problem. It turns out (see Section \ref{sec:tangent-complex-of-a-formal-moduli-problem}) that this sends the formal spectrum of $B \in \cdga_{/A}$ to the $B$-module $f^* \Tt_B$. We would like to compare it to the tangent complex to its associated formal thickening. 
\begin{Lem}\label{lem:tangent formal spectrum}
	Take $f: B \to A$ with $A$ connective of almost finite presentation. Then we have:
	\[ f^*\Tt_B \simeq f^* \Tt_{\pund{\Spf_A(B)}}\] 
	More generally, if $F \in \FMP_A$, then we have:
	\[ f^*\Tt_F \simeq f^* \Tt_{\pund{F}}\]  
	where $f$ denotes the map from $\Spec(A)$ on both side. 
\end{Lem}
\begin{proof}
	For all $n \geq 0$ we have:
	\[\begin{split}
		\relg{f^*\Tt_{\pund{F}}[n]} \simeq & \Mapsub{\dst_{\Spec(A)/}} \left(\Spec(A[\epsilon_n]), \pund{F}\right) \\
		_{(	A[\epsilon_n] \in \small)} \simeq &   F(A[\epsilon_n])\\
		_{\text{(Section \ref{sec:tangent-complex-of-a-formal-moduli-problem})}}	\simeq & \relg{f^* \Tt_F[n]}
	\end{split}\]
	
	where $\epsilon_n$ is seen in degree $n$ and squares to zero\footnote{This gives exactly the square zero extension $A[\epsilon_n] := A\boxplus A[-n]$.}. The collection for all $n\geq 0$ of these equivalences induces the equivalence between the complexes since they become equivalent as $\Omega$-spectra represented by the pullbacks of the cotangent complexes.	
\end{proof}

\begin{Cor}
	\label{cor:relative tangent complexe formal stack}
	
	For $X:= \Spec(A)$ affine of almost finite presentation, we have an equivalence: 
	\[ \Ttr{X}{\pund{\Spf_A(B)}} \simeq \Ttr{A}{B} \] 
	
	In fact, we will see in Corollary \ref{cor:lie algebroid of formal spectrum} that this is even an equivalence of \emph{Lie algebroids} for the Lie algebroid structure on the relative tangent described in Proposition \ref{prop:lie algebroid koszul duality context}.
\end{Cor}
\begin{proof}
	Using Lemma \ref{lem:tangent formal spectrum}, we have that the following two fiber sequences are equivalent: 
	\[ \Ttr{X}{\pund{\Spf_A(B)}} \to \Tt_A \to f^*\Tt_{\pund{\Spf_A(B)}} \]
	\[ \Ttr{A}{B}  \to \Tt_A \to f^*\Tt_B \]
\end{proof}

\begin{RQ} 
	The extension functor $\und{(-)}$ is closely related to the formal completion functor. Indeed, given $p: X \to Y$ such that $Y$ is formal, under some technical assumptions\footnote{Assume that $X$ satisfies Assumptions \ref{ass:very good stack}. These are assumptions ensuring that formal moduli problems (and therefore formal thickenings) of $X$ are equivalent to Lie algebroids over $X$ (see Section \ref{sec:quotient-stack-of-a-lie-algebroid}).}, we have an equivalence: 
	\[ \comp{Y_X} \simeq \pund{R(Y)} \]
	
	and the natural morphism $\comp{Y_X} \to Y$ corresponds to the counit of the adjunction of Theorem \ref{th:fmp are formal thickenings}.
\end{RQ}

\subsection{De Rham Algebra and (Closed) $p$-Forms}\label{sec:de-rham-complex-and-closed-p-forms}\

\medskip

This Section is mostly a recollection of \cite{CPTVV}. In the first part, we focus on setting up notations for the different structures we can put on the de Rham algebra and its relative version in derived algebraic geometry. In the second part, we use this formalism to defined the notion of $n$-shifted (closed) $p$-form on a derived stack.  

\subsubsection{De Rham complex}\label{sec:de-rham-complex}\

\medskip

Classically we think of the de Rham algebra as the algebra of differential forms together with the de Rham differential, $\left(\bigwedge_{\_O_X} \Gamma(T^*X), \dr\right)$. In derived geometry, we want to replace $T^*X$ by $\Ll_X$, and we would like the de Rham algebra to be something like the derived global sections of $\Sym_{\_O_X} \Ll_X[-1]$. But this object has much more structure than its classical counterpart. Namely we have that: 
\begin{itemize}
	\item This is a cdga for the grading induced by the grading on $\Ll_X[-1]$, which we will call \defi{internal grading} with a vertical differential induced by the differential on $\Ll_X[-1]$. 
	\item There is a weight grading given by $\Sym_{\_O_X}^p \Ll_X[-1]$, the symmetric powers of arity $p$.
	\item There is a ``mixed differential'' also known as the de Rham differential, which increases the weight and total degree by $1$.
\end{itemize}  

This all data gives the structure of \emph{graded mixed algebra}. We quickly recall the definitions and main properties related to graded mixed objects in Appendix \ref{sec:graded-mixed-objects}. We will first focus on the affine case, $X := \Spec(B)$.\\

It turns out that there is a universal construction of a ``de Rham object'' that will, thanks to Proposition \ref{prop:description de rham graded mixed algebra}, recover the heuristic we just discussed. We will define the de Rham graded mixed complex in the general setting of \emph{good model category} $\_M$ (see Appendix \ref{sec:good-model-structures}). In particular for $\_M = \Mod_k$, we recover the classical de Rham graded mixed complex of a cdga. Moreover, in this context, we can even extend the functor $\DR$ to all derived stacks.     

\begin{Lem}[{\cite[Proposition 1.3.8 and Proposition 1.3.16]{CPTVV}}] 
	\label{lem:de rham adjunction}
	
	There is an adjunction: 
	
	\[ \begin{tikzcd}
		\DR_{\_M} : \cdga_\_M \arrow[r, shift left] & \arrow[l, shift left]  \gmc{\cdga_\_M}: (-)(0)
	\end{tikzcd}\]
	where the right adjoint is the functor that sends a graded mixed algebra to the algebra in $\cdga_{\_M}$ given by the weight $0$ part of the graded mixed algebra. 
	
	Moreover, this adjunction holds in the relative setting. If $A \in \cdga_\_M$, then there is an adjunction\footnote{On the right hand side, $A$ is viewed as a graded mixed algebra concentrated in weight $0$.}: 
	
	\[ \begin{tikzcd}
		\DR_{\_M}(-/A) : \cdga_{\_M, \ A/} \arrow[r, shift left] & \arrow[l, shift left]  \gmc{\cdga_{\_M, \ A/}}: (-)(0)
	\end{tikzcd}\]
\end{Lem}

\begin{Def}
	\label{def:de rham complex and relative in mod}
	
	The \defi{de Rham algebra of $A \in \cdga_M$} is defined as the following graded mixed algebra $\DR_\_M(A)$. We are going to write $\DR(A)$ when $\_M$ is clear. For example we have: 
	\begin{itemize}
		\item If $\_M = \Mod_k$, then $\DR(B) := \DR_{\Mod_k}(B)$ is the de Rham graded mixed algebra associated to the algebra $B$. 
		\item If $\_M = \Mod_A$ for $A \in \cdga$ and $A \to B$ a map of algebras (and therefore $B \in \cdga_{\Mod_A}$), we define: \[\DR(B/A) := \DR_{\Mod_A}(B)\]
		More generally, we have: 
		\[ \DR(-/A) := \DR_{\Mod_A} \]
		
		From \cite[Lemma 1.3.18]{CPTVV}, we have that
		\[ \DR(B/A) \simeq \DR(B) \otimes_{\DR(A)} A \]
		
		where $A$ is viewed as a graded mixed complex concentrated in degree $0$. Note that $\DR(A) \simeq \DR(A/k)$ and this definition coincides with the relative de Rham construction via the adjunction of Lemma \ref{lem:de rham adjunction}. 
	\end{itemize}
\end{Def}

\begin{notation}\
	\label{not:de rham with different structures}
	There are various structures on the de Rham algebra and we set the following notations: 
	\begin{itemize}
		\item $\DR$ is the de Rham functor valued in graded-mixed algebras. 
		\item $\gr{\DR}$ is the underlying graded algebra of $\DR$ (see Lemma \ref{lem:graded mixed complex as graded modules}) where we forget the mixed structure. 
		\item As weak graded mixed complexes (Definition \ref{def:weak graded mixed complex}) are equivalent to complete filtered objects, we denote by $\cpl{\DR}$ the induced complete filtered algebra. It is described, thanks to Proposition \ref{prop:totalization weak graded mixed complex}, by the following filtered algebra: 
		\[ F^p \cpl{\DR}(A/B) \simeq  \prod_{p' \geq p} \Sym_A^{p'} \Llr{A}{B}[-1]\] 
		
		Since $\DR(A/B)$ is non-negatively weighted, each $F^p \cpl{\DR}(A/B)$ is a differential graded algebra with the differential being the \emph{total differential} summing the vertical and the mixed differentials. 
		
		\item $\rel{\DR}$ is the de total Rham algebra obtained by applying the realization functor $\rel{-}$ given by Proposition \ref{prop:realization of graded mixed complexes}:
		\[\rel{\DR(A/B)} \simeq \prod_{p \geq 0} \DR(A/B)(p) \]
		together with the total differential. Note that we have: 
		\[ \rel{\DR} \simeq F^0\DR \simeq \colim  F^p \cpl{\DR}\]
	\end{itemize}
\end{notation}

\begin{Prop}
	\label{prop:description de rham graded mixed algebra}
	
	There is an equivalence of graded algebras:  
	\[  \gr{\Sym}_B \Ll_B[-1] \overset{\sim}{\to} \gr{\DR}(B) \]
	
	Similarly, in the relative case we have an equivalence: 
	\[  \gr{\Sym}_B \Llr{B}{A}[-1] \overset{\sim}{\to} \gr{\DR}(B/A) \]
\end{Prop}

\begin{Cor}[{\cite[Corollary 1.3.14 and Remark 1.3.15]{CPTVV}}]\
	\label{cor:mixed structure on symmetric power on cotangent}
	For any object $A \in \cdga_\_M$, the algebra $ \gr{\Sym}_B \Llr{A}{B}[-1] $ possesses a canonical weak graded mixed structure\footnote{This structure is (weakly) transferred from $\DR(A/B)$. If $B \to A$ is nice enough (that is cofibrant) then this can be taken as a strict (non-weak) graded mixed algebra (see \cite[Section 1.3.3]{CPTVV}).} making it into a weak graded mixed algebra in $\_M$. The corresponding mixed differential is called the de Rham differential and is denoted $\dr$.
	
\end{Cor}

\newpage

\begin{Prop}
	\label{prop:realization and completion of de rham algebra}

	We have the following equivalences:
	\[ \rel{\DR}(B/A) \simeq \colim \cpl{\DR}(B/A) \]	
	\[ \gr{\DR}(B/A) \simeq \bf{Gr}\left( \cpl{\DR}(B/A)\right) \]
	
\end{Prop}
\begin{proof}
	The first equivalence is a direct application of Lemma \ref{lem:totalization and colimit} and the second is an application of Lemma \ref{lem:associated graded and forgetting mixed structure}. 
\end{proof}

So far we only constructed to de Rham graded mixed algebra for affine objects. It turns out that $\DR$ satisfies étale descent and is therefore a derived stack valued in graded mixed complexes. Therefore we can extend it to a functor (see \cite[Definition 1.13]{PTVV}) on all derived stacks and it defines\footnote{The de Rham complex extends to all derived stacks. However if $X$ does not admit a cotangent complex, there is no analogue to Proposition \ref{prop:description de rham graded mixed algebra}.} a graded mixed algebra $\DR(X)$ for all derived stack $X$. \\

The natural way to define this global de Rham graded mixed algebra is by taking the left Kan extension computed by following limit: 
\[ \DR(X) := \lim\limits_{\Spec(A) \to X} \DR(A)\]

Moreover the underlying graded of this limit is computed weight-wise and therefore: 
\[ \gr{\DR}(X) \simeq \lim\limits_{\Spec(A) \overset{f}{\to}X} \gr{\DR}(A)  \]

We get that: 
\[ \gr{\DR}(X)(p) \simeq \lim\limits_{\Spec(A) \overset{f}{\to}X} \Sym_A^p \Ll_A[-1] \]

\begin{Cor}\label{cor:global section and de rham algebra}
	If $X$ admits a cotangent complex there is a natural morphism\footnote{Here $\Gamma$ denotes the derived global section functor that sends $\_F \in \QC(X)$ to:
		\[ \Mapsub{\dst_{/X}}(X, F) \simeq \Mapsub{\QC(X)}(\_O_X, \_F) \simeq \lim\limits_{\Spec(A) \overset{f}{\to}X} \relg{f^* \_F}\]}: 
	\[ \Gamma\left(\gr{\Sym}_{\_O_X} \Ll_X[-1]\right) \simeq \lim\limits_{\Spec(A) \overset{f}{\to}X} \relg{\gr{\Sym}_A f^*\Ll_X} \to \relg{\gr{\DR}(X)} \] 
	
	Since the de Rham functor satisfies smooth descent, this natural morphism is in fact an equivalence for $X$ a Artin stack (\cite[Proposition 1.14]{PTVV}). 
\end{Cor}

\begin{RQ}\label{rq:spectrum enriched global section and graded de rham algebra}
	In this corollary, we consider the \emph{space} of sections by taking the mapping space of sections. If we take instead the \emph{mapping spectra}, we get a map to $\gr{\DR}(X)$ of graded algebras. In other words we get a map of graded algebras: 
	\[ \gr{\iHomsub{\QC(X)}}\left(\_O_X, \gr{\Sym}_{\_O_X} \Ll_X[-1] \right) \to \gr{\DR}(X)\]
\end{RQ}

Moreover, the weight $0$ component always recovers the algebra of functions of $X$ (even if $X$ is not Artin).

\begin{Lem}\label{lem:augmentation de rham functor}
	There is a natural augmentation to the algebra of global functions (viewed as a graded mixed algebra concentrated in weight $0$):
	\[ \DR(X) \to \_O_X(X)\]
\end{Lem}
\begin{proof}
	For any graded mixed algebra, there is a natural map to its weight $0$ part (viewed as a graded mixed algebra concentrated in weight $0$). In particular we get a map: 
	\[ \DR(X) \to \DR(X)(0)  \]
	
	Thank to Proposition \ref{prop:description de rham graded mixed algebra}, we get an equivalence: 
	\[ \DR(A)(0) \simeq A\]
	
	Therefore, we can use Definition \ref{def:sheaf of function on derived stacks} and show that: 
	\[ \DR(X)(0) \simeq \lim\limits_{\Spec(A) \to X} \DR(A)(0) \simeq \lim\limits_{\Spec(A) \to X} A \simeq \_O_X (X) \]
\end{proof}

\begin{Def}[{\cite[Definition 2.3.1]{CPTVV}}]\label{def:relative de rham global}
	Given a map $f: Y \to X$ of derived stacks with $X = \Spec(A)$, we can define the relative de Rham graded mixed algebra of $f: Y \to X$ as:
	\[ \DR(Y/X) := \lim\limits_{\Spec(B) \overset{f}{\to}Y}\DR(A/B)  \]
\end{Def}

\begin{Lem}\label{lem:relative de rham complex sequence} 
	Let $X \to Y \to Z$ be morphisms of derived stacks such that $X = \Spec(A)$ and $Y = \Spec(B)$ are affine. Then there is morphism of graded mixed complexes: 
	\[ \DR(X/Z) \otimes_{\DR(X)} \DR(Y) \to \DR(X/Y)  \]
	Moreover if $Z$ is also affine, then this is an equivalence.
\end{Lem}
\begin{proof}
	This is the natural morphism: 
	\[ \begin{split}
		\DR(X/Z)\otimes_{\DR(A)} \DR(B) \simeq & \left(\lim\limits_{\Spec(C) \to Z} \DR(C/A) \right) \otimes_{\DR(A)} \DR(B) \\
		\to & \lim\limits_{\Spec(C) \to Z} \left(\DR(C/A) \otimes_{\DR(A)} \DR(B)  \right) \\
		\simeq & \lim\limits_{\Spec(C) \to Z} \left(\DR(C/B)  \right)  \\
		\simeq & \DR(X/Y) 
	\end{split} \]
	
	This is an equivalence whenever the limit commutes with the tensor product. This is the case if $Z$ is affine.
\end{proof}

\subsubsection{Shifted (closed) $p$-forms}
\label{sec:shifted-closed-p-forms}\

\medskip

Recall from \cite{PTVV} that there are classifying stacks $\_A^p (\bullet, n)$ and $\_A^{p,cl}(\bullet, n)$ of respectively the space of $n$-shifted differential $p$-forms and the space of $n$-shifted closed differential $p$-forms. In this section we will recall their definition and basic properties.

\begin{Def}
	\label{def:p-forms (closed) on affine}
	
	Let $A \in \cdgacon$. The space of \defi{$p$-forms of degree $n$} and the space of \defi{closed $p$-forms of degree $n$} on $A$ are defined respectively by: 
	\[\_A^p(A, n) := \Mapsub{\gr{\Mod_k}} \left( k[-n-p]((-p)), \gr{\DR}(A) \right)\simeq \relg{\left(\Sym^p \Ll_A[-1]\right)[n+p]} \]
	\[ \_A^{p,\tx{cl}}(A, n) := \Mapsub{\gmc{\Mod_k}} \left( k[-n-p]((-p)), \DR(A) \right) \] 
	where $k[-n-p]((-p))$ is the graded (mixed) complex concentrated in weight $p$ and degree $n+p$. 
	All along, we denote the vertical differential, i.e. the differential on $\Ll_A$, by $\delta$ and the mixed differential, i.e. the de Rham differential, by $\dr$. 
\end{Def}

From \cite[Proposition 1.11]{PTVV}, we know that the functors  $\_A^p (-, n)$ and $\_A^{p,cl}(-, n)$ are in fact stacks and therefore we can extend the definition of (closed) $p$-forms of degree $n$ to any derived stacks.

\begin{Def}
	\label{def:p-form (closed) on stacks}
	
	The space of \defi{$p$-forms of degree $n$} on a derived stack $X$ is the mapping space: 
	\[\_A^p(X,n) := \Mapsub{\mathbf{dSt}}\left( X, \_A^{p}(\bullet, n) \right) \simeq  \Mapsub{\gr{\Mod_k}} \left( k[-n-p]((-p)), \gr{\DR}(X) \right)\] 
	
	The space of \defi{closed $p$-forms of degree $n$} on $X$ is:  \[\_A^{p, \tx{\tx{cl}}}(X,n):= \Mapsub{\mathbf{dSt}}\left( X, \_A^{p,\tx{cl}}(\bullet, n) \right) \simeq \Mapsub{\gmc{\Mod_k}} \left( k[-n-p]((-p)), \DR(X) \right) \]
	
\end{Def}
Now the following proposition says that in the case when $X$ is a derived Artin stack, the spaces of shifted differential forms are spaces of sections of quasi-coherent sheaves on $X$.

\begin{Prop}[Proposition 1.14 in \cite{PTVV}]
	\label{prop:proposition 1.14 PTVV}
	
	Let $X$ be a derived Artin stack (see Remark \ref{rq:artin and geometric derived stacks}) and $\Ll_X$ be its cotangent complex. Then there is an equivalence:
	\[\_A^{p}(X,n) \simeq \Mapsub{\QC(X)} \left( \_O_X, \left(\Sym_{\_O_X}^p \Ll_X[-1]\right)[p+n]\right)\]
\end{Prop}

\begin{RQ}\label{cor:section of T^*X are 1-forms for Artin}
	For a derived Artin stack, a $n$-shifted $p$-form is therefore equivalent to the data of a section of $\left(\Sym_{\_O_X}^p \Ll_X[-1]\right)[p+n]$. Even better, thanks to Corollary \ref{cor:linear stack section and zero section}, the space of  sections of $T^*X$ is equivalent to the space of $n$-shifted $1$-form on $X$:
	\[\Map_{\dst_{/X}}(X, T^*[n]X) \simeq \_A^{1}(X,n)\]  
\end{RQ}

\begin{RQ}	
	More concretely, we have from \cite{Ca21} and \cite{CPTVV} an explicit description of (closed) $p$-forms of degree $n$ on a derived Artin stack $X$. A $p$-form of degree $n$ is given by a (derived) global section 
	\[\omega \in \DR(X)(p) [n+p]  \simeq \Gamma \left(\left( \Sym_{\_O_X}^p \Ll_X[-1] \right) [n+p] \right) \] 
	
	such that $\delta \omega = 0$. 
	
	A closed $p$-form of degree $n$ is a semi-infinite sequence $\omega = \omega_0 + \omega_1 + \cdots$ with 
	\[ \omega_i  \in \DR(p+i) [n+p]= \Gamma \left( \left( \Sym_{\_O_X}^{p+i}  \Ll_X[-1] \right) [n+p]\right)\]
	such that $\delta \omega_0 = 0$ and $d \omega_{i} = \delta \omega_{i+1}$. Equivalently, being closed means that $\omega$ is a closed element in \[F^p\cpl{\DR}(X) \simeq \Gamma \left( \prod_{i \geq 0} \left( \Sym_{\_O_X}^{p+i} \Ll_X[-1] \right) [n+p] \right)\] whose total degree is given by $n+p+i$ and with the total differential.
	
\end{RQ}

In fact in general (without assuming $X$ Artin), we have that the following description of the spaces of (closed) shifted $p$-forms (see the comment following \cite[Definition 1.13]{PTVV}):

\begin{Prop}\label{prop:realization of de rham and (closed) forms}
	For $X$ a derived stack, we can also describe the spaces of (closed) differential forms as: 
	\[\_A^{p}(X, n) \simeq  \relg{\DR(p)(X)[n+p]}\] 
	\[\_A^{p,\tx{cl}}(X, n) \simeq  \relg{ F^p \cpl{\DR}(X)[n+p]} \simeq \relg{ \prod_{i\geq p} \DR(p+i)(X)[n+p] } \] 
	
	together with the total differential. 
\end{Prop}

\begin{Def}\
	\label{def:underlying p form}
	The natural projection: \[F^p \cpl{\DR}(X)[n] \to \DR(p)(X)[n+p]\] induces a map $\_A^{p,\tx{cl}}(X, n) \to \_A^{p}(X, n)$ that forgets the closed structure. It essentially sends \[ \omega := \omega_0 + \omega_1 + \cdots \mapsto \omega_0 \]
	and the image of $\omega$ by this map is called the \defi{underlying $p$-form of degree $n$ of $\omega$}. We will often denote it by $\omega_0$.  
\end{Def}

The collection of all the closures of a $p$-form of degree $n$ forms a space:

\begin{Def}
	\label{def:space of keys}
	
	Let $\alpha \in \_A^{p}(X,n)$ then the space of all closures of $\alpha$ is called the \defi{space of keys} of $\alpha$ denoted $\mathbf{key}(\alpha)$. It is given by the homotopy pull-back:
	\[	\begin{tikzcd}
		\mathbf{key}(\alpha) \arrow[r] \arrow[d] & \_A^{p,\tx{cl}}(X,n) \arrow[d]\\
		\star \arrow[r, "\alpha"] & \_A^p(X,n)
	\end{tikzcd}\]
\end{Def}

\begin{Prop}
	\label{prop:de rham differential of p-forms}
	The mixed differential of $\DR(X)$ induces a map: 
	\[\dr : \_A^p (X,n) \rightarrow \_A^{p+1, \tx{cl}}(X,n)\]
	that squares to zero. 
\end{Prop}
\begin{proof}
	The mixed differential increases the weight by (at least) $1$ therefore induces a map $\dr : F^p\cpl{\DR}(X) \to F^{p+1}\cpl{\DR}(X)$ of degree $1$. This induces the desired map since we have, from Definition \ref{def:filtration of graded mixed complexes}, that: 
	
	\[F^p \cpl{\DR}(X) \simeq \Hom_{\gmc{\cdga}}(k((-p))[-n-p], \DR(X))\]
\end{proof}

\begin{RQ}
	\label{rq:p-forms (closed) relative version}
	Given a map of derived Artin stack $f : Y \rightarrow X$, we define $\_A^{p,\tx{(cl)}}(Y/X, n)$, the space of $n$-shifted (closed) $p$-forms on $Y$ relative to $X$, to be the homotopy cofiber of the natural map:
	\[f^* : \_A^{p,\tx{(cl)}}(X,n) \rightarrow \_A^{p,\tx{(cl)}}(Y,n)\]
	Using Definition \ref{def:relative de rham global}, we can show\footnote{Because the pushout defining the relative de Rham graded mixed algebra is preserved by geometric realization.} that the relative forms can be directly obtain by using the relative de Rham $\DR(Y/X)$. For instance $n$-shifted relative $p$-forms are equivalent to the derived global sections of $\left( \Sym_{\_O_Y}^p \Llr{Y}{X}[-1] \right)[n+p]$. 
\end{RQ}

\begin{RQ}\
	\label{rq:closed as structure}
	We say that a $p$-form, $\omega_0$, of degree $n$ on a derived Artin stack $X$ can be lifted to a closed $p$-form of degree $n$ if there exists a family of $(p+i)$-forms $\omega_i$ of degree $n-i$ for all $i > 0$, such that $\omega = \omega_0 + \omega_1 + \cdots $ is closed in $F^p\cpl{\DR}(X)[n]$. In that situation, we can see that $\dr \omega_0$ is in general not equal to $0$ but is homotopic to $0$ with $\dr \omega_0 =  D\left(- \sum_{i>0} \omega_{i} \right)$. The choice of such a homotopy is the same as the structure of a closure of the $p$-form of degree $n$.
	
	Being closed is therefore no longer a property of the underlying $p$-form of degree $n$ but a structure given by a homotopy between $\dr\omega_0$ and zero.
\end{RQ}

\begin{War}\label{war:closed p-forms for non-Artin stacks}
	When we are working with formal stacks, they will usually not be Artin and therefore $n$-shifted $p$-forms might not always come from derived global sections of $\left(\Sym_{\_O_X}^p \Ll_X[-1]\right)[p+n]$. This will be a problem in order to define symplectic structure on formal stacks or more generally any kind of ``non-degeneracy condition''. To avoid this problem, recall from Corollary \ref{cor:global section and de rham algebra} that we have a map of graded spaces: 
	\[ \Gamma\left(\_O_X, \gr{\Sym}_{\_O_X} \Ll_X[-1] \right) \to \relg{\gr{\DR}(X)} \] 
	
	Therefore, in order to still view $p$-forms as global sections of elements of the $p$-th symmetric power of $\Ll_X[-1]$ (which is what we will need to speak of non-degeneracy conditions in Section \ref{sec:new-constructions-in-derived-symplectic-geometry}) when $X$ is a derived stack with a cotangent complex, we \emph{define} the notion of \defi{good} $p$-forms of degree $n$ as the $p$-forms of degree $n$ that are in the image of the map:
	\[  \Gamma\left(\_O_X, \left(\Sym_{\_O_X}^p \Ll_X[-1]\right)[n+p] \right) \to   \relg{\DR(p)(X)[n+p]}\simeq \_A^{p}(X,n)  \]
	
	More generally, an element in $\DR(p)(X)[n+p]$ will be called \defi{good} if it is in the image of the map from  Remark \ref{rq:spectrum enriched global section and graded de rham algebra}: 
	\[ \gr{\iHomsub{\QC(X)}}\left(\_O_X, \left(\Sym_{\_O_X}^p \Ll_X[-1]\right)[n+p] \right) \to \DR(X)(p)[n+p]\]
	
	A closed element will be considered good if its underlying $p$-form is good. In other words, the space of good closed $p$-forms of degree $n$ is the pullback of the diagram: 
	\[ \begin{tikzcd}
		& \_A^{p, \tx{cl}}(X,n) \arrow[d] \\
		\Gamma\left(\_O_X, \left(\Sym_{\_O_X}^p \Ll_X[-1]\right)[n+p] \right) \arrow[r] & \_A^{p}(X,n)  
	\end{tikzcd}\]
	
	Similarly, \emph{good closed} elements in  $\DR(X)(p)[n+p]$ are closed elements in the pullback of the diagram: 
	
	\[ \begin{tikzcd}
		& F^p\cpl{\DR}(X)[n+p] \arrow[d]\\
		\gr{\iHomsub{\QC(X)}}\left(\_O_X, \left(\Sym_{\_O_X}^p \Ll_X[-1]\right)[n+p] \right) \arrow[r] & \DR(X)(p)[n+p]
	\end{tikzcd}\]

	Note that if $X$ is Artin good and general objects coincide (because the horizontal maps of the previous pullback diagrams are equivalences). 
\end{War}

\begin{notation}\label{not:space of closed forms}
	From now on, we will make an abuse of notation and denote by $\_A^{p}(X,n)$ and $\_A^{p, \tx{cl}}(X,n) $ the spaces of \emph{good} (closed) $p$-forms of degree $n$. Moreover we will simply call them (closed) $p$-forms of degree $n$ (omitting the word ``good''). In particular, \emph{by definition} we get the analogue of Proposition \ref{prop:proposition 1.14 PTVV}:
	
	\[\_A^{p}(X,n) \simeq \Mapsub{\QC(X)} \left( \_O_X, \left(\Sym_{\_O_X}^p \Ll_X[-1]\right)[p+n]\right)\]   
\end{notation}

	\newpage
\section{New Constructions in Derived Symplectic Geometry}\ 
\label{sec:new-constructions-in-derived-symplectic-geometry} 

Following the philosophy of derived geometrie, derived symplectic geometry is an extension of classical symplectic geometry where we consider shifted symplectic structures and homotopies between them. Moreover, we again would like such objects to be well behaved with respect to derived intersections and derived quotients. The study of the derived symplectic geometry of derived quotients is the object of Section \ref{sec:equivariant-symplectic-geometry}. In this section, we will recall and expand upon the idea that taking derived ``Lagrangian intersections'' is a way to produce new \emph{shifted} symplectic structures out of old ones.\\

In Section \ref{sec:derived-symplectic-geometry}, we recall the main definitions and examples of symplectic, Lagrangian, Lagrangian correspondences and Lagrangian fibration structures.\\ 

Then Section \ref{sec:new-constructions-from-old-ones} is devoted to constructing new structures via derived \emph{Lagrangian intersections}. In particular, we recall the construction of shifted symplectic structures on a derived intersections of Lagrangian morphisms (Proposition \ref{prop:lagrangian intersection are shifted symplectic}), and extend this result to the construction of new Lagrangian fibrations (Theorem \ref{th:derived intersection lagrangian fibraton}).

Then Section \ref{sec:the-higher-categories-of-lagrangians} describes the (higher) categories of Lagrangians (and Lagrangian correspondences) in order to give again a result producing new Lagrangian correspondences from derived Lagrangian intersections (Theorem \ref{th:lagragian correspondence pullback}). This will enable us to take derived ``equivariant Lagrangian intersections'' of moment maps in Section \ref{sec:equivariant-symplectic-geometry}. \\

Finally Section  \ref{sec:example-of-constructions-of-lagrangian-fibrations} gives concrete examples for the construction of these new Lagrangian fibration structures.

\subsection{Derived Symplectic Geometry}\
\label{sec:derived-symplectic-geometry} 

\medskip

In this section we will recall the basic definitions and examples of symplectic, Lagrangian and Lagrangian fibration structures for derived stacks. 

All along we will assume that our derived stacks admit a cotangent complex. \\

Recall that Warning \ref{war:closed p-forms for non-Artin stacks} tells us that the underlying $p$-form of a closed $p$-form comes from derived global sections of the shifted symmetric power of the cotangent complex. This will ensure that all the non-degeneracy conditions we are going to define make sense. In other words, given a derived stack $X$, a $2$-form of degree $n$ is given by a map:

\[\omega_0 \in \Mapsub{\QC(X)}\left(\_O_X, \left(\Sym_{\_O_X}^2 \Ll_X[-1]\right)[n+2]\right)\]  

\subsubsection{Shifted symplectic structures}\ \label{sec:shifted-symplectic-structures}

\medskip

\begin{Def}
	\label{def:non degenerate 2-form}
	
	We say that a closed 2-form of degree $n$ is \defi{non-degenerate} if the underlying $2$-form $\omega_0$ (Definition \ref{def:underlying p form}) of degree $n$ induces a quasi-isomorphism:
	\[\omega_0^\flat : \Tt_X \rightarrow \Ll_X[n]\]  
	
	We denote by $\_A^{2, \tx{nd}}(X,n) $ the subspace of $\_A^2(X,n)$ generated by the non-degenerate $n$-shifted $2$-forms. 
\end{Def}

\begin{Def}
	\label{def:symplectic structure on a closed 2-form}
	
	A \defi{$n$-shifted symplectic structure} is a non-degenerate $n$-shifted closed 2-form on $X$. We can also define a space of $n$-shifted symplectic structures as the pullback: 
	\[\begin{tikzcd}
		\mathbf{Symp}(X,n) \arrow[r] \arrow[d] & \_A^{2,\tx{nd}}(X,n) \arrow[d]\\
		\_A^{2,\tx{cl}}(X,n)\arrow[r] &\_A^2(X,n)
	\end{tikzcd}\]
	
\end{Def}

\begin{Ex}\label{ex:canonical symplectic structure cotangent}
	Suppose that $X$ in a derived stack.
	As in the classical case, we can construct the canonical Liouville 1-form. To do so, consider the identity: \[\id : T^*[n]X \rightarrow T^*[n] X\]
	
	It is determined, thanks to Proposition \ref{prop:map to linear stack}, by the data of:
	\begin{itemize}
		\item the projection $\pi : T^*[n]X \rightarrow X$.
		\item a section $\lambda_X \in \Mapsub{\QC(T^*[n]X)} \left( \_O_{T^*[n]X}, \pi^* \Ll_{X}[n]\right)$.
	\end{itemize} 
	Since we have a natural map $\pi^* \Ll_{X}[n] \rightarrow \Ll_{T^*[n]X}[n]$, $\lambda_X$ induces a $1$-form on $T^*[n]X$ called the tautological $1$-form. This $1$-form induces a closed $2$-form $\dr \lambda_X$ which happens to be non-degenerate whenever $X$ is Artin of locally finite presentation (see \cite[Section 2.2]{Ca19} for a proof of the non-degeneracy). 
\end{Ex}

The tautological $1$-form on the shifted cotangent is universal in the sense that it satisfies the usual universal property:

\begin{Lem}
	\label{lem:universal property tautological 1-form}
	If $X$ is stack, given a $1$-form of degree $n$, $\alpha : X \rightarrow T^*[n]X$, we have that $\alpha^* \lambda_X = \alpha$. 
\end{Lem} 

\begin{proof} 
	
	In general, if we take $f: X \rightarrow Y$, the pull-back of a $n$-shifted $1$-form, $\beta$, is described by the map $f^*\beta$ in the commutative diagram: 
	\[	\begin{tikzcd}
		T^*[n]X \arrow[dr, shift left] & f^* T^*[n] Y \arrow[l,"(df)^*"'] \arrow[d, shift left] \arrow[r] & T^*[n] Y \arrow[d, shift left] \\
		& X \arrow[ul, dashed, shift left, "f^* \beta"] \arrow[u, shift left, dashed] \arrow[r, "f"] & Y \arrow[u, dashed, shift left, " \beta"] 
	\end{tikzcd}
	\]
	
	Taking into account the fact that $\lambda$ factors through $\pi^* T^*[n]X$, we consider the following diagram:

	$$\begin{tikzcd}
		T^*[n] X & \alpha^* T^* T^*[n] X \arrow[l, "(d\alpha)^*"'] \arrow[r] & T^* T^*[n] X   \\
		& T^*[n] X = \alpha^* \pi^* T^*[n] X \arrow[ul, dashed, "\tx{Id}"] \arrow[r] \arrow[u, dashed, "(d\pi)^*"'] \arrow[d, shift left] & \pi^* T^*[n] X  \arrow[d, shift left] \arrow[u, "(d\pi)^*"] \\
		&X \arrow[u, dashed, shift left, "\tilde{\lambda}"] \arrow[r, "\alpha"]  & T^*[n] X \arrow[u, dashed, shift left, "\lambda"]
	\end{tikzcd} 
	$$
	
	This proves that the pull-back along $\alpha$ of $\lambda_X$ seen as a 1-form of degree $n$ on $T^*[n]X$ is the same as the pull-back along $\alpha$ of the section $\lambda_X : T^*[n]X \rightarrow \pi^* T^*[n]X$.\\
	
	We denote by $\alpha_1$ the associated section in $\Mapsub{\QC(X)} \left( \_O_X, \Ll_X[n] \right)$ of degree $n$. We use the fact that $\tx{Id}\circ \alpha = \alpha$: \begin{itemize}
		\item On the one hand, $\alpha$ is completely described by $\alpha_1 \in \Mapsub{\QC(X)} \left( \_O_X, \Ll_X[n] \right)$. 
		
		\item On the other hand, the map $\tx{Id}: T^*[n] X \rightarrow T^*[n]X$ is described by:\[\pi : T^*[n]X \rightarrow X\]
		\[\lambda_X \in \Mapsub{\QC(T^*[n]X)} \left( \_O_{T^*[n]X}, \pi^* \Ll_X \right)\]
		Therefore the composition\footnote{If we have a composition $Y \to \Aa_X(\_F) \to \Aa_X(\_G)$ where $Y \to \Aa_X(\_F)$ is given by $f:Y \to X$ and $s_f \in \Map(\_O_Y, f^* \_F)$ and $g:  \Aa_X(\_F) \to \Aa_X(\_G)$ is given by $\pi: \Aa_X(\_F) \to X$ and $s_g$ is induced by a map $g: \_F \to \pi^*\_G$, then the composition $g \circ f$ is described by $f: Y \to X$ and $s_{ g\circ f} \in \Map(\_O_Y, f^*\pi^*\_G)$ given by $s_{ g\circ f} := h \circ \pi^*s_f$.} $\tx{Id} \circ \alpha$ is also a section of $\pi$ and is described by $ \alpha^* \lambda_X \in \Mapsub{\QC(X)} \left( \_O_X, \Ll_X[n] \right)$.
	\end{itemize}

	This proves that $\alpha^* \lambda_X = \alpha_1$. Since these maps characterise the sections of $\pi$ they represent, we have $\alpha^* \lambda_X = \alpha$.  
\end{proof}

\begin{Ex}\label{ex:quotient coadjoint action is shifted symplectic} 
	
	In Lemma \ref{lem:cotagent BG and coadjoint quotient}, we will see that if $G$ is an affine algebraic group then: 
	\[\QS{\G_g^*[n]}{G} \simeq T^*[n+1]\bf{B}G\]
	Together with Example \ref{ex:canonical symplectic structure cotangent}, this shows that $\QS{\G_g^*[n]}{G}$ is canonically $(n+1)$-shifted symplectic.    
\end{Ex}

\subsubsection{Lagrangian structures} \label{sec:lagrangian-structure}\

\medskip

We recall from \cite{PTVV} the definition and standard properties of Lagrangian structures and Lagrangian correspondences. These notions will be the building block of all the ``Lagrangian intersections'' theorems discussed later on. \\

Recall that classically, a Lagrangian is a sub-manifold $f: L \hookrightarrow M$ of a symplectic manifold $M$ such that: 
\begin{itemize}
	\item It is isotropic, i.e. $f^*\omega = 0$.
	\item $\omega$ induces an isomorphism to the conormal of $f$: 
	\[ TL \to N^*f \]
\end{itemize} 

The first condition we will translate as the data of an isotropic \emph{structure} on $f$. 

\begin{Def}
	\label{def:isotropic structure}
	
	Let $ f : L \rightarrow X$ be a map of derived Artin stacks. An \defi{isotropic structure on $f$} is a homotopy, in $\_A^{2,\tx{cl}}(L,n)$, between $f^* \omega$ and $0$ for some $n$-shifted symplectic structure $\omega : \star \rightarrow \Symp(X,n)$. Isotropic structures on $f$ form a space described by the homotopy pullback:
	
	\[ \begin{tikzcd}
		\Iso(f,n) \arrow[r] \arrow[d] & \Symp(X,n) \arrow[d, "f^*"] \\
		\star \arrow[r, "0"] & \_A^{2,\tx{cl}}(L,n)
	\end{tikzcd}\]
	
	If we fix a given $n$-shifted symplectic structure $\omega:  \star \rightarrow \Symp(X,n)$, we can define the space of isotropic structures on $f$ at $\omega$ defined by by the pullback:
	
	\[ \begin{tikzcd}
		\Iso(f,\omega) \arrow[r] \arrow[d] & \star \arrow[d, "\omega"] \\
		\Iso(f,n) \arrow[r] & \Symp(X,n)
	\end{tikzcd}\]
\end{Def}

\begin{RQ}\label{rq:isotropic structures are good}
	Take $\gamma$ a homotopy between $\alpha$ and $\beta$ in $\_A^{p, \tx{cl}}(X,n)$. Then $\gamma$ is given by a homotopy between the images of $\alpha$ and $\beta$ in $\relg{F^p\cpl{\DR}(X)[n+p]}$ such that the underlying weight $2$ component: 
	\[ \gamma_0 : \alpha_0 \rightsquigarrow \beta_0\]
	is equivalent to a homotopy in $\Gamma\left( \left(\Sym_{\_O_X} \Ll_X[-1]\right)[n+p] \right)$. 
\end{RQ}

\begin{RQ}
	\label{rq:description isotropic structure}
	More explicitly, if $X$ is Artin, an isotropic structure is given by a family of forms of total degree $(p+n-1)$, $(\gamma_i)_{i \in \Nn}$ with: \[\gamma_i \in \DR(L){(p+i)}[p+n+i-1]\] such that $\delta \gamma_0 = f^*\omega_0$ and $\delta \gamma_i + d \gamma_{i-1} = f^*\omega_i$. This can be rephrased as $D \gamma = f^* \omega$, in other words, $f^*\omega$ in exact in $F^p \cpl{\DR}(X)$ and $\gamma$ is indeed a homotopy between $f^*\omega$ and $0$.     
\end{RQ}

The second condition to be Lagrangian is a non-degeneracy condition of an isotropic structure $\gamma$. 

\begin{Def}
	\label{def:lagrangian structure}
	
	An isotropic structure $\gamma$ on $f: L \rightarrow X$ is called a \defi{Lagrangian structure on f} if the leading term, $\gamma_0$, viewed as an isotropic structure on the morphism\footnote{Given by the homotopy $\gamma_0$ between $f^*\omega_0$ and $0$.} $\Tt_L \rightarrow f^* \Tt_X$, is non-degenerate. We say that $\gamma_0$ is \defi{non-degenerate} if the following null-homotopic sequence\footnote{This sequence only makes sense because the underlying weight $2$ component of the symplectic structure and of the homotopy come from derived global sections as explained in Remark \ref{rq:isotropic structures are good}.} (homotopic to 0 via $\gamma_0$) is fibered:
	
	\begin{equation}
		\label{dia:lagrangian nd}
		\begin{tikzcd}
			\Tt_L \arrow[r] \arrow[rr, bend right, "(f^*\omega_0)^{\flat}"'] & f^*\Tt_X \simeq f^* \Ll_X[n] \arrow[r] &\Ll_L[n]
		\end{tikzcd}
	\end{equation}
	
	The space of $n$-shifted Lagrangian structures on $f$ is denoted $\Lag(f,n)$. There are natural morphisms of spaces: \[\Lag(f,n) \rightarrow \Iso(f,n) \rightarrow \Symp(X,n)\]
\end{Def}

\begin{RQ}
	\label{rq:lagrangian nd}
	
	To say that the sequence \eqref{dia:lagrangian nd} is fibered can be reinterpreted as a more classical condition involving the conormal. Since $\QC(X)$ is a stable $\infty$-category, the homotopy fiber of $f^* \Ll_X [n] \rightarrow \Ll_L[n]$ is 
	$\Llr{L}{X}[n-1]:=\Ll_f[n-1]$ and the non-degeneracy condition can be rephrased by saying that the natural map $\Theta_f : \Tt_L \rightarrow \Ll_f [n-1]$ is a quasi-isomorphism.   
\end{RQ}

\begin{RQ} 
	Unlike in the classical case, being Lagrangian for a morphism $f$ is a \emph{structure} and not a property. To simplify the notations, we will abusively say that a morphism $f: X \rightarrow Y$ \emph{is} Lagrangian when we consider $f$ together with a fixed Lagrangian structure on $f$ and a fixed symplectic structure $\omega$.    
\end{RQ}

\begin{Ex}
	\label{ex:lagrangian closed 1-form}
	A $1$-form of degree $n$ on an Artin stack $X$ is equivalent to a section $\alpha : X \rightarrow T^*[n]X$. This section is a Lagrangian morphism if and only if $\alpha$ admits a closure, i.e. $\bf{Key}(\alpha)$ is non-empty. This is \cite[Theorem 2.15]{Ca19}.  
\end{Ex} 

\begin{Prop}
	\label{prop:space of keys and isotropic structures}
	There is a canonical homotopy equivalence $\Iso(\alpha) \rightarrow \bf{Key}(\alpha)$ between the space of isotropic structures on the $1$-form $\alpha$ and the space of keys of $\alpha$.
\end{Prop}

\begin{proof}
	\begin{equation}
		\begin{tikzcd}
			\bf{key}(\alpha) \arrow[r] \arrow[d] & \_A^{1,\tx{cl}}(X,n) \arrow[r] \arrow[d] & \star \arrow[d, "0"]\\
			\star \arrow[r,"\alpha"]& \_A^1(X,n) \arrow[r, "\dr"] & \_A^{2,\tx{cl}}(X,n)
		\end{tikzcd}
	\end{equation}
	
	The left most square is Cartesian by definition of $\bf{key}(\alpha)$ in Definition \ref{def:space of keys}. By definition, the pullback of the outer square is $\Iso(\alpha)$ because $\dr\alpha = \alpha^* \omega$ (by universal property of the Liouville 1-form, Lemma \ref{lem:universal property tautological 1-form}). 
	It turns out that the rightmost square is also Cartesian. This is simply saying that the space of closed $1$-forms of degree $n$ is the same as the space of $1$-forms of degree $n$ whose de Rham differential is homotopic to $0$. We obtain that $\bf{key}(\alpha)$ and $\Iso(\alpha)$ are both pullbacks of the outer square and therefore are canonically homotopy equivalent. 
\end{proof}

\begin{RQ}
	\label{rq:closed 1-forms are lagrangian}
	It turns out that \cite[Theorem 2.15]{Ca19} says that all the isotropic structures on $\alpha$ (or equivalently the lifts of $\alpha$ to a closed form) are in fact non-degenerate, which implies Example \ref{ex:lagrangian closed 1-form} and even the fact that the space of Lagrangian structures on $\alpha$ is equivalent to the space of keys of $\alpha$. 
\end{RQ}

\begin{Ex} 
	\label{ex:lagrangian in quotient cadjoint action.}

	Let $G$ be an affine algebraic group. The map $\bf{B}G \to \QS{\G_g^*[n]}{G}$, induced by the $G$-equivariant inclusion $0 \to \G_g^*[n]$, is Lagrangian. In fact, under the equivalence of Lemma \ref{lem:cotagent BG and coadjoint quotient}, it is equivalent to the zero section:
	\[\bf{B}G \to T^*[n+1]\bf{B}G\] 
\end{Ex}

\begin{Ex} \label{ex:lagrangians in coadjoint action quotient}
	
	We will see in Sections \ref{sec:some-symplectic-structure-on-quotient-by-a-group-action} and \ref{sec:for-groups} that if $G$ is an affine algebraic group, the following maps have the following Lagrangian structures: 
	\begin{itemize}
		\item The natural projection (Proposition \ref{prop:lagrangian on projection to quotient codadjoint action}): 
		\[ \G_g^* \to \QS{\G_g^*}{G}\]
		\item If $\mu : X \to \G_g^*$ is a moment map, then it induces a Lagrangian (Proposition \ref{prop:moment map equivariant is lagrangian}):
		\[\eq{\mu} : \QS{X}{G} \to \QS{\G_g^*}{G}\]
		\item The map \[\bf{B}{G} \to \QS{\G_g^*}{G}\]
		corresponding to the zero section $ \bf{B}G \to T^* [n+1]  \bf{B}{G}$ via the equivalence of Lemma \ref{lem:cotagent BG and coadjoint quotient}.
	\end{itemize}
\end{Ex}

\begin{Lem}[{\cite[Example 1.26]{Ca21}}]
	\label{lem:lagrangian structure over a point}
	
	Consider the map $X \rightarrow \star_n$ where $\star_n$ is the point endowed with the canonical $n$-shifted symplectic structure given by $0$. Then a Lagrangian structure on this map is equivalent to an $(n-1)$-shifted symplectic structure on $X$. \\
\end{Lem}
\begin{proof}
	
	Pick an isotropic structure $\gamma$ on $p$. We know that $\gamma$ is a homotopy between 0 and 0 which means that $D\gamma = 0$. Therefore $\gamma$ is a closed 2-form of degree $n-1$. We want to show that $\gamma$ is non-degenerate as an isotropic structure if and only if it is non-degenerate as a closed 2-form on $X$. The non-degeneracy of the Lagrangian structure, as described in Remark \ref{rq:lagrangian nd}, corresponds to the requirement that the natural map $ \Tt_X \rightarrow \Ll_X[n-1]$ is a quasi-isomorphism. This map depends on $\gamma_0$ and we want to show that this map is in fact $\gamma_0^{\flat}$. This map is the natural map that fits in the following homotopy commutative diagram:
	
	$$ \begin{tikzcd}[row sep = small]
		\Tt_X \arrow[r]\arrow[d] & \Ll_X[n-1] \arrow[r] \arrow[d] & 0 \arrow[d] \\
		0 \arrow[r, "0^{\flat}"] & 0 \arrow[r] & \Ll_X [n] 
	\end{tikzcd}$$
	We can show that by strictifying the homotopy commutative diagram:
	
	$$\begin{tikzcd}[row sep = small]
		\Tt_X \arrow[dr] \arrow[drr, "p^*\omega^{\flat}=0"] & &  \\
		& 0 \arrow[r]& \Ll_X[n]
	\end{tikzcd}
	$$
	Note that this diagram is already commutative but we see it as homotopy commutative using the homotopy $\gamma_0$. We use the homotopy $\gamma_0$ to strictify the previous diagram and we obtain:
	
	$$\begin{tikzcd}[row sep = small]
		\Tt_X \arrow[dr, "\gamma_0^{\flat} + 0"'] \arrow[drr, "p^*\omega^{\flat}=0"] & &  \\
		& \Ll_X[n-1]\oplus \Ll_X[n] \arrow[r, "\tx{pr}"']& \Ll_X[n]
	\end{tikzcd}
	$$
	
	The homotopy fiber and also strict fiber of the projection \[\tx{pr} : \Ll_X[n-1] \oplus \Ll_X[n] \rightarrow \Ll_X[n]\] is $\Ll_X [n-1]$, and therefore the natural map we obtain is $\gamma_0^{\flat} : \Tt_X \rightarrow \Ll_X[n-1]$.\\

	Since the non-degeneracy condition of the isotropic structure $\gamma$ is the same as saying that the map $\gamma_0^{\flat}$ is a quasi-isomorphism, we have shown that an isotropic structure $\gamma$ is an $(n-1)$-shifted symplectic structure on $X$ if and only if it is non-degenerate as an isotropic structure on $X \rightarrow \star_n $.
\end{proof}

A closely related notion to Lagrangian morphisms is that of Lagrangian correspondences. We will see in Section \ref{sec:equivariant-symplectic-geometry} that they are at the heart of \emph{symplectic reduction} and in Section \ref{sec:derived-perspective-of-the-bv-complex} we use Lagrangian correspondences as a defining feature of the notion of ``generalized symplectic reduction'' and of the BV construction.   

\begin{Def}
	\label{def:lagrangian correspondences}
	
	Let $X$ and $Y$ be derived Artin stacks with $n$-shifted symplectic structures. A \defi{Lagrangian correspondence} from $X$ to $Y$ is given by a derived Artin stack $L$ with morphims:
	\[ \begin{tikzcd} 
		&L \arrow[dl] \arrow[dr]& \\
		X & & Y
	\end{tikzcd}\]
	
	and a Lagrangian structure on the map $L \rightarrow X \times \overline{Y}$ where $X \times \overline{Y}$ is endowed with the $n$-shifted symplectic structure $\pi_X^* \omega_X - \pi_Y^* \omega_Y$.  
	
	For example, a Lagrangian structure on $L \rightarrow X$ is equivalent to a Lagrangian correspondence from $X$ to $\star$.
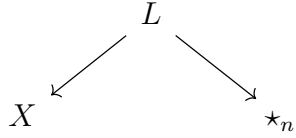
	\[
	\begin{tikzcd}
		& L \arrow[dl] \arrow[dr] & \\
		X&&\star_n
	\end{tikzcd}\]
\end{Def}

A more extensive study of Lagrangian correspondences and their properties is done in Section \ref{sec:the-higher-categories-of-lagrangians}. 

\begin{Ex}[{\cite[Example 2.3]{Ca21}}]\label{ex:conormal lagrangian correspondence}
	Given a  map $f: X \to Y$, its graph $X \to X \times Y$ has a conormal $N(f^* ) := T^*Y \times_X Y $ with a Lagrangian correspondence: 
	
	
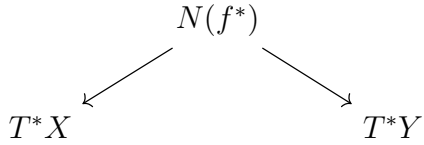
\[	 \begin{tikzcd}
		& N(f^* )\arrow[dl] \arrow[dr] & \\
		T^* X && T^* Y
	\end{tikzcd}\]
	which is moreover natural in $f$. 
\end{Ex}

\subsubsection{Lagrangian fibrations}\label{sec:lagrangian-fibrations}\

\medskip

We recall in this section the definition and standard properties of Lagrangian fibrations (see \cite{Ca19}).  

\begin{Def}
	\label{def:lagrangian fibration}
	Let $f : Y \rightarrow X$ be a map of derived stacks. A \defi{Lagrangian fibration} on $f$ is given by: 
	
	\begin{itemize}
		\item A homotopy, denoted $\gamma$, between ${\omega}_{/X}$ and $0$, where ${\omega}_{/X}$ is the image of $\omega$ under the natural map $\_A^{2,\tx{cl}}(Y,n) \rightarrow \_A^{2,\tx{cl}}(Y/X, n)$ (see Remark \ref{rq:p-forms (closed) relative version}) for some $n$-shifted symplectic structure $\omega : \star \rightarrow \Symp(Y,n)$. This forms a space of isotropic fibrations\footnote{The notion of isotropic fibration makes sense even if $\omega$ is not non-generate by replacing $\Symp(Y,n)$ by $\_A^{2, \tx{cl}}(Y,n)$.} given by the pullback: 
		
		\[ \begin{tikzcd}
			\IsoFib(f,n)  \arrow[r]  \arrow[d] & \Symp(Y,n) \arrow[d] \\
			\star \arrow[r, "0"] & \_A^{2,\tx{cl}} ( Y/X, n) 
		\end{tikzcd}\]
		
		\item A non-degeneracy condition which says that the following sequence (homotopic to $0$ via $\gamma_0$) is fibered\footnote{Where again this sequence only makes sense because of Remark \ref{rq:isotropic structures are good}.} 
		
		\begin{equation*}
			\label{eq:lagrangian ND fiber sequence_LagrangianFiberSequence}
			\Tt_{Y/X} \rightarrow \Tt_Y \simeq \Ll_Y [n] \rightarrow \Ll_{Y/X}[n]  
		\end{equation*}
	\end{itemize}
	
	In particular, the non-degeneracy condition can be rephrased by saying that there is a canonical equivalence $\alpha_f : \Tt_{Y/X} \rightarrow f^*\Ll_{X}[n]$ (similar to the criteria for Lagrangian morphism in Remark \ref{rq:lagrangian nd}) that makes the following diagram commute: 
	
	\begin{equation}
		\label{dia:lagrangian fibration ND criteria}
		\begin{tikzcd}
			\Ttr{Y}{X} \arrow[d] \arrow[r, "\alpha_f"] & f^* \Ll_X[n] \arrow[d] \arrow[r] & 0 \arrow[d] \\
			\Tt_Y \arrow[r, "\omega^{\flat}"] & \Ll_Y[n] \arrow[r] & \Llr{Y}{X}[n] 
		\end{tikzcd}
	\end{equation}

	The subspace of $\IsoFib(f,n)$ generated by the non-degenerate objects is the space of Lagrangian fibration structures on $f$ and is denoted by $\LagFib(f,n)$. There are natural maps $\LagFib(f,n) \rightarrow \IsoFib(f,n) \rightarrow \Symp(Y,n)$.\\
	
	Similarly to the case of isotrope and Lagrangian morphisms, we can fix a $n$-shifted symplectic structure on $Y$ and define Lagrangian and isotropic fibration of $f$ at a given $\omega$ as the pullbacks:
	
	\[ \begin{tikzcd}
		\IsoFib(f, \omega) \arrow[r] \arrow[d] & \star \arrow[d, "\omega"] \\
		\IsoFib(f, n) \arrow[r]	& \Symp(Y,n)
	\end{tikzcd}  \qquad \begin{tikzcd}
		\LagFib(f, \omega) \arrow[r] \arrow[d] & \star \arrow[d, "\omega"] \\
		\LagFib(f, n) \arrow[r]	& \Symp(Y,n)
	\end{tikzcd}\] 
	
	and  
\end{Def} 

\begin{RQ}
	To simplify the notations, we will abusively say that a morphism $f: X \rightarrow Y$ \emph{is} a Lagrangian fibration when we consider $f$ together with a fixed shifted symplectic structure $\omega$ and a fixed structure of Lagrangian fibration on $f$ at $\omega$.    
\end{RQ}

\begin{Ex}
	\label{ex:lagrangian fibration cotangent}
	The natural projection $\pi_X : T^*[n]X \rightarrow X$ is a Lagrangian fibration.  The Liouville 1-form is a section of $\pi_X^* \Ll_X[n]$ which is part of the fiber sequence: 
	
	$$ \pi_X^* \Ll_X[n] \rightarrow \Ll_{T^*[n]X} [n] \rightarrow \Llr{T^*[n]X}{X}[n]$$
	
	Thus the 1-form induced by $\lambda_X $ in $\Llr{T^*[n]X}{X}[n]$ is homotopic to 0. The non-degeneracy condition is more difficult and is proven in Section 2.2.2 of \cite{Ca19}. 
	
	It turns out that the morphism expressing the non-degeneracy condition, $\alpha_{\pi_X}$, is given by a canonical equivalence described in Proposition \ref{prop:relative cotangent complex for linear stacks}.
	
\end{Ex}

\begin{Prop}
	\label{prop:cotangent ND lag fibration is natural}
	
	If $X$ is a derived Artin stack, then the equivalence of Example \ref{ex:lagrangian fibration cotangent} expressing the non-degeneracy of the canonical Lagrangian fibration on the shifted cotangent stacks, $\alpha_{\pi_X} : \Ttr{T^*[n]X}{X} \rightarrow \pi_X^* \Ll_X[n]$, is the canonical quasi-isomorphism from Proposition \ref{prop:relative cotangent complex for linear stacks}.
\end{Prop}
\begin{proof}
	
	Consider the following diagram: 
	\[ \begin{tikzcd}
		\Ttr{T^*X}{X} \arrow[d, "\beta"] \arrow[dd, bend right=40, "\alpha"'] & & \\ 
		\pi^* \Ll_X \arrow[r] \arrow[d, "\phi_X"] & \Tt_{T^*X} \arrow[r] \arrow[d, "\omega_0^\flat"] & \pi^* \Tt_X  \arrow[d]\\
		\pi^*\Ll_X \arrow[r] & \Ll_{T^*X} \arrow[r] & \pi^*\Tt_X
	\end{tikzcd} \]
	
	From \cite[Lemma 2.6]{Ca19}, the map $\phi_X$ is natural in $X$ and therefore is an equivalence (\cite[Corollary 2.7]{Ca19}). Moreover, the naturality implies that $\phi_X$ is the multiplication by a scalar. Looking at $X = \Aa^1$ in local coordinates, it is not hard to see that this scalar is in fact $1$\footnote{This should be related to the fact that the canonical symplectic form is induced by $\id: T^*X \to T^*X$ and there is no dilatation along the fiber.} and therefore $\phi_X$ is the identity. 
\end{proof} 

\begin{Lem}
	\label{lem:lagrangian fibration from a point}
	
	Let $x: \star_n \rightarrow X$ be a point of $X$. Then, given a Lagrangian fibration structure on $x$, the non-degeneracy condition is given by a quasi-isomorphism: \[x^* \Tt_X \rightarrow x^*\Ll_X[n+1]\] 
\end{Lem}
\begin{proof}
	The Lagrangian fibration structure on $\star_n \rightarrow X$ is a homotopy between 0 and itself in $\_A^{2,\tx{cl}}(\faktor{\star_n}{X},n)$. As in the proof of Lemma \ref{lem:lagrangian structure over a point}, this is given by an element $\gamma \in \_A^{2,\tx{cl}}(\faktor{\star_n}{X},n-1)$. Similarly to what was done in the proof of Lemma \ref{lem:lagrangian structure over a point}, we can show that $\gamma$ is non-degenerate as a Lagrangian fibration if and only if it is non-degenerate as a closed 2-form of degree $n$. Again it boils down to the fact that the natural morphism in the non-degeneracy criteria for Lagrangian fibrations is in fact $\gamma_0^{\flat} : \Ttr{\star_n}{X} \rightarrow \Llr{\star_n}{X}[n-1]$.  
	
	Moreover, we have natural equivalences, $\Ttr{\star_n}{X} \simeq x^*\Tt_X [-1]$ and $ \Llr{\star_n}{X}[n-1] \simeq x^*\Ll_X[n]$ because the sequence
	\[ \begin{tikzcd}
		\Ttr{\star_n}{X} \arrow[r] & \Tt_{\star_n} \simeq 0 \arrow[r] &  x^* \Tt_X[n]
	\end{tikzcd} \]
	
	is fibered. This concludes the proof. 
\end{proof}

\begin{Ex}
	\label{ex:lagrangian fibration projection cadjoint quotient to BG}
	
	The natural map $\QS{\G_g^*}{G} \to \bf{B}G$ is equivalent to the projection $T^*[n+1]\bf{B}G \to \bf{B}G$ through the equivalence of Lemma \ref{lem:cotagent BG and coadjoint quotient}. Therefore this is a Lagrangian fibration. 
\end{Ex}

\subsection{New Constructions from Old Ones}\ \label{sec:new-constructions-from-old-ones}

\medskip

In this section we will recall how to construct new shifted symplectic structures by taking derived intersection of Lagrangian morphisms (Proposition \ref{prop:lagrangian intersection are shifted symplectic}) and then extend this result to construct new Lagrangian fibrations (Theorem \ref{th:derived intersection lagrangian fibraton}).

The gymnastic of constructing new Lagrangians and Lagrangian correspondences is very well behaved, and in Section \ref{sec:the-higher-categories-of-lagrangians} we use this to described a category and a weak $2$-category structure on the space of Lagrangians. This leads to another construction that produces new Lagrangian correspondences out of derived intersections (Theorem \ref{th:lagragian correspondence pullback}).   

\subsubsection{Derived intersection of Lagrangians morphisms}\label{sec:derived-intersection-of-lagrangians-morphisms}

\medskip

\begin{Prop}[{\cite[Section 2.2]{PTVV}}]
	\label{prop:lagrangian intersection are shifted symplectic}
	Let $Z$ be a derived Artin stack with a $n$-shifted symplectic structure $\omega$.  Let $f: X \rightarrow Z$ and $g : Y \rightarrow Z$ be morphisms with $\gamma$ and $\delta$ Lagrangian structures on $f$ and $g$ respectively. Then the homotopy pull-back $X \times_Z Y$ possesses a canonical $(n-1)$-shited symplectic structure called the residue of $\omega$ and denoted $R(\omega, \gamma,\delta)$. 
\end{Prop}

\begin{RQ}
	\label{rq:lagrangian intersection map}
	If we fix $f$ and $g$ as above, we can extend the previous theorem to obtain a map of spaces (see Theorem 2.4 in \cite{Ca15}):
	\[ \mathbf{Lag}(f,n) \times_{\mathbf{Symp}(X,n)} \mathbf{Lag}(g,n) \rightarrow \mathbf{Symp}(X\times_Z Y, n-1)\]
\end{RQ}

\begin{Ex}\  
	\label{ex:derived critical locus}	The derived critical locus is by definition the derived pullback of two closed $1$-form $s_0, df: X \to T^*X$ which are therefore Lagrangian thanks to Example \ref{ex:lagrangian closed 1-form}. Therefore  $\RCrit(f)$ is $(-1)$-shifted symplectic. 
\end{Ex}

\begin{RQ}
	\label{rq:self intersection df=0}
	When $X$ is a derived stack and $df=0$, we have that $\RCrit(f) \simeq T^*[-1]X$ and $\omega_{\RCrit(f)}$ is the canonical $(-1)$-shifted symplectic structure on $T^*[-1]X$.  
\end{RQ}

\begin{RQ}
	Proposition \ref{prop:lagrangian intersection are shifted symplectic} can also be seen as a consequence of the procedure of composition of Lagrangian correspondences (Theorem \ref{th:lagrangian correspondence composition}). To se that, consider the following composition of Lagrangian correspondences: 
	
	
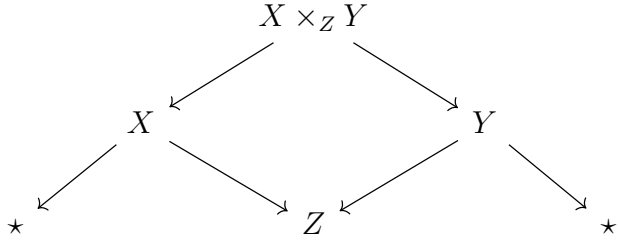
\[ \begin{tikzcd}
		& & X \times_Z Y \arrow[dl] \arrow[dr] & & \\
		& X \arrow[dl] \arrow[dr]& & Y \arrow[dl] \arrow[dr] & \\
		\star & & Z  & & \star 
	\end{tikzcd}\]
	
	The maps $X \rightarrow Z \times \bar{\star}$ and $Y \rightarrow Y \times \bar{\star}$ are Lagrangian correspondences because $X \rightarrow Z$ and $Y \rightarrow Z$ are Lagrangian. Therefore, by composition, $X \times_Z Y \rightarrow \star \times \bar{\star} $ is also a Lagrangian correspondence, thus $X \times_Z Y \rightarrow \star$ is Lagrangian. From Lemma \ref{lem:lagrangian structure over a point}, since the point is $n$-shifted symplectic, then $X \times_Z Y$ is $(n-1)$-shifted symplectic.
\end{RQ}
%pull-back along Lagrangian morphims 

\subsubsection{Derived intersections and Lagrangian fibrations}\label{sec:derived-intersection-and-lagrangian-fibrations}\

\medskip
The goal of this section is to prove a similar statement to Proposition \ref{prop:lagrangian intersection are shifted symplectic} but with the additional data of a Lagrangian fibration. This result is from \cite{Gr22} and has also been generalized in \cite{Sa20}.

\newpage

Before going to the main statement, we need the following proposition:

\begin{Prop}
	\label{prop:lagrangian fibration prop 1}
	
	Suppose we have a sequence $\begin{tikzcd}
		L \arrow[r, "f"]&  Y \arrow[r,"g"]& X
	\end{tikzcd}$ of Artin stacks and $\omega$ a $n$-shifted symplectic form on Y. Assume that $f $ is a Lagrangian morphism and $g$ is a Lagrangian fibration. Then there is a canonical equivalence $\Ttr{L}{X} \rightarrow \Llr{L}{X}[n-1]$. \\
\end{Prop}

\begin{proof}

	Consider the following commutative diagram: \\
	
	\adjustbox{scale=0.8,center}{
		\begin{tikzcd}[row sep=0.8cm, column sep=0.5cm]
			\Tt_L \arrow[dd, "\simeq", near end] \arrow[rr] & & f^* \Tt_Y \arrow[dd, "\simeq", near end] \arrow[rr] & & (g\circ f )^* \Tt_X \arrow[dd, dashed, "\simeq", near end] & \\
			& \Ttr{L}{X} \arrow[dd, dashed,crossing over,  "\simeq", near start] \arrow[lu] \arrow[rr, crossing over]& & f^* \Ttr{Y}{X} \arrow[dd,crossing over,  "\simeq", near start] \arrow[lu] \arrow[rr, crossing over] & & 0 \arrow[dd, crossing over, "\simeq"] \arrow[ul] \\
			\Llr{L}{Y}[n-1] \arrow[rr] & & f^* \Ll_Y [n] \arrow[rr] & & f^*\Llr{Y}{X}[n] & \\
			&\Llr{L}{X}[n-1] \arrow[ul] \arrow[rr] & &  (g \circ f)^* \Ll_X [n] \arrow[ul] \arrow[rr] & & 0 \arrow[ul]
		\end{tikzcd}
	}\\

	In the upper face, every square is bi-Cartesian because both the outer square and the right most square are bi-Cartesian. Every non-dashed vertical arrow is an equivalence by assumption (because of the various non-degeneracy conditions). Focusing on the right hand cube, it sends the upper homotopy bi-Cartesian square to the bottom square which is also homotopy bi-Cartesian. The homotopy cofiber of $(g \circ f)^* \Ll_X [n] \rightarrow f^* \Ll_Y[n]$ is $f^* \Llr{Y}{X} [n]$ and we obtain a quasi-isomorphism $(g \circ f)^* \Tt_X \rightarrow f^* \Llr{Y}{X}[n]$ depicted as a dashed arrow. 
	
	By the same reasoning, since the upper outer square is homotopy bi-Cartesian, it maps to the lower outer square who is also homotopy bi-Cartesian. Moreover, the homotopy fiber of the map $\Llr{L}{Y} [n-1] \rightarrow f^* \Llr{Y}{X}[n]$ is exactly $\Llr{L}{X}[n-1]$. This proves that there is a canonical quasi-isomorphism $\Ttr{L}{X} \rightarrow \Llr{L}{X} [n-1]$. 
\end{proof}

\begin{Th}[{\cite[Theorem 3.5]{Gr22}}]
	\label{th:derived intersection lagrangian fibraton}
	Let $Y$ be a $n$-shifted symplectic derived Artin stack. Let $f_i: L_i \rightarrow Y$ be Lagrangian morphisms (for $i= 1 \cdots 2$) and $\pi : Y \rightarrow X$ a Lagrangian fibration. Suppose that the maps $\pi \circ f_i : L_i \rightarrow X$ are weak equivalences. Then $P : Z = L_1 \times_Y L_2 \rightarrow X$ is a Lagrangian fibration.
\end{Th}
\begin{proof}
	
	We summarize the notation in the following diagram:
	
	$$\begin{tikzcd}
		Z \arrow[r, "p_1"] \arrow[dr, "F"] \arrow[d, "p_2"] & L_1 \arrow[d, "f_1"]  & \\
		L_2 \arrow[r,"f_2"] & Y \arrow[dr, " \pi"] & \\
		& & X 
	\end{tikzcd}
	$$
	
	We also denote $P := \pi \circ F : Z \rightarrow X $. \\

	To show that we can obtain an isotropic structure, we will show that we have a map of spaces (dropping at first the non-degeneracy condition of the Lagrangian fibration): 
	\[ \Lag(f_1,n) \times_{\Symp(Y,n)} \Lag(f_2,n) \times_{\Symp(Y,n) }\IsoFib(\pi, n) \rightarrow\IsoFib(P,n-1)  \]
	
	If we forget the non-degeneracy of the Lagrangian structure, we obtain an element in $ \Iso(f_1,n) \times_{\Symp(Y,n)} \Iso(f_2,n) \times_{\Symp(Y,n) }\IsoFib(\pi, n)$ and we can show, with formal manipulations of the pullbacks defining the spaces of isotropic structures and isotropic fibrations, that:
	\[  \Iso(f_1,n) \times_{\Symp(Y,n)} \Iso(f_2,n) \times_{\Symp(Y,n) }\IsoFib(\pi, n)\] 
	\[ = \star \times_{\_A^{2,cl}(L_1,n)} \Symp(Y,n) \times_{\_A^{2,cl} (L_2,n)} \star  \times_{\_A^{2,cl}(Y/X,n)} \star \]
	
	Using the pullback to $\_A^{2,cl}(L_1 \times_Y L_2, n)$ we obtain a morphism: 
	\[  \Iso(f_1,n) \times_{\Symp(Y,n)} \Iso(f_2,n) \times_{\Symp(Y,n) }\IsoFib(\pi, n) \rightarrow\] 
	\[  \star \times_{\_A^{2,cl}(L_1 \times_Y L_2,n)} \_A^{2,cl}(L_1 \times_Y L_2,n) \times_{\_A^{2,cl} (L_1 \times_Y L_2,n)} \star \times_{\_A^{2,cl}(Y/X,n)} \star\]
	
	This last space naturally maps to: 
	\[ \_A^{2,cl}(L_1 \times_Y L_2, n-1) = \star \times_{\_A^{2,cl}(L_1 \times_Y L_2,n)} \_A^{2,cl}(L_1 \times_Y L_2,n) \times_{\_A^{2,cl} (L_1 \times_Y L_2,n)} \star \] Moreover, if we restrict this map to non-degenerate isotropic structures, then it is valued in $\Symp(L_1 \times_Y L_2, n-1)$ (thanks to Proposition \ref{prop:lagrangian intersection are shifted symplectic}).
	
	We also have that:
	\[  \_A^{2,cl}(Y/X,n-1) = \star \times_{\_A^{2,cl}(Y/X,n)} \star \]
	
	We have the commutative diagram:
	\[ \begin{tikzcd}
		\_A^{2,cl}(L_1 \times_Y L_2, n-1) \times \_A^{2,cl}(Y/X,n-1) \arrow[r] \arrow[d] & \_A^{2,cl}(Y/X,n-1) \arrow[d, "P^*"] \\
		\_A^{2,cl}(L_1 \times_Y L_2, n-1) \arrow[r] &  \_A^{2,cl}(L_1 \times_Y L_2/X, n-1)
	\end{tikzcd}\]
	
	Since the map $\_A^{2,cl}(Y/X,n-1)  \rightarrow \_A^{2,cl}(L_1 \times_Y L_2/X, n-1)$ factors through $\_A^{2,cl} (L_i/X, n-1) \simeq \star$, we get a morphism: 
	\[ 	\_A^{2,cl}(L_1 \times_Y L_2, n-1) \times \_A^{2,cl}(Y/X,n-1) \rightarrow \_A^{2,cl}(L_1 \times_Y L_2, n-1) \times_{\_A^{2,cl}(L_1 \times_Y L_2/X,n-1)} \star \]
	
	Now if we restrict to $\Symp(L_1 \times_Y L_2, n-1) \subset \_A^{2,cl}(L_1 \times_Y L_2, n-1)$ (which amounts to restricting to non-degenerate isotropic structures), we get a map: 
	\[ \Symp(L_1 \times_Y L_2, n-1) \times \_A^{2,cl}(Y/X,n-1) \rightarrow\IsoFib(P, n-1) \]
	
	Therefore we get the desired map and we will consider the isotropic fibration on $P$ given by the image along the morphism we just described of the Lagrangian structures and Lagrangian fibration structure given on $f_1$, $f_2$ and $\pi$ respectively. We are left to prove the non-degeneracy condition. To do that, we first consider the following diagram: 
	
	\[ \begin{tikzcd}
		\Ttr{Z}{X} \arrow[d] \arrow[r] &   p_1^* \Ttr{L_1}{X}\oplus p_2^* \Ttr{L_2}{X} \arrow[d] \arrow[r] & F^*\Ttr{Y}{X} \arrow[d] \\
		\Tt_{Z} \arrow[d] \arrow[r] &   p_1^* \Tt_{L_1} \oplus p_2^* \Tt_{L_2} \arrow[d] \arrow[r] & F^*\Tt_{Y} \arrow[d] \\
		P^*\Tt_X \arrow[r] &  P^*\Tt_{X} \oplus  P^*\Tt_{X}  \arrow[r] & P^*\Tt_{X} 
	\end{tikzcd} \]
	
	The vertical sequences and the last two horizontal sequences are fibered and therefore so is the first horizontal sequence. The last two horizontal sequences are fibered because the following diagrams are Cartesian: 
	\[ \begin{tikzcd}
		\Tt_Z \arrow[r] \arrow[d] &  p_1^* \Tt_{L_1} \arrow[d]  \\
		p_2^* \Tt_{L_2} \arrow[r] &  F^*\Tt_{Y}
	\end{tikzcd} \qquad
	\begin{tikzcd}
		\Tt_X \arrow[r] \arrow[d] & \Tt_X \arrow[d]  \\
		\Tt_X \arrow[r] &  \Tt_X
	\end{tikzcd}
	\] 
	
	Using Proposition \ref{prop:lagrangian fibration prop 1} and non-degeneracy, we get the following commutative diagram: 
	
	\[\begin{tikzcd}
		\Ttr{Z}{X} \arrow[r] \arrow[d, dashed] & p_1^* \Ttr{L_1}{X} \oplus p_2^* \Ttr{L_2}{X} \arrow[d] \arrow[r] & F^*\Ttr{Y}{X} \arrow[d] \\
		P^*\Ll_X[n-1] \arrow[r] & \left(p_1^* \Ttr{L_1}{X} \oplus p_2^* \Ttr{L_2}{X}\right)[n-1] \arrow[r] & P^* \Ll_X[n]
	\end{tikzcd}\]  
	
	where all vertical morphisms are quasi-isomorphisms. The fiber in the lower sequence is exactly $\Ll_X[n-1] $ because $p_1^*\Ttr{L_1}{X} \oplus p_2^* \Ttr{L_2}{X} \simeq 0$ since $L_i \to X$ are equivalences. We will call $\alpha : 	\Ttr{Z}{X} \to 	P^*\Ll_X[n-1] $ the dashed equivalence obtained. 
	
	We still need to show that $\alpha$ is the morphism used in the criteria for the non-degeneracy of the Lagrangian fibration. Recall that this morphism is given by means of the universal map filling Diagram \eqref{dia:lagrangian fibration ND criteria}:
	
	\begin{equation*}
		\begin{tikzcd}
			\Ttr{Z}{X} \arrow[d] \arrow[r, "\alpha_P"] & P^* \Ll_X[n-1] \arrow[d] \arrow[r] & 0 \arrow[d] \\
			\Tt_Z \arrow[r, "\sim"] & \Ll_Z[n-1] \arrow[r] & \Llr{Z}{X}[n-1] 
		\end{tikzcd}
	\end{equation*}
	
	To compare $\alpha$ and $\alpha_P$, we summarize the construction of $\alpha$ and all the equivalences coming from non-degeneracy conditions in the following diagram: 
	
	\begin{equation}
		\label{dia:lagrangian fibration derived intersection}
		\adjustbox{scale=0.75,center}{
			\begin{tikzcd}[row sep=2cm, column sep=tiny]
				\Tt_Z \arrow[rr] \arrow[dd, "\simeq", near start] & & p_1^*\Tt_{L_1} \oplus p_2^* \Tt_{L_2} \arrow[rr] \arrow[dd, "\simeq", near start] & & F^*\Tt_Y  \arrow[dd, "\simeq", near start] &  \\
				&\Ttr{Z}{X} \arrow[rr, crossing over] \arrow[ul] \arrow[dd, dashed, "\simeq", crossing over,near start] & & p_1^* \Ttr{L_1}{X} \oplus p_2^* \Ttr{L_2}{X} \arrow[rr, crossing over] \arrow[ul] \arrow[dd, "\simeq", crossing over, near start]& & F^* \Ttr{Y}{X} \arrow[ul] \arrow[dd, "\simeq", crossing over, near start]\\
				\Ll_Z [n-1]  \arrow[rr] & & \left( p_1^* \Llr{L_1}{Y} \oplus p_2^* \Llr{L_2}{Y}\right) [n-1] \arrow[rr]& & F^* \Ll_Y[n]  & \\
				& P^*\Ll_X[n-1] \arrow[ul] \arrow[rr, crossing over]& &\left(  p_1^* \Llr{L_1}{X} \oplus p_2^* \Llr{L_2}{X}\right)[n-1] \arrow[ul] \arrow[rr, crossing over] & & P^*\Ll_X[n] \arrow[ul] 
			\end{tikzcd} 
		}\\
	\end{equation}
	
	where all the vertical maps are quasi-isomorphism obtained from the non-degeneracy conditions. We want to prove that $\alpha_P$ and  $\alpha$ are homotopic. The relevant data extracted from the Diagram \eqref{dia:lagrangian fibration derived intersection} is:
	
	\begin{equation*}
		\begin{tikzcd}
			\Ttr{Z}{X} \arrow[d] \arrow[r, "\alpha"] & P^* \Ll_X[n-1] \arrow[d] \arrow[r] & 0 \arrow[d] \\
			\Tt_Z \arrow[r, "\sim"] & \Ll_Z [n-1] \arrow[r] & p_1^* \Llr{L_1}{Y}[n-1] \oplus  p_2^* \Llr{L_2}{Y}[n-1]  
		\end{tikzcd}
	\end{equation*}

	The composition: 
	\[P^* \Ll_X[n-1] \rightarrow \Ll_Z[n-1] \rightarrow p_1^* \Llr{L_1}{Y}[n-1] \oplus  p_2^* \Llr{L_2}{Y}[n-1] \] factorizes through $0 \simeq p_1^* \Llr{L_1}{X}[n-1] \oplus  p_2^* \Llr{L_2}{X}[n-1]$. This implies that the map $ \Ll_Z[n-1] \rightarrow p_1^* \Llr{L_1}{Y}[n-1] \oplus  p_2^* \Llr{L_2}{Y}[n-1] $ factorizes through $\Llr{Z}{X}[n-1]$ and therefore $\alpha$ satisfies the same universal property as $\alpha_P$, proving that $\alpha$ and $\alpha_P$ are homotopic.   
\end{proof} 

\begin{RQ}
	This result was generalized in \cite[Theorem 1.14]{Sa20}.
\end{RQ}

\begin{RQ}
	Similarly to Proposition \ref{prop:lagrangian intersection are shifted symplectic}, this theorem can be extended to a map of spaces:
	\[ \Lag(f_1,n) \times_{\Symp(Y,n)} \Lag(f_2,n) \times_{\Symp(Y,n)} \LagFib(\pi,n) \rightarrow \LagFib(P,n) \] 
	
	This is the simply the restriction the map described in the proof of Theorem \ref{th:derived intersection lagrangian fibraton} to the non-degenerate elements. Forgetting the extra Lagrangian fibration recovers the map in Remark \ref{rq:lagrangian intersection map}, that is the following diagram is commutative: 
	
	\[  \begin{tikzcd}
		\Lag(f_1,n) \times_{\Symp(Y,n)} \Lag(f_2,n) \times_{\Symp(Y,n)} \LagFib(\pi,n)  \arrow[r] \arrow[d] & \LagFib(P,n)  \arrow[d] \\
		\Lag(f_1,n) \times_{\Symp(Y,n)} \Lag(f_2,n) \arrow[r] & \Symp(L_1 \times_Y L_2, n-1)
	\end{tikzcd}  \]

\end{RQ}

\begin{Ex}\label{ex:derived critical locus lagrangian fibration}
	
	Since the projection $T^*X \to X$ is a Lagrangian fibration from Example \ref{ex:lagrangian fibration cotangent}, and the derived intersection defining the derived critical locus (Definition \ref{def:derived critical locus}) clearly satisfies the condition of Theorem \ref{th:derived intersection lagrangian fibraton}, we get that $\RCrit(f) \to X$ is a Lagrangian fibration. This Lagrangian fibration will be studied in details in Section \ref{sec:non-degenerate-functionals} in good situations. 
\end{Ex}

\begin{RQ}
	\label{rq:self intersection df=0 lagragnian fibration}
	When $X$ is a derived stack and $df=0$, we have that $\RCrit(f) \simeq T^*[-1]X$,$\omega_{\RCrit(f)}$ is the canonical $(-1)$-shifted symplectic structure on $T^*[-1]X$ and the Lagrangian fibration obtain is the canonical Lagrangian fibration $T^*[-1]X \to X$ of Example \ref{ex:lagrangian fibration cotangent}.  
\end{RQ}

\subsubsection{The higher categories of Lagrangians} \label{sec:the-higher-categories-of-lagrangians}\

\medskip

The idea behind the categorical structure on Lagrangian structures is that given $X$ and $Y$, $n$-shifted symplectic derived stacks, a Lagrangian correspondence from $X$ to $Y$ (Definition \ref{def:lagrangian correspondences}) can be though of as a kind of arrow from $X$ to $Y$. But it turns out that this idea is a specific instance of a categorical structure on the set of Lagrangians structures over a fixed base given by a $n$-shifted symplectic derived stacks $X$ (see Definition \ref{def:lagrangian 1 category}). 

In full generality, it is shown in \cite{CHS21} that we can construct an $(\infty, n)$-category of Lagrangians as a $n$-fold Segal space. We are going to restrict ourselves to the $1$- and $2$-categorical structures. \\

The goal of this section is to present rather explicitly the (weak) $2$-category of Lagrangian as described in \cite{AB17} and give a new constructions of Lagrangian correspondences. The main goal of this section is to give the framework to prove Theorem \ref{th:lagragian correspondence pullback} which is another  ``derived intersection theorem'' that produces new Lagrangian correspondences (see Theorem \ref{th:lagragian correspondence pullback}). \\

First recall, from Definition \ref{def:lagrangian correspondences}, that a Lagrangian correspondence from $X$ to $Y$ is given by a Lagrangian $L \to X \times \overline{Y}$ and is depicted by the diagram: 

\[ \begin{tikzcd}
	& L \arrow[dl] \arrow[dr]  & \\
	X & & \overline{Y} 
\end{tikzcd}\]

where $\overline{Y}$ is the same derived stack as $Y$ but endowed with the ``opposite'' symplectic structure $- \omega_Y$.\\ 

Observe that there is an equivalence: 
\[ \begin{tikzcd}
	\Lag(X) \arrow[r, "\sim"] & \Lag(\overline{X})
\end{tikzcd} \]

In particular, any Lagrangian correspondence $L$ from $X$ to $Y$ induces an opposite\footnote{The ``opposite'' Lagrangian does not means that it is an inverse for the composition. However we will see in Corollary \ref{cor:natural 2-morphism} that there is a natural $2$-morphism: \[1_{X \times \overline{X}} \Rightarrow L \circ \overline{L}\]} Lagrangian correspondence, $\overline{L}$ from $Y$ to $X$ and therefore we get a Lagrangian correspondence:
\[
\begin{tikzcd}
	& \overline{L} \arrow[dl] \arrow[dr]& \\
	Y & & \overline{X}
\end{tikzcd} \]

\begin{War}\ 
	\label{war:no overline notation}
	From now on, we will make an abuse of notation and not write the overline $\overline{X}$ in the diagrams describing the Lagrangian correspondence with the convention that the stack on the right is the one with the opposite structure. A Lagrangian from $X$ to $Y$ and its opposite will be respectively denoted by
	\[
	\begin{tikzcd}
		& L \arrow[dl] \arrow[dr]& \\
		Y & & X
	\end{tikzcd} \quad \tx{and} \quad	\begin{tikzcd}
		& \overline{L} \arrow[dl] \arrow[dr]& \\
		X & & Y
	\end{tikzcd} \]
	
\end{War}

The first step toward finding categorical structure on Lagrangian structures is the following theorem giving a way to compose Lagrangian correspondences.

\begin{Th}[{\cite[Theorem 4.4]{Ca15}}]  \label{th:lagrangian correspondence composition}
	
	Given $L_1$ and $L_2$ Lagrangian correspondences from $X$ to $Y$ and $Y$ to $XZ$ respectively, then there is a Lagrangian structure from $X$ to $Z$ given by $L_1 \times_Y L_2$. The construction gives a map: 
	\[ \mathbf{Lag}\left(X \times \overline{Y}\right) \times_{\Symp(Y,n)} \mathbf{Lag}\left(Y \times \overline{Z} \right) \to \mathbf{Lag}\left(X \times \overline{Z} \right)   \]

	This composition can be depicted by the following diagram:
	
	
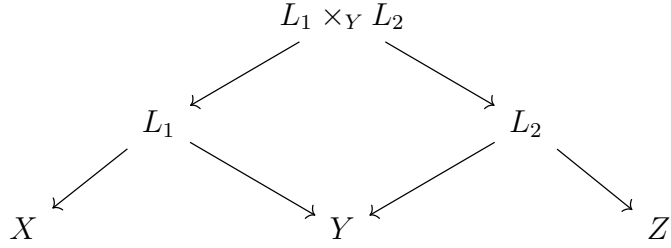
\[ \begin{tikzcd}
		& & \arrow[dl] L_1 \times _Y L_2 \arrow[dr]  && \\
		& L_1 \arrow[dl] \arrow[dr]  && L_2 \arrow[dl] \arrow[dr] & \\
		X & & Y  & & Z  
	\end{tikzcd}\]
\end{Th}

This Theorem is itself a special case of Lemma \ref{lem:lagrangian correspondence triple intersection} that will enable us to describe a composition is some category whose objects will be Lagrangian morphisms.

\begin{RQ}\
	\label{rq:category of Lagrangian correspondences}
	From this we can define a category $\LagCor_n$ such that:
	\begin{itemize}
		\item the objects are $n$-shifted symplectic derived stacks. 
		\item $\LagCor_n(X,Y)$ is the equivalence classes of Lagrangian correspondences from $X$ to $Y$ for the equivalence relation given by \emph{Lagrangeomorphism} (see Warning \ref{war:composition up to lagrangeomorphisms} for an explanation of what Lagrangeomorphisms are).
	\end{itemize}
\end{RQ}

Now we can rephrase and generalize this category solely in terms of Lagrangian structures thanks to the following observations:
\begin{itemize}
	\item A $n$-shifted symplectic structure on $X$ is given by a Lagrangian morphism $X \to \star_{n+1} $ where $\star_{n+1}$ is seen together with the $n+1$-shifted symplectic structure $0$ (see Lemma \ref{lem:lagrangian structure over a point}). 
	\item A Lagrangian structure on $L \to X$ is equivalent to a Lagrangian correspondence: 
	
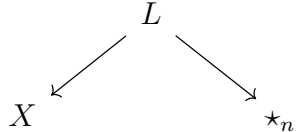
\[ \begin{tikzcd}
		& L \arrow[dl] \arrow[dr] & \\
		X & & \star_n
	\end{tikzcd}\] 
\end{itemize}

Therefore the objects of $\LagCor_n$ are Lagrangian morphisms $X \to \star_{n+1}$, and a morphism from $X$ to $Y$ is a Lagrangian in the derived intersection\footnote{The shifted symplectic structure on the derived intersection is obtained from Proposition \ref{prop:lagrangian intersection are shifted symplectic} where the construction of the shifted symplectic structure proves that $X \times_{\star_{n+1}} Y \simeq X \times \overline{Y}$ as $n$-shifted symplectic derived stacks.} $L \to X \times_{\star_{n+1}} Y \simeq X \times \overline{Y}$ (up to Lagrangeomorphism).\\ 

In order to generalize this, pick a base given by a $n$-shifted symplectic derived stack $X$ (here $X$ plays the role of $\star_{n+1}$). Given two Lagrangians $L_1 \to X$ and $L_2 \to X$, we can define a morphism from $L_1$ to $L_2$ to be a data of a Lagrangian morphism $N \to L_1 \times_X L_2$, where $L_1 \times_X L_2$ is endowed with the $(n-1)$-shifted structure obtained by Proposition \ref{prop:lagrangian intersection are shifted symplectic} by taking the derived intersection of Lagrangian morphisms. We then need a notion of composition which is given by the following lemma:   

\begin{Lem}[{\cite[Corollary 2.15]{AB17}}]
	\label{lem:lagrangian correspondence triple intersection}
	
	Let $L_i$, $i=1..3$ be Lagrangian morphism in $X$ (for a fixed symplectic structure on $X$). Then there is a map: 
	\[ \Lag(L_1 \times_X L_2) \times \Lag(L_2 \times_X L_3) \to \Lag(L_1 \times_X L_3)\]
	that sends the Lagrangians $N_1 \to L_1 \times_X L_2$ and $N_2 \to L_2 \times_X L_3$ to a Lagrangian structure on:
	\[N_1 \times_{L_2} N_2 \to L_1 \times_X L_3 \] 
\end{Lem}

\begin{War}
	\label{war:composition up to lagrangeomorphisms}
	
	The composition obtained by Lemma \ref{lem:lagrangian correspondence triple intersection} is not associative on the nose, but is so \emph{up to Lagrangeomorphism}. 
	A Lagrangeomorphism (see \cite[Definition 3.2]{AB17}) between two Lagrangian $L_1 \to X$ and $L_2 \to X$ is an equivalence of derived stacks $L_1 \to L_2$ commuting with the morphisms to $X$ together with a Lagrangian structures on the natural map:
	\[ L_1 \to L_1 \times_X L_2 \] 
	
	Then \cite[Proposition 3.7]{AB17} tells us that the composition defined by Lemma \ref{lem:lagrangian correspondence triple intersection} is associative up to canonical Lagrangeomorphism. 
\end{War}

We can now define a category of Lagrangians over a fixed $n$-shifted shifted symplectic derived stack.

\begin{Def}
	\label{def:lagrangian 1 category}
	Let $X$ be a $n$-shifted symplectic stack. The category of Lagrangians over $X$ denoted $\Lagc(X)$ is the category: 
	\begin{itemize}
		\item Whose objects are Lagrangian morphisms $f: L \to X$. 
		\item Whose morphisms $\Lagc(f,g)$, with $f: L_1 \to X$ and $g: L_2 \to X$, are given by Lagrangian morphisms: \[h: L \to L_1 \times_X {L}_2\] up to Lagrangeomorphisms (see Warning \ref{war:composition up to lagrangeomorphisms}). The composition is given by Lemma \ref{lem:lagrangian correspondence triple intersection}. 
		\item The unit morphism of a Lagrangian $L \to X$ is given by: \[1_L : L \to L \times_X L\] with the diagonal Lagrangian structure (see Proposition \ref{prop:lagrangian correspondence unit}). 
	\end{itemize}
\end{Def}

\begin{Prop}[{\cite[Corollary 2.19]{AB17}}]
	\label{prop:lagrangian correspondence unit}
	
	Consider a Lagrangian morphism $L \to X$. Then the diagonal map $1_L : L \to L \times_X L$ is a Lagrangian morphism for the natural $(n-1)$-shifted symplectic structure on $L \times_X L$ obtained by derived intersection.  
\end{Prop}

\begin{Ex}
	$\Lagc(\star_{n+1})$ is exactly the category $\LagCor_n$. 
\end{Ex}

As the morphisms in $\Lagc(X)$ are given by Lagrangians, we can further give the collection of morphisms a structure of category, making $\Lagc(X)$ into a $2$-category\footnote{Actually the category of $1$-morphisms will also have a space of morphisms given by Lagrangians, and therefore it can also be enriched in categories. This leads to a structure of $3$-category of Lagrangians. This idea goes on to define a structure of $n$-category on Lagrangians, and is made more precise in the $\infty$-categorical context in \cite{CHS21}, where they described a $(\infty, n)$-category of Lagrangian as a $n$-fold Segal category.} that we will denote by $\Lagb(X)$. 

\begin{Def}\label{def:lagrangian 2-category}
	
	We can describe $\Lagb(X)$ to be the weak $2$-category (see \cite[Definition 4.1]{AB17}) defined as follows: 
	\begin{itemize}
		\item Objects are given by Lagrangian structures $L \to X$. 
		\item The $1$-morphisms from $L_1$ to $L_2$ are given by Lagrangians structures\footnote{Here we do not take Lagrangians up to Lagrangeomorphisms since we define a weak $2$-category, where the associativity is only required up to a $2$-morphism.}: \[L \to L_1 \times_X {L}_2\]
		\item The unit $1$-morphism is the diagonal Lagrangian: \[1_L : L \times L\times_X {L}\] 
		\item Given $N_1$ and $N_2$ two morphisms from $L_1$ to $L_2$, given by Lagrangians $N_i \to L_1 \times_X L_2$. Then a $2$-morphism from $N_1$ to $N_2$ is defined as a Lagrangian: 
		\[T \to N_1 \times_{L_1 \times_X L_2} N_2\] 
		up to Lagrangeomorphism. 
		
		\item The unit $2$-morphism $1 : N \Rightarrow N$ is given by the diagonal Lagrangian: \[N \to N \times_{L_1 \times_X L_2} N\]  
	\end{itemize}
	The $2$-categorical structure is given by: 
	\begin{itemize}
		\item the composition of $1$-morphism is the same as in the $1$-categorical setting. In other words, the composition of $N_1 : L_1 \to L_2$ and $N_2 : L_2 \to L_3$, given by Lagrangian $N_1 \to L_1 \times_X {L}_2$ and $N_2 \to L_2 \times_X {L}_3$, is given by the Lagrangian: 
		\[N_1 \times_{L_1} N_2 \to L_1 \times_X L_3\]
		obtained from Lemma \ref{lem:lagrangian correspondence triple intersection}.
		
		\[ \begin{tikzcd}
			L_1 \arrow[r, "N_1"] \arrow[rr, bend left, "N_1 \times_{L_2} N_2"] & L_2 \arrow[r, "N_2"] & L_3
		\end{tikzcd}\]
		
		\item The vertical composition of $2$-morphisms: 
		
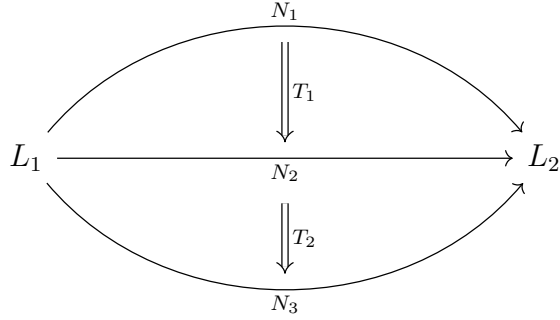
\[\begin{tikzcd}[sep=huge]
			L_1 \arrow[rrr, bend left=50,  "N_1"{name=LU,inner sep=2pt, above}]
			\arrow[rrr, bend right=50, "N_3"{name=LD,inner sep=2pt, below}]
			\arrow[rrr, "N_2"{name=LC,inner sep=2pt, below}]
			&&& L_2 
			\arrow[r,Rightarrow,to path=(LU) -- (LC)\tikztonodes, "T_1", shorten <= 6pt, shorten >= 6pt]
			\arrow[r,Rightarrow, to path=(LC) -- (LD)\tikztonodes, "T_2", shorten <= 6pt, shorten >= 6pt]
		\end{tikzcd}\]
		
		is given by the natural Lagrangian structure on: \[T_1 \times_{N_2} T_2 \to N_1 \times_{L_1 \times_X L_2} N_3\]
		provided again by Lemma \ref{lem:lagrangian correspondence triple intersection}. 
		
		\item The horizontal composition:
		
		
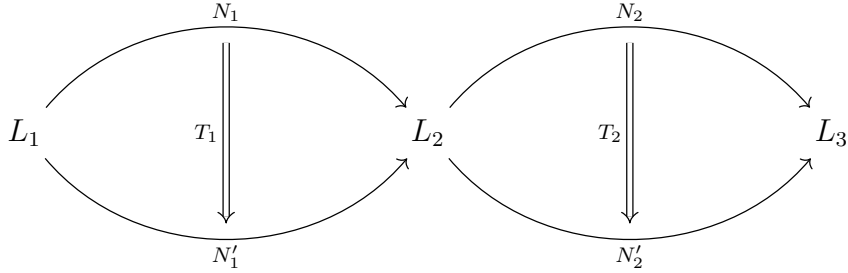
\[ \begin{tikzcd}[sep=large]
			L_1 \arrow[rrr, bend left=50, "N_1"{name=LU,inner sep=2pt, above} ]
			\arrow[rrr, bend right=50, "N_1'"{name=LD,inner sep=2pt, below}]
			& & &L_2 \arrow[rrr, bend left=50,  "N_2"{name=RU,inner sep=2pt, above}]
			\arrow[rrr, bend right=50, "N_2'"{name=RD,inner sep=2pt, below}] & & &L_3
			\arrow[r,Rightarrow,to path=(LU) -- (LD)\tikztonodes, "T_1"', shorten <= 6pt, shorten >= 6pt]
			\arrow[r,Rightarrow,to path=(RU) -- (RD)\tikztonodes, "T_2"', shorten <= 6pt, shorten >= 6pt]
		\end{tikzcd}
		\]
		
		The horizontal composition is the data of a $2$-morphism from $P_1 := N_2 \times_{L_2} N_1$ to $P_2 := N_2' \times_{L_2} N_1'$ given by a Lagrangian morphism: \[T_1 \times_{L_2} T_2 \to P_1 \times_{L_1 \times_X L_3} P_2\]
		constructed in \cite[Proposition 2.20]{AB17}.
	\end{itemize}
\end{Def}

\begin{Ex}
	$\Lagb(\star_{n+1})$ is the $2$-category whose objects are $n$-shifted symplectic stacks, $1$-morphism $X \to Y$ are Lagrangian correspondences $L \to X \times \overline{Y}$ and a $2$-morphism $L_1 \to L_2$ is a Lagrangian structure $N \to L_1\times_{X \times \overline{Y}} L_2$ up to Lagrangeomorphisms. 
\end{Ex}

\begin{Cor}\label{cor:natural 2-morphism}
	
	Take $N \in \Lagb(L_1, L_2)$. Then $\overline{N} \in \Lagb (L_2, L_1)$ and we can consider the composition $N \circ \overline{N} \in \Lagb (L_1, L_1)$. Then there is a natural $2$-morphism: \[1_{L_1} \Rightarrow N \circ \overline{N}\]
	
	given by the diagonal Lagrangian morphism: 
	\[ N \to N \times_{L_2} N \times_{L_1 \times_X L_1} L_1 \simeq N \times_{L_1 \times_X L_2} N \]
\end{Cor}

The purpose of presenting this $2$-category in such details is to obtain a $2$-categorical proof of the following theorem, which gives us a mean to construct \emph{Lagrangian correspondences} via derived Lagrangian intersections.

\begin{Th} \label{th:lagragian correspondence pullback}
	
	Consider a Lagrangian correspondence $L$ from $X$ to $Y$. Let $L_1$, $L_2$ be two Lagrangian morphisms in $X$. Consider the diagram: 
	
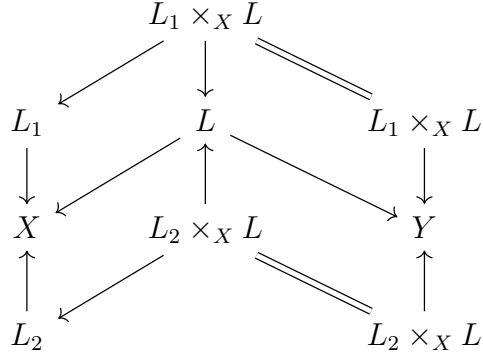
\[ \begin{tikzcd}
		& L_1 \times_X L \arrow[dl] \arrow[dr, equals] \arrow[d]& \\
		L_1 \arrow[d]& L \arrow[dl] \arrow[dr]& L_1 \times_X L \arrow[d]\\
		X & L_2 \times_X L \arrow[dl] \arrow[dr, equals] \arrow[u] & Y  \\
		L_2 \arrow[u] && L_2 \times_X L \arrow[u]
	\end{tikzcd} \]
	
	The maps $L_i \times_X L \to Y$ are Lagrangian and taking the limit long the vertical morphisms, we get a Lagrangian correspondence: 
	
	
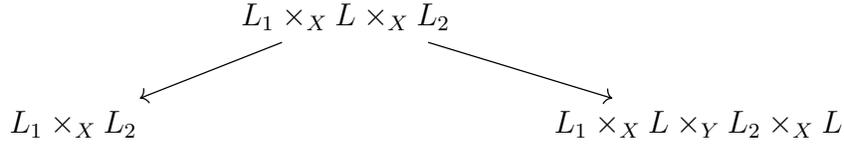
\[ \begin{tikzcd}
		& L_1 \times_X L \times_X L_2  \arrow[dl] \arrow[dr]& \\
		L_1 \times_X L_2& & L_1 \times_X L \times_Y L_2 \times_X L
	\end{tikzcd} \]
\end{Th}
\begin{proof}
	We work in the $2$-category $\Lagb (\star_{n+1})$. 
	To show that $L_i \times_X L \to Y$ is Lagrangian it is enough to view it as the composition of $1$-morphisms in $\Lagb (\star_{n+1})$: 
	\[ \begin{tikzcd}
		\star \arrow[r, "L_1"] & X \arrow[r, "L"] & Y 
	\end{tikzcd}\]
	
	Consider the following sequence of $1$-morphisms:
	
	\[ \begin{tikzcd}
		\star \arrow[r, "L_1"] & X \arrow[r, "L"] & Y \arrow[r, "\overline{L}"] & X \arrow[r, "\overline{L}_2"] & \star
	\end{tikzcd}\]
	
	The direct composition of these morphism is given by: 
	
	\[ L_1 \times_X L \times_Y L \times_X L_2 \]
	
	The key is to remember that there is a $2$-morphism (from Corollary \ref{cor:natural 2-morphism}): 
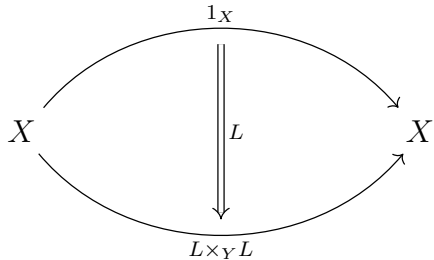
	\[  \begin{tikzcd}[sep=large]
		X \arrow[rrr, bend left=50, "1_X"{name=U,inner sep=2pt, above}]
		\arrow[rrr, bend right=50, "L\times_Y L"{name=D,inner sep=2pt, below}]
		& & &X
		\arrow[Rightarrow,to path=(U) -- (D)\tikztonodes, shorten <= 6pt, shorten >= 6pt]{r}{L}
	\end{tikzcd}\]
	
	where $1_X$ the unit corresponding to the diagonal Lagrangian $X \to X \times X$ and $L$ is defined as the diagonal Lagrangian morphism $L \to \simeq L \times_{X \times Y} L$. Precomposing by $L_1$ and post-composing by $\overline{L}_2$ we sucessively get the $2$-morphisms: 
	
	\[ L_1\times_X L : L_1 \Rightarrow L_1 \times_X L \times_Y L \]
	\[L_1 \times_X L \times_X L_2 : L_1 \times_X L_2 \Rightarrow L_1 \times_X L \times_YL \times_X L_2\] 
	
	This last $2$-morphism is exactly the Lagrangian correspondence we want since both $ L_1 \times_X L_2$ and $L_1 \times_X L \times_Y L_2 \times_X L$ are in fact symplectic so that this $2$-morphism is in fact a Lagrangian: 
	\[ \begin{split}
		L_1 \times_X L \times_X L_2 \to & ( L_1 \times_X L_2) \times_{\star_{n+1}}( L_1 \times_X L \times_Y L_2 \times_X L) \\
		& \simeq ( L_1 \times_X L_2 )\times (\overline{ L_1 \times_X L \times_Y L_2 \times_X L})
	\end{split}\]
\end{proof}

\subsection{Examples of Constructions of Lagrangian Fibrations}\label{sec:example-of-constructions-of-lagrangian-fibrations}
\medskip

\subsubsection{Derived critical locus}\label{sec:derived-critical-locus}\

\medskip

We discuss here some general result on the shifted symplectic geometry of the derived critical locus. First we want to understand in general the $(-1)$-shifted symplectic form on $\RCrit(f)$. We use the universal property of the tautological $1$-form (Lemma \ref{lem:universal property tautological 1-form}) to see that $(df)^* \omega = 0$ (with $\omega = \dr\lambda_X$ the canonical symplectic structure on $T^*X$). Using the appropriate replacement to compute the derived tensor product appearing in the proof of Proposition \ref{prop:derived critical locus corepresentability}, $\omega$ induces a closed 2-form on $\Spec_X \left( \Sym_{\_O_X} \left( \Tt_X[1] \oplus  \Tt_X\right)\right)$. Since the differential on the resolution, $\Sym_{\_O_X} \left( \Tt_X[1] \oplus  \Tt_X\right)$, is induced by $\tx{Id} : \Tt_X \rightarrow \Tt_X[1]$ (plus the differentials on $\Tt_X$ and $\_O_X$), the tautological $1$-form $\omega_{-1}$ on $T^*[-1]X$ induces a closed 2-form on $\Spec_X \left( \Sym_{\_O_X} \left( \Tt_X[1] \oplus  \Tt_X\right) \right)$ which gives a homotopy\footnote{Note that both $\omega$ and $\omega_{-1}$ are closed for the vertical differential since they are $1$-forms. Therefore, only the part of the differential given by $\Tt_X[1] \to \Tt_X$ is involved and sends $\omega_{-1}$ to $\omega$.} between $\omega$ and $0$. We then have that the $(-1)$-shifted symplectic form is described by the self-homotopy of $0$ given by $\omega_{-1}$: 

\[\begin{tikzcd}
	0 \arrow[r, "\omega_{-1}"] & p^*\omega = 0
\end{tikzcd}\]

The proof of Proposition \ref{prop:lagrangian intersection are shifted symplectic} tells us that  $\omega_{-1}$ is the $(-1)$-shifted symplectic form obtained on $\RCrit(f)$.\\

We have seen in Remark \ref{ex:derived critical locus lagrangian fibration} that $\pi : \RCrit(f) \rightarrow X$ is a Lagrangian fibration. The morphism $\alpha_{\pi}$ controlling the non-degeneracy condition of the Lagrangian fibration (see Diagram \ref{dia:lagrangian fibration ND criteria}) is still natural in the sense given by the following proposition.

\begin{Prop}
	\label{prop: lagrangian fibration cannical ND map}
	
	$\alpha_{\pi}$ is equivalent to the following composition of equivalences: 
	
	\begin{equation}
		\label{eq:ND lagrangian libration derived critical locus}
		\begin{tikzcd}
			\Llr{\RCrit(f)}{X} \simeq  0 \times_{\Llr{T^*X}{X}} 0 \arrow[r, "0 \times_{\beta} 0"] &  0 \times_{\pi_X^* \Ll_{X}} 0 \simeq \pi^* \Ll_X[-1]
		\end{tikzcd}
	\end{equation}

	where $\beta$ is the equivalence $ \Llr{T^*X}{X} \simeq \pi_X^* \Ll_X$ obtained from the Lagrangian fibration $\pi_X: T^*X \to X$.
\end{Prop}
\begin{proof}
	The strategy here is to express the Diagram \eqref{dia:lagrangian fibration ND criteria} as a pull-back of the same type of diagrams. 
	
	First we express $\Ll_{\RCrit(f)}[-1]$ as a pull-back above $\Ll_{T^*X}$. This can be done by observing that all squares in the following diagram are bi-Cartesians:
	
	\[ \begin{tikzcd}  
		\Ll_{\RCrit(f)} [-1] \arrow[r] \arrow[d] &\pi_1^* 0^* \Llr{T^*X}{X}  \arrow[r] \arrow[d] & 0 \arrow[d] \\
		\pi_2^* df^* \Llr{T^*X}{X}  \arrow[r] \arrow[d] & \pi_1^* 0^* \Ll_{T^*X} \simeq \pi_2^* df^*\Ll_{T^*X}  \arrow[r] \arrow[d] & \pi_2^*\Ll_X \arrow[d] \\
		0 \arrow[r] & \pi_1^* \Ll_X \arrow[r] & \Ll_{\RCrit(f)} 
	\end{tikzcd} \]

	Where $\pi_1$ and $\pi_2$ are the natural projections $\RCrit(f) \to X$  given by the pullback diagram (and $\pi_1 = \pi_2$). We write Diagram \eqref{dia:lagrangian fibration ND criteria} for $\pi : \RCrit(f) \rightarrow X$ as: 
	
	\[ \begin{tikzcd}
		0 \times_{\Llr{T^*X}{X}} 0 \arrow[r, "\alpha_{\pi} \simeq 0 \times_{\alpha_{\pi_X}} 0"] \arrow[d] & 0 \times_{\pi_X^* \Ll_X} 0 \arrow[d] \arrow[r] & 0 \arrow[d] \\
		\Tt_X \times_{\Tt_{T^*X}} \Tt_X \arrow[r] & \Llr{T^*X}{X} \times_{\Ll_{T^*X}}\Llr{T^*X}{X} \arrow[r, " \tx{Id} \times_{\tx{pr}} \tx{Id}"] & \Llr{T^*X}{X} \times_{\Llr{T^*X}{X}} \Llr{T^*X}{X}
	\end{tikzcd}\]
	
	In this diagram, pullbacks have been omitted to keep the diagram easy to read. We need to describe the morphism $\omega_{\RCrit(f)} : \Tt_X \times_{\Tt_{T^*X}} \Tt_X \rightarrow \Llr{T^*X}{X} \times_{\Ll_{T^*X}}\Llr{T^*X}{X}$. Recall from Remark \ref{rq:lagrangian nd} and the proof of the non-degeneracy in Proposition \ref{prop:lagrangian intersection are shifted symplectic} that $\omega_{\RCrit(f)}$ is $\Theta_{df} \times_{\omega} \Theta_0$ where $\Theta_h : \Tt_X \rightarrow \Ll_{h}[-1] \simeq \Llr{X}{T^*X}[-1] \simeq \Llr{T^*X}{X}$ is the natural morphism expressing the non-degeneracy of the Lagrangian structure (see Definition \ref{def:lagrangian structure}). 
\end{proof}

\subsubsection{Non-degenerate functionals}\label{sec:non-degenerate-functionals}\

\medskip

We consider the example of the derived critical locus where $f$ may have a family of critical points which are all non-degenerate in the directions normal to the strict critical locus. We want to understand the quasi-isomorphism describing the non-degeneracy condition of the Lagrangian fibration $\RCrit(f)\to X$ in that situation. 

We denote by $S$ the strict critical locus, which comes with a closed immersion $i: S \rightarrow X$ and whose structure sheaf is $\_O_S = i^{-1}\faktor{\_O_X}{I}$ with $I = \langle df.v, \ v \in \Tt_X \rangle$. \\

We assume that both $X$ and $S$ are smooth algebraic varieties. We denote by $\RCrit(f)$ the derived critical locus of $f$ and we get a canonical morphism: \[\lambda : S \rightarrow \RCrit(f)\]

In order to define the Hessian quadratic form and the non-degeneracy condition, we need to assume that the closed immersion $S \hookrightarrow X$ has a first order splitting. Concretely, we assume in this Section that the following fiber sequence splits:

\begin{equation}\label{eq:splitting assumption}
	\begin{tikzcd}
		\Tt_S \arrow[r, shift left] & i^* \Tt_X \arrow[l, shift left, dashed] \arrow[r, shift left] & \Ttr{S}{X}[1]\arrow[l, shift left, dashed]
	\end{tikzcd}
\end{equation}  

This assumption is necessary to be able to restrict $Q$ to the normal part $\Ttr{S}{X}[1] $.

\begin{Def}\
	The \defi{Hessian quadratic form} is defined by the symmetric bilinear map:
	
	$$\begin{array}{ccccl}
		Q & : &  \Sym_{\_O_S}^2 i^* \Tt_X  & \to & \_O_S \\
		& & (w,v) & \mapsto & d(df.v).w \\
	\end{array}$$
	
	We define non-degeneracy to be along the "normal" direction to $S$, by considering the following diagram:  
	
	\begin{equation}
		\label{dia:hessian nd}
		\begin{tikzcd}
			\Tt_S \arrow[r, shift left] \arrow[d, "0"] &\arrow[l, shift left, dashed] i^* \Tt_X \arrow[r, shift left] \arrow[d,crossing over, "Q", near end] &\arrow[l, shift left, dashed] \Ttr{S}{X}[1] \arrow[d, "0"] \arrow[dll,  dashed, "\widetilde{Q}", near start] \\
			\Llr{S}{X} [-1] \arrow[r, shift left] &\arrow[l, shift left, dashed] i^* \Ll_X \arrow[r, shift left] &\arrow[l, shift left, dashed] \Ll_S 
		\end{tikzcd}
	\end{equation}
	
	Both rows are split fiber sequences (by assumption in Diagram \eqref{eq:splitting assumption}). The left and right vertical maps are the zero map because $Q$ restricted to $\Tt_S$ is zero and, since $Q$ is symmetric, $Q$ composed with the projection to $\Ll_S$ is also zero.  We obtain a map $\widetilde{Q}$ (using $Q$ and following the section and retract of the fiber sequences) which corresponds to the map induced by $Q$ on the normal bundle. Then the \defi{non-degeneracy condition} is the requirement that $\widetilde{Q}$ is a quasi-isomorphism. 
	
\end{Def}
Since the differential on $\_O_{\textbf{Crit}(f)}$ is $ \delta = \iota_{df}$ (see Proposition \ref{prop:derived critical locus corepresentability}), we have the commutative diagram in $\textbf{QC}(S)$: 

\begin{equation}
	\label{Dia_HessianVSdifferential}
	\begin{tikzcd}
		i^*\Tt_X \arrow[dr, "Q"'] \arrow[r, "\iota_{df}"] & i^* \_O_X \arrow[d, "\dr"] \\
		& i^*\Ll_X
	\end{tikzcd}
\end{equation}

We will abusively write $Q = \dr \circ \delta : i^*\Tt_X[1] \to i^*\Ll_X$ for the map of degree $1$ corresponding to the composition $d \circ \iota_{df} : i^*\Tt_X \to i^* \Ll_X$ of degree $0$.\\

In general, the natural map $\lambda : S \rightarrow \RCrit(f)$ is not an equivalence. This is due to the fact that the partial derivatives of $f$ will not in general form a regular sequence and therefore $\RCrit(f)$ has higher homology. The default to be a regular sequence comes from vector fields that annihilate $df$. Such vector fields are in fact vector fields on $S$ when $f$ is non-degenerate. With that idea in mind, we show that an equivalent description of $\RCrit(f)$ is given by $T^*[-1]S$ when $Q$ is non-degenerate.\\

\begin{Prop}
	\label{prop:map exitence}
	There exists a natural map $\Phi : T^*[-1]S \rightarrow \RCrit(f)$ making the following diagram commute: 
	
	$$ \begin{tikzcd}
		T^*[-1]S \arrow[d, "\pi_S"] \arrow[r, "\Phi"] & \RCrit(f) \arrow[d, "\pi"] \\
		S \arrow[r,"i"] & X
	\end{tikzcd}$$
\end{Prop} 

\begin{proof}
	Under our first order splitting assumption (Diagram \eqref{eq:splitting assumption}), the natural map $\Tt_S \rightarrow i^*\Tt_X$ admits a retract, and therefore the natural map $i^*T^*X \rightarrow T^*S$ admits a section: $T^*S \dashrightarrow i^* T^*X$. We consider the following diagram:
	
	\begin{equation*}
		\begin{tikzcd}
			T^*X & \arrow[l] i^*T^*X \arrow[r, shift left] & \arrow[l, dashed, shift left] T^*S  \\
			X \arrow[u, "0"] & S \arrow[l, "i"] \arrow[r, equals] \arrow[u, "0"] & S \arrow[u, "0"]
		\end{tikzcd}
	\end{equation*}
	
	We want to pull-back these zero sections along the maps induced by $df$ represented by the vertical morphisms in the following commutative diagram: 
	
	\begin{equation*}
		\begin{tikzcd}
			T^*X & \arrow[l] i^*T^*X \arrow[r, shift left] & \arrow[l, dashed, shift left] T^*S  \\
			X \arrow[u, "df"] & S \arrow[l, "i"] \arrow[r, equals] \arrow[u, " i^* df = 0"] & S \arrow[u, "0"]
		\end{tikzcd}
	\end{equation*}  
	
	This induces the following morphisms between the pull-backs: 
	
	$$ \begin{tikzcd}
		\RCrit(f)& \arrow[l] S \times_{i^* T^*X} S \arrow[r, shift left] & T^*[-1]S \arrow[l, dashed
		, shift left]
	\end{tikzcd}$$
	
	We obtain a map $\Phi : T^*[-1]S \rightarrow \RCrit(f)$. The maps we obtain come from the universal properties of the pull-backs therefore if we denote $s_0: X \rightarrow T^*X$ the zero section, we have $s_0 \circ \pi \circ \Phi = s_0 \circ i \circ \pi_S$. If we compose by the projection $\pi_X :  T^*X \rightarrow X $, we get $\pi \circ \Phi = i \circ \pi_S$.
\end{proof}

$\Phi$ gives a relationship between the Lagrangian fibration structures on $T^*[-1]S \rightarrow S$ and $\RCrit(f) \rightarrow X$ which we now analyse. The idea is to show that the difference between these Lagrangian fibrations is in fact controlled by $\widetilde{Q}$ (see Proposition \ref{prop:ND and phi} and Remark \ref{rq:lagrangian fibration vs hessian}).
\begin{Lem}
	\label{lem:naturality lagrangian fibation ND morphism}
	$\Phi$ induces a morphism $\Ttr{T^*[-1]S}{S} \rightarrow \Phi^* \Ttr{\RCrit(f)}{X}$ that fits in the commutative diagram
	
	\begin{equation}
		\begin{tikzcd}
			\Ttr{T^*[-1]S}{S} \arrow[r] \arrow[d, "\alpha_{\pi_S}"] & \Phi^* \Ttr{\RCrit(f)}{X} \arrow[d, "\alpha_{\pi}"] \\
			\pi_S^* \Ll_S[-1] \arrow[r] & \Phi^* \pi^* \Ll_X[-1] \simeq \pi_S^* i^* \Ll_X[-1]
		\end{tikzcd}
	\end{equation}
	where the bottom horizontal arrow is the pull-back along $\pi_S$ of the section $\Ll_S[-1] \rightarrow i^* \Ll_X [-1]$ in the dual of the split fiber sequence \eqref{eq:splitting assumption}. 
\end{Lem}
\begin{proof}
	The homotopy pull-back, $\RCrit(f) = X \times_{T^*X}^h X$ lives over $X$. We get the equivalences:
	
	\[ \begin{tikzcd}
		\Ttr{\RCrit(f)}{X} \arrow[r, "\simeq"] & \Ttr{X}{X} \times_{\faktor{\Tt_{T^*X}}{X}}^h \Ttr{X}{X} \arrow[r, "\simeq"] &   \star  \times_{\faktor{\Tt_{T^*X}}{X}}^h \star \arrow[r, "\simeq"] &  \pi^* \Ll_X[-1] 
	\end{tikzcd}\]

	Proposition \ref{prop:cotangent ND lag fibration is natural} tells us that the canonical fibrations on the cotangent stacks are the canonical ones and therefore behave functorially (using Proposition \ref{prop:relative cotangent complex functoriality for linear stacks}). This implies that the following commutative square is commutative: 
	\[ \begin{tikzcd}
		\Ttr{T^*S}{S} \arrow[r] \arrow[d, "\beta_S"]& \Ttr{T^*X}{X} \arrow[d, "\beta_X"] \\
		\pi_S^* \Ll_S \arrow[r, "\pi_S^* s"] & \pi_S^* i^*\Ll_X
	\end{tikzcd} \]
	where $s$ is the section in the dual of the split fiber sequence \eqref{eq:splitting assumption}. From Definition \ref{def:lagrangian fibration} 
	we know that both $\alpha_{\pi_S}$ and $\alpha_{\pi}$ are the morphism induced by the morphisms $\beta_S$ and $\beta_X$ via Diagram \eqref{eq:ND lagrangian libration derived critical locus}. We then obtain the commutative diagram:
	\[\begin{tikzcd} \Ttr{T^*[-1]S}{S} \arrow[r, "\simeq"] \arrow[d]& 0 \times_{\Ttr{T^*S}{S}}^h 0 \arrow[r, "0 \times_{\beta_S}^h 0"] \arrow[d] & \pi_S^* \Ll_S[-1] \arrow[d]\\
		\Phi^*\Ttr{\RCrit(f)}{X} \arrow[r, "\simeq"] & \Phi^* \left( 0 \times_{\Ttr{T^*X}{X}}^h 0 \right) \arrow[r, "0 \times_{\beta_X}^h 0"]  & \Phi^* \pi^* \Ll_X[-1]
	\end{tikzcd} \]
	
	where the composition of the horizontal maps are exactly $\alpha_{\pi_S}$ and $\alpha_{\pi}$ thanks to Proposition \ref{prop:cotangent ND lag fibration is natural}.
\end{proof}

\begin{Lem}
	\label{lem:differential on cotangent complex of Crit(f)}
	We first remark that $\Phi^*\Ll_{\RCrit(f)}$ can be described, as a sheaf of graded modules (forgetting the differential), by: 
	\[ \Phi^*\Ll_{\RCrit(f)} \simeq \Sym_{\_O_S} \left(\Tt_S[1] \right) \otimes_{\_O_S} \left( i^* \Ll_X \oplus i^* \Tt_X [1] \right) \] 
	
	where $\Ll_X$ is generated by terms of the form $dg$ with $g \in \_O_X$ and $\Tt_X[1]$ is generated by terms of the form $d \xi$ with $\xi \in \Tt_X[1] \subset \_O_{\RCrit(f)}$. Then, the internal differential on $\Phi^*\Ll_{\RCrit(f)}$ is characterised by $Q = d \circ \iota_{df}$ via $\delta(d \xi) = Q(\xi)$ and $\delta (dg) = 0$. 
\end{Lem}

\begin{proof}

	The differential on $\Sym_{\_O_S} \left(\Tt_S[1] \right) \otimes_{\_O_S} \left( i^* \Ll_X \oplus i^* \Tt_X [1] \right)$ is $\_O_{T^*[-1]S}$-linear because $\iota_{df}$ is zero on $\Tt_S[1]$. Moreover, for $\xi \in \Tt_X[1] \subset \_O_{\RCrit(f)}= \Sym_{\_O_X} \Tt_X[1]$, we have  $\delta \circ d (\xi) = d \circ \delta (\xi) = d \circ \iota_{df} (\xi) = Q(\xi)$ (see Diagram \eqref{Dia_HessianVSdifferential}), and for $g\in \_O_X$, $\delta \circ d (g) = d \circ \delta g = 0$. 
\end{proof}

\begin{Lem}
	\label{lem:hessian map from lagrangian fibration}
	The composition:
	\[ \begin{tikzcd} \pi_S^* i^* \Tt_X [-1] \arrow[r] & \Phi^* \Ttr{\RCrit(f)}{X} \arrow[r, "\alpha_{\pi}"] & \Phi^* \pi^* \Ll_X[-1] \end{tikzcd}\]
	is given by $\pi_S^* Q$. Similarly, the composition: 
	
	\[ \begin{tikzcd} \pi_S^*  \Tt_S [-1] \arrow[r] & \Ttr{T^*[-1]S}{S} \arrow[r, "\alpha_{\pi_S}"] & \pi_S^* \Ll_S[-1] \end{tikzcd}\]
	
	is $0$ (the restriction of $\pi_S^* Q$ to $S$). 
\end{Lem}
\begin{proof}
	The left morphism is the morphism fitting in the fiber sequence:
	
	\[ \begin{tikzcd} \pi_S^*  i^* \Tt_X [-1] \arrow[r] & \Phi^* \Ttr{\RCrit(f)}{X} \arrow[r] & \Phi^* \Tt_{\RCrit(f)} \end{tikzcd}\]
	
	Which gives us:
	
	\[ \begin{tikzcd}
		\pi_S^* i^* \Tt_{X} [-1] \arrow[r] \arrow[d, equals] & \Phi^* \Ttr{\RCrit(f)}{X} \arrow[r] \arrow[d, "\alpha_{\pi}"] & \Phi^* \Tt_{\RCrit(f)} \arrow[d, "\simeq"] \\
		\pi_S^* i^* \Tt_{X} [-1] \arrow[r, dashed] & \Phi^* \pi^* \Ll_{X}[-1] \arrow[r, hookrightarrow]  & \Phi^* \pi^* \Ll_{X}[-1] \oplus \Phi^* \pi^* \Tt_{X}
	\end{tikzcd} \]
	
	The second row can be seen as the extension (by $\pi_S^* $) of the fiber sequence: 
	
	\[ \begin{tikzcd}
		i^* \Tt_{X} [-1] \arrow[r, dashed] & i^* \Ll_{X}[-1] \arrow[r, hookrightarrow]  &i^*  \Ll_{X}[-1] \oplus i^* \Tt_{X}
	\end{tikzcd} 
	\]
	
	Since $X$ and $S$ are smooth, $i^* \Tt_{X} [-1]$ and $i^* \Ll_{X}[-1]$ are both quasi-isomorphic to complexes concentrated in a single degree. This imposes that the dashed arrow is equivalent to the connecting morphism of the induced long exact sequence in cohomology. Therefore, it is equivalent to the map that sends a section $s$ in $i^* \Tt_{X} [-1] $ to its differential, in $i^*  \Ll_{X}[-1] \oplus i^* \Tt_{X}$, which can in turn be seen as an element in   $i^* \Ll_{X}$. More concretely, denote $\tilde{s}$ any lift of $s$ to an element in $i^* \Ll_{X}[-2] \oplus i^* \Tt_{X}   [-1]$. Using Lemma \ref{lem:differential on cotangent complex of Crit(f)}, its differential is given by 
	
	\[  Q (s) =  Q (\tilde{s}) \in i^*  \Ll_{X}[-1] \subset i^* \Ll_{X}[-1]\oplus i^* \Tt_{X}. \]
	
	We then apply $\pi_S^*$ to get the sequence we want. The second part of the statement is proven the same way.
\end{proof}

\begin{Prop}
	\label{prop:ND and phi}
	The map $\Tt_{T^*[-1]S}  \rightarrow  \Phi^* \Tt_{\RCrit(f)}$ induced by $\Phi$ is an equivalence if and only if $Q$ is non-degenerate.
\end{Prop}

\begin{proof}
	
	First, using the equivalences $\alpha_{\pi}: \Phi^* \Ttr{\RCrit(f)}{X} \rightarrow \pi_S^* i^* \Ll_X[-1]$ and  $\alpha_{\pi_S} : \Phi^* \Ttr{T^*[-1]S}{S} \rightarrow \pi_S^* \Ll_S[-1]$, we can show that the cofiber of $\Ttr{T^*[-1]S}{S} \rightarrow \Phi^* \Ttr{\RCrit(f)}{X}$ is equivalent to $\pi_S^* \Llr{S}{X} [-2]$. Then Lemma \ref{lem:naturality lagrangian fibation ND morphism} and \ref{lem:hessian map from lagrangian fibration} ensure that the upper half of the following diagram is commutative:
	
	\begin{equation}
		\begin{tikzcd}
			\pi_S^* \Tt_S[-1] \arrow[d] \arrow[r] & \pi_S^* i^* \Tt_X[-1] \arrow[r] \arrow[d] & \pi_S^* \Ttr{S}{X} \arrow[d, "\widetilde{Q}"] \\
			\Ttr{T^*[-1]S}{S} \arrow[r] \arrow[d] & \Phi^* \Ttr{\RCrit(f)}{X} \arrow[r] \arrow[d] & \pi_S^* \Llr{S}{X}[-2] \arrow[d]\\
			\Tt_{T^*[-1]S} \arrow[r] & \Phi^* \Tt_{\RCrit(f)} \arrow[r] & \_F
		\end{tikzcd}
	\end{equation}
	
	This diagram is then commutative and all rows and columns are cofiber sequences and in particular $\_F$ is both the homotopy cofiber of $\Tt_{T^*[-1]S}  \rightarrow  \Phi^* \Tt_{\RCrit(f)}$ and the homotopy cofiber of $\widetilde{Q}$. In particular, the homotopy cofiber of $\widetilde{Q}$ is zero if and only the homotopy cofiber of $\Tt_{T^*[-1]S}  \rightarrow  \Phi^* \Tt_{\RCrit(f)}$ is also zero. 
\end{proof}

We now decompose $\alpha_\pi$ into a part along $S$ and a part normal to $S$. This decomposition is by means of split fibered sequences coming from the split fiber sequence \eqref{eq:splitting assumption}. \\

\begin{Prop}
	When $Q$ is non-degenerate, the maps expressing the non-degeneracy of the Lagrangian fibrations fit in the commutative diagram:
	\[ \begin{tikzcd}
		\Ttr{T^*[-1]S}{S} \arrow[r] \arrow[d, "\alpha_{\pi_S}"] & \Ttr{\RCrit(f)}{X} \arrow[r] \arrow[d, "\alpha_{\pi}"] & \Ttr{S}{X} \arrow[d, "\widetilde{Q}"] \\
		\pi_S^* \Ll_S [-1] \arrow[r] & \pi_S^* i^* \Ll_X [-1] \arrow[r] & \Llr{S}{X} [-1]
	\end{tikzcd}\]
	where the rows are fiber sequences.   
\end{Prop}
\begin{proof}
	First, when $Q$ is non-degenerate, the top horizontal sequence is fibered and comes from the following diagram:
	
	\[ \begin{tikzcd}
		\Ttr{T^*[-1]S}{S} \arrow[d] \arrow[r] & \Phi^* \Ttr{\RCrit(f)}{X} \arrow[d] \arrow[r, dashed] & \pi_S^* \Ttr{S}{X} \arrow[d]\\
		\Tt_{T^*[-1]S} \arrow[r] \arrow[d] & \Phi^*\Tt_{\RCrit(f)} \arrow[d] \arrow[r] & 0 \arrow[d] \\
		\pi_S^* \Tt_S  \arrow[r]  & \Phi^* i^* \Tt_X \arrow[r]  & \pi_S^* \Tt_{ \faktor{S}{X}}[1]  \\
	\end{tikzcd} \]
	
	where all rows and columns are fibered and the cofiber of the second row is $0$ thanks to Proposition \ref{prop:ND and phi} since we assumed that $Q$ is non-degenerate. Using Lemma \ref{lem:naturality lagrangian fibation ND morphism} and Lemma \ref{lem:hessian map from lagrangian fibration}, we obtain the following commutative diagram: 
	
	\begin{equation}\label{dia:lagrangian fibration cofiber}
		\begin{tikzcd}
			\pi_S^* \Tt_S [-1] \arrow[r] \arrow[d] \arrow[dd,  bend right = 70, "0"', near start,crossing over] & \Phi^* i^* \Tt_X [-1] \arrow[r] \arrow[d] \arrow[dd, bend right = 70, "Q"', near start,crossing over] & \pi_S^* \Ttr{S}{X} \arrow[d, equals] \arrow[dd,  bend right = 70, "\widetilde{Q}"', near start,crossing over] \\
			\Ttr{T^*[-1]S}{S} \arrow[d, "\alpha_{\pi_S}"] \arrow[r] & \Phi^*\Ttr{\RCrit(f)}{X} \arrow[d, "\alpha_{\pi}"] \arrow[r] & \pi_S^* \Ttr{S}{X} \arrow[d, dashed] \\
			\pi_S^* \Ll_S[-1] \arrow[r] & \Phi^* i^* \Ll_X[-1] \arrow[r] & \pi_S^* \Llr{S}{X}[-2]
		\end{tikzcd}
	\end{equation}
	
	The only map the dashed arrow can be, in order to make the diagram commutative, is $\widetilde{Q}$. 
\end{proof}

\begin{RQ}
	\label{rq:lagrangian fibration vs hessian}
	If we do not assume $Q$ non-degenerate, the cofiber $\_F$ of the map $\Tt_{T^*[-1]S} \rightarrow \Phi^* \Tt_{\RCrit(f)}$ will be non zero. We will denote by $\_G$ the fiber of the natural map $\_F \rightarrow \Ttr{S}{X}$. Then we can rewrite Diagram \eqref{dia:lagrangian fibration cofiber} as
	
	\[
	\begin{tikzcd}
		\pi_S^* \Tt_S [-1] \arrow[r] \arrow[d] \arrow[dd,  bend right = 70, "0"', near start,crossing over] & \Phi^* i^* \Tt_X [-1] \arrow[r] \arrow[d] \arrow[dd, bend right = 70, "Q"', near start,crossing over] & \pi_S^* \Ttr{S}{X} \arrow[d] \arrow[dd,  bend right = 70, "\widetilde{Q}"', near start,crossing over] \\
		\Ttr{T^*[-1]S}{S} \arrow[d, "\alpha_{\pi_S}"] \arrow[r] & \Phi^*\Ttr{\RCrit(f)}{X} \arrow[d, "\alpha_{\pi}"] \arrow[r] & \_G \arrow[d, "\alpha_N"] \\
		\pi_S^* \Ll_S[-1] \arrow[r] & \Phi^* i^* \Ll_X[-1] \arrow[r] & \pi_S^* \Llr{S}{X}[-2]
	\end{tikzcd}
	\]

	The map $\alpha_N : \_G \rightarrow \pi_S^*\Llr{S}{X}[-2] $ represent the "difference" between the maps $\alpha_\pi$ and $\alpha_{\pi_S}$ from the Lagrangian fibrations. $\alpha_N$ is still related to $\widetilde{Q}$ in the sense that the following diagram is commutative: 
	
	\[\begin{tikzcd}
		\Ttr{S}{X} \arrow[d] \arrow[dr, "\widetilde{Q}"]  & \\
		\_G \arrow[r, "\alpha_N"] & \Llr{S}{X} [-2]
	\end{tikzcd}\]
	
	Therefore the restriction of $\alpha_N$ to $\Ttr{S}{X}$ is again $\widetilde{Q}$. 
\end{RQ}

\begin{RQ}
	As a non-example if we take $f : \Aa^1 \rightarrow \Aa^1$ sending $X$ to $\frac{X^3}{3}$, the basic assumptions that made this section work are failing. The strict critical locus $S$ is not smooth since it is a fat point, and the sequence \eqref{eq:splitting assumption} does not split.  
\end{RQ}   

\subsubsection{$G$-equivariant twisted cotangents} \label{sec:g-equivariant-twisted-cotangent-bundles}\

\medskip

For $X$ a smooth scheme, a twisted cotangent stack is a twist of the ordinary cotangent stack by a closed $1$-form of degree $1$ on $X$, $\alpha \in H^1(X, \Omega_X^{1,  \tx{cl}})$. Such a closed form has an underlying 1-form of degree 1 that corresponds to a morphism $\alpha : X \rightarrow T^*[1]X$. The \defi{twisted cotangent bundle} associated to $\alpha$ is defined to be the following pull-back:

\[ \begin{tikzcd}
	T_\alpha^* X \arrow[r] \arrow[d] & X \arrow[d, "\alpha"] \\
	X \arrow[r, "0"] & T^*[1]X
\end{tikzcd}
\] 

We refer to \cite{Ha16} for more information on the relation between this definition and the usual definition of twisted cotangent bundles. This is then a derived intersection of Lagrangians (see Example \ref{ex:lagrangian closed 1-form}) and therefore it is $0$-shifted symplectic. Moreover the Lagrangian fibration $T^*[1]X \to X$ follows the assumptions of Theorem \ref{th:derived intersection lagrangian fibraton} and therefore the map $T_\alpha^* X \rightarrow X$ has a Lagrangian fibration structure. \\

Now take $G$ an algebraic group acting on the algebraic variety $X$. Consider a character $\chi : G \rightarrow \Gg_m$. We have the logarithmic form on $\Gg_m$ given by a map $\Gg_m \rightarrow \_A^{1,cl}(-,0)$ which sends $t$ to $t^{-1} dt$. We get a closed $1$-form on $G$ described by the composition:

$$ G \rightarrow \Gg_m \rightarrow \_A^{1,cl}(-,0) $$

This is also a group morphism for the additive structure on $\_A^{1,cl}(-,0)$. We can therefore pass to classifying spaces and obtain a $1$-shifted closed $1$-form on $\mathbf{B}G$:

$$ \alpha_\chi : \mathbf{B}G \rightarrow \mathbf{B}\_A^{1,cl}(-,0) = \_A^{1,cl}(-,1)$$

We can consider the pull-back of $\alpha_\chi$ along the $G$-equivariant moment map:

$$ \begin{tikzcd}
	\QS{T^*X}{G}  \times_{ \QS{\G_g^*}{G}  } \mathbf{B}G \arrow[r] \arrow[d] & \mathbf{B}G \arrow[d, "\alpha_\chi"] \\
	\QS{T^*X}{G}   \arrow[r, "\mu"] & \QS{\G_g^*}{G}   \simeq T^*[1]\mathbf{B}G
\end{tikzcd}$$ 

It turns out that the equivariant moment map $\eq{\mu}$ is Lagrangian (see Proposition \ref{prop:moment map equivariant is lagrangian}), which implies (Proposition \ref{prop:lagrangian intersection are shifted symplectic}) that this fiber product is $0$-shifted symplectic. It turns out that we have an equivalence of shifted symplectic derived Artin stacks: 

\[\QS{T^*X}{G}  \times_{ \QS{\G_g^*}{G}  }\mathbf{B}G \simeq T_{\widehat{\alpha}}^* \QS{X}{G}  \] 

Where $\widehat{\alpha}$ denotes the pull-back of $\alpha_\chi$ to a $1$-form of degree $1$ on $\QS{X}{G} $. Therefore, according to Theorem \ref{th:derived intersection lagrangian fibraton}, the natural projection

$$ \begin{tikzcd}
	T_{\widehat{\alpha}}^* \QS{X}{G}  \arrow[r] & \QS{X}{G} 
\end{tikzcd}$$
is a Lagrangian fibration.\\

To show the equivalence above, we use the composition of the following Lagrangian correspondences:

\begin{itemize}
	\item The Lagrangian structure on the section $ \QS{X}{G}  \rightarrow T^*[1] \QS{X}{G} $: 
	
	\[ \begin{tikzcd}
		& \QS{X}{G}  \arrow[dl] \arrow[dr, "0"]& \\
		\star& &  T^*[1] \QS{X}{G}  
	\end{tikzcd}\]
	
	\item Using Example \ref{ex:conormal lagrangian correspondence}, Proposition \ref{prop:pullbak quotient stack of groupoids} and Lemma \ref{lem:cotagent BG and coadjoint quotient}, we have: \[\QS{X \times \G_g^*}{G}  \simeq \QS{X}{G} \times_{\QS{\star}{G} } \QS{\G_g^*}{G}\] 
	We obtain the Lagrangian correspondence: 
	
	\[ \begin{tikzcd}
		& \QS{X \times \G_g^*}{G}  \simeq \QS{X}{G} \times_{\QS{\star}{G} } \QS{\G_g^*}{G}  \arrow[dl] \arrow[dr]& \\
		T^* [1]\QS{X}{G}  & & \QS{\G_g^*}{G}  \simeq T^*[1]\QS{\star}{G}  
	\end{tikzcd}\]
	
	\item The Lagrangian obtained from the closed $1$-form of degree $1$, $\alpha_\chi$:  	
	
	\[ \begin{tikzcd}
		& \mathbf{B}G \arrow[dl, "\alpha_\chi"] \arrow[dr]& \\
		\QS{\G_g^*}{G}  	& & \star
	\end{tikzcd}\]
	
\end{itemize}

We then compose these Lagrangian correspondences: 

$$ 
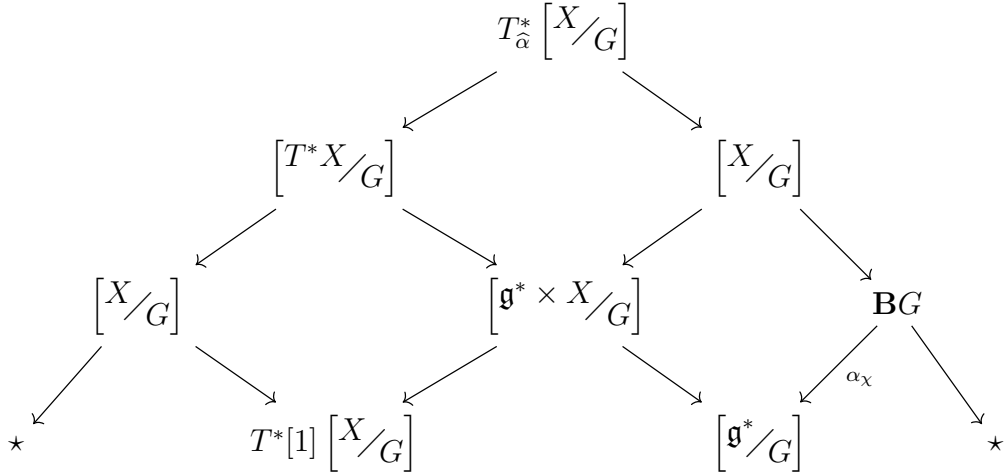
\begin{tikzcd}[column sep=small]
	& & & T_{\widehat{\alpha}}^* \QS{X}{G}  \arrow[dl] \arrow[dr]& & &\\
	& & \QS{T^*X}{G}  \arrow[dl] \arrow[dr] & & \QS{X}{G}  \arrow[dl] \arrow[dr] & & \\
	&\QS{X}{G}  \arrow[dl] \arrow[dr]& & \QS{\G_g^* \times X}{G}  \arrow[dl] \arrow[dr] & &\mathbf{B}G \arrow[dl, "\alpha_\chi"] \arrow[dr] & \\
	\star & & T^*[1] \QS{X}{G} & & \QS{\G_g^*}{G} & & \star 
\end{tikzcd}
$$

The only thing we need to show is that this is a diagram of Lagrangian correspondences and therefore we need to show that all squares in this diagrams are pull-backs. The right most square is a pull-back from Proposition \ref{prop:pullbak quotient stack of groupoids} and we can recognize the pullback square defining $T_{\widehat{\alpha}}^* \QS{X}{G} $.\\

We are left to prove that we have a natural equivalence:
\[ \QS{X}{G}  \times_{ T^*[1]\QS{X}{G} } \QS{ \G_g^* \times X}{G}  \simeq \QS{T^*X}{G}  \]

We are left to prove the following equivalence:

there it be\[ \QS{X}{G}  \times_{ T^*[1]\QS{X}{G} } \QS{ \G_g^* \times X}{G}   \simeq   \QS{X}{G}  \times_{ T^*[1]\QS{X}{G} }\QS{X}{G}  \times_{\QS{\star}{G} } \QS{ \G_g^*}{G} \]

We now use the fact that the self intersection of the zero section in $T^*[1]\QS{X}{G} $ is $T^*\QS{X}{G} $. This implies that:

\[ \QS{X}{G}  \times_{ T^*[1]\QS{X}{G} } \QS{ \G_g^* \times X}{G}   \simeq   T^*\QS{X}{G}  \times_{\QS{\star}{G} } \QS{ \G_g^*}{G} \]

We can now use the fact that we have a pull-back (see Proposition \ref{prop:cotangent quotient groupoid as a derived intersection} or Example 2.2.1 in \cite{Saf16}):
\[ T^* \QS{X}{G}  \simeq \QS{T^*X}{G} \times_{\QS{\G_g^* }{G}  }\mathbf{B}G\]

We use that to decompose $T^* \QS{X}{G}$ in a fiber product and we obtain:

\[\QS{X}{G}  \times_{ T^*[1]\QS{X}{G} } \QS{ \G_g^* \times X}{G}   \simeq \QS{T^*X}{G}  \times_{\QS{\G_g^* }{G} } \QS{\star}{G}   \times_{\QS{\star}{G} } \QS{\G_g^*}{G}  \simeq \QS{T^*X}{G}  \]

\newpage
\section{On the Derived Geometry of Lie Algebroids}\label{sec:on-the-derived-geometry-of-lie-algebroid}\

Lie algebroids are the natural objects controlling \emph{infinitesimal actions}. They are, in some sens, some kind of  ``generalized tangent bundle'' and they determine (through the \emph{anchor}) a sheaf of tangent vectors which are though as the infinitesimal directions along which a Lie algebroid acts. The reason Lie algebroids are the correct objects to encode infinitesimal actions can be explained by the heuristic saying that:

\begin{center}
	\noindent\fbox{%
		\parbox{7.5cm}{%
			\begin{center}
				Lie algebroids are to Lie groupoids what Lie algebras are to Lie groups
			\end{center}
		}%
	}
\end{center}
In fact given a smooth groupoid (over a  ``nice enough'' base): 
\[ \begin{tikzcd}
	\_G \arrow[r, shift left, "s"] \arrow[r, shift right, "t"'] & X 
\end{tikzcd} \]
with $e : X \to \_G$ the unit, we can construct a Lie algebroid (see Construction \ref{cons:lie algebroid from lie groupoid}) with anchor given by the natural map:
\[\rho : e^*\Ttr{\_G}{X}^s \to \Tt_X \]

The analogy is not perfect as not all Lie algebroids integrate to a \emph{smooth groupoid} (Definition \ref{def:smooth and formal groupoids}). However, we will see in Section \ref{sec:derivation-and-integration-of-lie-algebroids} that they integrate to a more general type of object, called ``formal groupoids'' (Definition \ref{def:smooth and formal groupoids}).\\ 

We will be interested, in Section \ref{sec:equivariant-symplectic-geometry}, by the study of ``infinitesimal'' equivariant geometry, as this kind of geometry naturally appears in the construction of the BV complex in Section \ref{sec:derived-perspective-of-the-bv-complex}.

As we will see, equivariant geometry is simply the study of the geometry of the ``quotient'' stack. In the infinitesimal case, we need to make sense of the notion of \emph{infinitesimal quotient stack}, which is the purpose of Section \ref{sec:quotient-stack-of-a-lie-algebroid}. \\

We start in Section \ref{sec:lie-and-linfty-algebroids} by recalling the definition and homotopy theory of Lie algebroids mainly following \cite{Nu19b, Nu19a}. In Section \ref{sec:quotient-stack-of-a-lie-algebroid}, we define the notion of infinitesimal quotient stack of a Lie algebroid. We will see that these are indeed infinitesimal versions of the quotient by a groupoid in Section \ref{sec:derivation-and-integration-of-lie-algebroids} by showing that they are (in good situations) the formal completions of the projections of the quotient stacks by smooth groupoids.

In Section \ref{sec:infinitesimal-action-of-lie-algebroids}, we discuss the notion of infinitesimal action of a Lie algebroid (up to homotopy) on a map, $f : Y \to X$, where $X$ is the base of the Lie algebroid. We also discuss the case of actions given by \emph{representations up to homotopy} (Section \ref{sec:representation-and-action-of-lie-algebroids}) with in particular the examples of the adjoint and coadjoint actions. \\

As we will work in formal geometry, from now on, all stacks will be locally finitely presented. Moreover, we will assume that all our stacks admit a cotangent complex.

\subsection{Lie and $\_L_\infty$-Algebroids} \label{sec:lie-and-linfty-algebroids}\

\medskip

In this section, we recall the definitions and homotopy theory of Lie and $\_L_\infty$ algebroids. This section mainly follows \cite{Nu19a,Nu19b}.

\subsubsection{Definitions and basic properties of Lie algebroids} 
\label{sec:definitions-and-basic-properties-of-lie-algebroids}\

\medskip

We want to discuss the notion of Lie algebroid over a derived stack $X$. In differential geometry, this would be defined as some structure on the module of global sections of a vector bundle. 	\\

In algebraic geometry it is not enough to work on the algebra of global sections (because it does not recover the structure of a Lie algebroid locally), and we should therefore define Lie algebroids as an \emph{$\infty$-sheaf} of Lie algebroid structures.

We will not do that as this is beyond our goals to develop such a theory. Therefore, in order to make sense of the homotopy theory of Lie algebroids, we restrict ourselves to the context used in \cite{Nu19b, Nu19a}. In practice, we are restricting ourselves to work over an affine base $X := \Spec(A)$ of almost finite presentation. In other words, since the base is affine, we go back to the situation where we can define the structure of a Lie algebroid on the module of global sections.   

\begin{Def}[{\cite[Definition 2.1]{Nu19a}}]
	\label{def:lie algebroid affine}
	
	A \defi{Lie algebroid over $A \in \cdgacon$} is an $A$-module $\_L$ together with a $k$-linear Lie algebra structure and an \defi{anchor map} $\rho : \_L \to \Tt_A$ such that: 
	\begin{itemize}
		\item $\rho$ is a map of $A$-module and of Lie algebras\footnote{The fact that this is a Lie algebra morphism is in fact a consequence of the Leibniz rule and the Jacobi identity of the Lie bracket.}. 
		\item The failure of the Lie bracket to be $A$-linear is controlled by a Leibniz formula: 
		\[ \left[ w, fv\right] = (-1)^{\vert f \vert \vert w \vert}f \left[ w,v\right] + \rho(w)(f).v \]
	\end{itemize}    
\end{Def}  

Recall from Notation \ref{not:linear stack and their sheaves} that we denote by $\_L$ the linear stack associated to $\_L$ with anchor $\rho: L \to TX$. We will interchangeably refer to $L$ and $\_L$ as a Lie algebroid over $X$.

\begin{Def} \label{def:lie algebroid morphism same base}\
	A morphism of Lie algebroids over $A$ is a morphism $f: \_L \to \_L'$ such that: 
	\begin{itemize}
		\item $f$ commutes with the anchors. 
		\item $f$ defines a morphism of the underlying Lie algebras.
	\end{itemize}

	This defines a category of Lie algebroids over $A$ together with the notion of morphisms from Definition \ref{def:lie algebroid morphism same base}. Moreover, we say that $f$ is a weak equivalence if the underlying map of $A$-modules is a quasi-isomorphism\footnote{In particular, this definition implies that the forgetful functor $\algbd_A \to \Mod_A$ is conservative.}. Localizing at these weak-equivalences, this defines the $\infty$-category of Lie algebroids (which we also denote by $\algbd_A$). We will see in Section \ref{sec:homotopy-theory-for-lie-algebroid} that this category can also be viewed as a semi model category. 
\end{Def}

\begin{RQ}\label{rq:lie structure on a module}
	Take $A \in \cdgacon$ and $\_L \in \Mod_A$. A structure of Lie algebroid on $\_L$ is the data of a Lie algebroid $\_L'$ together with a weak equivalence of the underlying $A$-modules: 
	\[ \_L \simeq \_L' \]
	
	In fact there is an $\infty$-category of Lie algebroid structures on $\_L$ given by the pullback of $\infty$-categories: 
	\[ \begin{tikzcd}
		\algbd_A(\_L) \arrow[d] \arrow[r] & \algbd_A \arrow[d] \\
		\star \arrow[r, "\_L"] & \Mod_A
	\end{tikzcd} \] 
\end{RQ}

Therefore a Lie algebroid structure on an $A$-module $\_L$ is the data of a Lie algebroid whose underlying $A$-module is only weakly equivalent to $\_L$. Unfortunately, it is not possible to transfer the structure of Lie algebroid along weak equivalences. However, it is possible to transfer a ``homotopy Lie algebroid structure'' and get $\infty$-quasi-isomorphisms. \\

Indeed, similarly to the case of Lie algebras, we can consider the notion of homotopy Lie algebroid, which we will call $\_L_\infty$-algebroid. Just like for Lie algebras and $\_L_\infty$-algebras, it is a more complicated object but thanks to Corollary \ref{cor:equivalence lie algebroids lie infity algebroid}, it turns out that it describes an $\infty$-category equivalent to that of Lie algebroids. The main advantage of $\_L_\infty$-algebroids lies in the notion of $\infty$-morphism (Definition \ref{def:lie infty algebroid infty morphism same base}) for which we have a version of the homotopy transfer theorem (Section \ref{sec:infty-morphism-and-homotopy-transfer-theorem}). 

\begin{Def}[{\cite[Definition 2.5]{Nu19b}}]  
	\label{def:lie infty algebroid affine}
	
	A \defi{$\_L_\infty$-algebroid over an algebra $A \in \cdga$} is an $A$-modules $\_L$ with the structure of a $k$-linear $\_L_\infty$-algebras and an anchor map $\rho : \_L \to \Tt_A$ such that: 
	\begin{itemize}
		\item $\rho$ is both a map of $A$-module and of $\_L_\infty$-algebra.
		\item The failure of the $2$-bracket to be $A$-linear is controlled by the Leibniz rule: 
		\[ \left[ w, fv\right] = (-1)^{\vert f \vert \vert w \vert}f \left[ w,v\right] + \rho(w)(f).v   \]
		\item For $n \geq 3$, the $n$-ary bracket is $A$-linear: 
		\[ \left[v_1, \cdots, fv_n \right]_n = (-1)^{\vert f \vert \left( \vert v_1 \vert + \cdots + \vert v_{n-1} \vert \right)} f \left[ v_1, \cdots, v_n \right]\] 
	\end{itemize}
	
	A morphism of $\_L_\infty$-algebroids over $A$ is a morphism $f: \_L \to \_L'$ such that: 
	\begin{itemize}
		\item $f$ commutes with the anchors. 
		\item $f$ defines a morphism of the underlying $\_L_\infty$-algebras.
	\end{itemize}

	This defines a category of $\_L_\infty$-algebroids over $A$.
	
	Moreover, we say that $f$ is a weak equivalence if the underlying map of $A$-modules is a quasi-isomorphism\footnote{In particular, this definition implies that the forgetful functor $\ialgbd_A \to \Mod_A$ is conservative.}. Localizing at these weak-equivalences, this defines the $\infty$-category of $\_L_\infty$-algebroids (which we also denote by $\ialgbd_A$). This $\infty$-category is equivalent to the localization of a semi-model category also denoted by $\ialgbd_A$.  
\end{Def}

Although this definition of morphism is simple, it is too strict. Instead, we want to define a notion of morphism that only respects the $\_L_\infty$-structures up to homotopy. Again, up to homotopy, this changes nothing. However this is the correct notion of ``weak'' morphism that will appear in the homotopy transfer theorem.  

\begin{Def}[{\cite[Definition 2.2]{PS20}}] \label{def:lie infty algebroid infty morphism same base}
	
	A $\_L_\infty$-morphism of $\_L_\infty$-algebroids over $A$, $f: \_L \rightsquigarrow \_L'$, is given by $A$-linear skew-symmetric maps for $n\geq 1$: 
	\[ f_n : \_L^{\otimes n} \to \_L'[1-n] \] such that: 
	\begin{itemize}
		\item $f_1$ commutes with the anchors. 
		\item $f$ defines an $\infty$-morphism of the underlying $k$-linear $\_L_\infty$-algebras (see \cite[Proposition 10.2.13]{LV}).
	\end{itemize}
\end{Def}

\begin{Def}\label{def:perfect algebroids}
	A Lie or $\_L_\infty$-algebroid $\_L$ is called \defi{perfect} if the underlying $A$-module of $\_L$ is perfect.
\end{Def}

\begin{Prop}\label{prop:lie algebroid relative tangent}
	Consider $X$ an affine stack satisfying Assumptions\footnote{These are the assumptions necessary to identify a formal moduli problem under $X$ with a Lie algebroid on $X$. This will be discussed in details in Section \ref{sec:quotient-stack-of-a-lie-algebroid}.} \ref{ass:very good stack} and a map of formal derived stacks: $f: X \to Y$ such that $Y$ admits a tangent complex. Then the relative tangent complex has a structure of Lie algebroid\footnote{In the sens of Remark \ref{rq:lie structure on a module}.} with anchor the natural map: 
	\[ \Ttr{X}{Y} \to \Tt_{X}\] 
	This construction is functorial in $Y \in \dstfp_{/X}$.
\end{Prop}
\begin{proof}
	From Lemma \ref{lem:formal completion factorization and morphims properties} we have a factorization: 
	\[ X \to \comp{Y_{X}} \to Y\]
	Since $Y$ is formal, $\comp{Y_X}$ is also formal and the map $X \to \comp{Y_{X}}$ is a formal thickening (Proposition \ref{prop:formal stack and de rham stack and formal completions}). Therefore it can be viewed as a formal thickening in pre-stack (see Remark \ref{rq:stackification issues} and Proposition \ref{prop:stackification and formal geometry}), and corresponds to a formal moduli problem $F$ (thanks to Theorem \ref{th:fmp are formal thickenings}). Then we have an equivalence (from Lemma \ref{lem:tangent formal spectrum} and the fact that the map $\comp{Y_X} \to Y$ is formally étale): \[\Ttr{X}{F} \simeq \Ttr{X}{\comp{Y_{X}}} \simeq \Ttr{X}{Y} \] 
	
	This uses the fact that since $Y$ and $\comp{Y_X}$ are stacks, their relative tangents coincide with their relative tangents when viewed as pre-stacks (see Proposition \ref{prop:stackification and formal geometry}). 
	
	From Corollary \ref{cor:relative tangent and Lie algebroids}, the relative tangent of a formal moduli problem, $\Ttr{X}{F}$ has a structure of Lie algebroid and therefore so does $\Ttr{X}{Y}$. 
	
	Taking the formal completion, the associated formal moduli problem and its associated Lie algebroid are all functors.		
\end{proof}

We will also see in Proposition \ref{prop:functoriality relative tangent base change} that this construction is functorial in the choice of $X$ (base change).

\begin{Ex}\label{ex:action lie algebroid}
	Let $G$ be an affine algebraic group acting on $X$ an affine stack. Let $\G_g$ denote it Lie algebra. Then $X \times \G_g$ has a structure of a Lie algebroid on $X$ given by: 
	\begin{enumerate}
		\item \[ \rho : X \times \G_g \to \Tt_X\]
		given by the tangent of the action $\gamma : G \times X \to X$ at the unit of $G$, given by the composition: 
		\[ X \times \G_g \overset{e}{\to} X \times G \times \G_g \simeq X \times TG \overset{s_0}{\to} TX \times TG \to TX  \] 
		\item The Lie bracket is defined by the Lie bracket on $\G_g$ extended to $X \times \G_g$ using the Leibniz property. 
	\end{enumerate}
	This is the \defi{infinitesimal action} associated to the action of $G$. 
	
	If $G$ acts smoothly on $X$ an Artin stack, this Lie algebroid is equivalent to the Lie algebroid given by the relative tangent $\Ttr{X}{\QS{X}{G}}$ from Example \ref{prop:lie algebroid relative tangent} of the projection $ p: X \to \QS{X}{G}$.
\end{Ex}

\begin{Ex}\label{ex:Lie algebroid}\
	
	\begin{itemize}
		\item The tangent complex with the identity, $\id : \Tt_X \to \Tt_X$, for anchor is a Lie algebroid.
		
		\item Example \ref{ex:action lie algebroid} extends to any map of $\_L_\infty$-algebra $\G_g \to \Tt_X$. Following \cite[Example 2.8]{Nu19b}, we can construct an \defi{action Lie or $\_L_\infty$-algebroid} on $A \otimes \G_g$ defining left adjoint functors to the forgetful functors (see \cite[Lemma 2.10]{Nu19b}):
		\[ \begin{tikzcd}
			\tx{Lie}_{/\Tt_A} \arrow[r, shift left] & \arrow[l, shift left] \algbd_A
		\end{tikzcd} \qquad \begin{tikzcd}
			(\tx{Lie}_\infty)_{/\Tt_A} \arrow[r, shift left] & \arrow[l, shift left] \ialgbd_A
		\end{tikzcd}\] 
		\item A (possibly singular) foliation of $X$ is equivalent to a Lie algebroid on $X$ with injective anchor map. 
		\item A Lie algebroid whose anchor is zero is the same thing as asking the Lie bracket to be $A$-linear. 
		
	\end{itemize}
\end{Ex}

\subsubsection{Homotopy theory for Lie algebroid} \label{sec:homotopy-theory-for-lie-algebroid}\

\medskip

We will extract the essential elements of homotopy theory of Lie and $\_L_\infty$ algebroid from \cite{Nu19b, Nu19a}. This sets up to necessary homotopy theory and the most important idea of this section is that up to homotopy, Lie algebroids and $\_L_\infty$-algebroids give equivalent $\infty$-categories. This holds essentially for the same reasons than the fact that the categories of Lie algebras and $\_L_\infty$-algebras are Quillen equivalent.

\begin{Th}[{\cite[Theorem 3.1]{Nu19b}}] \label{th:model (semi) structure on algebroids}
	
	The categories $\algbd_A$ and $\ialgbd_A$ both admit a right proper, tractable semi-model structure such that: \begin{itemize}
		\item A map is a weak equivalence if and only if it is a quasi-isomorphism\footnote{These are the same weak equivalences we discussed in Definition \ref{def:lie algebroid morphism same base}. Therefore, this makes this definition a (semi)-model for the $\infty$-categories of Lie and $\_L_\infty$ algebroids over $A$.}. 
		\item A map is a fibration if and only if it is a degreewise surjection.
	\end{itemize}
\end{Th}

\begin{RQ}\label{rq:model structure algebroid is a semi-model structure}
	The classes of weak equivalences and fibrations given by Theorem \ref{th:model (semi) structure on algebroids} do not define a model structure as there is in general no fibrant replacement as explained in \cite[Example 3.2]{Nu19b}. 
\end{RQ}

\begin{Th}[{\cite[Theorem 3.3]{Nu19b}}] 
	\label{th:adjunction quillen algebroid-module}
	
	The forgetful functors: 
	\[ U : \algbd_A \to \Mod_A^{/\Tt_A} \qquad \ialgbd_A \to \Mod_A^{/\Tt_A}\]
	are right Quillen functors with the following properties: \begin{itemize}
		\item they preserve cofibrant objects. 
		\item they preserve sifted homotopy colimits. 
	\end{itemize}
\end{Th}

The homotopy theory of Lie algebras and Lie algebroids are also related thanks to the following result:

\begin{Prop}[{\cite[Proposition 3.4]{Nu19b}}] 
	\label{prop:adjunction quillen Lie algebra-Lie algebroid}
	
	There is a Quillen adjunction\footnote{The model structure on $\lie_A$ is the usual one, transferred from the free-forget adjunction $\lie_A \dashv \Mod_A$.}:
	\[\begin{tikzcd}
		\lie_A \arrow[r, shift left, "i"] & \arrow[l, shift left, "\tx{ker}"] \algbd_A
	\end{tikzcd}\]
	
	The right adjoint, $\tx{ker}$ send a Lie algebroid to the kernel of its anchor map. Moreover its right derived functor detects equivalences and preserves all sifted homotopy colimit. 
\end{Prop}

\begin{Cor}[{\cite[Corollary 3.8]{Nu19b}}]    
	\label{cor:equivalence lie algebroids lie infity algebroid}
	
	The inclusion $j: \algbd_A \to \ialgbd_A$ is the right adjoint of a Quillen equivalence. In particular, any $\_L_\infty$-algebroid is weakly equivalent to a Lie algebroid.  
\end{Cor}

Note that we have a commutative diagram of right Quillen functors: 
\[ \begin{tikzcd}
	\algbd_A \arrow[r] \arrow[d, "\tx{ker}"]& \ialgbd_A \arrow[d, "\tx{ker}"] \\
	\lie_A \arrow[r, "w^*"] & \ilie_A 
\end{tikzcd}\]
where $w^*$ is the restriction functor coming from the map of operads: \[w : \bf{Lie}_\infty \to \bf{Lie}\]
This map is an equivalence of operads (even a cofibrant replacement) and induces a Quillen equivalence between the categories of Lie algebras and $\_L_\infty$-algebras (see \cite[Theorem 4.7.4]{Hi97}).

\subsection{Infinitesimal Quotient by a Lie Algebroid and Geometric Properties}\ \label{sec:quotient-of-lie-algebroid-and-geometric-properties}

\medskip

The notion of ``infinitesimal quotient'' by an infinitesimal action is not an obvious one. More often than not in the literature, given a Lie algebroid (i.e. an infinitesimal action on its base), its infinitesimal quotient is solely defined algebraically by the \CE algebra of the Lie algebroid (see for example \cite{PS20}). However there are two problems; the choice of the \CE algebra as a representative of derived infinitesimal invariants is a priori not canonical, and the relation to a notion of geometric \emph{infinitesimal quotient} is unclear\footnote{As we will see, the infinitesimal quotient is in particular \emph{not} the spectrum of the \CE algebra and therefore care must be taken in order to relate algebraic operations on the \CE algebras and geometric operations on the infinitesimal quotients.}.\\

In this section we provide another construction of a  \emph{formal stack} that we call \emph{infinitesimal quotient} and explain its relationship with the \CE algebra. However, it will still not be clear from the definition that this indeed corresponds to a sensible notion of infinitesimal quotient. It is only in Section \ref{sec:derivation-and-integration-of-lie-algebroids} that can explain (in good situations) that the infinitesimal quotient of $X$ by a Lie algebroid $\_L$, is the formal completion of the projection to the quotient by a groupoid $\_G$ \emph{integrating} $\_L$, giving a factorization: 
\[ X \to \QS{X}{\_L}\simeq \comp{\QS{X}{\_G}_X} \to \QS{X}{\_G}\] 

We will start by describing the relationship between Lie algebroids and \emph{formal moduli problems}. It turns out to be, under some technical conditions, an instance of an equivalence obtained from of a \emph{Koszul duality context} (Definition \ref{def:koszul duality context} and Theorem \ref{th:FMP equivalence tkoszul duality context}). 

Then from formal moduli problems, we have seen in Section \ref{sec:formal-stack-from-formal-moduli-problems} that we can obtain formal thickenings of the base. This will define the \emph{infinitesimal} quotient of $X$ by a Lie algebroid $\_L$.	    

\subsubsection{Quotient stack of a Lie algebroid} \label{sec:quotient-stack-of-a-lie-algebroid}\

\medskip

It is a well known result (Theorem \ref{th:Lurie-Pridham}) that the category of Lie algebras is equivalent to the category of ``pointed formal moduli problems''\footnote{Formal moduli problems for the deformation context on $\cdga_{/k}$, viewed as pointed affine spaces.}.  
We will explain that, in a similar way, Lie algebroids over an affine base $X$ correspond (under some technical assumptions on $X$) to formal moduli problems \emph{under} $X$ (see Theorem \ref{th:lie algebroid and FMP equivalence}). We refer to Appendix \ref{sec:formal-moduli-problem-and-koszul-duality-context} for a recollection on formal moduli problems and Koszul duality contexts.

We claim that such formal moduli problems represent ``infinitesimal quotients'' of $X$. We will see in Section \ref{sec:derivation-and-integration-of-lie-algebroids} why this is the case.

\begin{Th}[{\cite[Theorem 6.1]{Nu19a}}] 
	\label{th:lie algebroid and FMP equivalence}
	
	Suppose that $A \in \cdga^{\leq 0}$ is cofibrant and recall that $\FMP_A$ denotes the $\infty$-category of formal moduli problems under $X := \Spec(A)$. Then we have an adjunction: 
	\[\begin{tikzcd}
		\MC : \algbd_A  \arrow[r, shift left] & \arrow[l, shift left] \bf{FMP}_A : \Ttrl{X}{-} 
	\end{tikzcd}\]
	
	Moreover this is an equivalence whenever $A$ is eventually coconnective. The notation $\Ttrl{X}{-}$ is similar to the one for the relative cotangent complex for reasons that will be made clear with Lemma \ref{lem:lie algebroid and relative tangent maurer cartan functor}. 
\end{Th}

To have the equivalence of this theorem, we will often use the following assumptions: 
\begin{assumption}\label{ass:very good stack}
	For a derived stack $X$ we consider the following assumption: \begin{itemize}
		\item $X$ is affine, and $X = \Spec(A)$ with $A \in \cdgacon$.
		\item $A$ is cofibrant.
		\item $A$ is almost finitely presented. 
		\item $A$ is eventually coconnective. 
	\end{itemize}
\end{assumption}

\begin{Def}\label{def:infinitesimal quotient stack for eventually coconnective A}
	Given a Lie algebroid $\_L$ over $X$ satisfying Assumptions \ref{ass:very good stack}. Then the \defi{infinitesimal quotient} of $X$ by $\_L$ is defined as: 
	\[ \QS{X}{\_L} := \und{\MC_\_L}\]  
	
	where $\und{(-)}$ denotes the functor of Theorem \ref{th:fmp are formal thickenings} extending formal moduli problems to formal thickenings of $X$. This stack is defined as the stackification of the prestack: 
	\[ \pQS{X}{\_L} := \pund{\MC_\_L} \]
\end{Def}

\begin{RQ}
	It is a priori unclear why this is called a quotient. This will be explained in Section \ref{sec:derivation-and-integration-of-lie-algebroids}, where we will show that it can be identified (in good situations) with the formal completion of a quotient map:
	\[ p: X \to \QS{X}{\_G}\]
\end{RQ}

We now proceed to explain Theorem \ref{th:lie algebroid and FMP equivalence}. It turns out that the equivalence comes from a \emph{Koszul duality context} (Definition \ref{def:koszul duality context}) using Theorem \ref{th:FMP equivalence tkoszul duality context}. This Koszul duality context is given by an adjunction that involves the \CE functor: 

\begin{Def}\label{def:chevalley--eilenberg functor}
	The underived \CE functor is the functor:
	\[ \begin{tikzcd}[row sep = 1mm]
		\ceu: \algbd_A \arrow[r] & \left( \cdga_{/A}\right)^{\tx{op}} \\
		\qquad \qquad \_L \arrow[r, mapsto] &   \left( \iHom(\Sym_A \_L[1], A), \delta_{\ceu}\right) 
	\end{tikzcd}\]
	With the differential $\delta_{\ceu}$ determined by: \begin{itemize}
		\item its action on $A$, given by $\delta_A + \rho^* \circ \dr$ with  $\rho^* \circ \dr$ the composition: \[ A \to \Ll_A[-1]  \to \_L^\vee[-1] \]
		\item its action on $\_L^{\vee}[-1]$, given by $\delta_{\_L^\vee} + \Delta$ where $\delta_{\_L^\vee}$ is the dual of the differential on $\_L$,
		and $\Delta$ is such that for all $\alpha \in \_L^\vee[-1]$, and $v,w \in \_L$: 
		\[ (\Delta \alpha) (v,w) = \alpha([v,w]) - \rho(v)(\alpha(w)) + \rho(w)(\alpha(v))\] 
		\item In general the formula is given, up to signs, by: 
		\[ \begin{split}
			\delta_\ce \alpha (v_1, \cdots, v_n) = & \delta_A \alpha(v_1, \cdots, v_n )  - \sum\limits_i \alpha(v_1, \cdots, \delta_\_L v_i, \cdots, v_n)\\
			& + \sum\limits_i \rho(v_i)\left(\alpha(v_1, \cdots, v_{i-1}, v_{i+1}, \cdots, v_n)\right) - \sum\limits_{i<j} \alpha\left( [v_i, v_j], v_1, \cdots, v_n\right)
		\end{split} \]
		
	\end{itemize}
	
	If $\_L$ is perfect, then this is equivalently a differential on\footnote{In which case the differential is completely determined by its action on $A$ and $\_L^\vee[-1]$.}: 
	\[\left( \cSym_A \_L^\vee[-1], \delta_{\ceu}\right)\] 
	
	Moreover, this functor induces a functor between the associated $\infty$-categories: 
	\[ 	\ce: \algbd_A \to \left( \cdga_{/A}\right)^{\tx{op}} \]
	The underived \CE functor preserves weak equivalences between $A$-cofibrant\footnote{An $A$-cofibrant Lie algebroid is a Lie algebroid whose underlying module is cofibrant.} Lie algebroids (see the proof of \cite[Proposition 4.1]{Nu19a}). Therefore the derived \CE functor can be computed by taking an $A$-cofibrant replacement.
\end{Def}

\begin{Prop}
	\label{prop:lie algebroid koszul duality context}
	Let $A\in \cdga^{\leq 0}$ be cofibrant of almost finite presentation. Then the \CE functor has a right adjoint: 
	\[\begin{tikzcd}
		\ce : \algbd_A  \arrow[r, shift left] & \arrow[l, shift left] \left( \cdga_{/A}\right)^{\tx{op}} : \G_D
	\end{tikzcd}\] 
	
	such that $\G_D(B \to A) := \Ttrl{A}{B}$ is given by the $A$-module $\Ttr{A}{B}$ together with Lie algebroid structure given by the sub-lie algebroid (Remark \cite[Remark 4.20]{Nu19a}):
	\[\Ttr{A}{B} \simeq \tx{Der}_B(A,A) \to \Tt_A\] 
	
	We see with Example \ref{ex:koszul duality context lie algebroid} that when $A$ is eventually coconnective (in other words if $X = \Spec(A)$ satisfies Assumptions \ref{ass:very good stack}), this adjunction forms a Koszul duality context for the (dual) deformation contexts discussed in Examples \ref{ex:deformation contextes} and \ref{ex:dual deformation contexts}. This implies Theorem \ref{th:lie algebroid and FMP equivalence} thanks to Theorem \ref{th:FMP equivalence tkoszul duality context}.
\end{Prop}

\begin{RQ}
	\label{rq:description of MC} 
	
	We can now describe the functor $\MC$ in the context of Assumptions \ref{ass:very good stack}. From Corollary \ref{cor:koszul duality context and FMP}, we have that for all $\_L \in \algbd_A$: 
	\[ \MC_\_L := \Mapsub{\algbd_A}(\G_D(-), \_L)\]
\end{RQ}

\begin{Lem}\label{lem:lie algebroid and relative tangent maurer cartan functor}
	If $X = \Spec(A)$ satisfies Assumptions \ref{ass:very good stack}, then we have a commutative diagram: 
	\[ \begin{tikzcd}
		\algbd_A \arrow[r, "\MC"] \arrow[dr, "U"'] &  \FMP_A \arrow[d, "\Ttr{A}{-}"]\\
		& \Mod_A 
	\end{tikzcd}  \]
	
	Where the functor $U$ forgets the Lie algebroid structure. 
\end{Lem}
\begin{proof}

	From Proposition \ref{prop:FMP equivalence and tangent}, we have the commutative diagram: 
	
	\[ \begin{tikzcd}
		\algbd_A \arrow[rr, "\Psi"] \arrow[rd, "\varTheta"'] & & \FMP_A  \arrow[dl, "T"] \\
		& \mathbf{Sp} & 
	\end{tikzcd}\]

	where $\varTheta (\_L)$ is the spectrum object given for each $n\geq 0$ by: 
	\[ \Mapsub{\algbd_A}\left( \tx{Free}(A[-n-1]), \_L \right)\]
	
	This is equivalently the spectrum object obtained by applying the stable Dold--Kan functor to spectrum object in $A$-modules given for each $n\geq 0$ by:
	\[ \iHomsub{\algbd_A}\left( \tx{Free}(A[-n-1]), \_L \right) \simeq \iHomsub{\Mod_A^{/\Tt_A}}\left( A[-n-1], U(\_L)\right) \] 
	
	Similarly $T$ is the tangent functor (see Section \ref{sec:tangent-complex-of-a-formal-moduli-problem}) and can also be viewed as an $A$-module.  The following equivalence shows that $T \circ \MC$ gives back the same $A$-module as $\varTheta$:
	\[ \begin{split}
		\iHomsub{\FMP_A} \left(\Spf_A(A\boxplus A[n]), \MC_\_L \right) \simeq & \Mapsub{\algbd_A}\left(\G_D (A \boxplus A[n]), \_L  \right)\\
		\simeq & \Mapsub{\algbd_A} \left( \tx{Free}(A[-n-1]), \_L  \right)
	\end{split} \]
	
	where the second equivalence follows from the fact that for a Koszul duality context, the functor $\G_D$ sends the spectrum object ($A\boxplus A[n]$) of the deformation context to the spectrum object of the dual deformation context ($\tx{Free}(A[-n-1])$) (see Appendix \ref{sec:koszul-duality-context} for more details).\\
	
	In both cases, the element in $\bf{Sp}$ comes from the stable realization of the spectrum object given for all $n\geq 0$ by: 
	\[\begin{split}
		\iHomsub{\Mod_A^{/\Tt_A}}\left( A[-n-1], U(\_L)\right) \simeq&  \iHomsub{\Mod_A}\left( A[-n-1], \tx{fiber}(U(\_L) \to \Tt_A)\right)  \\
		\simeq &\tx{fiber}(U(\_L) \to \Tt_A) )[n+1]
	\end{split}\]
	
	Therefore, this spectrum object corresponds to the $A$-module\footnote{Recall that if $V$ and $W$ are complexes, and $f: V \to W$ is a map of complexes, then we denote by $ V[1] \oplus^f W $ the complex whose underlying module is $V[1] \oplus W$ with differential given by the sum of the differential on $V$, the differential on $W$ and $f$ viewed as a map of degree $1$.}: \[U(\_L)[1] \oplus^\rho \Tt_A\] 
	
	Therefore the functors $\varTheta$ and $T$ both factor trough $\Mod_A$ (and the factorizations commute with $\MC$). However these new functors are not quite the functors we want. To obtain the diagram we want, (with the relative tangent) we need to observe that the $A$-module we obtain is the pullback on $A$ of the tangent of $\MC_\_L$ and we want the relative tangent instead. \\
	
	We have a natural map $\Tt_A \to U(\_L)[1] \oplus^\rho \Tt_A$ and the relative tangent is the fiber of this map, which clearly gives $U(\_L)$, the underlying module of $\_L$. 
\end{proof}

\begin{Cor}\label{cor:relative tangent and Lie algebroids}
	Let $X = \Spec(A)$ satisfying Assumptions \ref{ass:very good stack}.
	If $F \in \FMP_A$, then thanks to Theorem \ref{th:lie algebroid and FMP equivalence}, there exists a Lie algebroid $\_L$ such that $F \simeq \MC_\_L$. Moreover the relative tangent $\Ttr{X}{F}$ has a structure of Lie algebroid making it equivalent to $\_L$. In other words, the underlying module of $\_L$ is weakly equivalent to $\Ttr{X}{F}$.  
\end{Cor}

This is a rather convoluted way to obtain a Lie algebroid structure on the relative tangent, $\Ttr{X}{F}$. The main problem being that the Lie algebroid structure obtained through the functor $\Ttrl{X}{-}$ is not explicit. However, in the situation where we take $F = \Spf_A(B)$ for $B \in \cdga_{/A}$, we have the following: 

\begin{Lem}\label{lem:formal spectrum is the fmp associated to the relative tangent}
	If $X = \Spec(A)$ satisfies Assumptions \ref{ass:very good stack}, we have an equivalence of formal moduli problems: 
	\[ \MC_{\G_D(B)}  \simeq \Spf_A(B)\]
\end{Lem} 
\begin{proof}
	As functors on $\small$, we have the equivalences:
	\[ \begin{split}
		\MC_{\G_D(B)} \simeq &  \Mapsub{\algbd_A}\left( \G_D(-), \G_D(B) \right) \\ 
		\simeq & \Mapsub{\op{\left(\cdga_{/A}\right)}}\left(\ce(\G_D(-)), B\right) \\
		\simeq & \Mapsub{\cdga_{/A}}\left( B, \ce(\G_D(-))\right) \\
		\simeq & \Mapsub{\cdga_{/A}}\left( B, -\right) \\
		\simeq & \Spf_A(B)
	\end{split} \]
	
	where the fourth equivalence is due to the fact that $\ce$ and $\G_D$ defines an equivalence between Artinian and good algebras (Proposition \ref{prop:properties koszul duality context}) and therefore the unit and counit of the adjunction are equivalences when restricted to these algebras. 
\end{proof}

\begin{Cor}
	\label{cor:lie algebroid of formal spectrum}
	Let $X := \Spec(A)$ satisfying Assumptions \ref{ass:very good stack} and $B \in \cdga_{/A}$. We have the following: 
	
	\begin{enumerate}
		\item The underlying module of the Lie algebroid associated with the formal spectrum $\Spf_A(B)$ is $\Ttr{A}{B} \simeq \Ttr{X}{\Spf_A(B)}$. 
		
		\item Therefore, from Corollary \ref{cor:relative tangent and Lie algebroids}, $\Ttr{A}{B}$ has a structure of Lie algebroid making it equivalent to $\Ttrl{X}{\Spf_A(B)}$.
		
		\item Using Lemma \ref{lem:formal spectrum is the fmp associated to the relative tangent}, there are equivalences of Lie algebroids: 
		\[ \Ttrl{X}{\Spf_A(B)} \simeq \Ttrl{X}{\MC_{\G_D(B)}} \simeq \G_D(B)\] 
		
	\end{enumerate}
	
	Therefore the Lie algebroid structure on $\Ttrl{A}{\Spf_A(B)}$ is the one coming from $\G_D(B)$, described in Proposition \ref{prop:lie algebroid koszul duality context} with underlying module $\Ttr{A}{B}$. 
\end{Cor}

\begin{Cor}\label{cor:relative tangent completion of formal spectrum}
	There is an equivalence of Lie algebroids: 
	\[ \Ttr{A}{\pund{\Spf_A(B)}} \simeq \Ttr{A}{B}\]
\end{Cor}
\begin{proof}
	Using Corollary \ref{cor:relative tangent complexe formal stack}, we have an equivalence of $A$-modules: 
	\[\Ttr{A}{\pund{\Spf_A(B)}} \simeq \Ttr{A}{\Spf_A(B)} \]
	
	This gives $\Ttr{A}{\pund{\Spf_A(B)}}$ the structure of Lie algebroid coming from Corollary \ref{cor:lie algebroid of formal spectrum}
\end{proof}

\begin{RQ}\label{rq:fmp and ce}
	In the literature (see for example \cite{PS20}), the infinitesimal quotient of $X$ by a Lie algebroid $\_L$ is often represented by the \CE algebra of $\_L$. Indeed, these algebras describe some kind of ``derived invariants''. However, we will see in Section \ref{sec:tangent-and-cotangent-of-quotient-stack} that $\MC_\_L$ (and its extension to a formal stack) is a better candidate to be some kind of quotient. We can relate $\MC_\_L$ to the \CE algebra of $\_L$ by using the counit of the adjunction: 
	\[ \_L \to \G_D(\ce(\_L)) \]
	It gives us a morphism: 
	\[ \MC_\_L \to \MC_{\G_D(\ce(\_L))} \simeq \Spf_A(\ce(\_L)) \] 
\end{RQ}

However, this will in general not be an equivalence as this would correspond to saying that the derived \CE functor is fully-faithful. We expect that we need a variant of $\Spf_A(-)$ that would remember the graded mixed structure on the \CE algebra to solve this issue for perfect Lie algebroids. \\

Even if the derived \CE functor behaves rather badly, it turns out that the \emph{underived} \CE functor has nice properties: 

\begin{Lem}\label{lem:map of ce algebras induce maps of lie algbroids}
	Take $\_L$ and $\_L'$ two Lie algebroids over $A \in \cdgacon$ a finitely presented\footnote{In particular, from \cite[Theorem 7.4.3.18]{Lu17}, it is equivalent to saying that $H^0(A)$ is finitely generated and $\Tt_A$ is perfect.} algebra. We consider a map (over $A$) of graded mixed complexes: \[ \Phi :\gmc{\ceu}(\_L') \to \gmc{\ceu}(\_L) \]
	
	Assume that the map between the weight $1$ part of these graded mixed algebras is given by the dual of the diagram in $\Mod_A^{/\Tt_A}$: 
	\[ \begin{tikzcd}
		\_L\arrow[r, "f"] \arrow[dr, "\rho"'] & \_L' \arrow[d, "\rho'"] \\
		& \Tt_A 
	\end{tikzcd} \] 
	
	Then $f$ is a morphism of Lie algebroids.
\end{Lem}
\begin{proof}
	First we want to show that $f$ commutes with the anchors. But because the mixed differential restricted to $A$ is a derivation $A \to \_L^\vee[-1]$, it factors through $\Ll_A[-1]$ and we get a commutative diagram: 
	\[ \begin{tikzcd}
		\Ll_A \arrow[r] \arrow[dr] & (\_L')^\vee \arrow[d, "f^\vee"] \\
		& \_L^\vee
	\end{tikzcd} \] 
	
	This diagram is the dual (by assumption) to the commutative diagram: 
	\[ \begin{tikzcd}
		\_L\arrow[r, "f"] \arrow[dr, "\rho"'] & \_L' \arrow[d, "\rho'"] \\
		& \Tt_A 
	\end{tikzcd} \]

	To show that $f$ preserves the Lie bracket, we take $\alpha \in (\_L')^\vee[-1]$ and use the fact that: 
	\[ \Phi (\delta_{\ce} \alpha) = \delta_{\ce} (\Phi(\alpha)) = \delta_{\ce} (\alpha \circ f) \]
	
	For all $v,w \in \_L$, this equality implies that: 
	\[\begin{split}
		\alpha \left( [f(v), f(w)]'\right)  + \rho'(f(v))(\alpha(f(w))) - \rho'(f(w))(\alpha(f(v))) &\\
		=\alpha(f([v,w])) + \rho(v)(\alpha(f(w))) - \rho(w) (\alpha(f(v))) &
	\end{split} \]
	
	using the fact that $\rho' \circ f = \rho$, this implies that: 
	\[ \alpha \left( [f(v), f(w)]'\right) =\alpha(f([v,w])) \]
	and therefore $f$ preserves the Lie bracket, and is a morphism of Lie algebroids.
\end{proof}

\begin{Lem}\label{lem:ce underived functor}
	Let $A \in \cdgacon$. The \CE functor factors through graded mixed algebras over $A$ (where $A$ is seen as a graded mixed algebra concentrated in weight $0$): 
	
	\[\begin{tikzcd}
		\algbd_A \arrow[rr, "\ceu"] \arrow[dr, "\gmc{\ceu}"'] & &  \left( \cdga_{/A}\right)^{\tx{op}} \\
		& \left(\gmc{\cdga}_{/A}\right)^{\tx{op}} \arrow[ur, "\rel{-}"']
	\end{tikzcd}\]
	If we assume that $A \in \cdgacon$ is finitely presented.
	Then the essential image of $\gmc{\ceu}$ restricted to perfect Lie algebroids is given by the graded mixed algebras $E$ such that there is an equivalence\footnote{In other words, it is a semi-free graded mixed algebra generated by $A$ in weight $0$ and $\_L^\vee[-1]$ in weight $1$.}: 
	\[ \gr{E} \simeq \iHom_A\left(\gr{\Sym}_A \_L[1], A \right)\]
	for some $\_L \in \perf(A)$. 
	
	Moreover, restricted to strictly perfect\footnote{Strictly perfect modules are modules finitely presented projective instead of being only quasi-isomorphic to such a module.} Lie algebroids, the \emph{strict} functor, $\gmc{\ceu}$, is fully-faithful. 
\end{Lem} 
\begin{proof}
	Define the weight $p$ part to be:
	\[\gmc{\ceu(\_L)}(p):= \iHom_A\left(\Sym_A^p \_L[1], A \right) \]
	From the definition of $\delta_{\ce}$, the \CE differential clearly splits in a part preserving the weight (given by the internal differentials on $A$ and $\_L^\vee[-1]$) and a part increasing it by exactly $1$ (given by the duals of the anchor and the Lie bracket). This defines the structure of a graded mixed complex\footnote{another way to view this is to observe that $\delta_{\ce}$ respects the natural complete filtration on the completed symmetric algebra. Moreover the differential decomposes in the weight components $\delta_{\ce} = \delta_0 + \delta_1$ and therefore, under the equivalence of Corollary \ref{cor:equivalence weak and strict graded mixed complexes}, we get a graded mixed structure with mixed differential $\delta_1$.}. From the definition of the realization functor (Definition \ref{def:realization functor graded mixed complexes}) we can see that: \[\rel{\gmc{\ceu}(\_L)} \simeq \ceu(\_L)\]
	
	We now turn to the proof of faithfulness and assume $A$ finitely presented. A map of Lie algebroid $f: \_L \to \_L'$ is completely determined by the map it induces on the \CE graded mixed algebras by looking at the induced map $(\_L')^\vee[-1] \to \_L^\vee[-1]$ in weight $1$, and knowing that the strict dualization $(-)^\vee$ is fully-faithful on strictly perfect complexes.\\
	
	For the essential image, take $E$ a graded mixed algebra over $A$ as in the statement having $\_L^\vee[-1]$ in weight $1$. Then we can construct a Lie algebroid structure on $\_L$ as follows:
	\begin{itemize}
		\item The mixed differential restricted to $A$ gives a map: 
		\[ A \to \_L^\vee [-1] \]
		This map is a derivation and factors though $\dr : A \to \Ll_A[-1]$. We get a map $\Ll_A \to \_L^\vee$ (of degree $0$) whose dual defines the anchor (since $\_L$ and $\Ll_A$  are perfect). 
		\item Let $\alpha \in \_L^\vee[-1]$ and $v,w \in \_L$, then we consider: 
		\[ (\epsilon\alpha) (v,w) + \rho(v)(\alpha(w)) - \rho(w)(\alpha(v))\]
		where $\epsilon$ is the mixed differential. Note that from Definition \ref{def:chevalley--eilenberg functor}, this is what we expect to be $\alpha([v,w])$. We can show that this expression is $A$-linear in $\alpha$. This gives us a $k$-linear map: 
		\[ \_L \otimes \_L \to \_L \]
		sending $(v,w)$ to the unique element, denoted $[v,w]$, such that: 
		\[\alpha([v,w]) =  (\epsilon\alpha) (v,w) + \rho(v)(\alpha(w)) - \rho(w)(\alpha(v))\]
		Using this formula and the fact that $\epsilon^2 = 0$, we can check that this is a Lie bracket and that the Leibniz rule holds. 
	\end{itemize}

	To show that $\gmc{\ceu}$ is full (when restricted to strictly perfect Lie algebroids), consider any map over $A$: \[\phi : \gmc{\ceu}(\_L) \to \gmc{\ceu}(\_L')\] 
	This map is determined by what it does on each weight and since the \CE graded mixed algebra is generated is weight $0$ and $1$ (for strictly perfect Lie algebroids), $\phi$ is completely determined by: \[\phi_0 : A \to A := \id \qquad \phi_1 : \_L^\vee[-1] \to (\_L')^\vee[-1]\]
	Using the fully-faithfullness of the functor $(-)^\vee$ (on perfect Lie algebroids), we get a map of linear stacks $\_L' \to \_L$. The fact that $\phi$ preserves the \CE differentials implies that this is a map of Lie algebroids over $A$ (since $A$ is finitely presented, thanks to Lemma \ref{lem:map of ce algebras induce maps of lie algbroids}).
\end{proof}

This lemma has an analogue where we replace Lie algebroids by $\_L_\infty$-algebroids and graded mixed algebras by weak graded mixed algebras.

\begin{Lem}\label{lem:ce conservative} 
	The derived graded mixed \CE functor, $\gmc{\ce}$, is conservative when restricted to perfect Lie algebroids. 
\end{Lem}
\begin{proof}
	We consider the following commutative diagram: 
	\[ \begin{tikzcd}
		\Mapsub{\algbd_A}(\_L, \_L') \arrow[r] \arrow[d] & \Mapsub{\gmc{\cdga}_{/A}}\left( \gmc{\ce}(\_L'), \gmc{\ce}(\_L)\right) \arrow[d] \\
		\Map_A(\_L, \_L') \arrow[r] & \Map_{A} (\_L'^\vee, \_L^\vee)
	\end{tikzcd} \]
	Take a morphism of Lie algebroids $f: \_L \to \_L'$ such that it induces an equivalence $\gmc{\ce}(\_L') \to \gmc{\ce}(\_L)$. In particular, it induces an equivalence $(\_L')^\vee \to \_L^\vee$. The dualization functor $(-)^\vee$ is fully-faithful on perfect modules and therefore conservative. Moreover, the forgetful functor $\algbd_A \to \Mod_A$ is also conservative which implies that the morphism of Lie algebroids $\_L \to \_L'$ is also an equivalence, and therefore $\ce$ is a conservative functor. 
\end{proof}

We will use the following notations: 

\begin{notation}\ \label{not:ce}
	\begin{itemize}
		\item $\ceu$ is the \CE functor, giving an algebra over $A$.
		\item $\gmc{\ceu}$  the \CE functor giving a graded mixed  algebra (see Proposition \ref{lem:ce underived functor}). It can also be viewed as a weak graded mixed algebra\footnote{Considering the \CE algebra as a weak graded mixed algebra is relevant when viewing it as a $\_L_\infty$-algebroid.} that we denote $\gmch{\ceu}$. The \CE differential is the total differential for this graded mixed structure and therefore $\ceu := \rel{\gmc{\ceu}}$, where $\rel{-}$ is the realization functor given by Proposition \ref{prop:realization of graded mixed complexes}. 
		\item $\gmch{\ceu}$ is the weak graded mixed \CE algebra of a $\_L_\infty$-algebroid.
		\item $\gr{\ceu}$ is the underlying graded algebra of $\gmc{\ceu}$ (see Lemma \ref{lem:graded mixed complex as graded modules}) where we forget the mixed structure. 
		\item Since weak graded mixed complexes are equivalent to complete filtered objects (see Appendix \ref{sec:complete-filtered-commutative-algebras-and-weak-graded-mixed-complexes}), we denote by $\cpl{\ceu}$ the associated complete filtered algebra to $\gmc{\ceu}$. It is described, thanks to Proposition \ref{prop:totalization weak graded mixed complex}, by the following filtration: 
		\[ F^p \cpl{\ceu}(\_L) \simeq  \prod_{p' \geq p} \left(\Sym_A^{p'} \_L[1]\right)^\vee\] 
		
		Since $\ceu(\_L)$ is non-negatively weighted, each $F^p \cpl{\ceu}(\_L)$ is a differential graded algebras with the differential being the \CE differential. The same hold for $\_L_\infty$-algebroids.
	\end{itemize}
\end{notation}

\begin{RQ}\label{quotient of infinty algebroids}
	As the  $\infty$-category of $\_L_\infty$-algebroids is equivalent to the category of Lie algebroids (Corollary \ref{cor:equivalence lie algebroids lie infity algebroid}) we can define the infinitesimal quotient of a $\_L_\infty$-algebroid  as the infinitesimal quotient of any Lie algebroid equivalent to it. 
\end{RQ}

\subsubsection{Pullback and base change of Lie algebroids} \label{sec:pullback-and-base-change-of-lie-algebroids}\

\medskip

In this section we will use $X := \Spec(A)$, $Y:= \Spec(B)$ and $Z:=\Spec(C)$ with $A$, $B$ and $C$ non-positively graded cofibrant cdgas of almost finite presentation. 

We will described both the base change and pullbacks of Lie algebroids, defining in the process the notion of morphisms of Lie algebroids over different bases. We will show that these constructions can also be viewed from an algebraic point of view using their \CE algebras. Many of these constructions on Lie algebroids are adapted and generalized from \cite{Kl17}.\\

\begin{Def}
	\label{def:base change lie algebroid}
	
	Take $\_L$ a Lie algebroid over $X$.  
	We define $f^! \_L$ to be the \defi{Lie algebroid} over $Y$ defined as the fiber product:  
	\[ \begin{tikzcd}
		f^! \_L \arrow[r] \arrow[d, "f^!\rho"] & f^* \_L \arrow[d, "f^*\rho"]\\
		\Tt_Y \arrow[r, "f_*"] & f^*\Tt_X 
	\end{tikzcd}  \]
	
	together with:
	\begin{itemize}
		\item An anchor $f^!\rho$ given by the projection: 
		\[ 
		f^! \_L \to \Tt_Y 
		\]
		\item Picking a model where the map $A \to B$ is a cofibration, this pullback becomes strict and the Lie bracket is defined on $\Tt_Y \times_{f^* \Tt_X} f^* \_L$ by:
		\[ \left[ X + gv, Y + hw \right] = [X,Y] + \left(gh \left[ v,w \right] + \_L_X(h).w - \_L_Y(g).v \right) \]
	\end{itemize}
	This defines a Lie algebroid structure on $f^!\_L$ and given the choice of a cofibrant replacement $A \to B$, this construction defines a model for that Lie algebroid.
\end{Def}

\begin{RQ}
	In \cite{Kl17}, this base change is not always possible. The extra condition they required ensures that the strict fiber product is a \emph{vector bundle}. Since we work in the more general context of linear stacks, such a (derived) fiber product always exists and is linear (thanks to Proposition \ref{prop:map of linear stacks}). Then the rest of the definition still makes sense of a strict model of the homotopy pullback and defines the desired Lie algebroid structure.
\end{RQ}

\begin{War}\label{war:wrong definition}
	Although somewhat natural, we will see later in this section that Definition \ref{def:base change lie algebroid} might not be quite the definition we want. We will later on change this definition to Definition \ref{def:better pullback algebroid}. We expect both definitions to agree and everything we discuss in this section also makes sense for the second definition. 
\end{War}

\begin{Def}\label{def:lie algebroid morphism different base}
	Let $\_L'$ and $\_L$ be Lie algebroids over $Y$ and $X$ respectively. Then a morphism of Lie algebroids (over different bases) is defined by a commutative diagram: 
	\[ \begin{tikzcd}
		L' \arrow[d] \arrow[r, "\phi"] & L \arrow[d] \\
		Y \arrow[r,"f"] & X
	\end{tikzcd}\]
	where $\phi$ is induced by the composition of a morphism $\phi^! : \_L' \to f^! \_L$ of Lie algebroids over the same base (Definition \ref{def:lie algebroid morphism same base}) with the natural morphism from the pullback Lie algebroid $f^!\_L \to \_L$. 
\end{Def}

We will rephrase this base change in terms of base change for their \CE graded mixed algebras. To do so, we first need to make sense of the notion of morphism between \CE algebras over different bases. 

\begin{Def}\label{def:morphism ce algebra different bases}
	Given $f: A \to B$ and Lie algebroids $\_L_2$ and $\_L_1$ over $A$ and $B$ respectively, we say that $\phi: \gmc{\ceu}(\_L_2) \to \gmc{\ceu}(\_L_1)$ is a\defi{morphism of graded mixed complexes over $f$} if the following diagram commute\footnote{This is a diagram of graded mixed algebras where $A$ and $B$ are concentrated in weight $0$. The vertical morphisms are the natural projections.}:
	\[ \begin{tikzcd}
		\gmc{\ceu}(\_L_2) \arrow[r] \arrow[d]& \gmc{\ceu}(\_L_1)\arrow[d]\\
		A \arrow[r, "f"] & B  
	\end{tikzcd} \]
\end{Def}

Using this definition, the \CE functor can be extended over different bases:
\begin{Cons}\
	\label{cons:ce functor different bases}
	We can define a strict (not derived) functor: 
	\[ \gmc{\ceu} : \algbd \to \left(\gmc{\cdga}\right)^{\tx{op}} \]
	
	that sends any morphism of Lie algebroids: 
	\[ \begin{tikzcd}
		L_1 \arrow[r] \arrow[d]& L_2 \arrow[d] \\
		Y \arrow[r, "f"] & X
	\end{tikzcd} \]
	
	to a morphism of graded mixed algebras $\gmc{\ceu}(\_L_2) \to \gmc{\ceu}(\_L_1)$ over $f$. 
	
	The definition of this functor on objects given by the usual \CE algebra. Any morphism of Lie algebroids as described above induces a map of modules over the morphism $f: B \to A$: \[\phi : \_L_2^\vee[-1] \to \_L_1^\vee[-1]\]  We also get morphisms in higher weights: 
	\[ \left(\Sym_A^p \_L_2[1]\right)^\vee \to \left(\Sym_A^p \_L_1[1]\right)^\vee \]

	This defines a morphism between the graded \CE algebras (because $f$ and $\phi$ preserve the internal differentials). 
	We only need to check that this map preserves the mixed differential.  Essentially, this map preserves the mixed differential restricted to the weight $0$ part $\epsilon_{\vert A} := \rho^* \circ \dr$ because the map $\_L_1 \to \_L_2$ is compatible with the anchors (because $\_L_1 \to f^! \_L_2$ must be compatible with the anchors from Definition \ref{def:lie algebroid morphism same base}): 
	\[ \begin{tikzcd}
		\_L_1 \arrow[r] \arrow[d] & f^*\_L_2 \arrow[d]\\
		\Tt_Y \arrow[r] & f^* \Tt_X 
	\end{tikzcd} \]
	
	And the rest of the compatibilities are due to the fact that the map $\_L_1 \to \_L_2$ is a morphism of Lie algebras (because both $\_L_1 \to f^! \_L_2$ and $f^!\_L_2 \to \_L_2$ are maps of Lie algebras) and the fact that they preserve of the anchors, since the \CE differential is build out of the Lie bracket and anchor (see Definition \ref{def:chevalley--eilenberg functor}).
\end{Cons}

It turns out that the \CE functor over different bases behaves similarly to the one over a fixed base:

\begin{Lem}\label{lem:ce underived functor different bases}
	The strict \CE functor factors through graded mixed algebras: 
	\[\begin{tikzcd}
		\algbd \arrow[rr, "\ceu"] \arrow[dr, "\gmc{\ceu}"'] & &  \left( \cdga\right)^{\tx{op}} \\
		& \left(\gmc{\cdga}\right)^{\tx{op}} \arrow[ur, "\rel{-}"']
	\end{tikzcd}\]
	We assume that the bases are affine and finitely presented.
	The essential image of $\gmc{\ceu}$ restricted to strictly perfect Lie algebroids is given by the graded mixed algebras $E$ such that there is an equivalence\footnote{In other words, it is a semi-free graded mixed algebra generated by $A$, a non-positively graded and finitely presented algebra, in weight $0$ and $\_L^\vee[-1]$ in weight $1$.}: 
	\[ \gr{E} \simeq \iHom_A\left(\gr{\Sym}_A \_L[1], A \right)\]
	for $A \in \cdgacon$ finitely presented and $\_L \in \perf(A)$. 
	
	Moreover, restricted to strictly perfect Lie algebroids, the \emph{strict} mixed graded \CE functor, $\gmc{\ceu}$, is fully-faithful.
\end{Lem}

We can now rephrase the base change construction in algebraic terms. 
To do that, we observe that the anchor map $\_L \to \Tt_A$ is a map of Lie algebroids\footnote{$\Tt_A$ is viewed as a Lie algebroid over $X$ with the identity as anchor map (see Example \ref{ex:Lie algebroid}).} (over the same base) and therefore induces a map of graded mixed algebras\footnote{The first equivalence is due to the fact that $A$ is assumed to be cofibrant and finitely presented. Therefore thanks to Corollary \ref{cor:mixed structure on symmetric power on cotangent} we get the desired equivalence}:
\[\DR(A)\simeq \gmc{\ceu}(\Tt_A) \to \gmc{\ceu}(\_L)\] 

\begin{Lem}
	\label{lem:base change lie algebroid ce version}
	Given a Lie algebroid $\_L$ over $X$ an affine stack finitely presented and $f: Y \to X$, then we have a pushout square in the category of graded mixed algebras: 
	\[ \begin{tikzcd}
		\DR(A) \arrow[r] \arrow[d] & \DR(B) \arrow[d] \\
		\gmc{\ceu}(\_L) \arrow[r] & \gmc{\ceu}(f^! \_L)
	\end{tikzcd} \]
	where $f^! \_L$ is the pullback Lie algebroid of Definition \ref{def:base change lie algebroid} and the morphism: \[\gmc{\ceu}(\_L) \to \gmc{\ceu}(f^! \_L)\]
	is induced by the morphism of Lie algebroids $f^!\_L \to \_L$. 
\end{Lem}
\begin{proof}
	Up to taking a cofibrant replacement of $f:A \to B$, we can chose $B$ to be semi-free over $A$. This implies that $\DR(B)$ is semi-free over $\DR(A)$. Indeed, if $B = \Sym_A \_F^\vee$ with $\_F$ projective. Then using a connection on $\_F$, have an equivalence (see Section \ref{sec:connections-on-semi-linear-stacks}): 
	\[ \begin{split}
		\gr{\DR}(B) \simeq &\gr{ \Sym}_{\Sym_A \_F^\vee} \left(\Sym_A \_F^\vee \otimes_A \left(\Ll_A \oplus^\nabla \_F \right)^\vee[-1]\right) \\ 
		\simeq & \Sym_A \left( \_F^\vee \oplus \Ll_A[-1] \oplus \_F^\vee[-1]  \right) \\
		\simeq & \Sym_{\gr{\DR}(A)} \left(\gr{\DR}(A) \otimes_A (\_F^\vee \oplus \_F^\vee[-1])\right)
	\end{split}  \]
	
	where $\_F^\vee[-1]$ is in weight $1$. Moreover by functoriality there is a map $\DR(A) \to \DR(B)$, which is therefore a cofibration and we can compute the homotopy pushout as a strict pushout. Its underlying graded algebra can be computes as the pushout of the underlying graded algebras: 
	\[ \begin{tikzcd}
		\gr{\Sym}_A \Ll_A[-1] \arrow[r] \arrow[d] & \gr{\Sym}_B \Ll_B[-1] \arrow[d] \\
		\gr{\Sym}_A \_L^\vee[-1] \arrow[r] & \gr{\Sym}_B (f^*\_L^\vee[-1] \coprod_{f^*\Ll_A[-1]} \Ll_B[-1]) 
	\end{tikzcd} \] 
	
	Since a strict model for $f^!\_L$ is also obtained through a cofibrant (or just semi-free) replacement of  $A \to B$, we have for that model that:
	\[\gr{\ceu}(f^!\_L) \cong \gr{\Sym}_B (\_L^\vee[-1] \coprod_{\Ll_A[-1]} \Ll_B[-1]) \simeq \gr{\ceu}(\_L) \otimes_{\gr{\DR}(A)} \gr{\DR}(B) \]
	We only have to show that the graded mixed structures coincide.  
	
	\begin{itemize}
		\item The weight $0$ part of the \CE differential clearly coincides with the weight $0$ part of the differential of the pushout.
		\item The part of the differential of weight $1$, restricted to $B$, is $\rho^* \circ \dr$ with the anchor being the natural projection: \[ \rho : \Tt_B \times_{f^*\Tt_A} f^*\_L \to \Tt_B\]
		the natural projection. This coincides with the differential on the pushout given by the composition: 
		\[ \begin{split}
			B \simeq B \otimes_A A \overset{\dr}{\to} & \ \Ll_B[-1] \simeq \Ll_B[-1] \coprod_{f^*\Ll_A[-1]} f^*\Ll_A[-1] \\
			\overset{\rho^*}{\to}& \ \Ll_B[-1] \coprod_{f^*\Ll_A[-1]} f^*\_L^\vee[-1]
		\end{split} \]
		where the last map coincide with the natural inclusion: \[\Ll_B[-1] \to \Ll_B[-1] \coprod_{f^*\Ll_A[-1]} f^*\_L^\vee[-1]\]
		
		\item Take $\mu \in \Ll_B[-1]$ and $\alpha \in \_L^\vee[-1]$. Then $\alpha$ induces an element (which we all also call $\alpha$), in $f^*\_L^\vee[-1]$, seen as a map $f^*\_L \to B$ with $\alpha(v):= \alpha(1 \otimes v) = f(\alpha(v))$.
		
		Then for all $X, Y  \in \Tt_B[1]$, $g,h \in B$ and $v,w \in \_L[1]$ such that $f_*X = g\rho(v)$ and $f_* Y = h\rho(w)$, we have that the part of the \CE differential on $\ceu(f^! \_L)$ in weight $1$ can be described by: 
		
		\[ \begin{split}
			\delta_{\ce} (\mu + \alpha)(X + gv, Y + hw) = \ &  \mu([X,Y]) + gh \alpha[v,w] + X(h)\alpha(w) - Y(g)\alpha(v) \\
			& - X(\mu(Y)) - X(h \alpha( w)) \\
			& + Y(\mu(X)) + Y(g \alpha(v)) \\
			= \ &  \mu([X,Y]) + gh \alpha[v,w] + X(h)\alpha( w) - Y(g)\alpha( v) \\
			& - X(\mu(Y)) - X(h) \alpha( w) - hX(\alpha( w)) \\
			& + Y(\mu(X)) + Y(g) \alpha( v) + gY(\alpha( v)) \\
			= \ &  \mu([X,Y]) + gh \alpha[v,w] \\
			& - X(\mu(Y))  - hX(f(\alpha(w))) \\
			& + Y(\mu(X))  + gY(f(\alpha(v))) \\
			= \ &  \mu([X,Y])- X(\mu(Y))+ Y(\mu(X))  \\
			&  + gh \alpha[v,w] - hg\rho(v)(\alpha(w)) + gh \rho(w)(\alpha(v)) \\
			= \ &  \mu([X,Y])- X(\mu(Y))+ Y(\mu(X))  \\
			& gh\left( \alpha[v,w] - \rho(v)(\alpha(w)) +  \rho(w)(\alpha(v)) \right)
		\end{split}\]
		We can clearly identify both the mixed structure on $\DR(B)$ and the mixed structure on $\ceu(\_L)$ that defines the differential on the tensor product. 
	\end{itemize}
	
\end{proof}

\begin{Cor} \label{cor:ce of relative tangent as reltive de rham complex}
	
	If $\_L = 0$ is the trivial Lie algebroid, then $\gmc{\ce}(\_L) \simeq A$ and we get an equivalence:  \[\gmc{\ceu}(f^!\_L) \simeq \DR(B) \otimes_{\DR(A)} A \simeq \DR(B/A)\] where the second equivalence comes from Definition \ref{def:de rham complex and relative in mod}. 
	
	Moreover, $f^! 0$ is equivalent to $\Ttr{B}{A}$ as Lie algebroids, and we have:
	\[\gmc{\ceu}(\Ttr{B}{A}) \simeq \DR(B/A)\]
	Moreover if $Y = \Spec(B)$ satisfies Assumptions \ref{ass:very good stack}, it coincides with the Lie algebroid $\G_D(A)$ described in Proposition \ref{prop:lie algebroid koszul duality context} (where $B$ and $A$ are exchanged).
\end{Cor}
\begin{proof}
	The pushout from Lemma \ref{lem:base change lie algebroid ce version} becomes exactly the pushout describing the relative de Rham algebra in Definition \ref{def:de rham complex and relative in mod}. The Lie algebroid structure on $\Ttr{B}{A}$ from Proposition \ref{prop:lie algebroid koszul duality context} is clearly the one whose \CE algebra is equivalent to $\DR(B/A)$.
\end{proof}

We want to show that the relative tangent Lie algebroid construction (Proposition \ref{prop:lie algebroid relative tangent}) is functorial in the choice of base. However, to do that we need to understand the infinitesimal quotient by $f^!\_L$. 

We would expect, by ``naturality'' to have a commutative diagram: 
\begin{equation}\label{eq:morphism exits}
	\begin{tikzcd}
		Y \arrow[r] \arrow[d] & X \arrow[d] \\
		\pQS{Y}{f^! \_L} \arrow[r, dashed, "\psi"] & \pQS{X}{\_L}
	\end{tikzcd}
\end{equation} 

The problem is that the existence of $\psi$ is a priori unclear. However, if it exists we have the following:

\begin{Lem}
	\label{lem:infinitesimal quotient pullback Lie algebroid}
	Take $f: Y \to X$ a map of finitely presented stacks satisfying Assumptions \ref{ass:very good stack}. Take $\_L$ a  Lie algebroid over $X$ and suppose that the morphism $\psi$ in Diagram \eqref{eq:morphism exits} exists, then there is an equivalence of Lie algebroids: 
	\[ f^! \_L \simeq \Ttrl{Y}{\pQS{X}{\_L}} \]
	
	In particular, this implies that we have an equivalence of formal thickenings of $Y$: 
	\[ \pQS{Y}{f^!\_L} \simeq \comp{\left(\pQS{X}{\_L}\right)_Y} \]
\end{Lem}
\begin{proof}
	We have a natural morphism: \[\begin{split}
		\phi : \pQS{Y}{f^!\_L} \to & \pQS{X}{\_L} \times_{\left(\pQS{X}{\_L}\right)_{\tx{DR}}} \left(\pQS{Y}{f^!\_L}\right)_{\tx{DR}} \\
		&\simeq \pQS{X}{\_L} \times_{\left(\pQS{X}{\_L}\right)_{\tx{DR}}} Y_{\tx{DR}}\\
		& \simeq \comp{\left(\pQS{X}{\_L}\right)_Y}
	\end{split} \]
	
	$\phi$ is clearly a nil-equivalence between formal thickening of $Y$ in pre-stacks. It induces a morphism of Lie algebroids: 
	\[ \Ttrl{Y}{\pQS{Y}{f^!\_L}}\simeq f^!\_L  \to \Ttrl{Y}{\comp{\left(\pQS{X}{\_L}\right)_Y}}\] 
	
	We will show that the underlying module of this morphism in an equivalence.
	Consider the following commutative diagram: \[ \begin{tikzcd}
		f^!\_L \arrow[r] \arrow[d] & \Tt_Y \arrow[d] & \\
		f^*\_L \arrow[r] & f^*\Tt_{X}  \arrow[r]& f^* h^* \Tt_{\pQS{X}{\_L}} 
	\end{tikzcd}\]
	
	The lower sequence is fibered thanks to Lemma \ref{lem:lie algebroid and relative tangent maurer cartan functor}, Lemma \ref{lem:tangent formal spectrum} and Definition \ref{def:infinitesimal quotient stack for eventually coconnective A}. Moreover the square is a pullback. Therefore we have an equivalence of quasi-coherent sheaves on $Y$ (over $\Tt_Y$):
	\[ f^!\_L \overset{\sim}{\to}\tx{fiber} \left( \Tt_Y \to f^*h^* \Tt_{\pQS{X}{\_L}}  \right)  \simeq \Ttr{Y}{\pQS{X}{\_L}}  \]
\end{proof}

\begin{RQ}
	The Lie algebroid $\Ttrl{Y}{\comp{\left(\pQS{X}{\_L}\right)_Y}}$ as all the properties we want: \begin{enumerate}
		\item For any morphism of Lie algebroid $\_L' \to \_L$ over different bases, the underlying morphism of $A$-modules factors as follows: \[ \_L' \to \Ttrl{Y}{\comp{\left(\pQS{X}{\_L}\right)_Y}} \to \_L \]
		\item We have a commutative diagram: 
		\[ \begin{tikzcd}
			Y \arrow[r] \arrow[d] & X \arrow[d] \\
			\comp{\left(\pQS{X}{\_L}\right)_Y} \arrow[r, dashed, "\psi"] & \pQS{X}{\_L}
		\end{tikzcd}\] 
	\end{enumerate}
	The only downside is that according is that it is unclear whether or not these maps in (1) are morphisms of Lie algebroids according to the previous definition.  
\end{RQ}

However, the definition of Lie algebroid morphism over different bases in based on the fact that $f^! \_L \to \_L$ is a morphism of Lie algebroid \emph{by definition}. We could as well say that the natural morphism: 
\[ \Ttrl{Y}{\comp{\left(\pQS{X}{\_L}\right)_Y}} \to \Ttrl{X}{\pQS{X}{\_L}} \]
is a morphism of Lie algebroids and \emph{define} morphisms of Lie algebroids over different bases as a morphism of Lie algebroids over $Y$: 
\[ \_L' \to \Ttrl{Y}{\comp{\left(\pQS{X}{\_L}\right)_Y}} \]

Essentially  we claim that $f^! \_L$ should be \emph{defined} as $ \Ttrl{Y}{\comp{\left(\pQS{X}{\_L}\right)_Y}}$. 

\begin{Def}\label{def:better pullback algebroid}
	From now on, $f^! \_L$ will be defined as the Lie algebroid: \[\Ttrl{Y}{\comp{\left(\pQS{X}{\_L}\right)_Y}}\]
	The results and definitions we have discussed so far with the previous definition still hold with the new one. Moreover whenever the morphism of Diagram \eqref{eq:morphism exits} exists, both definitions are equivalent thanks to Lemma \ref{lem:infinitesimal quotient pullback Lie algebroid}.  
\end{Def}

\begin{RQ}
	The previous and the new definition of $f^! \_L$ have in fact equivalent underlying $A$-modules (even if we don't know if their Lie algebroid structures coincide). 
\end{RQ}

\begin{RQ}
	In fact we expect that Lemma \ref{lem:infinitesimal quotient pullback Lie algebroid} should holds without assuming that the map of Diagram \eqref{eq:morphism exits} exists (which is in fact equivalent to proving that the map exists). However we do not have a proof of this.
\end{RQ}

\begin{Prop}\label{prop:functoriality relative tangent base change}
	Consider $X$, $X'$ satisfying Assumption \ref{ass:very good stack}, and $Y$, $Y'$ two formal stacks. Assume that $X$ and $X'$ are finitely presented. We consider the commutative diagram:
	\[ \begin{tikzcd}
		X \arrow[r] \arrow[d, "f"] & Y \arrow[d] \\
		X' \arrow[r, "h"] & Y'
	\end{tikzcd}\]
	Then, there is a morphism of Lie algebroids over $f$ (for the Lie algebroids structures obtained in Proposition \ref{prop:lie algebroid relative tangent}):
	\[ \Ttr{X}{Y} \to \Ttr{X'}{Y'} \]
	
	This holds if $Y$ and $Y'$ are formal \emph{pre-stacks} instead.
\end{Prop}
\begin{proof}	We just need to find a morphism of Lie algebroids over $X$: 
	\[ \Ttr{X}{Y} \to f^! \Ttr{X'}{Y'}  \]
	
	Using Definition \ref{def:better pullback algebroid}, this is equivalent to having a morphism of Lie algebroids over the same base: 
	\[  \Ttr{X}{Y} \to \Ttr{X}{Y'} \simeq \Ttr{X}{\comp{Y'}_X} \]

	Notice that we have a commutative diagram: 
	\[ \begin{tikzcd}
		X \arrow[d] \arrow[dr] & \\
		\comp{Y_X} \arrow[r] & \comp{Y'_X}
	\end{tikzcd}\]
	
	By naturality of the construction in Proposition \ref{prop:lie algebroid relative tangent} of Lie algebroid structures on the relative tangent (over the same base), we get a morphism of Lie algebroids: 
	\[ \Ttr{X}{Y} \to \Ttr{X}{Y'} \]
\end{proof}

We also have some naturality properties of the base change construction.

\begin{Lem} \
	\label{lem:naturality pullback Lie algebroid}
	Consider $f: Y \to X$ and $\_L$, $\_L'$ two Lie algebroids over $X$ with a morphism $\psi : \_L \to \_L'$ of Lie algebroids and $f: Y \to X$. Then there is a morphism of Lie algebroids: \[f^! \psi : f^! \_L \to f^! \_L'\] such that the following commutative diagram of Lie algebroids (over different bases) commute: 	
	\[ \begin{tikzcd}
		f^! \_L \arrow[r] \arrow[d] & \_L \arrow[d] \\
		f^! \_L' \arrow[r] & \_L'
	\end{tikzcd}\] 
\end{Lem}
\begin{proof}
	Since $f^!$ amounts to taking an actual pullback (over $TX$), we get a natural linear morphism $f^!\psi : f^! \_L \to f^! \_L'$ of modules fitting in the diagram. We only have to show that this preserves the Lie bracket and is compatible with the anchor. 
	
	The compatibility with the anchor is a consequence of the definition of $f^!$ via a pullback along the map $\Tt_Y \to f^*\Tt_X$. The preservation of the Lie bracket is a consequence of the compatibility of the anchors and the fact that $\psi$ preserves the Lie brackets (looking at the formula of Definition \ref{def:base change lie algebroid}).
\end{proof}

\begin{Lem}
	\label{lem:associativity pullback lie algebroid}
	
	Just like an ordinary base change, if we have $g: Z \to Y$ and $f : Y \to X$ where $X$, $Y$ and $Z$ are all finitely presented, then there is an equivalence of Lie algebroids over $Z$: \[g^! f^! \_L \overset{\sim}{\to} (f \circ g)^! \_L\]
\end{Lem}
\begin{proof}
	Since $f^!$ and $g^!$ amounts to taking a pullback, we have an equivalence of linear stacks: 
	\[ g^! f^! \_L \overset{\sim}{\to} (f \circ g)^! \_L \]
	
	Then using Lemma \ref{lem:base change lie algebroid ce version}, the \CE algebras are given by pushouts and therefore we get a natural morphisms between the associated \CE algebras.  Then we can use Lemma \ref{lem:map of ce algebras induce maps of lie algbroids} to show that this is a map of Lie algebroids. Since the underlying map of modules is an equivalence, which defines an equivalence of Lie algebroids 
\end{proof}

The second type of operation we are interested in is the fiber product of Lie algebroids. We start with the fiber product over the same base.

\begin{Lem}
	\label{lem:fiber product algebroids on the same base}
	
	The pullback of Lie algebroids exists\footnote{Because homotopy limits exist in a semi-model category as explained in \cite[Remark 2.13]{Nu19b}.}. Taking the underlying modules is a right adjoint therefore the underlying module of the pullback is a pullback of the underlying module. 
\end{Lem}

Combining pullbacks of Lie algebroids and the base change of Lie algebroids, we can make sense of the pullback of Lie algebroids over different bases. 

\begin{Prop}
	\label{prop:pullback lie algebroid}
	We can take pullbacks of Lie algebroids over different bases. That is, if $\_L_X$, $\_L_Y $ and $\_L_Z$ are Lie algebroids over finitely presented bases $X$, $Y$ and $Y$ respectively, and if $\phi_f : \_L_X \to \_L_Z$, $\phi_g : \_L_Y \to \_L_Z$ are morphisms of Lie algebroids over $f: X \to Z$ and $g: Y \to Z$ respectively. Then there is a Lie algebroid structure on $\_L_X \times_{\_L_Z} \_L_Y$ over $X \times_Z Y$ and there are Lie algebroid morphisms making the following square a pullback square in the category of Lie algebroids over different bases:
	\[ \begin{tikzcd}
		\_L_X \times_{\_L_Z} \_L_Y \arrow[r] \arrow[d] & \_L_Y \arrow[d] \\
		\_L_X \arrow[r] & \_L_Z
	\end{tikzcd}\]
\end{Prop}
\begin{proof}
	This pullback is defined as the pullback of the Lie algebroids over the base $X \times_Z Y$ (using Lemma \ref{lem:fiber product algebroids on the same base}) given by the base change of $\_L_X$, $\_L_Y$ and $\_L_Z$ to $X \times_Z Y$ (using Definition \ref{def:base change lie algebroid} and Lemma \ref{lem:naturality pullback Lie algebroid}). More precisely, consider the pullback diagram:
	
	\[ \begin{tikzcd}
		X \times_Z Y \arrow[r, "v"]\arrow[d, "u"] & X \arrow[d, "f"] \\
		Y \arrow[r, "g"] & Z
	\end{tikzcd}\]
	
	By definition a morphism $\_L_X \to \_L_Z$ is a Lie algebroid morphism $\_L_X \to f^!\_L_Z$ (idem for $\_L_Y \to \_L_Z$). 
	Using Lemma \ref{lem:naturality pullback Lie algebroid} we get two morphisms of Lie algebroids: 
	\[ v^! \_L_X \to v^! f^! \_L_Z \qquad u^! \_L_Y \to u^! g^! \_L_Z \]
	
	Since $u^! g^! = (g\circ u)^! = (f \circ v)^! = v^! f^!$, we get a pullback diagram defining the pullback of Lie algebroids (where our notation omits the base changes in the fiber product in order to keep it simple).
	\[ \begin{tikzcd}
		\_L_X \times_{\_L_Z} \_L_Y \arrow[r] \arrow[d] & v^! \_L_X \arrow[d] \\
		u^! \_L_Y \arrow[r] & v^! f^! \_L_Z \simeq u^! g^! \_L_Z 
	\end{tikzcd}
	\] 
	
	This pullback is indeed a Lie algebroid on $X \times_Z Y$ and any cone of Lie algebroids other it gives a cone over the previous diagram of Lie algebroids over the base $X \times_Z Y$. Therefore it is a fiber product thanks to Lemma \ref{lem:fiber product algebroids on the same base}. 
\end{proof}

\begin{RQ}\label{rq:pullback and base change of lie infty algebroids}
	The notion of base change and pullback of $\_L_\infty$-algebroids along morphisms (but not $\infty$-morphisms) also makes sense. Essentially, adding higher brackets ($n\geq 3$) to the constructions add no difficulty as they are all $A$-linear. In particular, there is a \CE \emph{weak} graded mixed algebra functor, $\gmch{\ce}$, characterizing the $\_L_\infty$-algebroid structure (in the sens of a generalization of Lemma \ref{lem:ce underived functor different bases}).   
	
	The easiest way to understand pullbacks and base changes in that setting is to use the tensor product in the category of weak graded mixed algebras. The underlying graded spaces are unchanged compared to the Lie algebroid part, but the tensor product in weak graded mixed algebras automatically defines the higher brackets. 
\end{RQ}

\subsection{Infinitesimal Action of Lie Algebroids}
\label{sec:infinitesimal-action-of-lie-algebroids}\  

\medskip

The goal of this section is to make sense of the notion of the infinitesimal action (up to homotopy) of a Lie algebroid. \\

We start in Section \ref{sec:representation-and-action-of-lie-algebroids} by discussing the special case of ``linear actions'' up to homotopy, which are given by \emph{representations up to homotopy} of Lie algebroids. We then explain that these representations up to homotopy are equivalent to $\_L$-connections that are \emph{flat up to homotopy}. We also describe relevant examples (e.g. the adjoint and coadjoint representations) and properties of such representations. 

Then we define in Section \ref{sec:action-of-lie-algebroids} the general notion of infinitesimal action of a Lie algebroid, and prove that representations up to homotopy are, under some technical assumption, actions up to homotopy.

\subsubsection{Representation up to homotopy}\ \label{sec:representation-and-action-of-lie-algebroids}

\medskip

Our interest in representation up to homotopy of Lie algebroids lies in the fact that they are ``linear'' infinitesimal actions up to homotopy of Lie algebroids (Theorem \ref{th:representation up to homotopy define action of lie algebroid}) and provide useful examples of actions up to homotopy such as the adjoint and coadjoint infinitesimal actions (Example \ref{ex:adjoint and coadjoint represention}). \\

Given a vector space $V$, a linear action of a group $G$ on $V$ is the data of a representation of $k[G]$ on $V$, or in other words, the structure of a $k[G]$-module on $V$. Similarly, a representation up to homotopy will defined as a module over $\gmch{\ceu(\_L)}$. \\  

This section is mostly a rephrasing and a generalization of \cite{AC12}, adapted to our context of derived algebraic geometry. In that section we will only consider Lie algebroids over affine bases of almost finite presentation. 

\begin{Def} 
	\label{def:representation up to homotopy}
	
	A \defi{representation up to homotopy} of a perfect Lie algebroid $\_L$ over $X$ is the data of a weak graded mixed module\footnote{The category of weak graded mixed $\gmch{\ceu}(\_L)$-modules is the category of modules according to Definition \ref{def:module over a commuative monoid} over $\gmch{\ceu}(\_L) \in \gmch{\cdga}$ in the good model category $\_M = \gmch{\Mod_k}$.}: \[M \in \gmch{\Mod}_{\gmch{\ceu}(\_L)}\]
	such that it underlying graded module $\gr{M} \in \gr{\Mod}_{\gr{\ceu}(\_L)}$ is \emph{isomorphic}\footnote{A more homotopic version would be to require only a weak equivalence. For our use, this stricter notion will be enough.} to: 
	\[ \gr{M} \cong \gr{\ceu}(\_L) \otimes_A  \_E \]
	
	with $\_E \in \gr{\QC}(X)$ a graded quasi-coherent sheaf on $X$. We call this \defi{a representation up to homotopy of $\_L$ on $\_E$}. 
	We denote the category of representations up to homotopy of $\_L$ by $\bf{Rep}_\_L$ and the category of representations up to homotopy of $\_L$ on $\_E \in \gr{\QC}(X)$ by $\bf{Rep}_\_L(\_E)$. 
\end{Def}

\begin{RQ}\label{rq:description rep un to homotopy}
	In other words, a representation up to homotopy is the data of a weak graded mixed structure on $\gr{\ceu}(\_L)\otimes_A \_E$. This can be viewed, after taking the realization (Remark \ref{rq:realization for weak graded mixed complexes}), as the structure of $\ceu(\_L)$-module on $\ceu^\sharp(\_L)\otimes_A \_E$. In other words, it is a  differential $D$ on $\ceu^\sharp(\_L)\otimes_A \_E$ such that for all $v \in \ceu(\_L)$ and $e \in \_E$:   \[D(v\otimes e) = \delta_{\ceu}(v)\otimes e + v\otimes D e\] 
\end{RQ}

\begin{RQ}\label{rq:pullback representation}
	The natural augmentation from the \CE weak graded mixed complex to $A$ (where $X= \Spec(A)$) induces a functor: 
	\[ \gmch{\Mod}_{\gmch{\ceu}(\_L)} \to \gmch{\Mod}_A \overset{\rel{-}}{\to} \Mod_A  \]
	
	This is essentially the pullback functor along $\gmch{\ceu}(\_L) \to A$.
	This functor sends a representation up to homotopy $M$ to the underlying $A$-module $\_E$ on which it acts. 
\end{RQ}

It turns out that the notion of representation (up to homotopy) of a Lie algebroid $\_L$ coincide with the notion of $\_L$-connections that are flat up to homotopy.

\begin{Def} 
	\label{def:lie algebroid connection}
	Given a Lie algebroid $\_L$ over $X$, a \defi{$\_L$-connection} on a linear stack $E$ is the data of a $\_L$-covariant derivative, that is, a $k$-linear map: 
	\[ \nabla : \_L \otimes \_E \to \_E \]
	such that for all $f \in A$, $v \in \_L$ and $s \in \_E$:
	\[\nabla_{fv} = f \nabla_v\] 
	\[  \qquad \nabla_v (fs) = \rho(v)(f).s + f\nabla_v s\] 
\end{Def}

\begin{Cons}[{\cite[Definition 2.9]{AC12}}]
	\label{cons:basic connection}
	
	Consider $\_L$ a cofibrant Lie algebroid and $\nabla$ a connection\footnote{Recall that when $X$ is affine and $\_L$ is a cofibrant Lie algebroid (and therefore a cofibrant $A$-module thanks to Theorem \ref{th:adjunction quillen algebroid-module}), then Proposition \ref{prop:connection existence} shows that there exists such a connection.} on $\_L$. We can define the following $\_L$-connections: 
	\begin{itemize}
		\item Basic $\_L$-connection on $\_L$ defined by: 
		\[ \nabla_v^{\tx{bas}} w := \nabla_{\rho(w)}v + [v,w]\]
		\item Basic $\_L$-connection on $TX$ defined by:
		\[ \nabla_v^{\tx{bas}}X := \rho(\nabla_X v) + [\rho(v), X] \] 
	\end{itemize}
	
	Note that $\rho \circ \nabla^{\tx{bas}} = \nabla^{\tx{bas}} \circ \rho$ and therefore this defines a $\_L$-connection on the linear stack associated to the complex\footnote{Recall that if $V$ and $W$ are complexes, and $f: V \to W$ is a map of complexes, then we denote by $ V[1] \oplus^f W $ the complex whose underlying module is $V[1] \oplus W$ with differential given by the sum of the differential on $V$, the differential on $W$ and $f$ viewed as a map of degree $1$.} $\Tt_X \oplus^\rho \_L[1]$. 
	
	The basic curvature, \cite[Definition 2.10]{AC12}, is the curvature of that $\_L$-connection given by: 
	\[ R_\nabla^{\tx{bas}}(v,w)(X) := \nabla_X([v,w]) - [ \nabla_X v, w] - [v, \nabla_X w] - \nabla_{\nabla_w^{\tx{bas}} X} v - \nabla_{\nabla_v^{\tx{bas}} X} w \] 
\end{Cons}

The following proposition gives a more concrete description of what a representation up to homotopy is.  

\begin{Prop}\label{prop:representation up to homotopy}
	A representation up to homotopy on $\_E$ of a perfect Lie algebroid $\_L$ can be equivalently described as the following data:
	\begin{itemize}
		\item A $\_L$-connection $\nabla$ on $\_E$. 
		\item An element: 
		\[ \omega_2 \in  \Sym_A^2 \_L^\vee[-1] \otimes_A \bf{End}(\_E) [1]  \] 
		such that $d_E (\omega_2) + R_\nabla = 0$ where $R_\nabla$ is the curvature of $\nabla$.
		\item For each $i>2$, we have: 
		\[ \omega_i \in \Sym_A^i \_L^\vee[-1] \otimes_A \bf{End}(\_E) [i-1]   \]
		such that:
		\[ d_E (\omega_i) + d_\nabla \omega_{i-1} + \omega_{2} \circ \omega_{i-1} + \cdots = 0 \]
	\end{itemize}
	
	We write $D = d_E + d_\nabla + \omega_2 + \omega_3 + \cdots $. 
\end{Prop}
\begin{proof}
	Consider the weak graded mixed structure on $\gr{\ceu}(\_L)\otimes_A \_E$ defining the representation up to homotopy. Then the collection of $d_E$ together with the mixed differentials $\epsilon_p$ for $p\geq 1$ are maps:
	\[ \_E \to \Sym_A^p \_L^\vee[-1] \otimes_A \_E[p-1]\]
	
	For $p=1$, we have $\_E \to \_L^\vee \otimes_A \_E[-1]$ that defines an $\_L$-connection on $E$ (thanks to the Leibniz rule satisfied by $\epsilon_1$). Then, for $p \geq 2$ we rename the elements obtained $\omega_i$.\\
	
	If we decompose the equations defining a weak graded mixed structure (Definition \ref{def:weak graded mixed complex}) we get:
	
	\begin{itemize}
		\item For $p=0$ we get $d_E^2 = 0$. 
		\item For $p=1$ we have $d_E \circ d_\nabla + d_\nabla \circ d_E =  0$.
		\item For $p=2$ we have $[d_E, \omega_2] + d_\nabla^2 = 0$ and $d_\nabla^2$ is the curvature of $\nabla$.
		\item The other equations are straightforward. In fact all the equations can be written in the following compact form: 
		\[ \sum_{p+q =i} \omega_p \circ \omega_q =0\] 
		with the convention $\omega_0 = d_E$ and $\omega_1 = d_\nabla$.  
	\end{itemize}
	
	Conversely, given the data of a connection and $\omega_i$ as in the statement, we want to reconstruct a weak mixed structure on $\gr{\ceu}(\_L) \otimes \_E$. Again each $\omega_p$ gives a map: 
	\[ \_E \to \Sym_A^p \_L^\vee[-1] \otimes_A \_E[p-1]\]
	and the equations:
	\[ \sum_{p+q =i} \omega_p \circ \omega_q =0\] 
	are exactly the equations of Definition \ref{def:weak graded mixed complex} describing a weak graded mixed structures. 
\end{proof}

\begin{RQ}
	A strict representation is a representation up to homotopy such that the weak graded mixed structure is in fact a graded mixed structure on the nose, in other words, $\epsilon_{\geq 2} := \omega_{\geq 2} =0$. Therefore it is exactly given by a $A$-module $\_E$ together with \emph{flat} $\_L$-connection on $\_E$. 
	
	In general the terms $\omega_{\geq 2}$ are closure terms making the connection flat ``up to homotopy'' (see \cite[Section 3.1]{Cr04}).
\end{RQ}

\begin{Ex}[{Variation of \cite[Example 3.8]{AC12}}] \label{ex:representation 2 term complex} 
	
	Consider $\_E = \_F \oplus \_F[-1] \in \QC(X)$ with $X := \Spec(A)$, where $\_E$ is equipped with the differential induced by the identity $\partial : \_F \to \_F[-1]$. Then, given a $\_L$-connection on $\_F$, we can define a representation up to homotopy on $\_E$ given by the following differential on $\ceu(\_L)\otimes_A \_E$:
	\[ D := d_F + \partial + R^\nabla \]
	
	Indeed $D^2=0$ amount to the equations:
	\begin{itemize}
		\item $d_F^2 =0$ and $d_F$ commutes with everything because $\partial$ is the identity and the connection respects $d_F$. 
		\item $\partial^2 =0$, and we have:
		\[ \partial \circ R^\nabla = R^\nabla \circ \partial \]
		
		\item Because the connection preserve the differential on $\_F$, we also have: 
		\[ d_F \circ R^\nabla = R^\nabla \circ d_F \]
		
		\item $(R^\nabla)^2 = 0$.
	\end{itemize}
\end{Ex}

\begin{Ex}\label{ex:fiber of representation up to homotopy}
	
	Take $\tilde{E} \in \bf{Rep}_L(\_E)$ and $\tilde{F} \in \bf{Rep}_L(\_F)$ and a morphism $f: \tilde{E} \to \tilde{F}$ of weak graded mixed modules inducing a morphism $f: E \to F$. Then the fiber: 
	\[ \tilde{G} := \tx{fiber}(f: \tilde{E} \to \tilde{F})\] 
	is a representation up to homotopy on the fiber: \[\_G := \tx{fiber}(f:\_E \to \_F) \simeq \_E \oplus^f \_F[-1]\] 
\end{Ex}

\begin{Ex}
	\label{ex:adjoint and coadjoint represention}
	
	Recall from Section \ref{sec:cotangent-complex-in-m} that we can construct the cotangent complex in any good model category, in particular in the category of weak graded mixed complexes (see Section \ref{sec:model-categorical-setup}). This gives us the following:
	\[ \gmch{\Tt}_{\gmch{\ceu}(\_L)} \in \gmch{\Mod_{\gmch{\ceu}(\_L)}} \qquad \gmch{\Ll}_{\gmch{\ceu}(\_L)} \in \gmch{\Mod}_{\gmch{\ceu}(\_L)}\]

	The notions of adjoint and coadjoint representations is described in \cite[Section 3.2]{AC12}. We expect that the graded mixed tangent and cotangent complexes should recover these representations up to homotopy. 
	
	Indeed, in a somewhat similar manner to when we use connections to compute the tangent complex of linear stack and semi-linear representations of a derived stack, we expect that a graded mixed version of connection could also us to compute the graded mixed complexes and make them representations up to homotopy on: 
	\[p^*\gr{\Tt}_{\gr{\ceu(\_L)}} \qquad p^*\gr{\Ll}_{\gr{\ceu(\_L)}} \] 
\end{Ex}

\begin{Cons}[{Variation of \cite[Example 4.1]{AC12}}]
	\label{cons:dual representation up to homotopy}
	
	Given a representation up to homotopy $\_R \in \gmch{\Mod}_{\gmch{\ceu}(\_L)}$ on a $A$-module $\_E$, we can obtain a representation up to homotopy on it dual, $\_E^\vee$, by considering the strict dual\footnote{In other words the enriched Hom functor defining the dual is not derived.} of $\_R$ enriched in weak graded mixed modules: \[\_R^\vee := \gmch{\Hom_{\gmch{\ceu}(\_L)}}(\_R, \gmch{\ceu}(\_L)) \in \gmch{\Mod}_{\gmch{\ceu}(\_L)}\]
	
	We can check that $\gr{(\_R^\vee)} \simeq (\gr{\_R})^\vee \simeq \gr{\ceu}(\_L)\otimes_A \_E^\vee$. 
\end{Cons}

\subsubsection{Action of Lie algebroids}\ \label{sec:action-of-lie-algebroids}

\medskip

In this section, we are only going to consider \emph{cofibrant} Lie algebroids over affine bases of almost finite presentation. \\

The typical example of an action is the action of $\_L$ on its base $X$, with the idea in mind that the associated infinitesimal quotient stack associated to $\_L$ is the infinitesimal quotient of $X$ by the action of $\_L$. \\

Taking the motivating example of a Lie algebra $\G_g$ viewed as a Lie algebroid over the point $\star$, we see that $\G_g$ act trivially on the point. However, we want to make $\G_g$ act on more geometric spaces that just the point. Classically, an action of $\G_g$ on $Y$ is an \emph{infinitesimal action}\footnote{In this context, it is just a map of Lie algebras.} $\G_g \to \Tt_Y$ which always gives (thanks to the second point of Example \ref{ex:Lie algebroid}) the Lie algebroid $\_O_Y \otimes \G_g$ on $Y$ called the \emph{action Lie algebroid}. What we take from this example is that this action of $\G_g$ on $Y$ is exactly the structure of a Lie algebroid on $f^* \G_g$ over $Y$ where $f$ is the canonical map $Y \to \star$. Keeping this idea and extending it to Lie algebroids over $X$ instead of $\star$, we get the following definition:     

\begin{Def}
	\label{def:action of a Lie (infty) alebroid (up to homotopy)}
	An \defi{infinitesimal action} of a Lie algebroid $\_L$ over $X$ on $f: Y \to X$ is a structure of Lie algebroid on $f^*\_L$ over $Y$ such that the following commutative diagram: 
	\[ \begin{tikzcd}
		f^*L \arrow[r] \arrow[d] & L \arrow[d] \\
		Y \arrow[r, "f"] & X
	\end{tikzcd}\]  
	defines a morphism of Lie algebroids over different bases (Definition \ref{def:lie algebroid morphism different base}). 
	
	This infinitesimal action is \defi{up to homotopy} if $f^*\_L$ has a $\_L_\infty$-structure and the morphism is an $\infty$-morphism. 
	
	We denote the category of infinitesimal actions of $\_L$ up to homotopy by $\bf{Act}_\_L$.  
\end{Def}

\begin{RQ}
	To speak about infinitesimal quotients by such an action, we would need $Y$ to satisfy Assumptions \ref{ass:very good stack} so that it fits into the framework of Section \ref{sec:quotient-stack-of-a-lie-algebroid}. In concrete examples this will not always be the case. However, we will see in Section \ref{sec:derivation-and-integration-of-lie-algebroids} how to give a generalized notion of infinitesimal quotient.  
\end{RQ}

\begin{Th}
	\label{th:representation up to homotopy define action of lie algebroid}
	Let $X$ be affine of almost finite presentation and $\_L$ a  perfect Lie algebroid over $X$.	We have a strict functor:
	\[ 
	\psi \ : \ \bf{Rep}_\_L^{\tx{afp}, \geq 0} \to \bf{Act}_\_L  
	\]
	
	that sends a representation up to homotopy on $\_E \in \QC^{\tx{afp}, \geq 0}(X)$, a non-negatively graded quasi-coherent sheaf on $X$ of almost finite presentation, to an infinitesimal action on the linear stack $E$ along the projection $\pi_E : E \to X$. 
\end{Th}
\begin{proof}
	Given a representation up to homotopy given by: 
	\[M \in \gmch{\Mod_{\gmch{\ceu}(\_L)}}\]
	such that $\gr{M}\cong \gr{\ceu}(\_L)\otimes_A \_E$ with $\_E$ non-negatively graded of almost finite presentation.
	We consider the associated free weak graded mixed $\gmch{\ceu}(\_L)$-algebra, $\Sym_{\gmch{\ceu}(\_L)} M^\vee$, whose underlying graded algebra is isomorphic to:
	\[ \Sym_{\gr{\ceu}(\_L)} \gr{\ceu}(\_L)\otimes_A \_E^\vee \cong \gr{\Sym}_{\Sym_A \_E^\vee} \pi_E^* \_L^\vee[-1]  \]
	with $\pi_E: E \to X$. 
	
	This defines a $\_L_\infty$-structure on $\pi_E^* \_L$ over $E$ since a representation up to homotopy can only increase the arity in $\_L^\vee[-1]$ and therefore the right hand side is in the essential image of the \CE functor. This gives $\pi_E^* \_L$ a structure of Lie algebroid over $E$. Notice that because $\_E$ is almost presented and non-negatively graded, $E$ is affine of finite presentation.\\
	
	Moreover, there is a map (coming from the unit of the free-forget adjunction) $\gmch{\ceu}(\_L) \to  \Sym_{\gmch{\ceu}(\_L)} M^\vee$ commuting with the natural inclusion $A \to \Sym_A \_E^\vee$. Therefore, thanks to Lemma \ref{lem:ce underived functor different bases}, we get a  morphism of Lie algebroid $f^*\_L \to \_L$ that defines the desired action. \\ 
	
	This construction is natural in the \CE algebras and since the strict \CE functor is fully faithful on perfect Lie algebroids (thanks to Lemma \ref{lem:ce underived functor different bases}), this construction of the action is functorial.
\end{proof}

\begin{RQ}
	The reason we need to restrict to representations up to homotopy on \emph{coconnective} almost finitely presented modules is that this is the necessary condition to ensure that $E$ is affine of almost finite presentation which is the necessary context to speak about Lie algebroids.
\end{RQ}

\subsection{Homotopy Transfer Theorem} \label{sec:infty-morphism-and-homotopy-transfer-theorem}\

\medskip

This section aims at proving a version of the homotopy transfer theorem for $\_L_\infty$-algebroids. In the context of differential geometry, a transfer theorem for Lie algebroids is proven in \cite[Section 2.3]{PS20}. We recall the proof of this result in the context of derived algebraic geometry. In this section all the Lie algebroids are strictly perfect.

\begin{Th} \label{th:HHT lie algebroid affine case}
	Let $f: \_L \to \_L'$ be a linear quasi-isomorphism of strictly perfect $A$-modules. Then for any $\_L_\infty$ algebroid structure on $\_L'$ over $X = \Spec(A)$, there is a $\_L_\infty$-algebroid structure on $\_L$ over $X$ and an extension of $f$ to an $\infty$-quasi-isomorphism\footnote{An $\infty$-quasi-isomorphism is an $\infty$-morphism whose linear component $f_1: \_L \to \_L'$ is a quasi-isomorphism.}: \[f_{\infty} : \_L \to \_L'\]
\end{Th}

The rest of the section is going to be devoted to the proof of this theorem. First, we need to define the notion of deformation retract.

\begin{Def} \label{def:deformation retract}
	A \defi{deformation retract} is a pair of map in $\Mod_A$: 
	
	\[\begin{tikzcd}
		\arrow[loop left, "h"] \_L \arrow[r, "p", shift left] & \arrow[l, shift left, "i"] \_M
	\end{tikzcd}\]
	
	such that $pi =\tx{id}$, $p$ is a quasi-isomorphism and $[d,h] = \tx{id} - ip$. 
	
	This is a \defi{special deformation retract} if moreover $ph= hi = h^2 = 0$. 
\end{Def} 

\begin{Lem}
	For any linear quasi-isomorphism $f : \_L \to \_L'$ there is zig-zag of special deformation retracts: 
	
	\[ \begin{tikzcd}
		\_L \arrow[r, shift left, "i"] & \arrow[l, shift left, "p"] \_M \arrow[r, shift right, "p'"'] & \arrow[l, shift right, "i'"'] \_L' 
	\end{tikzcd}\]
	
	with $h, h': \_M \to \_M$ the homotopies of degree $-1$ of the deformation retracts and such that $p' \circ i = f$.  
\end{Lem}
\begin{proof}
	$f$ comes from a quasi-isomorphism of $A$-modules, $\_L \to \_L'$, and therefore factor as: 
	
	\[ \begin{tikzcd}
		\_L \arrow[r, "\tx{triv. cof.}"] & \_M \arrow[r, "\tx{triv. fib.}"] & \_L'
	\end{tikzcd} \]
	
	Moreover we can chose $\_M$ to be the mapping cylinder of $f$ given by:
	\[\_L \oplus \_L[1] \oplus \_L'\]
	with differential given by $\delta (a_1, a_2, a' ) = (da_1 + a_2, -d a_2, da' - f(a_2))$.
	
	We get a special deformation retract: 
	
	\[\begin{tikzcd}
		\arrow[loop left, "h'"] \_L \oplus \_L[1] \oplus \_L' \arrow[r, "p'", shift left] & \arrow[l, shift left, "i'"] \_L'
	\end{tikzcd}\]
	where $i'$ is the inclusion, $p'(a_1, a_2, a') = f(a_1) + a'$ and $h'(a_1, a_2, a') = (0, a_1, 0)$.  \\
	
	We have another special deformation retract:
	
	\[\begin{tikzcd}
		\arrow[loop left, "h"] \_L \arrow[r, "p", shift left] & \arrow[l, shift left, "i"]  \_L \oplus \_L[1] \oplus \_L'
	\end{tikzcd}\]
	
	Indeed, we find a retract of $\_L \to   \_L \oplus \_L[1] \oplus \_L'$ by having a lift: 
	\[ \begin{tikzcd}
		\_L \arrow[d, "\tx{triv. cof.}"' ] \arrow[r, "\id"] & \_L \arrow[d, "\tx{fib.}"] \\
		\_L \oplus \_L[1] \oplus \_L' \arrow[r] \arrow[ur, dashed, "p"] & 0
	\end{tikzcd} \]
	
	This implies that $p$ is a homotopy inverse for $i$, since $i$ is a quasi-isomorphism and $p \circ i = \id$. Therefore $i\circ p$ is homotopic to the identity with homotopy $h$.\\ 
	
	These are in fact only deformation retracts but this is a classical result (see \cite[Remark 2.3]{Cr04}) that any deformation retract can be modified to a special deformation retract. 
\end{proof}

We are now going to deal with transfer on the left and on the right along a special deformation retract: 

\[ \begin{tikzcd}
	\_L \arrow[r, shift left, "i"]  & \arrow[l, shift left, "p"]\arrow[loop right, "h"] \_M 
\end{tikzcd}\]

\begin{Lem}\label{lem:transfer on the right} 
	
	If $\_L$ is a $\_L_\infty$-algebroid, then there exists a $\_L_\infty$-algebroid structure on $\_M$ such that both $p$ and $i$ are $\_L_\infty$-morphisms. 
\end{Lem}
\begin{proof}
	Using the deformation retract we get a splitting (an isomorphism):
	\[ \_M \cong \_L \oplus \_N \]
	
	with $\_N = \ker(p)$ a contractible complex so that there is an isomorphism:
	\[\_N \cong \tx{cone}(\id : \_K \to \_K)\]
	Indeed, first notice that $h$ restricted to $\_N$ is valued in $\_N$ and restricted to $\_N$ we have $[d,h] = \id$ (since $i \circ p =0$ on $\_N$). Then if we take $\_K = \ker(d)$, we can show that $\_N = \_K \oplus \_K'$ where $\_K'$ is the cokernel of $d$. It turns out that $h : \_K' \to \_K[1]$ is an isomorphism. With the condition  $[d,h] = \id$, we get an isomorphism $\_N \cong \_K \oplus \_K[1]$ together with the differential $\_K[1] \to \_K$, which is exactly the cone of $\id : \_K \to \_K$. \\
	
	Then picking a connection on $\_K$ (which exists since $\_K$ is projective as a direct summand of $\_M$ which is projective), we can define a representation up to homotopy on $\_N$ thanks to Example \ref{ex:representation 2 term complex}. This induces an action of $\_L$ on $\_N$ (using Theorem \ref{th:representation up to homotopy define action of lie algebroid}) which we will denote by $\tilde{\_N} \in \gmch{\Mod_{\gmch{\ceu}(\_L)}}$. Doing so, we can define a structure of Lie algebroid on $\_M$ by defining its \CE weak graded mixed algebra as the free weak graded mixed complex: 
	\[ \gmch{\ceu}(\_M) := \tx{Free}\left(\tilde{\_N}^\vee[-1]\right) \]
	
	This is a \CE algebra of $\_M$ since the underlying graded object of this free weak graded mixed algebra is: 
	\[ \gr{\Sym_A} \left(\_L^\vee[-1] \oplus \_N^\vee[-1]\right) \cong  \gr{\Sym_A} \left(\_M^\vee[-1]\right)  \]
	and therefore we get a $\_L_\infty$-structure on $\_M$.
	
	The natural projection and inclusion induced by $p$ and $i$ are compatible with the differentials and therefore induce strict morphisms of $\_L_\infty$-algebroids since the strict \CE functor is fully-faithful on perfect $\_L_\infty$-algebroids (thanks to a $\_L_\infty$ variation of Lemma \ref{lem:ce underived functor different bases}). 
\end{proof}

\begin{Lem}
	\label{lem:transfer on the left}
	
	If $\_M$ is an $\_L_\infty$-algebroid, then there exists an $\_L_\infty$-algebroid structure on $\_L$ and an extension of $i$ to an $\infty$-quasi-isomorphism $i_\infty$.
\end{Lem}
\begin{proof}
	We define an anchor on $\_L$ by the composite $\rho \circ i$. 
	
	To define the brackets, we use the homotopy transfer theorem for $\tx{Lie}_\infty$-algebras (see \cite[Theorem 10.3.2]{LV}). This gives us a sequence of $k$-linear maps:
	\[i_n : \_L^{\times n} \to \_M \] 
	
	These maps are in fact $A$-linear for thanks to the side condition $h^2 = 0$ and $h \circ i =0$. The binary bracket is defined by: 
	\[ [x,y]_\_L := p([i(x), i(y)]_\_M)  \]
	We can show that it satisfies the Leibniz rule using the equation $pi=\id$. 
\end{proof}

\begin{RQ}
	So far we talked about Lie algebroids, the objects given by Definition \ref{def:lie algebroid affine} and Lie algebroid structures on a module, defined via Remark \ref{rq:lie structure on a module}. The first type of objects are models of Lie algebroids but this notions is not invariant by weak equivalence (unlike the second notion). The homotopy transfer theorem ensures that if a module $\_L$ has a structure of Lie algebroid, then $\_L$ \emph{is} a $\_L_\infty$-algebroid such that the algebroid structures are equivalent via the transfer along $\infty$-equivalences between the underlying modules of the different models.  
\end{RQ}

\newpage
\section{Equivariant Symplectic Geometry}\ 
\label{sec:equivariant-symplectic-geometry}

In this section, we are going to develop equivariant and ``infinitesimally equivariant'' geometry, as well as derived symplectic geometry in that context. The study of (infinitesimal) equivariant geometry boils down to the study of the geometry of (infinitesimal) quotient stacks. The end goal is to make sense of a generalized notion of symplectic reduction for Segal groupoids actions and their infinitesimal version, Lie algebroids actions.  \\

We start by setting up the necessary derived equivariant geometry needed for Segal groupoids actions in Section \ref{sec:g-equivariant-geometry}, and for Lie algebroid actions in Section \ref{sec:l-equivariant-symplectic-geometry}. 

In both cases, we study the geometric properties of the (infinitesimal) quotients. In particular we are interested in taking pullbacks of such quotients and in computing their tangent complexes. \\

Moreover, we will discuss the heuristic saying that infinitesimal quotients should be viewed as an ``infinitesimal version'' of quotients by Segal groupoids. This requires a precise understanding of the relationship between Segal groupoids and Lie algebroids. We discuss the ``derivation'' and ``integration'' of Segal groupoids and Lie algebroids in Section \ref{sec:derivation-and-integration-of-lie-algebroids}. 

Furthermore, we also show that much of the constructions for Lie algebroids are but the ``infinitesimal versions'' of the same constructions for Segal groupoids. \\

Then, in Section \ref{sec:shifted-moment-maps-and-derived-symplectic-reduction} we recall the notion of moment maps and symplectic reduction for a group action following \cite{AC21}. We then generalize these results to Segal groupoid and Lie algebroid actions and show in each case that we can create new moment maps and symplectic reductions by a procedure of ``derived intersection of Lagrangians''. This produces interesting examples of $(-1)$-shifted moment maps on the derived critical locus, which are going to be at the heart of the BV constructions in Section \ref{sec:derived-perspective-of-the-bv-complex}. 

\subsection{$\_G$-Equivariant Symplectic Geometry} \label{sec:g-equivariant-geometry}\ 

\medskip

We are interested in the equivariant geometry with respect to the action of a Segal groupoid. As we will see, it amounts to studying the associated quotient. 

We will pay particular attention to the tangent and cotangent complexes of these quotient stacks, as well as their derived symplectic geometry. 

\subsubsection{$\_G$-equivariant maps and quotients}\
\label{sec:g-equivariant-maps-and-quotient}

\medskip
We start by defining more precisely the different notions of Segal groupoids we are going to use, and make sense of the notion of \emph{action} of a Segal groupoid.  

\begin{Def}[{\cite[Definition 1.3.1.6]{TV08}}] \label{def:groupoid in derived stacks}
	We say that a simplicial\footnote{Unlike the usual conventions, we denote the simplicial stack with a \emph{lower} bullet to have convenient notations later on.} derived stack $\_G^\bullet$ is a \defi{Segal groupoid over $X$} if the following conditions holds: 
	\begin{enumerate}
		\item $\_G^0 \simeq X$
		\item For all $n > 0$, there are weak equivalences: \[  \prod\limits_{0\leq i < n} \sigma_i : \_G^n \overset{\sim}{\to} \underbrace{\_G^1 \times_{X} \cdots \times_{X} \_G^1}_{\tx{n-times}} \]
		\item There is a weak equivalence: 
		\[ d_1 \times d_2 : \_G^2 \to \_G^1\times_{X}^{d_0, d_0} \_G^1\]
	\end{enumerate}
	
	Given a Segal groupoid $\_G^\bullet$ over $X$, we define the \defi{quotient stack} as the colimit of the simplicial diagram denoted by: \[\QS{X}{\_G}\]
	It comes with a natural projection that we will generically denote by: 
	\[ p: X \to \QS{X}{\_G} \]
	
	We denote by $s_n : \_G^n \to X$ the composition of all the ``source maps'' $d_0 : \_G^n \to \_G^{n-1}$ and $t_n: \_G^n \to X$ the composition of the ``target maps'' $d_{n-1}: \_G^n \to \_G^{n-1}$.  
\end{Def}

\begin{RQ}
	The first condition of a Segal groupoid $\_G^\bullet$ is essentially saying that it is a kind of $1$-category object with objects $X$ and morphisms $\_G^1$. In particular the composition can morally be obtained by ``inverting'' the the equivalence in: 
	\[ \_G^1 \times_X \_G^1 \overset{\sim}{\leftarrow} \_G^2 \overset{d_1}{\to} \_G^1 \] 
	
	The second condition says that any horn diagram (element in $\_G^1\times_{d_0, X, d_0} \_G^1$): 
	\[ \begin{tikzcd}
		x \arrow[r] \arrow[dr] & y \\
		& z
	\end{tikzcd}\]   
	can be filled uniquely (up to homotopy). This implies that all ``morphisms'' in $\_G^1$ are invertible. 
\end{RQ}

\begin{Def}\label{def:groupoid morphism different base}
	Let $\_G_X^\bullet$ be a Segal groupoid over $X$ and $\_G_Y^\bullet$ a Segal groupoid over $Y$. Then a morphism of Segal groupoids, denoted $\phi : \_G_X^\bullet \to \_G_Y^\bullet$ over a map $f: X \to Y$ is a morphism between the simplicial derived stacks $\_G_X^\bullet \to \_G_Y^\bullet$ such that it is given by $f$ on the $0$-simplices.
\end{Def}

\begin{Ex}
	The nerve of a groupoid: \[ \begin{tikzcd}
		\_G \arrow[r, shift left, "s"] \arrow[r, shift right, "t"'] & X
	\end{tikzcd}\]
	is a Segal groupoid and the quotient stack of a groupoid is by definition (see \ref{ex:derived stacks}) the quotient if its nerve. 
\end{Ex}

\begin{Ex}
	Given a map $f : X \to Z$, we can define a Segal groupoid over $X$, $C(f)$ given by the \v{C}ech nerve: 
	\[ C(f)^n  := \underbrace{X \times_Z \cdots \times_Z X}_{\tx{n-times}} \]
	with the simplicial structure obtained from the natural projections and the diagonals. Condition (1) is the associativity of the pullback and the second condition is also clear.   
\end{Ex}

\begin{Def}\label{def:smooth and formal groupoids}
	Let  $\_G^\bullet$ be a Segal groupoid over $X$. 
	\begin{itemize}
		\item The Segal groupoid is called \defi{smooth} (\cite[Definition 1.3.4.1]{TV08}) if each $\_G^n$ is an Artin stack and all the structure maps $\_G^n \to \_G^{n-1}$ are smooth (in the sens of \cite[Definition 1.3.3.1]{TV08}). In particular, from \cite[Proposition 1.3.4.2]{TV08}, the associated quotient stack is again Artin\footnote{Artin stacks can even be \emph{defined} as successive quotients of such groupoids.}.
		\item The groupoid is called formal if each $\_G^n$ is a formal stack and all structure morphisms are nil-equivalences. 
	\end{itemize}
\end{Def}

\begin{RQ}
	We can also consider \emph{Segal groupoids in pre-stacks}. The definitions are identical except for the fact that $\_G^\bullet$ is a simplicial pre-stack instead. The stackification-inclusion adjunction induces an adjunction between the categories of Segal groupoids in pre-stacks and Segal groupoids in stacks\footnote{The stackification of a Segal groupoid in pre-stacks is a groupoid in stacks because the stackification functor preserves small limits.}.
\end{RQ}

A class of examples of such groupoids come from the action of groups. By group we mean, a \emph{group object} in derived stacks. A group can act on a stack $X$ via a morphism $\rho: G \times X \to X$ satisfying the usual conditions for an action.  

\begin{Ex}\label{ex:group action}
	Take $G$ a group acting on a derived stacks $X$ with action $\rho : G \times X \to X$. Then $G \times X$ has a groupoid structure given by:
	
	\[ \begin{tikzcd}
		G \times X \arrow[r, shift left, "\rho"] \arrow[r, shift right, "\tx{pr}_X"'] & X
	\end{tikzcd}\]   
	
	whose quotient is $\QS{X}{G}$ (see Example \ref{ex:derived stacks}). In fact, such a groupoid structure is equivalent to the data of an action of $G$ on $X$. We will say that $G$ acts smoothly on $X$ if the corresponding groupoid is smooth in the sens of Definition \ref{def:smooth and formal groupoids} (and in particular, $G$ and $X$ are Artin). 
\end{Ex}

\begin{RQ}\label{rq:quotient formal groupoid}
	Let $\_G^\bullet$ be a formal Segal groupoid over $X$. Then we can show that the projection: \[p:X \to \QS{X}{\_G}\] 
	is a nil-equivalence. However, nothing ensures that the quotient is formal. Therefore this projection might \emph{not} be a formal thickening according to Definition \ref{def:formal thickening}.
\end{RQ}

We can obtain formal Segal groupoids out of a rather large class of groupoids. 

\begin{Cons}\label{cons:formal completion of a groupoid}
	Take $\_G^\bullet$ a Segal groupoid over $X$ such that each $\_G^n$ is formal (in particular both smooth groupoids and formal groupoids verify this condition). Then we can define the \defi{formal completion}, $\comp{\_G}^\bullet$, of this groupoid as the Segal groupoid given by formal completion at the unit: 
	\[ \comp{\_G}^n := \comp{\_G_X^n}\]
	
	Since taking the formal completion commutes with taking the pullbacks (Corollary \ref{cor:formal completion and pullacks}), this is again a Segal groupoid.
\end{Cons}

Using the fact that derived stacks form a model topos, they satisfy a version Giraud's theorem (Proposition \ref{prop:girauds theorem}). In particular quotients are ``effective'':

\begin{Lem}\label{lem:projection is an effective epimorphism}
	Given a Segal groupoid $\_G^n$ over $X$, the natural projection: \[ p: X \to \QS{X}{\_G}\] 
	is an effective epimorphism and we have equivalences for all $n \geq 1$:
	\[ \_G^{n} \simeq \underbrace{X \times_{\QS{X}{\_G}} \cdots \times_{\QS{X}{\_G}} X}_{\text{n-times}} \simeq C(p)^n \]
	
	This defines an equivalence of groupoids over $X$:
	\[ \_G^\bullet \simeq C(p)^\bullet \]
\end{Lem}

Equivariant geometry is the study of the geometry of quotient stacks and of \emph{equivariant maps} between derived stacks. 		

\begin{Def}\ \label{def:equivariant maps}
	Let $f: X \to Y$ be a morphism of derived stacks. Consider $\_G^\bullet$ a Segal groupoid over $X$. Then the structure\footnote{Being equivariant is a structure on $f$ and not a property. In a strict setting (e.g. strict quotients of schemes) the map $\eq{f}$ would necessarily be unique and equivariance would become a property.} of an \defi{invariant map with respect to $\_G^\bullet$} on $f$ is the data of a factorization through the projection $p$: 
	
	\[ \begin{tikzcd}
		X \arrow[r, "f"] \arrow[d, "p"] & Y \\
		\QS{X}{\_G} \arrow[ur, "\eq{f}"']
	\end{tikzcd} \]
	
	If now $\_H^\bullet$ is a Segal groupoid over $Y$, then the structure of an \defi{equivariant map with respect to $\_G^\bullet$ and $\_H^\bullet$} on $f$ is a structure of an invariant map on: 
	\[p\circ f: X \to \QS{Y}{\_H}\]
	In other words, this is the data of a morphism $\eq{f}$ fitting in the commutative diagram: 
	\[ \begin{tikzcd}
		X \arrow[r, "f"] \arrow[d] & Y \arrow[d] \\
		\QS{X}{\_G} \arrow[r, "\eq{f}"'] & \QS{Y}{\_H}
	\end{tikzcd}\]
	
	In what follows, we are going to say that $f$ \emph{is} equivariant when referring to the data of $\eq{f}$ giving $f$ the structure on an equivariant map.
\end{Def}

We also have constructions on groupoids similar to the ones in Section \ref{sec:pullback-and-base-change-of-lie-algebroids}:
\begin{Cons} \label{cons:base change and pullback groupoids}\
	
	\begin{itemize}
		\item Let $f: Y \to X$ be a map of derived stacks and $\_G^\bullet$ a Segal groupoid over $X$. Then there is a Segal groupoid $f^!\_G^\bullet$ over $Y$ called the base change groupoid defined, by analogy with Definition \ref{def:better pullback algebroid}, as the Segal groupoid: 
		\[ f^! \_G^\bullet := C(p\circ f)^\bullet \simeq \underbrace{Y \times_{\QS{X}{\_G}} \cdots \times_{\QS{X}{\_G}} Y}_{\bullet\text{-times}} \]
		\item Given $\_G^\bullet$, $\_G_1^\bullet$ and $\_G_2^\bullet$ groupoids over $X$ with morphisms of Segal groupoids $\_G_1^\bullet \to \_G^\bullet$ and $\_G_2^\bullet \to \_G^\bullet$, then the pullback $\_G_1^\bullet \times_{\_G^\bullet} \_G_2^\bullet$ is a groupoid over $X$. 
		\item Combining these two constructions, we can define the pullback of groupoids over different bases. This is a pullback because of Lemma \ref{lem:base change a morphism of groupoids}.				
	\end{itemize}
\end{Cons}

\begin{Lem}\label{lem:base change a morphism of groupoids}
	Let $\_H^\bullet$, $\_G^\bullet$ be groupoids over $Y$ and $X$ respectively and $\phi: \_H^\bullet \to \_G^\bullet$ a morphism of groupoids over $f: Y \to X$. Then $\phi$ factors through a unique morphism: 
	\[ \_H^\bullet \to f^! \_G^\bullet \] 
\end{Lem}
\begin{proof}
	First notice that there is an equivalence:
	\[C(p \circ f)^n := \underbrace{Y \times_{\QS{X}{\_G}} \cdots \times_{\QS{X}{\_G}} Y}_{\text{n-times}} \simeq \underbrace{\left(Y \times_X \_G^1 \times_X Y\right)\times_Y \cdots \times_Y \left(Y \times_X \_G^1 \times_X Y\right) }_{\tx{n-times}} \]
	
	We have a unique natural (i.e. commuting with the source and target maps) morphism: 
	\[ \_H^1 \overset{s \times \phi^1 \times t}{\longrightarrow} Y \times_X \_G^1 \times_X Y \]
	
	Which induces unique natural morphisms: 
	\[ \_H^n \overset{\sim}{\to} \underbrace{\_H_1 \times_Y \cdots \times_Y \_H^1}_{\tx{n-times}} \to C(p \circ f)^n \] 
	
	which in turn define a unique morphism of groupoids: 
	\[ \_H^\bullet \to f^! \_G^\bullet \]
\end{proof}

We can now turn to the notion of \emph{action} of a Segal groupoid. Similarly to the motivating example in the introduction of Section \ref{sec:action-of-lie-algebroids}, we can look at the case of an action of a group $G$. $G$ is naturally seen as a groupoid over the point. An action of $G$ on $Y$ is equivalent to the structure of a groupoid given by $Y \times G$ over $Y$ (Example \ref{ex:group action}), whose source map is the natural projection (and the target map defines the action). Here the projection $Y \times G$ can be viewed as the pullback of the ``source map'' $G \to \star$ along the natural augmentation $X \to \star$. This motivates the following definition:

\begin{Def}\label{def:groupoid action}
	An \defi{action} of a Segal groupoid $\_G^\bullet$ over $X$ on $f: Y \to X$ is the data of a groupoid: 
	\[ f^*\_G^\bullet := \_G^\bullet \times_X^{s_\bullet} Y \]
	
	where $s_n: \_G^n \to X$ are the source maps\footnote{The choice of the source maps instead of the target maps is purely conventional as the source and target maps play an interchangeable role.}. 			
	
	Given an action of $\_G^\bullet$ on $f$ we will denote the quotient by $\QS{Y}{\_G}$. We will see with Proposition \ref{prop:groupoid and lie algebroid action} that a Segal groupoid action naturally induces an action of the associated Lie algebroids (Construction \ref{cons:lie algebroid from lie groupoid}). 
\end{Def}

\begin{Lem}\label{lem:pullback groupoids projection is a groupoid projection}
	Consider a pullback of derived stacks: 
	\[ \begin{tikzcd}
		Y \arrow[r, "f"] \arrow[d, "q"]& X \arrow[d, "p"] \\
		Z \arrow[r, "g"] & \QS{X}{\_G} 
	\end{tikzcd} \]
	then we have that: 
	\begin{enumerate}
		\item the map $Y \to Z$ is an effective epimorphism.
		\item we have an action of $\_G^\bullet$ on $f:Y \to X$. 
		\item we have an equivalence: \[Z \simeq \QS{Y}{\_G}\]
		and $g := \eq{f}$ gives $f$ a structure of equivariant map with respect to the action. 
	\end{enumerate}
\end{Lem}
\begin{proof}\
	
	\begin{enumerate}
		\item The pullback of the effective epimorphism: 
		\[p:X \to \QS{X}{\_G}\]
		is $Y \to Z$, and it is an effective epimorphism thanks to \cite[Proposition 6.2.3.15]{Lu09HTT}.
		
		\item We have equivalences: \[\begin{split}
			Y \times_Z Y \simeq & X \times_{\QS{X}{\_G}} Z \times_Z Z \times_{\QS{X}{\_G}} X \\
			\simeq & X \times_{\QS{X}{\_G}}  Z \times_{\QS{X}{\_G}} X \\
			\simeq & X \times_{\QS{X}{\_G}}  X \times_{\QS{X}{\_G}} Z \\
			\simeq & \_G \times_{\QS{X}{\_G}}^{p\circ s} Z \\
			\simeq & \_G \times_X^s Y
		\end{split}\]

		where the fourth equivalence uses Lemma \ref{lem:projection is an effective epimorphism}. This gives $\_G \times_X^s Y$ the structure of a groupoid over $Y$ and the source map is clearly the projection from $Y$. More generally we get: 
		\[ C(q)^n := \underbrace{Y \times_Z \cdots \times_Z Y}_{\tx{n-times}} \simeq \_G^n \times_X^{s_n} Y\]
		
		Indeed we can proceed inductively and get:
		\[ \begin{split}
			C(q)^n  \simeq & C(q)^{n-1} \times_Z Y \\
			\simeq & \_G^{n-1} \times_X^{s_{n-1}} Y \\ 
			\simeq & \underbrace{\_G^1 \times_X \cdots \times_X \_G^1}_{\tx{(n-1)-times}} \times_X^{s} Y \\
			\simeq &  \underbrace{\_G^1 \times_X \cdots \times_X \_G^1}_{\tx{(n-1)-times}} \times_X^{ s} \left(X \times_{\QS{X}{\_G}}X\right) \times_X Y \\
			\simeq & \underbrace{\_G^1 \times_X \cdots \times_X \_G^1}_{\tx{(n-1)-times}} \times_X^{ s,s} \_G^1 \times_X^t Y \\
			\simeq &  \underbrace{\_G^1 \times_X \cdots \times_X \_G^1}_{\tx{(n-1)-times}} \times_X^{ s,t} \_G^1 \times_X^sY \\
			\simeq & \underbrace{\_G^1 \times_X \cdots \times_X \_G^1}_{\tx{(n)-times}}\times_X^sY  \\
			\simeq & \_G^{n} \times_X^{s_n} Y 
		\end{split} \] 
		
		Moreover, having the commutative square of the statement ensures that there is a map of groupoids:
		\[C(q)^\bullet \to C(p)^\bullet  \]
		which induces a map of groupoids: 
		\[ \_G^\bullet \times_X^{s_\bullet} Y \to \_G^\bullet\]
		and therefore defines an action of $\_G$ on $f$.  
		\item The equivalence follows from Proposition \ref{prop:girauds theorem} and the previous equivalence of groupoids over $Y$: 
		\[ C(q)^n \simeq \_G^n \times_X^{s_n} Y\] This shows that $\_G^\bullet \times_X^{s_\bullet} Y$ is the groupoid whose quotient is equivalent to $Z$.  
	\end{enumerate} 
\end{proof}

\begin{RQ}
	The previous lemma is essentially unraveling the content of \cite[Proposition 1.3.5.1]{TV08}.
\end{RQ}

\begin{Lem}\ \label{lem:pulback of quotient stack along the projection}		Let $f: Y \to X$ be a map of derived stacks and $\_G^\bullet$ a Segal groupoid over $X$ acting on $Y \to X$. Then we have a pullback diagram: 
	\[ \begin{tikzcd}
		Y \arrow[r, "f"] \arrow[d]& X \arrow[d] \\
		\QS{Y}{\_G} \arrow[r, "\eq{f}"'] & \QS{X}{\_G} 
	\end{tikzcd} \]
\end{Lem}
\begin{proof}
	Given an action of $\_G^\bullet$, we consider the morphism of Segal groupoids $f^*\_G^\bullet \to \_G^\bullet$. This morphism is Cartesian in the sens given in  \cite[Proposition 1.3.5.1]{TV08}\footnote{In the sens that it is in the essential image of the functor described in that proposition.}. Denote by $Y_0$ the pullback of the diagram. Then from Proposition \ref{prop:pullbak quotient stack of groupoids}, there is an action of $\_G^\bullet$ on $f_0: Y_0 \to X$ and the morphism $f_0^* \_G^\bullet \to \_G^\bullet$ is also Cartesian. 
	
	This implies that the natural morphism $f^*\_G^\bullet \to f_0^*\_G^\bullet$ is also Cartesian. The image of this under the colimit functor is the equivalence: 
	\[ \QS{Y}{\_G} \overset{\sim}{\to} \QS{Y_0}{\_G} \]
	From Proposition \ref{prop:pullbak quotient stack of groupoids}, and by definition of $Y_0$, this is an equivalence. The colimit functor restricted to Cartesian morphism over $f_0^*\_G^\bullet$ is fully-faithful thanks to \cite[Proposition 1.3.5.1]{TV08} and therefore reflects equivalences. In particular, we get a natural equivalence\footnote{This is because the category of Segal groupoids in \cite[1.3.5]{TV08} is viewed with its level-wise projective model structure, for which fibrations and equivalences are defined level-wise.}: 
	\[ Y \overset{\sim}{\to} Y_0\]   
\end{proof}

\begin{Prop}\ \label{prop:pullbak quotient stack of groupoids}
	We consider the morphisms $f:X \to Z$ and $g:Y \to Z$. Let $\_G^\bullet$ be a Segal groupoid over $Z$, and take actions of $\_G^\bullet$ on $f$ and $g$. We obtain morphisms between the quotients:
	\[\QS{X}{\_G} \to \QS{Z}{\_G} \qquad \QS{Y}{\_G} \to \QS{Z}{\_G}\] 
	
	Then there is a canonical \defi{pullback action}\footnote{By \emph{pullback action}, we mean that this is an action for any of the maps in the pullback diagram of $f$ and $g$.} of $\_G^\bullet$ on $X\times_Z Y$ such that we get a pullback square\footnote{Recall that that in that square, the pullback Segal groupoids are $\_G^\bullet \times_Z^{s_\bullet} (-)$.}:
	\[ \begin{tikzcd}
		\QS{X\times_Z Y}{\_G} \arrow[r] \arrow[d] & \QS{X}{\_G} \arrow[d] \\
		\QS{Y}{\_G} \arrow[r] & \QS{Z}{\_G}
	\end{tikzcd}\]
\end{Prop}
\begin{proof}
	
	First we now from Lemma \ref{lem:projection is an effective epimorphism} that $\_G^\bullet \simeq C(p_X)^\bullet$ where $p_X$ is the projection:
	\[ p_X: X \to \QS{X}{\_G}\]
	
	Now consider the following diagram: 
	
	\[\begin{tikzcd}
		& X \times_Z Y \arrow[dl] \arrow[rr] \arrow[dd] & & Y \arrow[dl] \arrow[dd] \\
		\QS{X}{\_G} \times_{\QS{Z}{\_G}} \QS{Y}{\_G} \arrow[dd] \arrow[rr] && \QS{Y}{\_G} \arrow[dd] & \\
		& X \arrow[rr] \arrow[dl] && Z \arrow[dl] \\
		\QS{X}{\_G} \arrow[rr] && \QS{Z}{\_G}&  
	\end{tikzcd} \]
	
	All squares are pullback squares. Indeed, the front and back squares are pullbacks by definition. The other squares (except for the left side square) are also pullbacks thanks to Lemma \ref{lem:pulback of quotient stack along the projection}. Since all squares but the left face in this cube are pullbacks, then the left face must also be a pullback square by formal properties of pullbacks in a commutative cube. \\

	Now using Lemma \ref{lem:pullback groupoids projection is a groupoid projection}, this shows that the map
	\[ X \times_Z Y \to \QS{X}{\_G} \times_{\QS{Z}{\_G}} \QS{Y}{\_G} \]
	is an effective epimorphism and we have the equivalences thanks to the proof of Lemma \ref{lem:pullback groupoids projection is a groupoid projection}: 
	
	\[ \begin{split}
		C(p_{X} \times_{p_Z} p_Y)^n \simeq & C(p_Y)^n \times_Y^{s_n} \left(X \times_Z Y\right) \\
		\simeq &C(p_Z)^n \times_Z^{s_n} Z \times_Z \left(X \times_Z Y\right)  \\
		\simeq &C(p_Z)^n \times_Z^{s_n} \left(X \times_Z Y\right)\\
		\simeq & \_G^n \times_Z^{s_n} \left(X \times_Z Y\right)
	\end{split}\]
	
	Moreover we have a commutative diagram of Segal groupoids: 
	\[ \begin{tikzcd}
		C(p_{X} \times_{p_Z} p_Y)^\bullet \arrow[r] \arrow[d] & C(p_{Y})^\bullet \arrow[d] \\
		C(p_{X} )^\bullet \arrow[r] & C(p_Z)^\bullet
	\end{tikzcd} \] 
	
	This makes $C(p_{X} \times_{p_Z} p_Y)^\bullet $ the groupoids given by the ``pullback action''. And because the $p_{X} \times_{p_Z} p_Y$ is an effective epimorphism, we get an equivalence:  
	\[ \QS{X}{\_G} \times_{\QS{Z}{\_G}} \QS{Y}{\_G} \simeq \QS{X \times_Z Y }{\_G}\] 
\end{proof}

\subsubsection{Tangent and cotangent complexes of quotient stacks}\ \label{sec:tangent-and-cotangent-of-quotient-stack}

\medskip

To study $\_G$-equivariant geometry is to study the quotient stacks. In particular, we would like to have a better understanding of quasi-coherent sheaves on such quotient and on their tangent and cotangent complexes. 

A first remark we can make is that since $p$ is an effective epimorphism, using \cite[Lemma 6.2.3.16.]{Lu09HTT}, the pullback along $p$ is a conservative functor between the slice categories. In particular, this implies that the functor: 
\[p^*: \QC\left(\QS{X}{\_G}\right) \to \QC(X)\] 
is conservative.\\

Moreover, it turns out that $p^* \Tt{\QS{X}{\_G}}$ is easy to compute. However, in order to describe it, we need to understand the notion of ``tangent to a groupoid'' $\_G$.

In the case of a group, the tangent is simply $G \times \G_g$, with $\G_g$ the Lie algebra associated to $G$. For a groupoid, the situation is more complicated, but we can produce a \emph{Lie algebroid} associated to $\_G$ playing the same role as $\G_g$. We will describe this in details with Construction \ref{cons:lie algebroid from lie groupoid}. For now it suffices to say that we can construct a Lie algebroid over $X$, when $X$ satisfies Assumptions \ref{ass:very good stack}, encoding the infinitesimal action of $\_G$ on $X$.  

In fact for this section, we do not need the full Lie algebroid, but only the underlying anchored linear stack. Given a Segal groupoid $\_G^\bullet$ we will therefore consider the following anchored module associated to $\_G^\bullet$:
\[  \_L := \Ttr{X}{\QS{X}{\_G}} \overset{\rho}{\to} \Tt_X \]
Therefore, we can work in a more general context where $X$ is a stack that admits a cotangent complex.

\begin{Prop}\label{prop:tangent of quotient stack groupoid}
	Let $\_G^\bullet$ be a Segal groupoid over $X$. Then we have that: 
	\[ p^* \Tt_{\QS{X}{\_G}} = \_L[1] \oplus^\rho \Tt_X\]
	
	where $\rho : \_L \to \Tt_X$ is the anchor. 
\end{Prop}
\begin{proof}
	By definition of $\_L$, using the fiber sequence defining the relative tangent complex, we get: 
	\[ p^*\Tt_{\QS{X}{\_G}} \simeq \tx{cofib}(\rho : \_L \to \Tt_X) \simeq   \Tt_X \oplus^\rho \_L[1] \]
\end{proof}

\begin{RQ}\label{rq:tangent quotient group and equivariant module}
	Unfortunately, it is more difficult to compute the tangent complex $\Tt_{\QS{X}{\_G}}$ because quasi-coherent sheaves on $\QS{X}{\_G}$ are not always easy to describe.
	
	However, if $\_G = X \times G$ comes from a smooth action of an affine group acting on a derived Artin stack $X$, then we have an equivalence:
	\[ G-\QC(X) \simeq  \QC\left(\QS{X}{G}\right) \]
	It sends a $G$-equivariant sheaf $\_F \in \QC(X)$ to the limit of the cosimplicial diagram $\_O_{G^\bullet}\otimes \_F$. The pullback functor $p^*$ along the projection amounts to forgetting the $G$-action. More details are given in the proof of Proposition \ref{prop:cotangent quotient group is good case}
\end{RQ}

\begin{Prop}\label{prop:cotangent quotient group is good case}
	Under the identification and assumptions of Remark \ref{rq:tangent quotient group and equivariant module}, we have an equivalence of $G$-equivariant modules: 
	\[ \Ll_{\QS{X}{G}} \simeq  \Ll_X \oplus^\rho \_O_X \otimes \G_g^*[-1] \]
	where $\rho$ is $G$-equivariant with $G$ acting on $\Ll_X$ by the cotangent action and on $\G_g^*[-1] \otimes \_O_X$ by the coadjoint action.
\end{Prop}
\begin{proof}[Sketch of proof]
	When $G$ is affine, we have an equivalence: 
	\[G-\QC(X) \to \QC(X \times G^\bullet)^{\tx{Cart}} \]
	where the right hand side is the category of \emph{Cartesian} cosimplicial presheaves (see \cite[Definition 1.2.12.1.]{TV08}). This equivalence sends $V \in \QC(X)$ to the cosimplicial set $\_O_{G^\bullet}\otimes_{\_O_X} V$. 
	
	Moreover, there is an equivalence: 
	\[ \lim : \QC(X \times G^\bullet)^{\tx{Cart}} \to \QC\left(\QS{X}{G}\right) \] 
	This works under the assumption that the cosimplicial diagram satisfies descent according to \cite[Definition 1.2.12.1]{TV08} which is the case when $G$ acts smoothly on $X$ a derived Artin stacks, in the sens that $X \times G$ is a smooth groupoid (Definition \ref{def:smooth and formal groupoids}). This is because the functor $\QC$ satisfies smooth descent (since it satisfies fpqc descent).  
	
	Moreover, the $G$-module $\Ll_X \oplus^{\rho^*} \_O_X \otimes \G_g^*[1] $ is equivalent to the limit of the cosimplicial diagram of $G$-equivariant modules: 
	\[ \lim  (\Ll_X \oplus^{\rho^*} \_O_X \otimes (\G_g^*)^{\oplus \bullet}) \]
	
	This is sent to the Cartesian cosimplicial diagram: 
	\[ \lim ( \_O_{G^\star} \otimes (\Ll_X \oplus^{\rho^*} \_O_X \otimes (\G_g^*)^{\oplus \bullet})) \]
	
	Since this is the limit of a bicosimplicial diagram, it is computed by the diagonal and we get:
	\[  \_O_{G^\bullet} \otimes (\Ll_X \oplus^{\rho^*} \_O_X \otimes (\G_g^*)^{\oplus \bullet})) \simeq \Ll_{X \otimes G^\bullet} \]
	
	Therefore taking the limit of this diagram recovers $\Ll_{\QS{X}{G}}$. 
\end{proof}

\begin{Cor}\label{cor:quotient of the underlying linear stack}
	If $G$ if affine and acting smoothly on $X$ an Artin stack, the action of $G$ on $\Ll_X \oplus^{\rho^*} \_O_X \otimes \G_g^*[-1]$ gives an action on its associated linear stack such that: 
	\[ \QS{\Aa_X(\Ll_X \oplus^{\rho^*} \_O_X \otimes \G_g^*[-1])}{G} \simeq T^*\QS{X}{G}\]
\end{Cor}
\begin{proof}
	We have a pullback diagram: \[ \begin{tikzcd}
		\Aa_X\left(p^*\Ll_{\QS{X}{G}}\right) \arrow[r] \arrow[d] & X \arrow[d, "p"] \\
		T^*\QS{X}{G} \arrow[r] & \QS{X}{G}
	\end{tikzcd}\]
	This is a pullback since this is the pullback of the linear stack $T^*\QS{X}{G}$ along $p$. Therefore from Proposition \ref{prop:map of linear stacks}, this pullback is the linear stack associated to $p^* \Ll_{\QS{X}{G}}$. 
	
	We can then use Proposition \ref{prop:tangent of quotient stack groupoid} and Lemma \ref{lem:pullback groupoids projection is a groupoid projection} to show that $T^*\QS{X}{G}$ is a quotient of $\Aa_X(\Ll_X \oplus^{\rho^*} \_O_X \otimes \G_g^*[-1])$ by an action of $G$. \\
	
	We need to show that this is exactly the action described in Proposition \ref{prop:cotangent quotient group is good case}. We observe that if we take the associated linear stack to the Cartesian quasi-coherent sheaf we used in the proof of Proposition \ref{prop:cotangent quotient group is good case}, we get a \emph{Cartesian simplicial stack} over $X \times G^\bullet$, and we have a commutative diagram: 
	
	\[\begin{tikzcd}
		\QC \left(\QS{X}{G} \right) \arrow[r] \arrow[d] & \QC(X \times G^\bullet)^{\tx{Cart}} \arrow[d] \\
		\dst_{/\QS{X}{G}} \arrow[r] & \left(\bf{sdSt}_{/X \times G^\bullet}\right)^{\tx{Cart}}
	\end{tikzcd}
	\] 
	The horizontal arrows are equivalences with left adjoint given by the colimit functor. Therefore $T^*\QS{X}{G}$ is both the linear stack associated with $\Ll_{\QS{X}{G}}$ and the colimit of the simplicial diagram: 
	
	\[ T^* (X \times G^\bullet) \simeq   T^*X \times (\G_g^*)^{\times \bullet} \times G^\bullet \simeq \Aa_{X} \left( \Ll_X \oplus^{\rho^*} \_O_X \otimes \G_g^*[-1] \right)\times G^\bullet \]
	where these equivalence are the linear version of the equivalences explained in the proof of Proposition \ref{prop:cotangent quotient group is good case}. 
	
	Therefore $T^*\QS{X}{G}$ is the quotient of: \[\Aa_{X} \left( \Ll_X \oplus^{\rho^*} \_O_X \otimes \G_g^*[-1] \right)\] by the natural action described in Proposition \ref{prop:cotangent quotient group is good case}.
	
	By effectivity of the quotient, this is the only action (up to equivalence) of $G$ on $\Aa_X\left(p^*\Ll_{\QS{X}{G}}\right)$ whose quotient is $T^*\QS{X}{G}$.
\end{proof}

\subsubsection{Some symplectic structures on quotients by a group action} \label{sec:some-symplectic-structure-on-quotient-by-a-group-action}\

\medskip

In this section, we discuss some element of the symplectic $G$-equivariant geometry for $G$ a group acting smoothly on $X$ a derived Artin stack. 

\begin{Ex}\label{ex:codajoint action}
	The Lie algebra $\G_g$ of an affine group $G$ admits an adjoint and a coadjoint action: 
	\[ G \times \G_g \to \G_g\ \qquad G \times\G_g^* \to \G_g^* \]
	
	We can see these actions as the natural actions on $G$ on the tangent and cotangent complexes of $\bf{B}G$ (see Proposition \ref{prop:cotangent quotient group is good case}). 
\end{Ex}

\begin{Lem}[{\cite[Example 1.6]{AC21}}]
	\label{lem:cotagent BG and coadjoint quotient}
	
	If $G$ if an affine group, there is an equivalence:
	\[ T^*[n+1]\bf{B}G \simeq \QS{\G_g^*[n]}{G}\]
\end{Lem}
\begin{proof}
	Consider the following pullback: 
	
	\[ \begin{tikzcd}
		\G_g^*[n] \arrow[r] \arrow[d] & \star \arrow[d] \\
		T^*[n+1]\bf{B}G \arrow[r] & \bf{B}G
	\end{tikzcd}\]
	
	This is a pullback because this is a pullback of linear stacks (Lemma \ref{lem:linear pullback of linear stacks}) and therefore is the linear stack associated with $p^*\Ll_{\bf{B}G}[n+1]$ given, thanks to Proposition \ref{prop:tangent of quotient stack groupoid}, by $\G_g^*[n]$. We have seen in Lemma \ref{lem:pullback groupoids projection is a groupoid projection} that it implies that $T^*[n+1]\bf{B}G$ is equivalent to a quotient of the pullback, $\G_g^*[n]$ by $G$. This quotient is the quotient by the coadjoint action thanks to Corollary \ref{cor:quotient of the underlying linear stack}.  
\end{proof}

\begin{RQ}\label{rq:structure symplectic on coadjoint quotient}
	$T^*[n+1]\bf{B}G$, and thus $\QS{\G_g^*[n]}{G}$, is canonically $(n+1)$-shifted symplectic. 
\end{RQ} 

\begin{Prop}
	\label{prop:lagrangian on projection to quotient codadjoint action}
	
	The natural projection $\G_g^* \to \QS{\G_g^*}{G}$ is Lagrangian. 
	
	Moreover,  the pullback of the symplectic form on $\G_g^*$ is zero on the nose and the Lagrangian structure can be chosen to be zero. 
\end{Prop}
\begin{proof}
	Since the symplectic structure described in Remark \ref{rq:structure symplectic on coadjoint quotient} is the canonical symplectic structure on a shifted cotangent, there is a model where it is strict in the sens that $\omega$ is equal to its underlying $2$-form. 
	The pullback of this $2$-form is given by the composition: 
	\[ \Tt_{\G_g^*} \to p^* \Tt_{\QS{\G_g^*}{G}} \overset{\omega^\flat}{\to}  p^* \Ll_{\QS{\G_g^*}{G}} \to \Ll_{\G_g^*} \]

	Using Proposition \ref{prop:tangent of quotient stack groupoid}, we have that: 
	\[ p^* \Tt_{\QS{\G_g^*}{G}} \simeq \Tt_{\G_g^*} \oplus^\rho \_O_{\G_g^*} \otimes \G_g[1] \simeq \_O_{\G_g^*} \otimes (\G_g^* \oplus \G_g [1]) \]
	
	Then our composition is given by:
	\[ \_O_{\G_g^*} \otimes \G_g^* \to \_O_{\G_g^*} \otimes (\G_g^* \oplus \G_g[1]) \overset{\tx{id}}{\to} \_O_{\G_g^*} \otimes (\G_g^* \oplus \G_g[1]) \to \_O_{\G_g^*} \otimes \G_g[1]\]
	
	This composition is \emph{strictly} zero and is both a homotopy and strict fiber sequence, which proves that the ``zero homotopy'' is a Lagrangian structure. 
\end{proof}

\begin{Prop}\
	\label{prop:lagrangian on zero section cotangent quotient stack}
	The zero map $\star \to \G_g^*[n]$ induces a Lagrangian morphism:
	\[ \bf{B}G \to \QS{\G_g^*[n]}{G}\]
	Moreover,  the pullback of the symplectic form on $\bf{B}G$ is zero on the nose and the Lagrangian structure can be chosen to be the zero isotropic structure. 
\end{Prop}
\begin{proof}
	This is the map induced by the zero section $\bf{B}G \to T^*[n+1]\bf{B}G$ using the equivalence of Lemma \ref{lem:cotagent BG and coadjoint quotient}. From Remark \ref{rq:closed 1-forms are lagrangian}, since the zero section on $T^*[n+1]\bf{B}G$ is strictly closed, we can pick the closure terms to be zero, which corresponds to taking a Lagrangian structure given by the $0$ isotropic structure on the zero section.
\end{proof}

\begin{Prop}\label{prop:cotangent quotient groupoid as a derived intersection}
	If $X$ is Artin and $G$ is affine, acting smoothly on $X$, we have a pullback: \[ \begin{tikzcd}
		T^*\QS{X}{G} \arrow[r] \arrow[d] & \bf{B}G \arrow[d] \\
		\QS{T^*X}{G} \arrow[r] & \QS{\G_g^*}{G}
	\end{tikzcd}\]
	where $G$ acts on $T^*X$ by the \emph{cotangent action}\footnote{For any $g \in G$, the action induces a morphism $\phi_g : X \to X$ that induces a map $T^*X \to T^*X$ by pulling-back along $\phi_g$. This induces an action $G \times T^*X \to T^*X$. This the action induced by the dual of the anchor of the action Lie algebroid $\Ll_X \to \_O_X \otimes \G_g^*[-1]$.} and $T^*X \to \G_g^*$ is the dual of the infinitesimal action.
\end{Prop}
\begin{proof}
	The first step is to notice that there is an equivalence:
	\[T^*X \times_{\G_g^*} \star \simeq \Aa_X\left(p^* \Ll_{\QS{X}{G}} \right)\]
	
	To see that, notice that we have a commutative diagram: 
	\[ \begin{tikzcd}
		T^*X \times_{\G_g^*} \star \arrow[r] \arrow[d] & X  \arrow[r] \arrow[d] & \star \arrow[d] \\ 
		T^*X \arrow[r] &  X\times \G_g^* \arrow[r] & \G_g^*
	\end{tikzcd} \]
	
	where all squares are Cartesian. The leftmost square is a pullback of linear maps between linear stacks . Therefore, from Lemma \ref{lem:linear pullback of linear stacks}, the pullback is the linear stack associated to the fiber: 
	\[ \tx{fiber} \left( T^*X \to X \times \G_g^* \right)\] 
	which is exactly $\Aa_X\left(p^* \Ll_{\QS{X}{G}} \right)$.
	
	Then taking the quotient of this equivalence by the natural action of $G$, the result becomes a consequence of Proposition \ref{prop:pullbak quotient stack of groupoids} since we have a pullback of quotients by actions of $G$. Then we use Corollary \ref{cor:quotient of the underlying linear stack} to identify the quotient of $\Aa_X\left(p^* \Ll_{\QS{X}{G}} \right)$ by $G$ with $T^*\QS{X}{G}$.
\end{proof}

We will see in Section \ref{sec:for-groups}, that the map $T^*X \to \G_g^*$ is a moment map and therefore the canonical symplectic structure on $T^*\QS{X}{G}$ coincides with the one obtained from the pullback viewed as a derived intersection of Lagrangian morphisms.

\subsection{$\_L$-Equivariant Symplectic Geometry}\
\label{sec:l-equivariant-symplectic-geometry}

\medskip

In this section, we adapt the results of Section \ref{sec:g-equivariant-geometry} to ``infinitesimal actions'' of Lie algebroids and discuss the relationship between $\_G$-equivariant and $\_L$-equivariant geometry.\\

We start in Section \ref{sec:derivation-and-integration-of-lie-algebroids} by exploring the relationship between Lie algebroids and groupoids, describing the ``derivation'' of a groupoid and the ``integration'' of a Lie algebroid. 

Moreover, we will explain that in good situations, the infinitesimal quotients of Lie algebroids are the formal completions of the projections to the quotients by the groupoids integrating the Lie algebroids. This in particular explains the terminology ``infinitesimal quotient'' that we use.\\

Then we will see in Section \ref{sec:l-equivariant-maps-and-quotient} the definitions and basic properties of $\_L$-equivariant geometry. This is very similar to Section \ref{sec:g-equivariant-maps-and-quotient} and in particular, $\_L$-equivariant geometry is the geometry of the infinitesimal quotients by Lie algebroids and their infinitesimal actions. \\

Then Section \ref{sec:tangent-and-cotangent-of-infinitesimal-quotient-stack} deals with the description of the tangent and cotangent complexes of infinitesimal quotients.

\subsubsection{Derivation and integration of Lie algebroids} \label{sec:derivation-and-integration-of-lie-algebroids}\

\medskip

In this section, we discuss the procedure of ``derivation'' that produces a Lie algebroid out of a Segal groupoid. We will then show that Lie algebroids always ``integrates'' to a \emph{formal} Segal groupoid. This is the analogue of the Lie differentiation and integration between Lie groups and Lie algebras.

We then proceed to show that in good situations, the infinitesimal quotient of a Lie algebroid can be identified with the formal thickening given by the formal completion of the projection to the quotient stack associated to a Segal groupoid integrating the Lie algebroid:
\[ \QS{X}{\_L} \simeq \comp{\QS{X}{\_G}_X}  \]

We start with the description of the derivation procedure.  Take $\_G^\bullet$ a Segal groupoid over $X= \Spec(A)$ satisfying Assumptions \ref{ass:very good stack}. We would like to define the Lie algebroid associated to $\_G^\bullet$ as the relative tangent: 
\[\_L := \Ttr{X}{\QS{X}{\_G}}\]
together with the Lie algebroid structure obtained from Proposition \ref{prop:lie algebroid relative tangent}. However, to use this proposition, we need the quotient stack to be \emph{formal}, which is not quite general enough for the integration procedure. 

Indeed, typical examples of groupoids whose quotient stacks are formal are \emph{smooth} groupoids\footnote{The quotient stack of a smooth Segal groupoid is \emph{Artin} and therefore formal thanks to Example \ref{ex:formal derived stacks}.} (Definition \ref{def:smooth and formal groupoids}), and it is well known that Lie algebroids do not always integrate to smooth Segal groupoids (see \cite{CF03}). In general, it does not even seem true that Lie algebroids integrate to Segal groupoids whose quotients are \emph{formal}, but we will see that they integrate to \emph{formal Segal groupoids} (Definition \ref{def:smooth and formal groupoids}).

Therefore we are interested in Segal groupoids that admit a formal completion (Construction \ref{cons:formal completion of a groupoid}) in particular, Segal groupoids such that $\_G^n$ is formal for all $n\geq 0$. \\

Furthermore, the correspondences between Lie algebroids, formal moduli problems and groupoids in much better behaved in \emph{pre-stacks}. Since we still want to obtain stacks and Segal groupoids in stacks, we need to consider Segal groupoids in stack that are obtained as the stackification of a Segal groupoid in pre-stacks (the Segal groupoid in pre-stacks being some additional data). \\

In what follows, the bases, $X$ and $Y$, are always affine. They will interchangeably be viewed as stacks or pre-stacks and we will not make any difference in notation between when they are seen as stacks or pre-stacks.

\begin{Cons}\label{cons:lie algebroid from lie groupoid} 
	Take $\_G^\bullet$ a Segal groupoid over $X= \Spec(A)$ satisfying Assumptions \ref{ass:very good stack}. We assume that $\_G^\bullet$ is given by the stackification of a groupoid $\pre{\_G}^\bullet$ in formal pre-stacks\footnote{In particular if $\_G^\bullet$ is already formal, then the assumption is satisfied with $\pre{\_G}^\bullet =j(\_G^\bullet)$ (viewed as a pre-stack).}.  Then we can construct a Lie algebroid as follows:
	
	\begin{enumerate}
		\item Take the formal completion of the groupoid $\pre{\_G}^\bullet$, denoted $\comp{\pre{\_G}^\bullet}$ (Construction \ref{cons:formal completion of a groupoid}). 
		\item Use \cite[Theorem 2.3.2]{GR20} and the equivalence:
		\[ \_B_X : \bf{fpGpd} \overset{\sim}{\to} \thickp(X)\] 
		where $\bf{fpGpd}$ denotes the category of formal Segal groupoids over $X$ in the category of pre-stacks. We get $\_B_X(\comp{\pre{\_G}^\bullet}) \in \thickp(X)$, which is equivalent to the data of a formal moduli problem under $X$ thanks to Theorem \ref{th:fmp are formal thickenings}. 
		
		\item We take the associated Lie algebroid using the equivalence of Theorem \ref{th:lie algebroid and FMP equivalence}: 
		\[ \_L_\_G :=  \Ttrl{X}{\_B_X(\comp{\pre{\_G}^\bullet})} \]
	\end{enumerate}
	This construction is clearly natural with respect to morphisms of groupoids $\pre{\_G}^\bullet \to \pre{(\_G')}^\bullet$ over a fixed base. Moreover we have: 
	\[ \_B_X(\comp{\pre{\_G}^\bullet}) \simeq \pQS{X}{\_L_\_G} \]
\end{Cons} 

\begin{RQ}
	This construction depends on the choice of $\pre{\_G}^\bullet$. When $\_G^\bullet$ is formal, we will by default consider that $\pre{\_G}^\bullet = j(\_G^\bullet)$ where $j: \dst \to \dpst$ is the inclusion of stacks in pre-stacks. Moreover since $j$ preserves formal completions (Proposition \ref{prop:stackification and formal geometry}), we have that: 
	\[ \comp{\pre{\_G}^\bullet} \simeq j(\comp{\_G^\bullet}) \]
\end{RQ}

\begin{RQ}\label{rq:classical lie algebroid differentiation} 
	This is a priori different from the classical differentiation of a groupoid (see for Example \cite[Section 1.2.3]{Ca21}). Classically, the Lie algebroid is obtained as a Lie algebroid structure on: \[ \_L_\_G := e^* \Ttr{\_G}{X}^s\]
	where $\Ttr{\_G}{X}^s$ denotes the relative tangent along the source map $s$ with anchor given by $t_*$. We will see in Lemma \ref{lem:relative tangent quotient groupoid source} that there is an equivalence of anchored linear stacks: 
	\[ e^* \Ttr{\_G^1}{X}^s \simeq \Ttr{X}{\QS{X}{\_G}} \]     
	
	We do not know whether it is an equivalence of Lie algebroids whenever the classical construction makes sense or not (i.e. if the Lie brackets coincide).  
\end{RQ}
\begin{Lem}\label{lem:relative tangent quotient groupoid source}
	Let $\_G^\bullet$ be any Segal groupoid over $X$ with $X$ any derived stack such that both $\_G^1$ and $X$ admit a cotangent complex. Then we have an equivalence of anchored $A$-modules: \[\begin{tikzcd}
		e^*\Ttr{\_G^1}{X}^s  \arrow[r, "\sim"] \arrow[dr, "\rho"'] & \Ttr{X}{\QS{X}{\_G}} \arrow[d] \\
		& \Tt_X
	\end{tikzcd}\]
	where $\rho$ denotes the following composition:
	\[e^*\Ttr{\_G^1}{X}^s \to e^*\Tt_{\_G^1} \overset{t_*}{\to} \Tt_X \]
\end{Lem}
\begin{proof}
	Consider the pullback square: 
	\[ \begin{tikzcd}
		\_G^1 \arrow[r, "t"] \arrow[d, "s"]  & X \arrow[d, "p"] \\
		X \arrow[r, "p"] & \QS{X}{\_G}
	\end{tikzcd}\] 
	
	Since this is a pullback, we have an equivalence in $\QC(\_G)$: 
	\[  \Ttr{\_G}{X}^s \simeq t^*\Ttr{X}{ \QS{X}{\_G}}\]
	
	This equivalence is part of the commutative diagram: 
	\[ \begin{tikzcd}
		\Ttr{\_G^1}{X}^s  \arrow[r, "\sim"] \arrow[d] & t^*\Ttr{X}{\QS{X}{\_G}} \arrow[d] \\
		\Tt_{\_G^1} \arrow[r, "t_*"]	& t^*\Tt_X
	\end{tikzcd}\]
	
	We then only have to pullback this last diagram along the unit $e$ to show the result.  
\end{proof}

\begin{Prop}\label{prop:relative tangent are algebroids more generaly}
	Let $\_G^\bullet$ be a Segal groupoid over $X$, with $X$ a stack that satisfies Assumptions \ref{ass:very good stack}. Assume that each $\_G^n$ is formal. Then there is an equivalence of $A$-modules:
	\[ \_L_\_G :=  \Ttr{X}{\_B_X\left(j(\comp{\_G^\bullet})\right)} \simeq  \Ttr{X}{\QS{X}{\_G}} \]
\end{Prop}
\begin{proof}
	In \cite[Section 2.4]{GR20}, it is shown that there is a pullback diagram in $\dpstfp$: 
	\[ \begin{tikzcd}
		j(\comp{\_G^1}) \arrow[r] \arrow[d] & X \arrow[d] \\
		X \arrow[r] & \_B_X\left(j(\comp{\_G^\bullet})\right)
	\end{tikzcd} \]
	
	This implies that: 
	\[ \Ttr{X}{\_B_X\left(j(\comp{\_G^\bullet})\right)} \simeq e^*\Ttr{j(\comp{\_G^1})}{X}^s \simeq e^*\Ttr{\comp{\_G^1}}{X}^s\]
	
	Note that since $\comp{\_G^\bullet}$ and $X$ are stack, taking their tangents in stacks or pre-stack is the same. 
	Moreover, from Lemma \ref{lem:formal completion factorization and morphims properties}, the map $\comp{\_G^\bullet} \to \_G^\bullet$ is formally étale, and therefore we have: 
	\[ e^*\Ttr{\comp{\_G^1}}{X}^s \simeq e^*\Ttr{\_G^1}{X}^s\]
	
	Finally, we can conclude using Lemma \ref{lem:relative tangent quotient groupoid source}.
\end{proof}

If fact the assumption that for each $n\geq 0$, $\_G^n$ is formal has very nice consequences on the properties of the associated Lie algebroid and its quotient.

\begin{Prop}\label{prop:formal completion and infinitesimal quotient}
	Take $\_G^\bullet$ a Segal groupoid over $X$ with $X$ satisfying Assumptions \ref{ass:very good stack} and such that for each $n\geq 0$, $\_G^n$ is a formal stack. Then we have an equivalence: 
	\[ \pQS{X}{\_L_\_G} := \_B_X\left( j(\comp{\_G}) \right) \simeq j \left(\comp{\QS{X}{\_G}_X} \right) \]
\end{Prop}
\begin{proof}
	We have an equivalence of formal groupoids in pre-stacks: 
	\[ \begin{split}
		X \times_{j \left(\comp{\QS{X}{\_G}_X}\right)} X \simeq & j \left( X \times_{ \comp{\QS{X}{\_G}_X}} X  \right) \\
		\simeq &  j \left( X \times_{ \QS{X}{\_G}\times_{\QS{X}{\_G}_{\tx{DR}}} X_{\tx{DR}}}  X  \right) \\
		\simeq & j \left( \left(X \times_{ \QS{X}{\_G}} X\right) \times_{X_{\tx{DR}} \times_{\QS{X}{\_G}_{\tx{DR}}} X_{\tx{DR}}} X_{\tx{DR}}  \right) \\
		\simeq & j \left( \_G^1 \times_{\_G^1_{\tx{DR}}} X_{\tx{DR}} \right) \\
		\simeq & j (\comp{\_G^1}) 
	\end{split}  \]
	
	This equivalence extends to the $n$-simplices: 
	\[ \underbrace{X \times_{j \left(\comp{\QS{X}{\_G}_X}\right)}\cdots \times_{j \left(\comp{\QS{X}{\_G}_X}\right)} X}_{\tx{n-times}} \simeq  j (\comp{\_G^n}) \]
	
	Therefore, since $\_B_X$ is the inverse of the functor sending $f:X \to Y$ to $C(f)^\bullet$, we get an equivalence: 
	\[ j \left(\comp{\QS{X}{\_G}_X}\right) \simeq \_B_X(j(\comp{\_G^\bullet}))  \]
\end{proof}	

\begin{Cor}\label{cor:formal completion lie algebroid and relative tangent}
	Take $\_G^\bullet$ a Segal groupoid over $X$ with $X$ satisfying Assumptions \ref{ass:very good stack} and such that for each $n\geq 0$, $\_G^n$ is a formal stack. Then we have the following: 
	\begin{enumerate}
		\item We have an equivalence: 
		\[ \QS{X}{\_L_\_G} \simeq \comp{\QS{X}{\_G}_X}\]
		\item We have equivalences of $A$-modules: 
		\[ \Ttr{X}{\QS{X}{\_G}} \simeq \Ttr{X}{\QS{X}{\_L_\_G}} \simeq \_L_\_G \]
		
		Moreover, these equivalences give these modules the Lie algebroid structure coming from the one on $\_L_\_G$.
	\end{enumerate}
\end{Cor}
\begin{proof}\
	
	\begin{enumerate}
		\item By definition we have: \[ \_B_X\left(j(\comp{\_G^\bullet})\right) \simeq \pQS{X}{\_L_\_G}\]
		Using Proposition \ref{prop:formal completion and infinitesimal quotient} and the stackification functor, we get:
		\[ \QS{X}{\_L_\_G} := \ST\left(\pQS{X}{\_L_\_G} \right) \simeq \ST \left( j \left( \comp{\QS{X}{\_G}_X}\right)\right) \simeq \comp{\QS{X}{\_G}_X}\]
		\item We have the following equivalences:
		\[ \_L_\_G \simeq \Ttr{X}{\pQS{X}{\_L_\_G}} \simeq \Ttr{X}{\_B_X \left(j (\comp{\_G})\right)}  \simeq \Ttr{X}{j\left(\comp{\QS{X}{\_G}_X}\right)} \simeq \Ttr{X}{\comp{\QS{X}{\_G}_X}} \simeq \Ttr{X}{\QS{X}{\_L_\_G}}\]
		
		Moreover since the map $\comp{\QS{X}{\_G}_X} \to \QS{X}{\_G}$ is formally étale, we get: 
		\[ \Ttr{X}{\QS{X}{\_G}} \simeq \Ttr{X}{\comp{\QS{X}{\_G}_X}} \simeq \Ttr{X}{\QS{X}{\_L_\_G}} \]
	\end{enumerate}
\end{proof}

\begin{RQ}\label{rq:why infinitesimal quotient}
	This shows that if $\_G^\bullet$ a Segal groupoid over $X$ with $X$ satisfying Assumptions \ref{ass:very good stack} and $\_G^\bullet$ is a formal stack, then the formal completion of a quotient stack is equivalent to the infinitesimal quotient of $X$ by the associated Lie algebroid. This explains why we call them ``infinitesimal quotients''.
\end{RQ}

So far, picking $\pre{\_G^\bullet}$ in Construction \ref{cons:lie algebroid from lie groupoid} is not useful if we restrict to Segal groupoids such that $\_G^\bullet$ is formal, which covers a rather large class of examples. However, when we want to integrate a Lie algebroid, it is unclear whether we will get a Segal groupoid of that sort. We will see that the integration procedure actually produces a Segal groupoid in pre-stacks $\pre{\_G^\bullet}$.\\

We now turn toward the integration procedure.

\begin{Def}\label{def:groupoid integrating a lie algebroid}
	Given a Lie algebroid $\_L$ over $X$ with $X$ satisfying Assumptions  \ref{ass:very good stack}. We say that a Segal groupoid $\_G^\bullet$ over $X$ \defi{integrates} $\_L$ if $\_L$ is equivalent (as Lie algebroids) to $\_L_{\_G}$, the Lie algebroid associated to $\_G^\bullet$.
\end{Def}  

\begin{RQ}
	Note that it is hopeless to have an integration procedure such that for all Segal groupoid $\_G^\bullet$ over $X$, $\_G^\bullet$ is the integration of its associated Lie algebroid $\_L$ since this property fails even for groups and Lie algebras. However, we can hope in good situations to recover the formal completion, $\comp{\_G^\bullet}$, of the groupoid we started with. 
\end{RQ} 

We are now going to prove a version of Lie's third theorem, stating that every Lie algebroid ( over a base satisfying some assumptions) admits a formal integration.

\begin{Prop}\label{prop:algebroid integration}
	Take $\_L$ a Lie algebroid over $X$ with $X$ satisfying Assumptions \ref{ass:very good stack}. Consider the infinitesimal projection: 
	\[ h: X \to \QS{X}{\_L} \]
	We define the Segal groupoid:
	\[\_G_\_L^\bullet := C(h)^\bullet \simeq \underbrace{X \times_{\QS{X}{\_L}} \cdots \times_{\QS{X}{\_L}} X}_{\tx{\bullet-times}} \] 
	Then $\_L$ is the Lie algebroid associated to $\_G_\_L^\bullet$, obtained from Construction \ref{cons:lie algebroid from lie groupoid} using the Segal groupoid in pre-stacks given by: \[ \pre{(\_G_\_L)}^\bullet := C(\pre{h})^\bullet \simeq \underbrace{X \times_{\pQS{X}{\_L}} \cdots \times_{\pQS{X}{\_L}} X}_{\tx{\bullet-times}} \]
	with \[\pre{h}: X \to \pQS{X}{\_L} \]
\end{Prop}  
\begin{proof}
	First, we want ot show that $\pre{(\_G_\_L)}^\bullet$ is a formal Segal groupoid whose stackification in $\_G_\_L^\bullet$. Clearly, $X$ and $\pre{(\_G_\_L)}^\bullet$ are formal stacks. Moreover, the projection:
	\[p:X \to \pQS{X}{\_L}\]
	is a nil-equivalence and this implies that all the structure morphisms of the Segal groupoid are nil-equivalences, making $\_G_\_L^\bullet$ a formal Segal groupoid. Moreover since the stackification functor preserves pullbacks, we have: \[\ST(\pre{(\_G_\_L)}^\bullet) \simeq \_G_\_L^\bullet\] 
	
	We clearly have an equivalence: 
	\[ \_B_X(\pre{(\_G_\_L)}^\bullet) \simeq \pQS{X}{\_L} \]
	since $\_B_X$ is the inverse of the functor sending  a formal thickening $f:X \to Y$ to the formal Segal groupoid given by $C(f)$. Then under Assumptions \ref{ass:very good stack}, the Lie algebroid associated to $\QS{X}{\_L}$ is $\_L$ itself (using the Segal groupoid in pre-stack $\pre{(\_G_\_L)}^\bullet$).   
\end{proof}

This proves that any Lie algebroid over $X$ satisfying Assumptions \ref{ass:very good stack} can be \emph{integrated} to a \emph{formal} Segal groupoid (in our case $\_G_\_L^\bullet$). Moreover this defines a functor (over a fixed base) that sends $\_L$ to $\_G_\_L^\bullet$. 

\begin{RQ}
	The Segal groupoid we produce, $\_G_\_L^\bullet$, is \emph{not} smooth and it is well known that Lie algebroids \emph{do not} always integrate to a smooth groupoids (see \cite{CF03}). In general, $\_G_\_L$ is instead the stackification of a formal groupoid in formal pre-stacks. 
	It is worth noting that the natural projections: 
	\[ X \to \QS{X}{\_G_\_L} \qquad X \to \QS{X}{\pre{(\_G_\_L)}} \]
	are nil-equivalences. However, they are \emph{not} formal thickenings as the quotients need not be formal (pre)-stacks.
\end{RQ}

One of the main drawback in having to use these pre-stacks is that the relative tangent of the infinitesimal quotient stack does not recover the Lie algebroid\footnote{This is essentially because the stackification functor does not preserve the tangent complexes}.\\

However, we have seen earlier that if $\_L$ is a Lie algebroid associated with a good class of Segal groupoid, then the  relative tangent of the infinitesimal quotient stack is equivalent to $\_L$. This motivates the following definition: 

\begin{Def}\label{def:good integration}
	A Lie algebroids $\_L$ over $X$ an affine stack satisfying Assumptions \ref{ass:very good stack}. Then $\_L$ is said to \defi{integrate well} if there exists a Segal groupoid $\_G^\bullet$ such that for each $n\geq 0$, $\_G^n$ is a formal stack and such that $\_G^\bullet$ integrates $\_L$, in other words, there is an equivalence of Lie algebroids:
	\[ \_L \simeq \_L_\_G\]
	
	Such a $\_G^\bullet$ will be called a \defi{good integration} of $\_L$. If $\_G^\bullet$ is smooth, it is called a \defi{smooth integration} of $\_L$.   
\end{Def}

\begin{Lem}\label{lem:good algebroid relative tangent}
	Let $\_L$ be a Lie algebroid over $X$ an affine stack satisfying Assumptions \ref{ass:very good stack} that integrates well. Then there is an equivalence: 
	\[ \Ttr{X}{\QS{X}{\_L}} \simeq \_L \]
\end{Lem}
\begin{proof}
	Pick $\_G^\bullet$ a good integration of $\_L$. Then thanks to Corollary \ref{cor:formal completion lie algebroid and relative tangent} and the fact that $\_L \simeq \_L_\_G$, we have the equivalences: 
	\[ \Ttr{X}{\QS{X}{\_L}} \simeq \Ttr{X}{\QS{X}{\_L_\_G}} \simeq \_L_\_G \simeq \_L \]
\end{proof}

\begin{Prop}\label{prop:integration only recover the formal completion}
	Let $\_G^\bullet$ be a Segal groupoid over $X$ an affine stack satisfying Assumptions \ref{ass:very good stack}. We assume that $\_G^\bullet = \ST\left(\pre{\_G}^\bullet\right)$ and consider $\_L_\_G$ the Lie algebroid obtained from Construction \ref{cons:lie algebroid from lie groupoid} using $\pre{\_G}^\bullet$. Then there is an equivalence of Segal groupoids: 
	\[ \ST\left(\comp{\pre{\_G}^\bullet}\right) \simeq \_G_{\_L_\_G}^\bullet \]
\end{Prop}
\begin{proof}
	
	We have by definition that:
	\[ \_B_X(\comp{\pre{\_G}^\bullet}) \simeq \pQS{X}{\_L_\_G} \]
	
	Therefore we get: 
	\[ \comp{\pre{\_G}^\bullet} \simeq C\left(X \to  \pQS{X}{\_L_\_G}\right) := \pre{(\_G_{\_L_\_G})^\bullet} \]
	
	Using the stackification functor, we get: 
	\[\ST\left(\comp{\pre{\_G}^\bullet}\right) \simeq \_G_{\_L_\_G}^\bullet \] 
\end{proof}

\begin{RQ}
	In particular, if for each $n\geq 0$, $\_G^n$ is a formal stack, then $\pre{\_G}^\bullet = j(\_G^\bullet)$ and since $j$ commutes with the formal completion we get: 
	\[ \comp{\_G^\bullet} \simeq \_G_{\_L_\_G}^\bullet \]
	Therefore the integration of the groupoid associated to $\_G^\bullet$ is \emph{not} $\_G^\bullet$ itself, but rather its formal completion.
\end{RQ}

We now turn to the behavior of the differentiation and integration constructions with respect to actions of Lie algebroids and Segal groupoids. First we are going to look at actions on $f: Y \to X$, a map of stacks satisfying Assumptions \ref{ass:very good stack}. \\

This differentiation procedure behaves well with respect to \emph{actions} in the sens that any action on $f: Y\to X$ of a Segal groupoid $\_G^\bullet$ such that for each $n\geq 0$, $\_G^n$ is a formal stack, induces an infinitesimal action of its associated Lie algebroid $\_L_\_G$ on $f$.

\begin{Prop}\label{prop:groupoid and lie algebroid action}
	Let $\_G^\bullet$ be a Segal groupoid over $X$ such that $\_G^\bullet$ is a formal stack and for each $n \geq 0$, $\_G^n$ is a formal stack. Then consider an action of $\_G^\bullet$ on $f: Y \to X$ and suppose that both $X$ and $Y$ satisfy Assumptions \ref{ass:very good stack}. Then this induces an action of $\_L_\_G$ on $f$, and this defines a functor: 
	\[ \bf{Act}_{\_G} \to \bf{Act}_{\_L_\_G}\]
\end{Prop}
\begin{proof}
	Consider the following commutative diagram: 
	\[ \begin{tikzcd}
		Y \arrow[r] \arrow[d] & X \arrow[d] \\
		\comp{\QS{Y}{\_G}_Y}\arrow[r] \arrow[d] & \QS{X}{\_L_\_G} \arrow[d] \\
		\QS{Y}{\_G} \arrow[r] & \QS{X}{\_G}
	\end{tikzcd} \]
	
	Where $\QS{X}{\_L_\_G} \simeq \comp{\QS{X}{\_G}_X}$ thanks to Corollary \ref{cor:formal completion lie algebroid and relative tangent}. The outer square is a pullback as a consequence of Lemma \ref{lem:pullback groupoids projection is a groupoid projection}. 
	
	The lower square is a pullback because of the equivalences: 
	\[ \begin{split}
		\comp{\QS{Y}{\_G}_Y} \simeq & \QS{Y}{\_G} \times_{\QS{Y}{\_G}_{\tx{DR}}} Y_{\tx{DR}} \\
		\simeq & \QS{Y}{\_G} \times_{\QS{Y}{\_G}_{\tx{DR}}} \QS{Y}{\_G}_{\tx{DR}} \times_{\QS{X}{\_G}_{\tx{DR}}} X_{\tx{DR}}\\
		\simeq & \QS{Y}{\_G} \times_{\QS{X}{\_G}_{\tx{DR}}} X_{\tx{DR}} \\
		\simeq & \QS{Y}{\_G}\times_{\QS{X}{\_G}} \QS{X}{\_G} \times_{\QS{X}{\_G}_{\tx{DR}}} X_{\tx{DR}}\\
		\simeq & \QS{Y}{\_G} \times_{\QS{X}{\_G}} \QS{X}{\_L}
	\end{split}\] 
	
	Therefore the upper square is also a pullback and the map: 
	\[Y \to \comp{\QS{Y}{\_G}_Y}\]
	is a formal thickening. 
	
	Since $f^*\_G^\bullet = \_G^\bullet \times_X^{s_\bullet} Y$ also satisfies that each $f^*\_G^n$ is formal, and its quotient stack is $\QS{Y}{\_G}$, then $\comp{\QS{Y}{\_G}_Y}$ is equivalent to the quotient of $Y$ by the Lie algebroid associated to $f^*\_G^\bullet$ (again thanks to Corollary \ref{cor:formal completion lie algebroid and relative tangent}). 
	
	Looking at the relative tangent of these pullbacks and using Corollary \ref{cor:formal completion lie algebroid and relative tangent}, we get the equivalences: 
	\[ \Ttr{Y}{\QS{Y}{\_G}} \simeq \Ttr{Y}{\comp{\QS{Y}{\_G}_Y}} \overset{\sim}{\to} f^*\Ttr{Y}{\QS{X}{\_G}} \simeq f^*\_L\] 
	
	This gives $f^*\_L$ the structure of a Lie algebroid. \\
	
	We want to find a morphism of Lie algebroids $f^*\_L \to \_L$. To do that, notice that we have a pullback diagram (because $j$ preserves pullbacks): 
	\[ \begin{tikzcd}
		Y \arrow[r] \arrow[d] & X \arrow[d] \\
		j\left(\comp{\QS{Y}{\_G}_Y}\right)\arrow[r] & j\left(\QS{X}{\_L_\_G} \right)
	\end{tikzcd} \]
	
	From Proposition \ref{prop:formal completion and infinitesimal quotient} we have the equivalences:
	\[ \_B_X\left(j(\comp{f^*\_G^\bullet})\right) \simeq j\left(\comp{\QS{Y}{\_G}_Y}\right) \qquad \_B_X\left(j(\comp{\_G^\bullet})\right) \simeq j\left(\comp{\QS{X}{\_G}_X}\right) \]
	
	Therefore these are all formal pre-stacks and we can use Proposition \ref{prop:functoriality relative tangent base change} which shows that: 
	\[ \Ttr{Y}{ j\left(\comp{\QS{Y}{\_G}_Y}\right)} \to \Ttr{X}{j\left(\comp{\QS{X}{\_G}_X}\right)}\]
	is a morphism of Lie algebroids. 
	
	Then we get a morphism of Lie algebroid: 
	\[ f^*\_L \simeq \Ttr{Y}{ \comp{\QS{Y}{\_G}_Y}} \simeq \Ttr{Y}{j\left(\comp{\QS{Y}{\_G}_Y}\right)} \to \Ttr{X}{j\left(\comp{\QS{X}{\_G}_X}\right)} \simeq \Ttr{X}{\comp{\QS{X}{\_G}_X}} \simeq \_L\]
\end{proof}

\begin{Prop}\label{prop:groupoid action from Lie algebroid action}
	Let $f: Y \to X$ be a morphism of affine stacks satisfying Assumptions \ref{ass:very good stack}. Let $\_L$ be a Lie algebroid on $X$ with an action on $f$. Then this induces an action of the formal groupoid integrating $\_L$ on $f$ defining a functor: 
	\[ \bf{Act}_\_L \to \bf{Act}_{\_G_\_L} \]
\end{Prop}
\begin{proof}Given an action of $\_L$ on $f: Y \to X$, we first need to prove this statement for the pre-stacks.  We first show that the following commutative diagram in pre-stacks is a pullback\footnote{We will show in Proposition \ref{prop:equivariant map are map of lie algebroids} that a map of Lie algebroid induces such a commutative diagram.}:
	\[ \begin{tikzcd}
		Y \arrow[r] \arrow[d] & X \arrow[d] \\ 
		\pQS{Y}{\_L} \arrow[r] & \pQS{X}{\_L}
	\end{tikzcd} \]
	Denote by $P$ the pullback, and $\phi: Y \to P$ the natural morphism. $P$ and $Y$ are formal and $\phi$ is a nil-equivalence because: \[P_{\tx{DR}} \simeq \left(\pQS{Y}{\_L}\right)_{\tx{DR}} \overset{\sim}{\leftarrow} Y_{\tx{DR}}\]
	
	Moreover we have a fiber sequence: 
	\[ \Ttr{Y}{P} \to \Ttr{Y}{\pQS{Y}{\_L}} \to \phi^* \Ttr{P}{\pQS{Y}{\_L}} \simeq  f^*\Ttr{X}{\pQS{X}{\_L}} \]
	The second morphism is clearly an equivalence as both complexes are equivalent to $\phi^*\_L$ and therefore $\phi$ is formally étale. This proves that $\phi$ is an equivalence thanks to Lemma \ref{lem:niliso + etal = isom}. \\
	
	Now we have the following equivalences: 
	\[ \begin{split}
		\pre{(\_G_{f^*\_L} )}^1 \simeq & Y  \times_{\pQS{Y}{\_L}} Y \\
		\simeq & X \times_{\pQS{X}{\_L}} \pQS{Y}{\_L} \times_{\pQS{Y}{\_L}} Y \\
		\simeq & X \times_{\pQS{X}{\_L}} Y \\
		\simeq & \left( X \times_{\pQS{X}{\_L}} X \right) \times_X Y \\
		\simeq & \pre{(\_G_\_L)^1} \times_X^s Y \\
		\simeq & f^* \pre{(\_G_\_L)^1}
	\end{split} \] 
	
	Moreover, by induction, we get that: 
	\[ \begin{split}
		\pre{(\_G_{f^*\_L} )}^n\simeq & \pre{(\_G_{f^*\_L} )}^{n-1} \times_Y \pre{(\_G_{f^*\_L} )}^1 \\
		\simeq & f^*\pre{\left(\_G_\_L\right)}^{n-1} \times_Y \pre{(\_G_{f^*\_L} )}^1 \\
		\simeq & f^*\pre{\left(\_G_\_L\right)}^{n-1} \times_Y Y  \times_{\pQS{Y}{\_L}} Y \\
		\simeq & f^*\pre{\left(\_G_\_L\right)}^{n-1} \times_Y  f^* \pre{(\_G_\_L)^1} \\
		\simeq & f^*\pre{\left(\_G_\_L\right)}^n
	\end{split}\]
	
	Therefore the Segal groupoid integrating $f^*\_L$ is equivalent to a Segal groupoid structure on $\pre{(\_G_\_L)}^\bullet \times_X^{s_\bullet} Y$. Moreover, since there is a natural morphism of groupoids $\pre{(\_G_{f^*\_L})}^\bullet  \to \pre{(\_G_\_L)}^\bullet$ over $f$, this defines an action of $\pre{(\_G_\_L)}^\bullet$ on $f$. 
	
	Then the statement is the stackification of this construction.  
\end{proof}

The main issue we will have is that all of this only works if $X$ and $Y$ satisfy Assumptions \ref{ass:very good stack} which is too strict for most examples we will be interested in (e.g. the derived critical locus, zero loci of moment maps etc...). 

We need to come up with a weakened notion of infinitesimal quotient. We actually only need this for infinitesimal quotients of actions on $f: Y \to X$ where $X$ satisfies Assumptions \ref{ass:very good stack} and $Y$ is affine of almost finite presentation.  

\begin{Def}\label{def:weak inf quotient action} 
	Let $f: Y \to X$ be a morphism of stack with $X$ satisfying Assumptions \ref{ass:very good stack} and $Y$ affine of almost finite presentation. Take $\_L$ a Lie algebroid over $X$. Then a \defi{weak infinitesimal quotient} of $Y$ by $\_L$ is a derived stack, denoted by $\QSW{Y}{\_L}$, that fits in the pullback diagram: 
	\[ \begin{tikzcd}
		Y \arrow[r] \arrow[d] & X \arrow[d] \\
		\QSW{Y}{\_L} \arrow[r] & \QS{X}{\_L}
	\end{tikzcd} \] 
	
	Note that the existence of such $\QSW{Y}{\_L}$ is not guarantied, nor is its uniqueness.
\end{Def}

\begin{RQ}
	This notion does not really use an action of $\_L$. Therefore, even if $\_L$ does not act on $f$ in the sens of Definition \ref{def:action of a Lie (infty) alebroid (up to homotopy)}, this defines a notion of weak infinitesimal quotient. 
	
	Note that this does not even depend on the Lie algebroid structure on $X$ but only on its infinitesimal projection: 
	\[ h: X \to \QS{X}{\_L} \]  
\end{RQ}

\begin{RQ}\label{rq:weak inf quotient of weak inf quotient}
	Since this construction only depends on the infinitesimal quotient projection, we can also define a weak infinitesimal quotient \emph{compatible} with a weak infinitesimal quotient $\QSW{Y}{\_L}$ along a morphism $Y' \to Y$ as a derived stack $\QSW{Y'}{\_L}$ fitting in a pullback diagram: 
	\[ \begin{tikzcd}
		Y' \arrow[r] \arrow[d] & Y \arrow[d] \\
		\QSW{Y'}{\_L} \arrow[r] & \QSW{Y}{\_L}
	\end{tikzcd} \]
	In particular it is also automatically a weak infinitesimal quotient of $\_L$ along the composition: 
	\[ Y' \to Y \to X \]
\end{RQ}

One of the main difficulty when dealing with (weak) infinitesimal quotients of Lie algebroids is that we do not know whether there is an equivalence: 
\[ \Ttr{X}{\QS{X}{\_L}} \simeq \_L \]

However, we have seen in Lemma \ref{lem:good algebroid relative tangent} that if $\_L$ \emph{integrates well} (in the sens of Definition \ref{def:good integration}), then the relative tangent recovers $\_L$. 	

\begin{Cor}\label{cor:relative tangent weak inf quotient} 
	Let ${\QSW{Y}{\_L}}$ be a weak infinitesimal quotient of $f: Y \to X$ as before. Assume that $\_L$ integrates well. Then using the pullback defining ${\QSW{Y}{\_L}}$ and Lemma \ref{lem:good algebroid relative tangent}, we get an equivalence in $\QC(Y)$: \[ \Ttr{Y}{{\QSW{Y}{\_L}}} \simeq f^*\Ttr{X}{\QS{X}{\_L}} \simeq f^* \_L\]
\end{Cor}

This Corollary does not give a Lie algebroids structure on $f^*\_L$ since $Y$ does not need to satisfy Assumptions \ref{ass:very good stack}. Therefore we cannot recover an action of $\_L$ as in Proposition \ref{prop:groupoid and lie algebroid action}. 

Clearly, if $Y$ satisfy Assumptions \ref{ass:very good stack} and $\_L$ has an action on $Y \to X$, then $\QS{Y}{\_L}$ is a valid choice of weak infinitesimal quotient (but not necessarily the only one even up to equivalence).

\begin{Ex}\label{ex:lie algebroid associated with a group action}
	The Lie algebroid associated with a smooth action of an affine group $G$ on $X$ (with $X$ satisfying Assumptions \ref{ass:very good stack}) is exactly the \emph{action Lie algebroid} of Example \ref{ex:action lie algebroid}. \\
	
	To see that, observe that $G$ is naturally a groupoid over the point $\star$ with Lie algebroid given by the Lie algebra $\G_g$ associated to $G$. Then a smooth action of $G$ on $X$ is exactly a smooth action as in Proposition \ref{prop:groupoid and lie algebroid action}. Therefore the Lie algebroid associated with the action groupoid $G \times X$ is $f^* \G_g = X \times \G_g$. Since the morphism $X \times \G_g \to \G_g$ must be a morphism of Lie algebroids, this implies that the Lie bracket on $X \times \G_g$ is exactly the Lie bracket described in Example \ref{ex:action lie algebroid}. 
\end{Ex}

\subsubsection{$\_L$-equivariant maps and (weak) infinitesimal quotients}\
\label{sec:l-equivariant-maps-and-quotient}

\medskip

This section is about having an analogue to Section \ref{sec:g-equivariant-maps-and-quotient} for infinitesimal action.
We start by defining the notion of $\_L$-equivariant map by analogy to Definition \ref{def:equivariant maps}.

\begin{Def}\ \label{def:equivariant maps algebroids}
	Let $f: X \to Y$ be a morphism of derived stacks. Assume that $\_L$ is a Lie algebroid over $X$, a stack satisfying Assumptions \ref{ass:very good stack}. Then the structure of an \defi{invariant map with respect to $\_L$} on $f$ is a factorization through the ``projection''\footnote{We still call the natural formal thickening: \[h: X \to \QS{X}{\_L}\] a ``projection'' by analogy with the groupoid case. But it needs not be an effective epimorphism.}: 
	
	\[ \begin{tikzcd}
		X \arrow[r, "f"] \arrow[d, "h"'] & Y \\
		\QS{X}{\_L} \arrow[ur, "\eq{f}"']
	\end{tikzcd} \]

	If $\_L'$ is a Lie algebroid over $Y$ (with $Y$ also satisfying Assumptions \ref{ass:very good stack}), then we say that $\eq{f}$ is the structure of an \defi{equivariant map with respect to $\_L$ and $\_L'$ on $f$} if the map:
	\[h\circ f: X \to \QS{Y}{\_L'}\]
	is $\_L$-equivariant in the sens given before. In other words, this is the data of a morphism $\eq{f}$ fitting in the commutative diagram: 
	\[ \begin{tikzcd}
		X \arrow[r, "f"] \arrow[d] & Y \arrow[d] \\
		\QS{X}{\_L} \arrow[r, "\eq{f}"'] & \QS{Y}{\_L'}
	\end{tikzcd}\]

	From now on we will say that $f$ \emph{is} equivariant when referring to the data of the map $\eq{f}$. 
\end{Def}

These notion of invariance and equivariance also make sens in pre-stacks by replacing all the infinitesimal quotient by the pre-stacks they are coming from. \\

In Section \ref{sec:g-equivariant-maps-and-quotient}, any morphism between quotient stacks would come from a morphism between Segal groupoids thanks to Lemma \ref{lem:projection is an effective epimorphism}. The infinitesimal projections are no longer effective epimorphisms in general but we still have the following:

\begin{Prop}\label{prop:equivariant map are map of lie algebroids} 
	Take map $f: Y \to X$ of finitely presented affine stacks satisfying Assumptions \ref{ass:very good stack} and $\_L'$ and $\_L$, Lie algebroids over $Y$ and $X$ respectively. Then we have that: 
	\begin{itemize}
		\item A morphism of Lie algebroids $\_L' \to \_L$ over $f: Y \to X$ induces an equivariant map $\eq{f}$. Moreover, it is obtained as the stackification of an equivariant map in pre-stacks:  
		\[ \pre{\eq{f}} : \pQS{Y}{\_L'} \to \pQS{X}{\_L} \]
		\item Given an equivariant map $\eq{f}$, if it comes from the stackification of an equivariant map in pre-stacks $\pre{\eq{f}}$, then there is a map, $\_L' \to \_L$, of Lie algebroids over $f$. 
		\item  Given an equivariant map $\eq{f}$, if $\_L'$ and $\_L$ integrate well in the sens of Definition \ref{def:good integration}, then there is a morphism of Lie algebroids $\_L' \to \_L$ over $f$.  
	\end{itemize}
\end{Prop}
\begin{proof}
	If we have a morphism of Lie algebroids $\_L' \to \_L$ over $f$, then by definition it factors as: 
	\[ \_L' \to f^! \_L \to \_L \]
	
	From Definition \ref{def:better pullback algebroid} we have a commutative diagram: 
	\[ \begin{tikzcd}
		Y \arrow[r] \arrow[d]  & X  \arrow[d] \\
		\pQS{Y}{f^!\_L} \simeq \comp{\left(\pQS{X}{\_L}\right)_Y} \arrow[r] & \pQS{X}{\_L}
	\end{tikzcd} \]
	
	Moreover the naturality of the infinitesimal quotient construction over the same base gives a commutative diagram: 
	\[ \begin{tikzcd}
		&Y \arrow[dl] \arrow[dr] &\\
		\pQS{Y}{\_L'} \arrow[rr] & & \pQS{Y}{f^!\_L} 
	\end{tikzcd}\]
	
	Therefore we get a commutative diagram: 
	\[ \begin{tikzcd}
		Y \arrow[r] \arrow[d]  & X  \arrow[d] \\
		\pQS{Y}{\_L'} \arrow[r] & \pQS{X}{\_L}
	\end{tikzcd} \]
	
	which induces an equivariant map structure after stackification.\\

	Conversely, given an equivariant map in pre-stack: 
	\[ \begin{tikzcd}
		Y \arrow[r] \arrow[d]  & X  \arrow[d] \\
		\pQS{Y}{\_L'} \arrow[r, "\eq{f}"] & \pQS{X}{\_L}
	\end{tikzcd} \]
	
	We have a factorization of $\eq{f}$ fitting in the commutative diagram: 
	\[ \begin{tikzcd}
		&Y \arrow[r] \arrow[dl] \arrow[d]  & X  \arrow[d] \\
		\pQS{Y}{\_L'} \arrow[r] & \pQS{Y}{f^!\_L} \simeq \comp{\left(\pQS{X}{\_L}\right)_Y}  \arrow[r] & \pQS{X}{\_L}
	\end{tikzcd} \]
	We have this factorization because the formal completion is equivalent to the fiber product: \[ \pQS{X}{\_L} \times_{\left(\pQS{X}{\_L}\right)_{\tx{DR}}} Y_{\tx{DR}} \simeq \pQS{X}{\_L} \times_{\left(\pQS{X}{\_L}\right)_{\tx{DR}}} \left(\pQS{Y}{\_L'}\right)_{\tx{DR}} \] and therefore receives a natural morphism from $\pQS{Y}{\_L'}$.\\
	
	The square is induced by the natural morphism of Lie algebroids $f^! \_L \to \_L$ over $f$ and the commutative triangle on the left induces (by naturality over a fixed base) a morphism of Lie algebroids $\_L' \to f^! \_L$, which is exactly the data of a morphism of Lie algebroids $\_L' \to \_L$ over $f$.\\
	
	For the last statement, observe that under the assumptions on $\_L'$ and $\_L$, we can take: 
	\[ \pre{\eq{f}} := j(\eq{f}) : \pQS{Y}{\_L'} \simeq j\left(\QS{Y}{\_L'}\right) \to  \pQS{X}{\_L} \simeq j\left(\QS{X}{\_L}\right) \] and use the previous result.
\end{proof}

\begin{Prop}
	Let $X$ and $Y$ be stacks satisfying Assumptions \ref{ass:very good stack} and $\_G^\bullet$ a Segal groupoid over $X$ that is a good integration of a Lie algebroid $\_L$. Then we have an equivalence of Lie algebroids: 
	\[ f^! \_L_\_G \simeq \_L_{f^! \_G} \] 
\end{Prop}
\begin{proof}
	It is enough to show that the associated formal groupoids in pre-stacks are equivalent. The formal groupoid in pre-stacks associated to $ f^! \_L_{ \_G}$ is the formal completion: 
	\[ Y \times_{\pQS{Y}{f^! \_L_\_G}} Y \simeq Y \times_{\comp{\left(\pQS{X}{\_L_\_G}\right)_Y}} Y   \]  
	
	Or the other side, the formal groupoid associated to $\_L_{f^! \_G}$ is the groupoid given by: 
	\[\comp{\left(Y \times_{\QS{X}{\_G}} Y\right)} \simeq Y \times_{\comp{\QS{X}{\_G}_Y}} Y \]
	
	We can then conclude by using the equivalences:			
	\[\begin{split}
		\comp{\left(Y \times_{\pQS{X}{\_G}} Y\right)} \simeq & \QS{X}{\_G} \times_{\QS{X}{\_G}_{\tx{DR}}} X_{\tx{DR}} \times_{X_{\tx{DR}}} Y_{\tx{DR}} \\
		\simeq &  \pQS{X}{\_L_\_G} \times_{X_{\tx{DR}}} Y_{\tx{DR}}\\
		\simeq & \pQS{X}{\_L_\_G} \times_{\left( \pQS{X}{\_L_\_G} \right)_{\tx{DR}}} Y_{\tx{DR}} \\
		\simeq & \comp{\left(\pQS{X}{\_L_\_G}\right)_Y}
	\end{split}  \] 
	
\end{proof}
It turns out that pullbacks of an infinitesimal quotient projection are themselves infinitesimal quotient given by an infinitesimal action (by analogy with Lemma \ref{lem:pullback groupoids projection is a groupoid projection}). 

\begin{Lem} 
	\label{lem:pullback and formal thickenings}
	
	Let $X$ be a stack satisfying Assumptions \ref{ass:very good stack}, $Z$ a formal stack and $\_L$ a Lie algebroid on $X$ that integrates well. Consider the following pullback diagram: 
	\[ \begin{tikzcd}
		Y \arrow[r] \arrow[d] & X \arrow[d] \\
		Z \arrow[r] & \QS{X}{\_L}
	\end{tikzcd}\]
	Assume that $Y$ also satisfies Assumptions \ref{ass:very good stack}. Then we have the following: 
	\begin{enumerate}
		\item $Y \to Z$ is a formal thickening and therefore is equivalent to the data of a formal moduli problem $F$. 
		\item There is an equivalence in $\QC(Y)$:
		\[\Ttr{Y}{Z} \overset{\sim}{\to} f^*\_L\]
		giving $f^*\_L$ a structure of Lie algebroid. 
		\item There is an infinitesimal action of $\_L$ on $Y \to X$ such that:
		\[ Z \simeq \QS{Y}{\_L}\]
	\end{enumerate}
\end{Lem}
\begin{proof}\
	
	\begin{enumerate}
		\item Applying $(-)_{\tx{DR}}$ (which is a right adjoint) we have that the morphism:
		\[Y_{\tx{DR}} \to Z_{\tx{DR}}\]
		is the pullback morphism of the equivalence:
		\[X_{\tx{DR}} \to \QS{X}{\_L}_{\tx{DR}}\] 
		Therefore $Y \to Z$ is a formal thickening (since both $Y$ and $Z$ are formal). Since $Y$ satisfies Assumptions \ref{ass:very good stack}, Theorem \ref{th:fmp are formal thickenings} implies that $Z$ is equivalent to the extension of a formal moduli problem under $Y$, $F$. 
		
		\item The relative cotangent of these formal thickenings are equivalent (because of the pullback) and thanks to Lemma \ref{lem:good algebroid relative tangent} (with our assumption that $\_L$ integrates well), we have: 
		\[ \Ttr{X}{\QS{X}{\_L}} \simeq \_L\] 
		Therefore $f^*\_L$ has a structure of Lie algebroid given by the Lie algebroid structure on the relative cotangent $\Ttr{Y}{F}$. Here we use the equivalences: 
		\[ \Ttr{Y}{F} \simeq \Ttr{Y}{\pund{F}} \simeq \Ttr{Y}{j(Z)} \simeq \Ttr{Y}{Z} \simeq f^*\Ttr{X}{\QS{X}{\_L}} \simeq f^*\_L\]
		using Lemma \ref{lem:tangent formal spectrum}, Proposition \ref{prop:formal stack and de rham stack and formal completions} and  Lemma \ref{lem:good algebroid relative tangent}.
		\item Proposition \ref{prop:functoriality relative tangent base change} gives us a morphism of Lie algebroids:
		\[ f^*\_L \simeq \Ttr{Y}{Z} \simeq  \Ttr{Y}{j(Z)} \to \Ttr{X}{j\left(\QS{X}{\_L}\right)} \simeq \Ttr{X}{\QS{X}{\_L}} \simeq \_L \] 
		
		This is exactly the data of an action of $\_L$ on $Y \to X$ such that there is an equivalence of Lie algebroids: 
		\[ f^*\_L  \simeq \Ttrl{Y}{F} \]
		Therefore there is an equivalence: \[ \MC_{f^*\_L} \simeq F\]
		
		This implies that $f^*\_L$ is a Lie algebroid on $Y$ whose infinitesimal quotient is:  \[Z \simeq \und{F} \simeq \QS{Y}{\_L}\]
	\end{enumerate}
\end{proof}

\begin{RQ}\label{rq:pullback algebroid for weak inf quotients}
	If we only assume that $Y$ is affine and almost finitely presented, and if $Z$ can be any derived stack, then $Z$ is by definition a \emph{weak} infinitesimal quotient of $Y$ by $\_L$. This makes the ``weak analogue'' of Lemma \ref{lem:pullback and formal thickenings} tautological. More precisely we get: 
	\begin{itemize}
		\item The morphism $Y \to Z$ is a nil-equivalence. 
		\item If $\_L$ integrates well then from Lemma \ref{lem:good algebroid relative tangent}, there is an equivalence: \[\Ttr{Y}{Z} \overset{\sim}{\to} f^*\_L\]
		\item For any $\_L$ (even if it does not integrate well), $Z$ is a weak infinitesimal quotient along:
		\[ Z := \QSW{Y}{\_L} \]  
	\end{itemize}
\end{RQ}

\begin{Lem}\label{lem:pullback of action quotient} 
	Take a morphism $f:Y\to X$ of affine stacks satisfying Assumptions \ref{ass:very good stack}, and an action of a Lie algebroids $\_L$ on $f$. Then we have a pullback diagram: 
	\[ \begin{tikzcd}
		Y \arrow[r] \arrow[d] & X \arrow[d]\\
		\QS{Y}{\_L} \arrow[r] & \QS{X}{\_L}
	\end{tikzcd} \]
\end{Lem}
\begin{proof}
	It is enough to show that we have a pullback diagram: 
	\[ \begin{tikzcd}
		Y \arrow[r] \arrow[d] & X \arrow[d] \\
		\pQS{Y}{\_L} \arrow[r] & \pQS{X}{\_L}
	\end{tikzcd}\]
	
	which is shown in the proof of Proposition \ref{prop:groupoid action from Lie algebroid action}.
\end{proof}

We will now turn to pullbacks of infinitesimal quotient stacks of Lie algebroids. First recall that any action of the Lie algebroid $\_L$ over $f: X \to Y$ induces a map between their infinitesimal quotient stacks thanks to Proposition \ref{prop:equivariant map are map of lie algebroids}: 
\[ \QS{X}{\_L} \to \QS{Y}{\_L} \] 

We want show by analogy to Proposition \ref{prop:pullbak quotient stack of groupoids} that the fiber product of the infinitesimal quotient stacks is the infinitesimal quotient stack associated to the fiber product of the Lie algebroids.\\ 

In practice, this would requires us to assume that $X$, $Y$, $Z$ and $X \times_Z Y$ all satisfy Assumptions \ref{ass:very good stack} which is very restrictive\footnote{In particular the fiber product of stacks satisfying Assumptions \ref{ass:very good stack} does not necessarily satisfy Assumptions \ref{ass:very good stack}.}.  Therefore we will only make an analogue for \emph{weak} infinitesimal quotients. 

\begin{Prop}
	\label{prop:pullback quotient lie algebroids}
	
	Let $\_L$ be a Lie algebroid over $Z$ satisfying Assumptions \ref{ass:very good stack}. Consider $Y$ and $X$ affine stack of almost finite presentation together with some maps: 
	\[ X \to Z \qquad Y \to Z\]
	
	Take $\QSW{X}{\_L}$ and $\QSW{Y}{\_L}$ to weak infinitesimal quotient along those morphisms. 
	Then we have that:			\[ \QSW{X \times_Z Y}{\_L} := \QSW{X}{\_L} \times_{\QS{Z}{\_L}} \QSW{Y}{\_L}\]
	
	is a weak infinitesimal quotient of along both $X \times_Z Y \to X$ and $X \times_Z Y \to Y$ (in the sens given by Remark \ref{rq:weak inf quotient of weak inf quotient}) inducing the same weak infinitesimal quotient structure along the maps $X \times_Z Y \to Z$.
\end{Prop}
\begin{proof}
	Consider the following commutative cube: 
	\[	\begin{tikzcd}
		& X \times_Z Y \arrow[dl, dashed] \arrow[rr] \arrow[dd] & & X \arrow[dd] \arrow[dl] \\
		\QSW{X}{\_L} \times_{\QS{Z}{\_L}} \QSW{Y}{\_L} \arrow[rr] \arrow[dd] && \QSW{X}{\_L} \arrow[dd] &\\
		& Y \arrow[dl] \arrow[rr] && Z \arrow[dl] \\
		\QSW{Y}{\_L} \arrow[rr] & &\QS{Z}{\_L} &
	\end{tikzcd}\]
	
	We only need to show that all the square involving the dashed arrow are pullback. But this is automatic since all the other squares are pullback squares by definition of weak infinitesimal quotients. 
\end{proof}

\begin{RQ}\label{rq:generalization pullback weak inf quotient}
	In the previous proposition we can replace $Z$ together with its Lie algebroid by $Z'$ affine of almost finite presentation together with a weak infinitesimal quotient: \[ Z' \to \QSW{Z'}{\_L}\]  
\end{RQ}

\begin{RQ}
	If we assume that $X$, $Y$, $Z$ and $X \times_Z Y$ all satisfy Assumptions \ref{ass:very good stack}, then we can replace the weak infinitesimal quotients by the actual infinitesimal quotients and by actual actions. In particular we can show that there is an induced action on the pullback such that: 
	\[ \QS{X \times_Z Y}{\_L} \simeq \QS{X}{\_L} \times_{\QS{Z}{\_L}} \QS{Y}{\_L}\]
\end{RQ}

\subsubsection{Tangent and cotangent of infinitesimal quotient stacks}\ \label{sec:tangent-and-cotangent-of-infinitesimal-quotient-stack}

\medskip

We already know from Lemma \ref{lem:lie algebroid and relative tangent maurer cartan functor} and Corollary \ref{cor:relative tangent complexe formal stack} that for $\_L$ a Lie algebroid over $X$ satisfying Assumptions \ref{ass:very good stack}, we have an equivalence of Lie algebroids: 
\[ \Ttr{X}{\pQS{X}{\_L}} \simeq \_L\]

From this we get, similarly to Proposition \ref{prop:tangent of quotient stack groupoid}, an equivalence: 
\[ p^* \Tt_{\pQS{X}{\_L}} \simeq \Tt_X \oplus^\rho \_L[1] \]

Moreover, if $\_L$ integrates well (in the sens of Definition \ref{def:good integration}) then we also have: 
\[ p^* \Tt_{\QS{X}{\_L}} \simeq \Tt_X \oplus^\rho \_L[1] 
\] 

We can, by analogy to to Corollary \ref{cor:quotient of the underlying linear stack}, that under some good conditions, $T^*\QS{X}{\_L}$ is a \emph{weak} infinitesimal quotient of $\Aa_X\left(p^* \Ll_{\QS{X}{\_L}}\right)$.  

\begin{Lem}\label{lem:pullback of cotangent projection along quotient map}
	Let $\_L$ be a Lie algebroid over $X$ satisfying Assumptions \ref{ass:very good stack} and such that $\_L$ integrates well. Then there is a pullback: 
	\[ \begin{tikzcd}
		\Aa_X(\Ll_A \oplus^\rho \_L^\vee[-1]) \arrow[r] \arrow[d] & T^*\QS{X}{\_L} \arrow[d]\\
		X \arrow[r, "h"] & \QS{X}{\_L}
	\end{tikzcd}  \] 
\end{Lem}
\begin{proof}
	First, this is the pullback of a linear stack, therefore this pullback is the linear stack $\Aa_X\left(p^* \Ll_{\QS{X}{\_L}}\right)$. 	Under Assumptions \ref{ass:very good stack}, we have seen that: 
	\[
	p^*\Ll_{\QS{X}{\_L}} \simeq  \Ll_X \oplus^\rho \_L^\vee[-1]
	\]
	
\end{proof}

\begin{Prop}\label{prop:quotient coadjoint action}
	Assume that $\Ll_A$ is non-negatively graded and almost finitely presented, and that $\_L$ is almost finitely presented, concentrated in non-positive degrees and integrates well. Then $T^*\QS{X}{\_L}$ is a weak infinitesimal quotient of $\Aa_X(\Ll_A \oplus^\rho \_L^\vee[-1])$. 
\end{Prop}
\begin{proof}
	The assumptions ensure that $\Aa_X(\Ll_A \oplus^\rho \_L^\vee[-1])$ is affine of almost finite presentation. Then Proposition \ref{lem:pullback of cotangent projection along quotient map} shows that  $T^*\QS{X}{\_L}$ is a weak infinitesimal quotient of $\Aa_X(\Ll_A \oplus^\rho \_L^\vee[-1])$ according to Definition \ref{def:weak inf quotient action}.
\end{proof}

\subsection{Shifted Moment Maps and Derived Symplectic Reduction}\ \label{sec:shifted-moment-maps-and-derived-symplectic-reduction}

\medskip

In this section, we develop the notion of \emph{moment maps} for Lie algebroids and Segal groupoids. The main motivation is to define \emph{symplectic reduction} in these contexts. \\		 

We start in Section \ref{sec:for-groups} by recalling the classical notion of moment map and symplectic reduction for group actions in the context of derived algebraic geometry as developed in \cite{AC21}. We recall the main important properties of these objects, notably the fact that they behave well under ``good Lagrangian intersections'' (Theorem \ref{th:symplectic reduction commutes with lagrangian intersection groups}). \\

Then in Section \ref{sec:for-lie-algebroid} we generalize these notion and define the notion of \emph{infinitesimal moment map} for the infinitesimal action of a Lie algebroid. We show that this definition naturally contains the notion of \emph{infinitesimal symplectic reduction} and that it is also well behave with respect to ``good Lagrangian intersections''. \\

Finally, we will explain in Section \ref{sec:for-groupoid} that the notion of infinitesimal moment maps and infinitesimal symplectic reduction transfers directly to a notion of moment map and symplectic reduction for Segal groupoids by mimicking the definitions of Section \ref{sec:for-lie-algebroid}. Therefore moment maps and symplectic reductions also have the same properties. 

The context of Segal groupoids is the most general and in particular, it recovers the case of moment maps for group actions (see Example \ref{ex:group moment are groupoid moment map}) and the example of the dual of the anchor of the associated Lie algebroid (Proposition \ref{prop:dual of the anchor is a moment map}). 

Moreover, we will see with Theorem \ref{th:formal completion of moment map and symplectic reductions} that in fact, the infinitesimal moment maps and infinitesimal symplectic reductions are in fact obtained via Construction \ref{cons:lie algebroid from lie groupoid}, Proposition \ref{prop:groupoid and lie algebroid action} and formal completions, from moment maps and symplectic reductions of Segal groupoids. This follows the ideas of Section \ref{sec:derivation-and-integration-of-lie-algebroids} making all the constructions for Lie algebroids the ``infinitesimal versions'' of the corresponding constructions for Segal groupoids.

\subsubsection{For groups}\label{sec:for-groups}\

\medskip

Let $X$ be a smooth symplectic manifold and $G$ a smooth action of a Lie group on $X$ acting by symplectomorphisms. Suppose further that the action is Hamiltonian\footnote{Meaning that the infinitesimal action $\G_g \to \Tt_X$ is given by Hamiltonian vector fields, $v \mapsto \overrightarrow{v}= \lbrace H_v, - \rbrace$ inducing a Lie algebra morphism $\mu^*: \G_g \to \_O_X$ sending $v$ to the Hamiltonian function $H_v$.} so that there exists a moment map: 
\[\mu : X \to \G_g^*\]
This map can be shown to satisfy the following properties: 
\begin{enumerate}
	\item $\mu$ is $G$-equivariant.  
	\item $\mu$ is Hamiltonian in the sens that for all $v \in \G_g$ we have a vector field $\overrightarrow{v} \in \Tt_X$ the image of $v$ under the ``infinitesimal action'' $\G_g \to \Tt_X$. Then the moment condition require $\overrightarrow{v}$ to be Hamiltonian\footnote{This condition is saying that $\mu^*(v)$ is the Hamiltonian vector fields controlling the action of $v$ where $\mu^*: \G_g \to \_O_X$ is the \defi{comoment map}. }: 
	\[ \dr(\langle \mu, v \rangle ) = \iota_{\overrightarrow{v}} \omega_X \]  
\end{enumerate}

We will want to rephrase these conditions. The first condition is unchanged as equivariant maps make sense in derived geometry. 
For the second condition, the following proposition explains that this condition implies that the symplectic structure on $X$ can be viewed as a Lagrangian structure.

\begin{Prop}\
	\label{prop:moment map equivariant is lagrangian}
	If $\mu : X \to \G_g^*$ is a moment map, then it is $G$-equivariant and there is a map:  \[\eq{\mu} : \QS{X}{G} \to \QS{\G_g^*}{G}\] This map  is Lagrangian and there is an equivalence of derived symplectic\footnote{Recall from Proposition \ref{prop:lagrangian on projection to quotient codadjoint action} that the map: \[ \G_g^* \to \QS{\G_g^*}{G}\] is also Lagrangian. Therefore $\QS{X}{G} \times_{\QS{\G_g^*}{G}} \G_g^*$ is the derived intersection of Lagrangian morphisms and is therefore $0$-shifted symplectic from Proposition \ref{prop:lagrangian intersection are shifted symplectic}.} stacks:
	\[X \simeq  \QS{X}{G} \times_{\QS{\G_g^*}{G}} \G_g^*\]
\end{Prop}
\begin{proof}			
	We already now thanks to Lemma \ref{lem:pulback of quotient stack along the projection} that this pullback is $X$ as derived stack. We only need to check that the symplectic structure on the Lagrangian intersection coincides with the symplectic structure on $X$. 
	
	First, recall from Lemma \ref{prop:lagrangian on projection to quotient codadjoint action} that the Lagrangian structure on $\G_g^* \to \QS{\G_g^*}{G}$ is given by the $0$ Lagrangian structure. Moreover, from \cite[Example 1.32]{Ca21}, we have that the moment condition implies that $\omega_X$, the symplectic structure on $X$ defines a Lagrangian structure on $\QS{X}{G} \to \QS{\G_g^*}{G}$. This is due to the second condition of moment maps. Then the construction of the symplectic structure on $X$ from the proof Proposition \ref{prop:lagrangian intersection are shifted symplectic} gives us that the symplectic structure is given by the following loop at zero in the space of closed $2$-forms of degree $1$: 
	\[ 0 \overset{\omega_X}{\rightsquigarrow} 0 \]  
\end{proof}

We will therefore rephrase the second condition by asking $\eq{\mu}$ to be a Lagrangian structure such that the derived intersection: 
\[\QS{X}{G} \times_{\QS{\G_g^*}{G}} \G_g^*\] 
is equivalent to $X$ as a \emph{symplectic} stack. We know that they are already equivalent as derived stacks thanks to Lemma \ref{lem:pulback of quotient stack along the projection} and the symplectic part is an extra structure required by the moment map. 
\begin{Def}[{\cite[Definition 2.3]{AC21}}]
	\label{def:moment map structure}
	
	Let $X$ be a $n$-shifted symplectic Artin stack together with a smooth action of $G$. Given a $G$-equivariant map $\mu: X \to \G_g^*[n]$, then the structure of a \defi{$n$-shifted moment map on $\mu$} is a Lagrangian structure on: 
	\[ \eq{\mu}: \QS{X}{G} \to \QS{\G_g^*[n]}{G}\]
	
	such that there is an equivalence of $n$-shifted symplectic stacks: \[X \simeq  \QS{X}{G} \times_{\QS{\G_g^*[n]}{G}} \G_g^*[n]\] 
\end{Def} 

\begin{RQ}
	In differential geometry, the conditions to be a moment map are \emph{properties} of the map. As often in derived geometry, these properties become extra structures. Indeed, being equivariant is the \emph{structure} of an equivariant map, as we have seen in Definition \ref{def:groupoid in derived stacks} and the second condition becomes the data of a Lagrangian \emph{structure}. Therefore being a moment map will be a \emph{structure} on the map $\mu$, and as often, we will keep saying that $\mu$ \emph{is} a moment map to refer to $\mu$ together with a structure of moment map. 
\end{RQ}

We are interested in the procedure of symplectic reduction. Classically, it amounts (under the condition that these constructions exist) to do the following: 
\begin{itemize}
	\item Take the zero locus\footnote{This can be generalized to taking the pre-image of any coadjoint orbit of $\G_g^*[n]$ (see \cite[Section 2.1.2]{Ca21}). We are not going to consider this degree of generality.} of the moment map:
	\[ \begin{tikzcd}
		Z(\mu ) \arrow[r] \arrow[d] & \star \arrow[d] \\
		X \arrow[r, "\mu"'] & \G_g^*[n]
	\end{tikzcd}\] 
	\item Take its quotient by the induced action of $G$: 
	\[ \red{X} := \QS{Z(\mu)}{G}\]
	\item Construct a symplectic structure on $\red{X}$ with a symplectic structure $\omega$ such that $p^*\omega = i^*\omega_X$ with:
	\[p:Z(\mu) \to \QS{Z(\mu)}{G} \qquad i: Z(\mu) \to X \]
\end{itemize}

In classical geometry, this procedure only exists under some assumptions (ensuring the existence of the zero locus and quotient). However, in derived geometry these constructions are always possible and generalize the classical constructions. 

\begin{Def}[{\cite[Definition 2.3]{AC21}}] 
	\label{def:sympletic reduction}
	
	We define the \defi{symplectic reduction of $X$} along the $n$-shifted moment map $\mu : X \to \G_g^*[n]$ to be the fiber product: 
	\[ \begin{tikzcd}
		\red{X} \arrow[r] \arrow[d] & \bf{B}G \arrow[d]\\
		\QS{X}{G} \arrow[r, "\eq{\mu}"'] & \QS{\G_g^*[n]}{G}
	\end{tikzcd} \]
	Since it is a derived intersection of Lagrangian in a $(n+1)$-shifted symplectic derived stack, it is naturally $n$-shifted symplectic. 
\end{Def}

\begin{RQ}
	Taking this pullback took care of all 3 steps of the classical construction at once. Indeed, from Proposition \ref{prop:pullbak quotient stack of groupoids}, this pullback is equivalent to: \[ \QS{X \times_{\G_g^*[n]} \star}{G} \]
	Therefore $\red{X}$ is exactly the quotient of the fiber of $\mu$. And the symplectic structure comes for free from the fact that it is a derived intersection of Lagrangian morphisms (Propositions \ref{prop:moment map equivariant is lagrangian} and \ref{prop:lagrangian on zero section cotangent quotient stack}). 
	
	Moreover, since the Lagrangian structure on $\bf{B}G \to \QS{\G_g^*[n]}{G}$ is the zero structure, and the Lagrangian structure on the moment map is given by the symplectic structure on $X$, the symplectic structure on $\red{X}$ is in some sens ``induced'' by the symplectic structure on $X$. 
\end{RQ}

The following proposition is one of the main observation that will motivate the definition generalizing of the notion of moment map for Lie algebroids in Section \ref{sec:for-lie-algebroid}. Moreover Sections \ref{sec:for-lie-algebroid} and \ref{sec:derived-perspective-of-the-bv-complex} will motivate the fact that these Lagrangian correspondences will be the defining feature of a generalized notion of symplectic reduction (Definitions \ref{def:generalized symplectic reduction groupoid} and \ref{def:generalized symplectic reduction}), viewed as a procedure to obtain symplectic stacks from quotients of ``almost derived critical loci''.   

\begin{Prop}[{\cite[Remark 2.4]{AC21}}] \
	\label{prop:symplectic reduction lagrangian correspondence}
	
	If $\mu: X \to \G_g^*[n]$ is a $n$-shifted moment map, there is a Lagrangian correspondence: 
	\[ \begin{tikzcd}
		& Z(\mu) \arrow[dl] \arrow[dr] & \\
		X_{\tx{red}} & & X
	\end{tikzcd}\]
\end{Prop}

\begin{Ex}\label{ex:symplectic reduction cotangent moment map}
	If we consider the moment map $\mu: T^*X \to \G_g^*$ from the cotangent action of $G$ on $T^*X$. Then its symplectic reduction is given, thanks to Proposition \ref{prop:cotangent quotient groupoid as a derived intersection} by $T^*\QS{X}{G}$. 
\end{Ex}

Following the ideas from Section \ref{sec:new-constructions-from-old-ones}, it turns out that we can construct moment maps by a ``Lagrangian intersection'' procedure and that symplectic reduction commutes with these Lagrangian intersections. To sketch the main ideas, \cite[Definition 2.8]{AC21} gives us a notion of \defi{derived symplectic reduction} of a \emph{Lagrangian structure} on $L \to X$.

To express this definition, we will use the $1$-category of Lagrangians over the $(n+1)$-shifted point\footnote{We consider the $1$-category of Lagrangians over the point. Objects are then $n$-shifted symplectic stacks and morphisms are (equivalence classes of) Lagrangian correspondences.}, $\tx{Lag}_1(\star_{n+1})$ (see Section \ref{sec:the-higher-categories-of-lagrangians}). We can rephrase their definition as follows:

\begin{Def}\label{def:lagrangian reduction group action}
	A \defi{Lagrangian reduction of a Lagrangian structure} on $L \to X$  is given by a factorization $\red{L} \in \tx{Lag}_1(\star_{n+1})(\star_{n+1}, \red{X})$  of $L \in \tx{Lag}_1(\star_{n+1})(\star_{n}, X)$ by $Z(\mu) \in \tx{Lag}_1(\star_{n})(X_{\tx{red}}, X)$: 
	
	\[ \begin{tikzcd}
		&& L \arrow[dl] \arrow[dr]& & \\
		& L_{\tx{red}} \arrow[dl] \arrow[dr] & & Z(\mu) \arrow[dl] \arrow[dr]& \\
		\star && X_{\tx{red}} & & X 
	\end{tikzcd}\]  
\end{Def}

\begin{RQ}
	\cite[Definition 2.8]{AC21} is spelled out differently but is in fact equivalent to asking for this factorization. Moreover, thanks to Proposition \ref{lem:pullback and formal thickenings}, the \emph{Lagrangian reduction} is obtained as a quotient for an action of $G$ on $L$: 
	\[ \red{L} \simeq \QS{L}{G}\]
\end{RQ}

\begin{Th}[{\cite[Theorem C]{AC21}}]\
	\label{th:symplectic reduction commutes with lagrangian intersection groups}
	
	Let $\mu: X \to \G_g^*[n]$ be a $n$-shifted moment map. Let $L, L'$ be Lagrangians in $X$ and $L_{\tx{red}}, L_{\tx{red}}'$ derived symplectic reductions of $L$ and $L'$. Then the moment map on $X$ induces a structure of moment map on: 
	\[ \mu_{-1} : L \times_X L' \to \star \times_{\G_g^*[n]} \star \simeq \G_g^*[n-1] \]
	
	Moreover there is a natural equivalence of $(n-1)$-shifted symplectic stacks: 
	\[ \left( L \times_X L' \right)_{\tx{red}} \simeq L_{\tx{red}} \times_{X_{\tx{red}}} L_{\tx{red}}' \]
\end{Th}

\begin{Ex} 
	\label{ex:moment map derived critical locus}
	
	The derived critical locus of a $G$-equivariant function is equipped with a $(-1)$-shifted moment map: 
	\[ \RCrit(f) \to \G_g^*[-1] \]
	
	Indeed, the induced action on $T^*X$ admit a moment map $\mu: T^*X \to \G_g^*$ whose symplectic reduction is $T^*\QS{X}{G}$ (see Example \ref{ex:symplectic reduction cotangent moment map}) and the Lagrangians $0, df: X \to T^*X$ have a symplectic reductions given by the natural Lagrangian structures on $[0]$ and $d[f]$: \[X_{\tx{red}} := \QS{X}{G} \to (T^*X)_{\tx{red}} \simeq T^*\QS{X}{G}\]

	where the last equivalence is a consequence of Proposition \ref{prop:cotangent quotient groupoid as a derived intersection}. Then Theorem \ref{th:symplectic reduction commutes with lagrangian intersection groups} implies that $\RCrit(f) \to \G_g^*[-1]$ is a $(-1)$-shifted moment map and its symplectic reduction is: 
	\[ \RCrit(f)_{\tx{red}} \simeq \RCrit([f]) \]
\end{Ex}

\subsubsection{For Lie algebroids}\
\label{sec:for-lie-algebroid}

\medskip 

We want to generalize the definition of a moment map for infinitesimal actions of Lie algebroids. As we are working with Lie algebroids again, we will assume that $X = \Spec(A)$ satisfies Assumptions \ref{ass:very good stack}.\\

The goal is two-fold, we want to recover the infinitesimal counterpart of Section \ref{sec:for-groups} (moment maps for groups actions) and extend it to infinitesimal actions of Lie algebroids. The critical example is given by the dual of the anchor map $\rho^* : T^*X \to L^*$ (Proposition \ref{prop:moment map structure on dual of anchor}) which we think of as an analogue of the cotangent moment map $\mu :T^*X \to \G_g^*$. \\

Given a map $\mu : Y \to L^*[n]$ of derived stacks\footnote{Recall that by convention $L = \Aa_X(\_L)$ and $L^* := \Aa_X(\_L^\vee)$.} over $X$, with $Y \to X$ where $Y$ is affine almost finitely presented and $X$ satisfies Assumptions \ref{ass:very good stack}, we ask ourself how to define a structure of an infinitesimal $n$-shifted moment map on $\mu$ when $\_L$ is a Lie algebroid over $X$.\\

The naive approach would be to try to mimic directly the definition for group action by asking for a Lagrangian structure on the infinitesimal quotient. This fails immediately as we do not have a natural infinitesimal action of $\_L$ on $L^*[n]$ giving an analogue to Lemma \ref{lem:cotagent BG and coadjoint quotient}, and there is also no obvious choice of an infinitesimal action of $\_L$ on $Y$.\\

However, it turns out that for $\mu : T^*X \to L^*$ the dual of an anchor, the fiber $Z(\mu)$ of the infinitesimal moment map admits a natural infinitesimal action of $\_L$ (up to homotopy) given by Proposition \ref{prop:quotient coadjoint action}. Using this fact and Proposition \ref{prop:symplectic reduction lagrangian correspondence} leads to the following definition: 

\begin{Def} 
	\label{def:moment map lie algebroid}
	Take $\_L$ a Lie algebroid on $X$, and a map $\mu : Y \to L^*[n]$ of stacks over $X$ where $Y$ is affine almost finitely presented and $X$ satisfies Assumptions \ref{ass:very good stack}. We suppose that the fiber of the moment map, denoted $Z(\mu)$, is also affine of almost finite presentation\footnote{For example if $Y = \Aa_X(\_E)$ with $\_E$ non-negatively graded and finitely presented, $\mu$ is a linear map of linear stacks and $\_L$ is non-positively graded, then $Z(\mu)$ is affine of almost finite presentation.}. Then the structure of an \defi{infinitesimal $n$-shifted moment map on $Y$}, where $Y$ is $n$-shifted symplectic, is the data of: 
	\begin{itemize}
		\item A weak infinitesimal quotient $\QSW{Z(\mu)}{\_L}$
		\item A Lagrangian correspondence: 
		
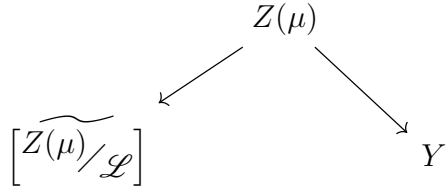
\[ \begin{tikzcd}
			& Z(\mu) \arrow[dl] \arrow[dr] & \\
			\QSW{Z(\mu)}{\_L}&& Y
		\end{tikzcd}\]
	\end{itemize}
	
	In that case, the \defi{infinitesimal symplectic reduction} of $\mu$ is defined to be:
	\[ Y_{\tx{red}} := \QSW{Z(\mu)}{\_L}\]
\end{Def}

\begin{RQ}
	If $Z(\mu)$ satisfies Assumptions \ref{ass:very good stack}, then we can instead ask for an action of $\_L$ on $Z(\mu) \to X$ and take:
	\[ \QSW{Z(\mu)}{\_L} := \QS{Z(\mu)}{\_L}\]
\end{RQ}

\begin{RQ}\label{rq:generalized infinitesimal moment maps}
	In this definition, the structure of Lie algebroid does not play any role on $L^*[n]$. In fact we can generalize this definition to any map $\mu: Y \to Z$ where $Z$ is a linear stack (or more generally, to any derived stack over $X$ together with a choice of a ``0 section'')\footnote{The requirement of this ``0 section'' is there to make sense of the fiber $Z(\mu)$.}.  
\end{RQ}

\begin{RQ}
	In general, the symplectic reduction is part the \emph{data} of the moment map $\mu$ due to the non-uniqueness of weak infinitesimal quotient.  
\end{RQ}

\begin{Ex}\label{ex:infintesimal moment map group action}
	If $\mu : X \to \G_g^*$ is a moment map, then it is an infinitesimal moment map. We view $\G_g$ as a Lie algebroid over the point $\star$, then $\G_g$ acts infinitesimally on $\G_g^*$ and $X$ (induced by the actions of $G$) and therefore the fiber $Z(\mu)$ also has a canonical infinitesimal action on $Z(\mu)$. 
	
	This is exactly the infinitesimal version of Section \ref{sec:for-groups}, and we will see with Lemma \ref{lem:formal completion lagrangian correspondence remains a lagrangian correpspondence} that the infinitesimal version of the Lagrangian correspondence of Proposition \ref{prop:symplectic reduction lagrangian correspondence} is a Lagrangian correspondence: 
	
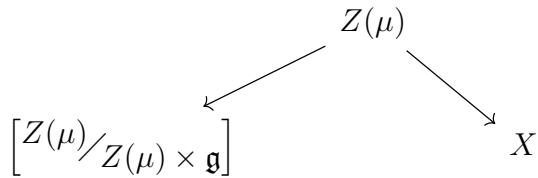
\[ \begin{tikzcd}
		& Z(\mu) \arrow[dr] \arrow[dl] & \\
		\QS{Z(\mu)}{Z(\mu) \times \G_g} & & X
	\end{tikzcd}\] 
	This gives $\mu : X \to \G_g^*$ the structure of an infinitesimal moment map. This is exactly the infinitesimal action we obtain by applying Theorem \ref{th:formal completion of moment map and symplectic reductions} to the moment maps for group actions.
\end{Ex}

Heuristically, the anchor map $\_L \to \Tt_X$ is supposed to represent the infinitesimal action of $\_L$ on $X$. Similarly to the fact that for a group action, the map $T^*X \to \G_g^*$ is a moment map, we want to prove that the map $T^*X \to L^*$ is an infinitesimal moment map.

\begin{Prop}
	\label{prop:moment map structure on dual of anchor}
	Assume that $X$ satisfies Assumptions \ref{ass:very good stack} such that $\Ll_X$ is non-negatively graded and finitely presented. Suppose that $\_L$ is a non-positively graded, almost finitely presented Lie algebroid that integrates well. If the canonical closed $2$-form on $T^*\QS{X}{\_L}$ is symplectic\footnote{It is in particular the case if $\_L$ integrates to a smooth Segal groupoid thanks to Proposition \ref{prop:formal completion and cotangent}.}, then the dual of the anchor $\rho^* : T^*X \to L^*$ can be equipped with the structure of an infinitesimal moment map with weak infinitesimal quotient given by: \[(T^*X)_{\tx{red}} \simeq T^*\QS{X}{\_L}\] 
\end{Prop}
\begin{proof}
	We can show that $Z(\rho^*) \simeq \Aa_X (\Ll_X \oplus^{\rho^*} \_L^\vee[-1])$. The assumptions on $\Ll_X$ and $\_L$ ensure that $Z(\rho^*)$ is affine of almost finite presentation so it makes sense to speak of moment map structure.  \\
	
	Moreover, Proposition \ref{prop:quotient coadjoint action} shows that $Z(\rho^*)$ admits a weak infinitesimal quotient given by: 
	\[ \QSW{Z(\rho^*)}{\_L} \simeq T^*\QS{X}{\_L}  \]
	It has a canonical symplectic structure.\\
	
	We now only have to find the structure of a Lagrangian correspondence: 
	\[ \begin{tikzcd}
		& Z(\rho^*) \arrow[dl, "f"'] \arrow[dr, "g"] & \\
		T^*\QS{X}{\_L} & & T^*X
	\end{tikzcd}\] 	 
	
	Both symplectic structures being the canonical ones, their pullback are in fact \emph{equal}\footnote{In the model given by: \[ Z(\mu) = \Aa_X\left(p^* \Ll_{\QS{X}{\_L}}\right) \simeq \Aa_X(\pi^* (\Ll_X \oplus^{\rho^*} \_L^\vee[-1]))\]}, and we can choose the $0$ isotropic correspondence. To show the non-degeneracy, we need to show that the following square is a pullback: 
	\[ \begin{tikzcd}
		\Tt_{Z(\rho^*)} \arrow[r] \arrow[d] & g^*\Tt_{T^*X} \arrow[d] \\
		f^*\Tt_{T^*\QS{X}{\_L}} \arrow[r] & \Ll_{Z(\rho^*)}
	\end{tikzcd} \]
	
	Taking the fiber of this commutative diagram we can show that we have:  
	
	\[ \begin{tikzcd}
		\pi^* \_L \arrow[d] \arrow[r, "\sim"] & \pi^* \_L \arrow[d] \\
		\Tt_{Z(\rho^*)} \arrow[r] \arrow[d] & g^*\Tt_{T^*X} \arrow[d] \\
		f^*\Tt_{T^*\QS{X}{\_L}} \arrow[r] & \Ll_{Z(\rho^*)}
	\end{tikzcd} \]
	
	where the fiber on the left is computed thanks to thanks to Corollary \ref{cor:relative tangent weak inf quotient} (since $\_L$ integrates well) and the fiber on the right is computed using connections on $T^*X$ and $Z(\rho^*)$.\\ 
	
	In general the fiber of a pullback in a stable category are equivalent if and only square on the right is a pullback: 
	\[ \begin{tikzcd}
		\tx{fib}(A \to B) \arrow[r] \arrow[d, "\sim"] & A \arrow[r] \arrow[d] & B \arrow[d] \\
		\tx{fib}(C \to D) \arrow[r] & C \arrow[r] & D 
	\end{tikzcd}\]
	
	We can see that by considering the following diagram: 
	
	\[ \begin{tikzcd}
		\tx{fib}(A \to B) \simeq 	\tx{fib}(C \to D)  \arrow[r] \arrow[d, "\sim"] & A \arrow[r] \arrow[d] & C \arrow[d] \\
		0 \arrow[r] & B \arrow[r] & D 
	\end{tikzcd}\]
	Then the left square is a pullback and the outer square is a pullback if and only if the fibers are equivalent. Then it is a general property of the  pullback in a stable category that if the outer square and the left square are Cartesians then the right square is also Cartesian. 		
\end{proof}

Now we are going to see that similarly to the situation for group action, this notion of moment map is well behaved with respect to Lagrangian intersections. To do so we will need an analogue to the notion of symplectic reduction of a Lagrangian morphism (Definition \ref{def:lagrangian reduction group action}).

\begin{Def}\label{def:symplectic reduction lagrangian morphism algebroid}
	Consider a Lagrangian $L_1 \to Y$ with a Lagrangian structure and $L_1$ affine of almost finite presentation. 
	An \defi{infinitesimal symplectic reduction of the Lagrangian $L_1 \to Y$}  with $Y$ equipped with an infinitesimal $n$-shifted moment map $\mu: Y \to L^*[n]$, is the data of a factorization of the Lagrangian morphism $L_1 \in \tx{Lag}_1(\star_{(n+1)})(\star_n, Y)$ by $Z(\mu) \in \tx{Lag}_1(\star_{(n+1)})(\red{Y}, Y)$ (see Definition \ref{def:lagrangian 1 category}).
	\[ \begin{tikzcd}
		&& L_1 \arrow[dl] \arrow[dr]& & \\
		& (L_1)_{\tx{red}} \arrow[dl] \arrow[dr] & & Z(\mu) \arrow[dl] \arrow[dr]& \\
		\star && Y_{\tx{red}} & & Y
	\end{tikzcd}\] 
\end{Def}

\begin{RQ}\ \label{rq:description of Lagrangian symplectic reduction}
	Definition \ref{def:symplectic reduction lagrangian morphism algebroid} is exactly the same as Definition \ref{def:lagrangian reduction group action}. We can however describe more explicitly the important features of this definition.
	
	By definition, an infinitesimal symplectic reduction of a Lagrangian $L_1 \to Y$ with an infinitesimal $n$-shifted moment map $\mu: Y \to L^*[n]$ is the data of: 
	\begin{itemize}
		\item A factorization of $L_1 \to Y$: 
		\[\begin{tikzcd}
			L_1  \arrow[r] & Z(\mu) \arrow[r] & Y 
		\end{tikzcd}\]
		\item A weak infinitesimal quotient $\red{(L_1)}$ of the map $L_1 \to Z(\mu)$ compatible with $Z(\mu)\to \red{X}$ according to Remark \ref{rq:weak inf quotient of weak inf quotient}.
		\item A Lagrangian structure on: 
		\[ \red{(L_1)} \simeq \QS{Z(\mu)}{\_L} \to \red{X}\]
	\end{itemize} 
\end{RQ}

We are now able to prove that the ``Lagrangian intersection'' of a moment map is a shifted moment map:

\begin{Th}
	\label{th:lagrangian intersection of moment map lie for lie algebroids}
	
	Let $\mu : Y \to L^*[n]$ be an infinitesimal $n$-shifted moment map and $L_1 \to Y$ and $L_2 \to Y$ Lagrangian morphisms with infinitesimal symplectic reductions. Then the infinitesimal moment map $\mu$ induces an infinitesimal moment map: 
	\[ \mu_{-1} : L_1 \times_Y L_2 \to L^*[n-1]\]
	
	such that the reduction is given by:
	\[ \red{\left(L_1 \times_Y L_2\right)} \simeq \red{(L_1)} \times_{\red{Y}} \red{(L_2)} \]
\end{Th}
\begin{proof}\
	
	\begin{itemize}
		\item We need a weak infinitesimal quotient of $Z(\mu_{-1})$. Since $L_1$, $L_2$ and $Z(\mu)$ all have a weak infinitesimal quotient, then from Proposition \ref{prop:pullback quotient lie algebroids} and Remark \ref{rq:generalization pullback weak inf quotient}, their pullback defines a weak infinitesimal quotient\footnote{This is a weak infinitesimal quotient of $L_1$, $L_2$ and $Z(\mu)$ and by extension also of $X$.} such that:
		
		\[ 
		\red{\left(L_1 \times_Y L_2\right)} := \QSW{Z(\mu_{-1})}{\_L} \simeq \red{(L_1)} \times_{\red{Y}} \red{(L_2)}
		\]
		Moreover, this is a derived intersection of Lagrangians and therefore it is $(n-1)$-shifted symplectic. 
		\item We are in the situation of Theorem \ref{th:lagragian correspondence pullback}: 
		\[ \begin{tikzcd}
			& L_1 \arrow[dl] \arrow[dr] \arrow[d] & \\
			\red{(L_1)} \arrow[d] &  Z(\mu) \arrow[dr] \arrow[dl] & L_1  \arrow[d] \\
			\QSW{Z(\mu)}{\_L} & L_2 \arrow[dr] \arrow[dl] \arrow[u]& Y \\
			\red{(L_2)} \arrow[u] & & L_2 \arrow[u] 
		\end{tikzcd}\]
		The maps $\red{(L_i)}\to \QSW{Z(\mu)}{\_L}$ are Lagrangian and from Definition \ref{def:symplectic reduction lagrangian morphism algebroid} we have the equivalences:
		\begin{equation}\label{eq:pullback symplectic reduction}
			L_i \simeq \red{(L_i)} \times_{\QSW{Z(\mu)}{\_L}} Z(\mu)
		\end{equation} 
		
		Moreover the vertical morphisms on the left and right sides are all clearly Lagrangians. Therefore we get a Lagrangian correspondence: 
		
		\adjustbox{scale=0.65,center}{
			\begin{tikzcd}[column sep = 1mm]
				& \red{(L_1)} \times_{\QS{Z(\mu)}{\_L}} Z(\mu) \times_{\QS{Z(\mu)}{\_L}} \red{(L_2)} \arrow[dl] \arrow[dr]& \\
				\red{(L_1)} \times_{\QSW{Z(\mu)}{\_L}} \red{(L_2)} & & \left(\red{(L_1)} \times_{\QSW{Z(\mu)}{\_L}} Z(\mu)\right) \times_E \left(\red{(L_2)} \times_{\QSW{Z(\mu)}{\_L}} Z(\mu)\right)
		\end{tikzcd}}

		\begin{itemize}
			\item Using Equation \eqref{eq:pullback symplectic reduction} we get the equivalences:
			\[ \begin{split}
				&	\red{(L_1)} \times_{\QSW{Z(\mu)}{\_L}} Z(\mu) \times_{\QSW{Z(\mu)}{\_L}} \red{(L_2)}  \\
				\simeq & \red{(L_1)} \times_{\QSW{Z(\mu)}{\_L}} Z(\mu)  \times_{Z(\mu)} Z(\mu)\times_{\QSW{Z(\mu)}{\_L}} \red{(L_2)}\\
				\simeq & L_1 \times_{Z(\mu)} L_2 \\
				\simeq & Z(\mu_{-1})
			\end{split}\]
			
			\item Using again Equation \eqref{eq:pullback symplectic reduction} we get the equivalence:\\
			
			\adjustbox{scale=0.85,center}{ 
				$\left(\red{(L_1)} \times_{\QSW{Z(\mu)}{\_L}} Z(\mu)\right) \times_Y \left(\red{(L_2)} \times_{\QSW{Z(\mu)}{\_L}} Z(\mu)\right) \simeq L_1 \times_Y L_2$
			}\\
			
			\item Using Remark \ref{rq:description of Lagrangian symplectic reduction} and Proposition \ref{prop:pullback quotient lie algebroids} we get:
			\[ \begin{split}
				\red{(L_1)} \times_{\QSW{Z(\mu)}{\_L}} \red{(L_2)} \simeq  & \QSW{L_1}{\_L}  \times_{\QSW{Z(\mu)}{\_L}} \QSW{L_2}{\_L} \\
				\simeq & \QSW{L_1 \times_{Z(\mu)} L_2}{\_L}\\
				\simeq & \QSW{Z(\mu_{-1})}{\_L}
			\end{split}\]
		\end{itemize}
		
		In short we get the following Lagrangian correspondence: 
		\[ \begin{tikzcd}
			& Z(\mu_{-1}) \arrow[dl] \arrow[dr]& \\
			\red{(L_1 \times_Y L_2)} \simeq \QSW{Z(\mu_{-1})}{\_L} & & L_1 \times_Y L_2
		\end{tikzcd}\]
	\end{itemize}
\end{proof}

\begin{Ex}\label{ex:shifted moment map algebroid derived critical locus} 
	Take $\_L$ a Lie algebroid over $X$ an affine stack satisfying Assumptions \ref{ass:very good stack} such that $\_L$ integrates well, and consider $L_1 = L_2 =X$ with Lagrangian morphisms $s_0, df: X \to T^*X$ where $s_0$ is the zero section. We have that $X$ has an infinitesimal symplectic reduction (for both Lagrangian morphisms), with respect to the infinitesimal moment map structure on $\rho^*: T^*X \to L^*$ given by Proposition \ref{prop:moment map structure on dual of anchor}, if and only if the map:
	\[X \overset{df}{\to} T^*X \overset{\rho^*}{\to} L^* \]
	
	is homotopic to zero\footnote{The same condition is necessary for $s_0$ to have a Lagrangian reduction, but this is automatic.}. Under that condition, we have factorization of Lagrangians:
	\[ \begin{tikzcd}
		&& X \arrow[dl] \arrow[dr, "\widetilde{df}"]& & \\
		& \QS{X}{\_L} \arrow[dl] \arrow[dr, "d\eq{f}"] & & Z(\rho^*) \arrow[dl] \arrow[dr]& \\
		\star && \red{T^*X}\simeq T^*\QS{X}{\_L} & & T^*X 
	\end{tikzcd}\] 
	where $\widetilde{df}$ is the natural factorization of $df$  obtained thanks to the condition $\rho^*\circ df \simeq 0$. This square is a pullback because the outer and the right squares (see Lemma \ref{lem:pullback of cotangent projection along quotient map}) in the following commutative diagram are pullbacks: 
	\[ \begin{tikzcd}
		X \arrow[r] \arrow[d] & Z(\rho^*) \arrow[r] \arrow[d] & X \arrow[d] \\
		\QS{X}{\_L} \arrow[r] & T^*\QS{X}{\_L} \arrow[r] & \QS{X}{\_L}
	\end{tikzcd} \]
	
	It follows from Definition \ref{def:symplectic reduction lagrangian morphism algebroid} that $X \to T^*X$ is a Lagrangian morphism with infinitesimal symplectic reduction given by: \[d\eq{f}: \QS{X}{\_L} \to T^*\QS{X}{\_L}\]
	
	This also works similarly for $s_0$ and we can use Theorem \ref{th:lagrangian intersection of moment map lie for lie algebroids} to show that: 
	\begin{itemize}
		\item The map $\RCrit(f) \to L^*[-1]$ is an infinitesimal $(-1)$-shifted moment map. 
		\item Its infinitesimal symplectic reduction is given by: 
		\[ \red{\RCrit(f)} \simeq \QS{X}{\_L} \times_{T^*\QS{X}{\_L}} \QS{X}{\_L} := \RCrit(\eq{f})  \]
	\end{itemize} 
\end{Ex}

\subsubsection{For groupoids} \label{sec:for-groupoid}\

\medskip

We will now give a definition of a moment map for actions of \emph{Segal groupoids} and then see that ``infinitesimal'' moment maps from Section \ref{sec:for-lie-algebroid} are the ``infinitesimal version'' of the moment maps for Segal groupoids (see Theorem \ref{th:formal completion of moment map and symplectic reductions}). \\

The critical examples we want to recover are the moment maps for a group actions (Example \ref{ex:group moment are groupoid moment map}) and the dual of the anchor $T^*X \to L^*[n]$ (Proposition \ref{prop:dual of the anchor is a moment map}) of a Lie algebroid. \\

\begin{RQ}\label{rq:anchored bundle target moment map}
	Recall that for any Segal groupoid $\_G^\bullet$ over $X$ a derived stack, such that the cotangent complexes of each $\_G^n$  exists, then we can define the anchored linear stack: 
	\[ L := \Aa_X\left(\Ttr{X}{\QS{X}{\_G}}\right) \to TX\]
	
	In this degree of generality, we do not know if this is a Lie algebroid, but in order to mimic the role of $L^*[n]$ in the previous section, having just an anchored linear stack is enough.
\end{RQ} 

The notion of moment map in this situation can be directly adapted from the infinitesimal case.  

\begin{Def} \
	\label{def:moment map lie groupoid}
	Take $\_G^\bullet$ a Segal groupoid over $X$ such that the cotangent complexes of each $\_G^n$  exists, and $\mu : Y \to L^*[n]$ a map of derived stacks over $X$, where $L$ is the linear stack of Remark \ref{rq:anchored bundle target moment map}. Assume that $Y$ is $n$-shifted symplectic. Then the structure of a $n$-shifted moment map on $\mu$ is the data of: 
	\begin{itemize}
		\item An action of $\_G^\bullet$ on $Z(\mu)$ (Definition \ref{def:groupoid action}). 
		\item A $n$-shifted symplectic structure on $\QS{Z(\mu)}{\_G}$.
		\item A Lagrangian correspondence: 
		\[ \begin{tikzcd}
			& Z(\mu) \arrow[dl] \arrow[dr] & \\
			\QS{Z(\mu)}{\_G}&& Y
		\end{tikzcd}\]
	\end{itemize}
	In that case, the \defi{symplectic reduction} of $\mu$ is defined to be:
	\[ Y_{\tx{red}} := \QS{Z(\mu)}{\_G}\]
\end{Def}

\begin{Ex}\label{ex:group moment are groupoid moment map}
	In the situation where $\_G^\bullet := G^{\times \bullet}$ seen as a groupoid over $\star$, we consider a smooth action of $G$ on $X$ a derived Artin stack with moment map $\mu: X \to \G_g^*$ (in the sens of Definition \ref{def:moment map structure}). Clearly, $\G_g$ is the Lie algebroid associated to $G^{\times \bullet}$. 
	Since $G$ acts on both $X$ and $\G_g^*$, it acts on $Z(\mu)$, and Proposition \ref{prop:symplectic reduction lagrangian correspondence} gives $\mu$ the structure of a moment map in the sens of Definition \ref{def:moment map lie groupoid}.  
\end{Ex}

\begin{Prop}\label{prop:dual of the anchor is a moment map}
	
	If we have a Segal groupoid $\_G^\bullet$ over $X$ as in Definition \ref{def:moment map lie groupoid} and $L$ as in Remark \ref{rq:anchored bundle target moment map}. Then the dual of the anchor map $\mu: T^*X \to L^*$ has a moment map structure such that its symplectic reduction is $T^*\QS{X}{\_G}$ whenever the canonical closed $2$-form on $T^*\QS{X}{\_G}$ is symplectic\footnote{This is the case whenever $\_G^\bullet$ is a smooth Segal groupoid since the quotient is Artin.}.
\end{Prop}
\begin{proof}
	
	Notice that there is an equivalence $Z(\mu) \simeq \Aa_X\left(p^* \Ll_{\QS{X}{\_G}}\right)$ (thanks to Proposition \ref{prop:tangent of quotient stack groupoid}) and therefore $Z(\mu)$ is part of the pullback diagram: 
	
	\[ \begin{tikzcd}
		Z(\mu) \arrow[r, "\pi"] \arrow[d] & X \arrow[d] \\
		T^*\QS{X}{\_G} \arrow[r] & \QS{X}{\_G}
	\end{tikzcd}\] 
	
	Therefore from Lemma \ref{lem:pullback groupoids projection is a groupoid projection}, there is an action of $\_G$ on $Z(\mu)$ whose quotient is $T^*\QS{X}{\_G}$. We are left to show that there is a Lagrangian correspondence: 
	\[\begin{tikzcd}
		& Z(\mu) \arrow[dl, "h"'] \arrow[dr, "i"] & \\
		T^*\QS{X}{\_G} && T^*X
	\end{tikzcd}\]
	
	It is not hard to see that the pullbacks of the (canonical) symplectic structures to $Z(\mu)$ coincide on the nose and therefore we can consider the $0$ isotropic correspondence. We want to show that it is non-degenerate. This amounts to showing that the following square is Cartesian: 
	\[ \begin{tikzcd}
		\Tt_{Z(\mu)} \arrow[r] \arrow[d] & h^* \Tt_{T^*\QS{X}{\_G}} \arrow[d] \\
		i^* \Tt_{T^*X} \arrow[r] & 	\Ll_{Z(\mu)} 
	\end{tikzcd}\] 	
	
	Similarly to the proof of Proposition \ref{prop:dual of the anchor is a moment map}, taking the fiber of the horizontal morphism we get in both case $\pi^* \_L$. Indeed, from the pullback describing $Z(\mu)$ we get: 
	
	\[ \Ttr{Z(\mu)}{T^*\QS{X}{\_G}} \simeq \pi^*\Ttr{X}{\QS{X}{\_G}} := \pi^*\_L \]
\end{proof}

Now we are again going to see that moment maps are well behaved with respect to Lagrangian intersections using a direct analogue of the notion of infinitesimal symplectic reduction of Lagrangian morphisms:

\begin{Def}\label{def:symplectic reduction lagrangian morphism groupoid}
	A symplectic reduction of a Lagrangian $L_1 \to Y$ with $Y$ equipped with a $n$-shifted moment map $\mu: Y \to L^*[n]$ (for an action of $\_G^\bullet$) is the data of a factorization  $\red{(L_1)} \in \tx{Lag}_1(\star_n, \red{Y})$ of the Lagrangian morphism $L_1 \in \tx{Lag}_1(\star_n, X)$ by $Z(\mu) \in \tx{Lag}_1(\red{Y}, Y)$:
	
	\[ \begin{tikzcd}
		&& L_1 \arrow[dl] \arrow[dr]& & \\
		& (L_1)_{\tx{red}} \arrow[dl] \arrow[dr] & & Z(\mu) \arrow[dl] \arrow[dr]& \\
		\star && Y_{\tx{red}} & & Y
	\end{tikzcd}\] 
\end{Def}

\begin{RQ}\ \label{rq:description of Lagrangian symplectic reduction groupoid}
	Definition \ref{def:symplectic reduction lagrangian morphism groupoid} is exactly the same as Definition \ref{def:lagrangian reduction group action}. We can describe more explicitly the important features of this definition.
	
	Thanks to Lemma \ref{lem:pullback groupoids projection is a groupoid projection}, a symplectic reduction of a Lagrangian $L_1 \to Y$ with a $n$-shifted moment map $\mu: Y \to L^*[n]$ is the data of: 
	\begin{itemize}
		\item A factorization of $L_1 \to Y$: 
		\[\begin{tikzcd}
			L_1  \arrow[r] & Z(\mu) \arrow[r] & Y 
		\end{tikzcd}\]
		\item An action of $\_G^\bullet$ on $L_1$ such that the map $L_1 \to Z(\mu)$ is equivariant. In particular, the groupoid on $L_1$ is coming from an action of the groupoid on $Z(\mu)$ (itself coming from an action on $Y$). Therefore there is an equivalence (Lemma \ref{lem:pullback groupoids projection is a groupoid projection}): 
		\[ L_1 \simeq Z(\mu) \times_{\red{Y}} \QS{L_1}{\_G} \] 
		\item A Lagrangian structure: 
		\[ \red{(L_1)} \simeq \QS{L_1}{\_G} \to \red{Y} \simeq \QS{Z(\mu)}{\_G}\]
	\end{itemize} 
\end{RQ}

\begin{Th}
	\label{th:lagrangian intersection of moment map lie for lie groupoids}
	
	Let $\mu : Y \to L^*[n]$ be a $n$-shifted moment map and $L_1 \to Y$ and $L_2 \to Y$ Lagrangian morphisms with symplectic reductions. Then the moment map $\mu$ induces a moment map: 
	\[ \mu_{-1} : L_1 \times_Y L_2 \to L^*[n-1]\]
	
	Moreover, there is a natural equivalence:
	\[ \red{\left(L_1 \times_Y L_2\right)} \simeq \red{(L_1)} \times_{\red{L}} \red{(L_2)} \]
\end{Th}
\begin{proof}\
	
	\begin{itemize}
		\item We need an action of $\_L$ on $Z(\mu_{-1})$. We have: 
		\[ \begin{split}
			Z(\mu_{-1}) \simeq & ( L_1 \times_Y L_2) \times_{(X \times_{L^*[n]}X)} (X \times_X X) \\
			\simeq & L_1 \times_{Z(\mu)} L_2
		\end{split}\]
		Since $L_1$, $L_2$ and $Y$ all have an action of $\_G^\bullet$, by pullback of actions of groupoids (Proposition \ref{prop:pullbak quotient stack of groupoids}). We get a pullback action of $\_G^\bullet$ on $Z(\mu_{-1})$. 
		\item We have thanks to Proposition \ref{prop:pullbak quotient stack of groupoids}:
		\[ \begin{split}
			\QS{Z(\mu_{-1})}{\_G} \simeq & \QS{L_1 \times_{Z(\mu)} L_2}{\_G} \\
			\simeq & \QS{L_1}{\_G} \times_{\QS{Z(\mu)}{\_G}} \QS{L_2}{\_G} \\
			\simeq & \red{(L_1)} \times_{\red{Y}} \red{(L_2)}
		\end{split}\]
		This is a derived intersection of Lagrangians by assumption therefore it is $(n-1)$-shifted symplectic. 
		\item We are in the situation of Theorem \ref{th:lagragian correspondence pullback}: 
		\[ \begin{tikzcd}
			& L_1 \arrow[dl] \arrow[dr] \arrow[d] & \\
			\red{(L_1)} \arrow[d] &  Z(\mu) \arrow[dr] \arrow[dl] & L_1  \arrow[d] \\
			\QS{Z(\mu)}{\_G} & L_2 \arrow[dr] \arrow[dl] \arrow[u]& Y \\
			\red{(L_2)} \arrow[u] & & L_2 \arrow[u] 
		\end{tikzcd}\]
		We have that the maps $\red{(L_i)}\to \QS{Z(\mu)}{\_G}$ are Lagrangian morphisms and thanks to Remark \ref{rq:description of Lagrangian symplectic reduction} and Lemma \ref{lem:pulback of quotient stack along the projection} we have: 
		\begin{equation}\label{eq:pullback symplectic reduction groupoids}
			L_i \simeq \red{(L_i)} \times_{\QS{Z(\mu)}{\_G}} Z(\mu)
		\end{equation}  
		
		Moreover, the vertical morphisms on the left and right sides are all Lagrangians by assumption and from Definition \ref{def:symplectic reduction lagrangian morphism algebroid}\\
		
		Therefore we get a Lagrangian correspondence: \
		
		\adjustbox{scale=0.6,center}{
			\begin{tikzcd}[column sep = 1mm]
				& \red{(L_1)} \times_{\QS{Z(\mu)}{\_G}} Z(\mu) \times_{\QS{Z(\mu)}{\_G}} \red{(L_2)} \arrow[dl] \arrow[dr]& \\
				\red{(L_1)} \times_{\QS{Z(\mu)}{\_G}} \red{(L_2)} & & \left(\red{(L_1)} \times_{\QS{Z(\mu)}{\_G}} Z(\mu)\right) \times_Y \left(\red{(L_2)} \times_{\QS{Z(\mu)}{\_G}} Z(\mu)\right)
		\end{tikzcd}}\\
		
		\begin{itemize}
			\item Using Equation \eqref{eq:pullback symplectic reduction groupoids} we get:
			\[ \begin{split}
				&	\red{(L_1)} \times_{\QS{Z(\mu)}{\_G}} Z(\mu) \times_{\QS{Z(\mu)}{\_G}} \red{(L_2)}  \\
				\simeq & \red{(L_1)} \times_{\QS{Z(\mu)}{\_G}} Z(\mu)  \times_{Z(\mu)} Z(\mu)\times_{\QS{Z(\mu)}{\_G}} \red{(L_2)}\\
				\simeq & L_1 \times_{Z(\mu)} L_2 \\
				\simeq & Z(\mu_{-1})
			\end{split}\]
			
			\item Using again Equation \eqref{eq:pullback symplectic reduction groupoids} we get:\\
			
			\adjustbox{scale=0.85,center}{ 
				$\left(\red{(L_1)} \times_{\QS{Z(\mu)}{\_G}} Z(\mu)\right) \times_Y \left(\red{(L_2)} \times_{\QS{Z(\mu)}{\_G}} Z(\mu)\right) \simeq L_1 \times_Y L_2$
			}\\
			
			\item Using Remark \ref{rq:description of Lagrangian symplectic reduction groupoid} and Proposition \ref{prop:pullbak quotient stack of groupoids} we get:
			\[ \begin{split}
				\red{(L_1)} \times_{\QS{Z(\mu)}{\_G}} \red{(L_2)} \simeq  & \QS{L_1}{\_G}  \times_{\QS{Z(\mu)}{\_G}} \QS{L_2}{\_G} \\
				\simeq & \QS{L_1 \times_{Z(\mu)} L_2}{\_G}\\
				\simeq & \QS{Z(\mu_{-1})}{\_G}
			\end{split}\]
		\end{itemize}
		
		In short have a Lagrangian correspondence: 
		\[ \begin{tikzcd}
			& Z(\mu_{-1}) \arrow[dl] \arrow[dr]& \\
			\QS{Z(\mu_{-1})}{\_G} & & L_1 \times_Y L_2
		\end{tikzcd}\]
	\end{itemize}
\end{proof}

\begin{Cor}\label{cor:moment map derived critical locus groupoid}
	Let $\_G^\bullet$ be a smooth Segal groupoid over $X$. The Lagrangian maps $s_0, df: X \to T^*X$ have symplectic reductions given by: \[s_0, d\eq{f}: \QS{X}{\_G} \to T^*\QS{X}{\_G}\] as soon as $f$ is $\_G$-invariant. Using Proposition \ref{prop:dual of the anchor is a moment map} and Theorem \ref{th:lagrangian intersection of moment map lie for lie groupoids}, we get a $(-1)$-shifted moment map $\RCrit(f) \to L^*[-1]$ whose symplectic reduction is exactly the $\_G$-equivariant derived critical locus, $\RCrit(\eq{f})$. 
\end{Cor}
\begin{proof}
	Since $f$ is equivariant, $df$ factors through $Z(\mu)$ (idem for $s_0$) and we have the diagram: 
	\[ \begin{tikzcd}
		&& X \arrow[dl] \arrow[dr]& & \\
		& \QS{X}{\_G} \arrow[dl] \arrow[dr] & & Z(\mu) \arrow[dl] \arrow[dr]& \\
		\star && T^*\QS{X}{\_G} & & T^*X
	\end{tikzcd}\]
	The square is clearly a pullback from Lemma \ref{lem:pulback of quotient stack along the projection} and the maps: 
	\[\QS{X}{\_G} \to T^*\QS{X}{\_G} \] 
	are either $d\eq{f}$ or the zero section (and therefore they are Lagrangian). 
	
	Therefore $df$ and $s_0$ both have a symplectic reduction. 
	
	We have seen in Proposition \ref{prop:dual of the anchor is a moment map} that $T^*X \to L^*$ is a moment map with symplectic reduction $T^*\QS{X}{\_G}$ and using Theorem \ref{th:lagrangian intersection of moment map lie for lie groupoids}, we get a $(-1)$-shifted moment map:
	\[ \mu_{-1} : \RCrit(f) \to L^*[-1]\]
	
	whose symplectic reduction is: 
	\[ \red{\RCrit(f)} \simeq \QS{X}{\_G} \times_{T^*\QS{X}{\_G}} \QS{X}{\_G} =: \RCrit(\eq{f}) \]  
\end{proof}

Following the idea of Section \ref{sec:derivation-and-integration-of-lie-algebroids} where we want to explain that infinitesimal quotients by Lie algebroids are formal completions along the projections to the quotients by groupoids. Similarly, we have that infinitesimal symplectic reductions are formal completions of symplectic reductions and more generally all the constructions of Section \ref{sec:for-lie-algebroid} are ``infinitesimal versions'' of the constructions of the current section. 

\begin{Th}\label{th:formal completion of moment map and symplectic reductions} 
	Let $\_L$ be a Lie algebroid over $X$, an affine stack satisfying Assumptions \ref{ass:very good stack}. We assume that $\_L$ integrates well and we pick $\_G^\bullet$ a good integration of $\_L$ (see Definition \ref{def:good integration}). 
	Take $\mu : Y \to L^*[n]$ a $n$-shifted moment map for $\_G^\bullet$ and $L_i \to X$ ($i= 1..2$) Lagrangian morphisms with symplectic reduction for $\_G$ such that $Y$, $L_i$ and $Z(\mu)$ are all affine of almost finite presentation. Then we have the following: 
	\begin{enumerate}
		\item $\mu$ has a canonical infinitesimal moment map structure with symplectic reduction given by the weak infinitesimal quotient obtained as the formal completion: 
		\[ \QSW{Z(\mu)}{\_L} := \comp{\QS{Z(\mu)}{\_G}_X} \]
		
		\item The infinitesimal symplectic reductions of $L_i$ (for $i= 1..2$) are given by the weak infinitesimal quotients obtained as the formal completions of the symplectic reductions of $L_i$: 
		\[ \QSW{L_i}{\_L} := \comp{\QS{L_i}{\_G}_X} \]
		\item The infinitesimal shifted moment map obtained by Lagrangian intersection (Theorem \ref{th:lagrangian intersection of moment map lie for lie algebroids}):
		\[ \mu_{-1} : L_1 \times_Y L_2 \to L^*[n-1]\]
		is the moment map obtained by the first construction applied to shifted moment map obtained via Lagrangian intersection  for groupoids (Theorem \ref{th:lagrangian intersection of moment map lie for lie groupoids}).
		
		Therefore, the infinitesimal symplectic reduction of the derived intersection is the formal completion of the projection to the symplectic reduction of the derived intersection. 	
	\end{enumerate} 
\end{Th} 

\begin{Lem}\label{lem:formal completion lagrangian correspondence remains a lagrangian correpspondence}
	Take $Z$ and $Y$ two affine stacks of almost finite presentation.
	Let $\_G$ be a groupoid over $Z$ and take a Lagrangian correspondence: 
	
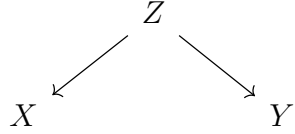
\[ \begin{tikzcd}
		& Z \arrow[dl] \arrow[dr] & \\
		X & & Y
	\end{tikzcd} \]
	
	then there is a canonical Lagrangian correspondence: 
	
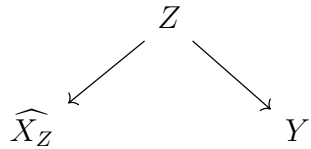
\[ \begin{tikzcd}
		& Z \arrow[dl] \arrow[dr] & \\
		\comp{X_Z} && Y
	\end{tikzcd} \]
\end{Lem}
\begin{proof}
	Recall from Lemma \ref{lem:formal completion factorization and morphims properties} that there is a factorization: \[ Z \to \comp{X_Z}   \to X \]
	where the first map is a formal thickening and the last map is formally étale. We get the diagram: 
	
	
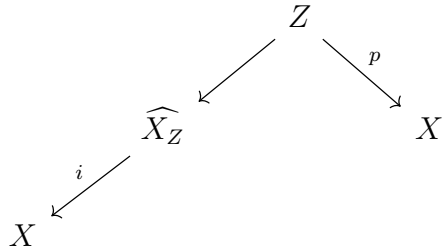
\[ \begin{tikzcd}
		&& Z \arrow[dl] \arrow[dr, "p"] & \\
		&\comp{X_Z}\arrow[dl, "i"']&& X\\
		X&&&
	\end{tikzcd} \]
	
	Pulling back the symplectic structure along $i$ gives again a symplectic structure because $i$ is formally étale. Therefore the isotropic correspondence can be chosen to be the same. Then since the non-degeneracy condition only depends on the tangent and cotangent complexes, $i$ being formally étale implies that the isotropic correspondence we obtained is also non-degenerate.
\end{proof}

\begin{proof}[Proof of Theorem \ref{th:formal completion of moment map and symplectic reductions}]
	The proof is based on taking formal completions to define weak infinitesimal quotients:
	\begin{enumerate}
		\item With Lemma \ref{lem:formal completion lagrangian correspondence remains a lagrangian correpspondence}, we need to show that the formal completion is a weak infinitesimal quotient. Consider the following commutative diagram: 
		\[ \begin{tikzcd}
			Z(\mu) \arrow[r] \arrow[d] & X \arrow[d] \\
			\comp{\QS{Z(\mu)}{\_G}_{Z(\mu)}} \arrow[r] \arrow[d]& \QS{X}{\_L} \simeq 	\comp{\QS{X}{\_G}_X}  \arrow[d] \\
			\QS{Z(\mu)}{\_G} \arrow[r] & \QS{X}{\_G}
		\end{tikzcd} \]
		
		First from Corollary \ref{cor:formal completion lie algebroid and relative tangent} (since $\_G^\bullet$ is a good integration of $\_L$), we have the equivalence: 
		\[ \QS{X}{\_L} \simeq \comp{\QS{X}{\_G}_X}\] 
		
		The outer square is a pullback because of Lemma \ref{lem:pulback of quotient stack along the projection}. The lower square is a pullback because of the equivalences: 
		\[ \begin{split}
			\comp{\QS{Z(\mu)}{\_G}_{Z(\mu)}} \simeq & \QS{Z(\mu)}{\_G} \times_{\QS{Z(\mu)}{\_G}_{\tx{DR}}} Z(\mu)_{\tx{DR}} \\
			\simeq & \QS{Z(\mu)}{\_G} \times_{\QS{Z(\mu)}{\_G}_{\tx{DR}}} \left(\QS{Z(\mu)}{\_G} \times_{\QS{X}{\_G}} X\right)_{\tx{DR}} \\
			\simeq & \QS{Z(\mu)}{\_G} \times_{\QS{Z(\mu)}{\_G}_{\tx{DR}}} \QS{Z(\mu)}{\_G}_{\tx{DR}} \times_{\QS{X}{\_G}_{\tx{DR}}} X_{\tx{DR}}\\
			\simeq & \QS{Z(\mu)}{\_G} \times_{\QS{X}{\_G}_{\tx{DR}}} X_{\tx{DR}}\\
			\simeq & \QS{Z(\mu)}{\_G} \times_{\QS{X}{\_G}} \QS{X}{\_G} \times_{\QS{X}{\_G}_{\tx{DR}}} X_{\tx{DR}} \\
			\simeq &  \QS{Z(\mu)}{\_G} \times_{\QS{X}{\_G}} \comp{\QS{X}{\_G}_X}
		\end{split}\]
		
		Therefore the upper square is also a pullback and $\comp{\QS{Z(\mu)}{\_G}_{Z(\mu)}}$ is both a weak infinitesimal quotient and the symplectic reduction for the infinitesimal moment map thanks to Lemma \ref{lem:formal completion lagrangian correspondence remains a lagrangian correpspondence}.
		
		\item We have the commutative diagram: 
		\[ \begin{tikzcd}
			L_i \arrow[r] \arrow[r]\arrow[d] & Z(\mu) \arrow[d]\\ 
			\comp{\QS{L_i}{\_G}_{L_i}} \arrow[r] \arrow[d]& \comp{\QS{Z(\mu)}{\_G}_{Z(\mu)} } \arrow[d]\\
			\QS{L_i}{\_G} \arrow[r] & \QS{Z(\mu)}{\_G}
		\end{tikzcd}\]
		
		where all squares are pullback squares (for the same reason as in the proof of (1)). Moreover, it is not hard to see (using again that the lower vertical maps are formally étale, as in the proof of Lemma \ref{lem:formal completion lagrangian correspondence remains a lagrangian correpspondence}) that the maps: 
		\[ \QSW{L_i}{\_L} := \comp{\QS{L_i}{\_G}_{L_i}} \to  \QS{Z(\mu)}{\_L} \] 
		have Lagrangian structures canonically determined by the Lagrangian structures on \[ \QS{L_i}{\_G} \to \QS{Z(\mu)}{\_G}\]
		
		Moreover, $\comp{\QS{L_i}{\_G}_{L_i}}$ is a weak infinitesimal quotient for exactly the same reason than $\comp{\QS{Z(\mu)}{\_G}_{Z(\mu)} }$ in the proof of (1).   
		This gives $L_i$ the structure of an infinitesimal Lagrangian symplectic reduction with reduction given by $\comp{\QS{L_i}{\_G}_{L_i}}$. 
		
		\item Consider the diagram where all squares are pullbacks: 
		\[ \begin{tikzcd}
			L_1 \arrow[r] \arrow[r]\arrow[d] & Z(\mu) \arrow[d] & \arrow[l] L_2 \arrow[d]\\ 
			\comp{\QS{L_1}{\_G}_{L_1}} \arrow[r] \arrow[d]& \comp{\QS{Z(\mu)}{\_G}_{Z(\mu)} } \arrow[d] & \arrow[l] \comp{\QS{L_2}{\_G}_{L_2}} \arrow[d]\\
			\QS{L_1}{\_G} \arrow[r] & \QS{Z(\mu)}{\_G} & \arrow[l] \QS{L_2}{\_G} 
		\end{tikzcd}\]
		The idea is that taking the fiber along the horizontal morphisms produces the Lagrangian intersections we are looking for, and the formal completion of the pullback morphisms is the pullback of the formal completions (Corollary \ref{cor:formal completion and pullacks}). 
		
		More precisely, the moment map structure on: 
		\[ \mu_{-1}: Z(\mu_{-1}) \simeq L_1 \times_{Z(\mu)} L_2 \to L^*[-1] \simeq X \times_{L^*} X\]
		is given by the pullback action (Proposition \ref{prop:pullbak quotient stack of groupoids}) and gives us a Lagrangian correspondence: 
		
		\[ \begin{tikzcd}
			& Z(\mu_{-1}) \arrow[dl] \arrow[dr] & \\
			\QS{Z(\mu_{-1})}{\_G} && L_1 \times_Y L_2
		\end{tikzcd}\]
		with: 
		\[ 	\QS{Z(\mu_{-1})}{\_G} \simeq\QS{ L_1}{\_G} \times_{\QS{Z(\mu)}{\_G}} \QS{L_2}{\_G} \]
		
		Then the infinitesimal moment map induced by first construction is given by the same underlying map $\mu_{-1}$, the same fiber of the moment map, $Z(\mu_{-1})$, and we only need to find the weak infinitesimal quotient. It is given by the formal completion which can be computed as (again thanks to Corollary \ref{cor:formal completion and pullacks}):
		\[ \begin{split}
			\comp{\left(\QS{ L_1}{\_G} \times_{\QS{Z(\mu)}{\_G}} \QS{L_2}{\_G}\right)_{Z(\mu_{-1})}} \simeq & \comp{\left(\QS{ L_1}{\_G} \times_{\QS{Z(\mu)}{\_G}} \QS{L_2}{\_G}\right)_{L_1 \times_{Z(\mu)} L_2 }}\\
			\simeq & \comp{\QS{ L_1}{\_G}_{L_1}} \times_{\comp{\QS{Z(\mu)}{\_G}_{Z(\mu)}}} \comp{\QS{ L_2}{\_G}_{L_2}}
		\end{split}\]
		
		But the last term is exactly the weak infinitesimal quotient obtained by Lagrangian intersection of the infinitesimal moment maps by the infinitesimal Lagrangian reductions described before. 
	\end{enumerate}
\end{proof}

To compute the particular example of this construction applied to the dual of the anchor map  $T^*X \to L^*$ (with respect to a groupoid action), we use the following proposition:

\begin{Prop}\label{prop:formal completion and cotangent}
	Let $X$ be an affine stack satisfying Assumptions \ref{ass:very good stack}. Take $\_L$ a Lie algebroid over $X$ that integrates well to a good groupoid $\_G$. Then there is an equivalence: 
	\[ T^*\QS{X}{\_L} \simeq \comp{\left(T^*\QS{X}{\_G}\right)_{Z(\mu)}}\] 
\end{Prop}
\begin{proof}
	We consider the following morphism: 
	\[ \QS{X}{\_L} \simeq \comp{\QS{X}{\_G}_X} \overset{\alpha}{\to} \QS{X}{\_G} \]
	
	Then we have a commutative diagram such that all squares are pullbacks: 
	\[ \begin{tikzcd}
		Z(\mu) \arrow[r] \arrow[d] & X  \arrow[d]\\
		\alpha^* T^*\QS{X}{\_G}  \simeq \comp{\left(T^*\QS{X}{\_G}\right)_{Z(\mu)}} \arrow[r] \arrow[d] & \QS{X}{\_L} \simeq \comp{\QS{X}{\_G}_X} \arrow[d] \\
		T^*\QS{X}{\_G} \arrow[r] & \QS{X}{\_G}
	\end{tikzcd}\]
	
	The outer square is clearly a pullback square since there is an equivalence: \[p^*\Ll_{\QS{X}{\_G}} \simeq \Ll_X \oplus^{\rho^*} \_L^\vee[-1] \]
	
	For the lower square, $\alpha^*  T^*\QS{X}{\_G}$ is the pullback by definition. 
	
	From  Corollary \ref{cor:formal completion lie algebroid and relative tangent}, since $\_L$ integrates well, we get the equivalence:
	\[\QS{X}{\_L} \simeq \comp{\QS{X}{\_G}_X} \]

	Finally using Corollary \ref{cor:formal completion and pullacks} this shows that $\comp{\left(T^*\QS{X}{\_G}\right)_{Z(\mu)}}$ is also the pullback of the lower square. \\
	
	Using Lemma \ref{lem:formal completion factorization and morphims properties}, the map: 
	\[ \comp{\QS{X}{\_G}_X} \overset{\alpha}{\to} \QS{X}{\_G}\]
	is formally étale. Therefore there is an equivalence: 
	\[ \alpha^* T^*\QS{X}{\_G} \overset{\sim}{\to} T^*\QS{X}{\_L} \]  
\end{proof}

In particular if $\_L$ integrates well to $\_G^\bullet$, then this proposition shows that $T^*\QS{X}{\_L}$ is canonically symplectic if and only if $ T^*\QS{X}{\_G}$ is also canonically symplectic.  

	\newpage
\section{Derived Perspective of the BV Complex} \label{sec:derived-perspective-of-the-bv-complex}\

The classical BV construction is a mean to build an algebra of classical observables associated with a space of fields $X$ together with an action functional $f$ taking into account some symmetries of the system. There are two different approaches to construct these algebras:

\begin{Cons}\label{cons:BV FK heuristic}
	The idea for the first construction is given by the following two steps: 
	\begin{enumerate}
		\item First, take $S$ a ``space  of solutions'' to the Euler--Lagrange equation. In our setting it amounts to solving $df= 0$.
		\item Then take a ``infinitesimal quotient'' of $S$ by a Lie algebroid of ``maximal symmetries'' so that this infinitesimal quotient is $(-1)$-shifted symplectic.
	\end{enumerate} 
\end{Cons} 

This is the approach corresponding to the construction in \cite{FK14} which we will discuss in more details in Construction \ref{cons:BV FK}.

\begin{Cons}\label{cons:BV CG heurestic}
	The second approach takes these operations in reverse:
	\begin{enumerate}
		\item First take a Lie algebroid ``of symmetries'' on $X$. In particular, $f$ must be equivariant with respect to the infinitesimal action of this algebroid and we get a map\footnote{Or rather ``being'' equivariant is the data of such a map (see Definition \ref{def:equivariant maps algebroids}).}: 
		\[  \eq{f} : \QS{X}{\_L} \to \Aa_k^1\] 
		\item We take the derived critical locus of the equivariant map viewed as the (derived) space of solutions of the Euler--Lagrange equations. 
	\end{enumerate}
\end{Cons} 

This is similar to the approach of \cite{CG21} which is going to be studied in more details in Section \ref{sec:lie-algebroid-moment-maps} (in a restricted setting).  \\

The goal of this section is to provide a setting fitting and generalizing both constructions. We will argue that both constructions are intances of a notion of ``generalized derived symplectic reduction'' of the derived critical locus by a choice of given infinitesimal symmetries. \\

Moreover, we will discuss a variation of this construction for symmetries which are not infinitesimal but given by \emph{Segal groupoids} of symmetries.  \\

Finally, we show in Section \ref{sec:examples-of-bv-constructions} that any groupoid or Lie algebroid of ``off-shell symmetries'' induces a canonical BV construction by symplectic reduction of a canonical $(-1)$-shifted moment map on the derived critical locus.

\subsection{BV as a Generalized Derived Symplectic Reduction}\ \label{sec:bv-as-a-generalized-derived-symplectic-reduction}

\medskip

In this section, we start in Section \ref{sec:context-and-construction-for-infinitesimal-actions} by explaining in more details Constructions \ref{cons:BV FK heuristic} and \ref{cons:BV CG heurestic} and show that they are both instances of a notion of ``generalized symplectic reduction'' of the derived critical locus (see Section \ref{sec:generalized-bv-construction}). \\

In Section \ref{sec:strictification} we try to explain to what extent this general construction still resembles the usual BV constructions.  We will in particular see that we get the same kind of algebras (Proposition \ref{prop:strictification}). However, in general we cannot strictify the $(-1)$-shifted symplectic structure as it may admit ``higher'' terms (Theorem \ref{th:strictification symplectic structure}).

\subsubsection{Context and construction for infinitesimal actions}\ \label{sec:context-and-construction-for-infinitesimal-actions} 

\medskip

Through all of Section \ref{sec:derived-perspective-of-the-bv-complex}, we will take $X= \Spec(A)$ a smooth affine algebraic variety\footnote{In particular, $X$ satisfies Assumptions \ref{ass:very good stack}} and a map $f: X \to \Aa_k^1$. We start by describing the BV algebra as constructed in \cite{FK14}: 

\begin{Cons}\label{cons:BV FK}\
	
	\begin{enumerate}
		\item 	The first step is to pick a resolution of the strict critical locus. We take a semi-free resolution called the \defi{Koszul--Tate} resolution (Construction \ref{cons:koszul tate resolution}) of the strict critical locus of a function $f$ (denoted by $\KT(f)$). We have: 
		\[ \KT(f) := \left( \Sym_A \left(\Tt_A[1] \oplus \_L_{\KT}[2] \right), \delta_{\KT} \right)\]
		
		where $\_L_{\KT}$ is projective concentrated in non-positive degrees. $\_L_{\KT}$ and $\delta_{\KT}$ are chosen so that the cohomology of this complex is concentrated in degree $0$ and giving in degree $0$ by: \[\faktor{A}{\langle df.X, X \in \Tt_A \rangle}\] 
		
		\item From here we consider the following graded algebra (with no differential):
		\[ \BV_{\FK}^\sharp = \Sym_A\left(\Tt_A[1]\oplus \_L_{\KT}[2] \oplus \widehat{\_L_{\KT}^\vee[-1]}\right) \] 
		where the symbol $\widehat{-}$ means that we take the completion along the \CE terms, $\_L_{\KT}^\vee[-1]$ (see Notation \ref{not:partial completion}). 
		\item This graded algebra is $(-1)$-shifted symplectic together with the canonical\footnote{Canonical is the sens that it is induced by the canonical symplectic structure on $\RCrit(f)$ (induced by the Poisson pairing between $B$ and $\Tt_B[1]$) and the canonical pairing between $\_L_{\KT}[2]$ and $\_L_{\KT}^\vee[-1]$. This is the strict structure from Definition \ref{def:symplectic structure strict}} symplectic structure. 
		
		\item By an inductive procedure, we can construct a BV charge (\cite[Theorem 4.5]{FK14}), $Q \in (\BV_{\FK}^\sharp)_0$ satisfying the \defi{Classical Master Equation}: 
		\[ \lbrace Q, Q \rbrace = 0\]
		
		This can be equivalently phrased as finding a differential $\delta_{\FK}$ on $\BV_{\FK}^\sharp$ that is Hamiltonian, $\delta_{\FK} = \lbrace Q, - \rbrace$. This differential also respects the filtration given by the symmetric power in $\_L^\vee[-1]$. Moreover, this differential restricts to the Koszul--Tate differential defining a map, $\BV_{\FK} \to \KT(f)$, of algebras. We get the full BV algebra: 
		\[ \BV_{\FK} := \left( \BV_{\FK}^\sharp, \lbrace -, - \rbrace, Q \right)  \]
	\end{enumerate}
\end{Cons}

We will try to rephrase this construction as follows: 
\begin{itemize}
	\item Start by taking the derived critical locus whose algebra of functions\footnote{This is also called the Koszul complex, to which we then add the Tate terms, $\_L_{\KT}$ to produce the Koszul--Tate resolution.} is $\Sym_A \Tt_A[1]$ with differential $\iota_{df}$ (thanks to \ref{prop:derived critical locus corepresentability}). Then taking the Koszul--Tate resolution amounts to adding ``anti-ghost fields'' (see Remark \ref{rq:ghost and ghosts of ghosts}) generating $\_L_{\KT}[2]$ that will kill the higher cohomology of the derived critical locus, in other  words, kill the symmetries of $f$. Note that there is a natural map: 
	\[ \iKT(f) \to \RCrit(f)\]
	\item We can pullback the sympletic form of the derived critical locus to a pre-symplectic form on $\iKT(f)$ and the ``kernel'' of this pre-symplectic form is essentially given by $\_L_{\KT}$, in other words the ``anti-ghost fields generating the symmetries''.
	\item We add ``ghost fields'' dual to the anti-ghost fields to kill this kernel, and we get a symplectic structure. Adding these ghost fields amounts to taking the \emph{\CE algebra} of a $\_L_\infty$-algebroid of maximal symmetries over the Koszul--Tate resolution. 
\end{itemize}

It also turns out that the use of $\iKT(f)$ is to restrictive if we want to recover Construction \ref{cons:BV CG heurestic}. Indeed, in most case, this second construction cannot recover a full resolution of the strict critical locus. This construction is given by: 

\begin{Cons}\label{cons:BV CG}\
	
	\begin{enumerate}
		\item Start with a Lie algebroid (or $\_L_\infty$-algebroid)  $\_L$ over $X$ of ``infinitesimal symmetries'', i.e. such that $f$ is $\_L$-equivariant, and we have a map:
		\[\eq{f}: \QS{X}{\_L} \to \Aa^1\]
		\item Take the derived critical locus $\RCrit(\eq{f})$ who is naturally $(-1)$-shifted symplectic.   
	\end{enumerate} 
\end{Cons}

Heuristically\footnote{As we have seen in Sections \ref{sec:on-the-derived-geometry-of-lie-algebroid} and \ref{sec:equivariant-symplectic-geometry}, this heuristic is not really clear. In practice the \CE algebra is use as a model for derived invariant although its relationship to infinitesimal quotient is not clear.}, the first step amounts to considering the \CE algebra of the Lie algebroid of symmetry we start with (which amounts to first add the ghost fields). 

Then taking the derived critical locus amounts to adding the anti-fields and anti-ghost fields. Let us have an incorrect motivating discussion. We somehow look at the derived critical locus of a functional on the \CE algebra. We would think that its algebra of functions would be like: 
\[ \Sym_{\ceu(\_L)} \Tt_{\ceu(\_L)}[1] \simeq \Sym_A \left( \Tt_A[1] \oplus \_L[2] \oplus \_L^\vee[-1]\right)\]

and we would recover the general shape of the BV algebra from \cite{FK14}. However, the differential on this algebra would be the contraction $\iota_{d\eq{f}}$ and linearization of the \CE differential which \emph{cannot} account for all the correction terms of the Koszul--Tate differential on $\_L[2]$ (in particular the terms increasing the symmetric power cannot appear). \\ 

The correct way to view the relation between the second and first constructions is discussed in Sections \ref{sec:with-moment-maps}, \ref{sec:lie-algebroid-moment-maps} and \ref{sec:bv-construction-for-groupoid-action}. However, this discussion motivates a generalization of Construction \ref{cons:BV FK} where we add anti-ghost fields, but not necessarily making it a resolution. This procedure of adding anti-ghost fields corresponds to choosing an \emph{almost derived critical locus} (Definition \ref{def:almost derived critical loci}). 

\subsubsection{Generalized BV construction}\ 
\label{sec:generalized-bv-construction}

\medskip

The idea that we had so far boils done to taking an almost derived critical locus $S \hookrightarrow \RCrit(f)$ and then take its quotient to make it $(-1)$-shifted symplectic. This is the general idea behind \emph{symplectic reduction}, and we will describe the BV construction in those terms. This motivates, by analogy with Definition \ref{def:moment map lie algebroid}, the following definition of generalized symplectic reduction and generalized infinitesimal symplectic reduction.

\begin{Def}\ \label{def:generalized symplectic reduction groupoid}
	Let $Y$ be a $n$-shifted symplectic derived stacks and $g: S \to Y$ a map. Then a \defi{generalized symplectic reduction of $Y$ along $g$} is the following data: 
	\begin{itemize}
		\item A Lie groupoid $\_G$ over $S$ such that $\red{Y} := \QS{S}{\_G}$ is $n$-shifted symplectic. 
		\item A Lagrangian correspondence: 
		\[\begin{tikzcd}
			& S \arrow[dr] \arrow[dl] & \\
			\red{Y} && Y
		\end{tikzcd}\]
	\end{itemize}
\end{Def}

\begin{Def}\ \label{def:generalized symplectic reduction}
	Let $S$ be affine of almost finite presentation and $Y$ a $n$-shifted symplectic derived stack. Take $g: S \to Y$ a map. Then a \defi{generalized infinitesimal symplectic reduction of $Y$ along $g$} is the following data: 
	\begin{itemize}
		\item A $n$-shifted symplectic ``infinitesimal quotient'' of $S$  denoted: \[\red{Y} := \QS{S}{\_L}\] 
		\item A Lagrangian correspondence: 
		\[\begin{tikzcd}
			& S \arrow[dr] \arrow[dl] & \\
			\red{Y} && Y
		\end{tikzcd}\]
	\end{itemize}
\end{Def}

\begin{Conv}\label{conv:notion of quotient on S}
	Note that the notion of ``infinitesimal quotient'' in the previous definition is unclear. We have the following possible notion of infinitesimal quotient: 
	\begin{enumerate}
		\item The first naive notion that mimics Construction \ref{cons:BV FK} is to consider a perfect $\_L_\infty$-algebroid $\_L$ on $S$ and define: \[ \red{Y} := \Spec(\ceu(\_L)) \]
		
		However, the natural projection $S \to \red{Y}$ needs not be a formal thickening, therefore the adjective ``infinitesimal'' is not really fitting. 
		
		\item We can could consider the formal completion of the previous example: 
		\[ \red{Y} := \comp{\Spec(\ceu(\_L))_S}\]
		
		From Lemma \ref{lem:formal completion lagrangian correspondence remains a lagrangian correpspondence}, if $\Spec(\ceu(\_L))$ is part of a Lagrangian correspondence, then so is its formal completion. Moreover, this formal completion is a formal thickening of $S$. 
		
		\item In a similar idea, we can also consider infinitesimal quotients of the form: 
		\[ \red{Y} := \und{\Spf_A(\ceu(\_L))} \]

		\item In Section \ref{sec:quotient-stack-of-a-lie-algebroid}, we define the notion of infinitesimal quotient of a Lie algebroid. If $S$ satisfies Assumptions \ref{ass:very good stack}, we can define: 
		\[ \red{Y} := \und{\MC_\_L} \simeq \QS{S}{\_L} \]
		
		For example, $S$ satisfies Assumptions \ref{ass:very good stack} if $S = \iKT(f)$, but $S$ will \emph{not} satisfy these Assumptions in the context of Construction \ref{cons:BV CG}.
		
		\item $S$ admits a map $S \to X$ where $X$ satisfies Assumptions \ref{ass:very good stack} and $\_L$ is a Lie algebroid on $X$. We can consider the weak infinitesimal quotients (Definition \ref{def:weak inf quotient action}) associated to $\_L$.
		
		We will typically use this notion of weak infinitesimal quotient in the situation where we have ``off-shell symmetries'' (Definition \ref{def:off-shell symmetries}). 
		
	\end{enumerate}
\end{Conv}

\begin{Claim}\label{prop: th different kind of inf quotient}
	The different notions of infinitesimal quotients given in Definition \ref{conv:notion of quotient on S} are related as follows: 
	\begin{itemize}
		\item Any infinitesimal quotient in the sens of (1) induces an infinitesimal quotient in the sens of (2) by formal completion. A generalized symplectic reduction for (1) induces a generalized symplectic reduction for (2). Therefore the second definition is strictly better as it is a formal thickening.
		\item We have a natural morphism: 
		\[ \pund{\Spf_A(\ceu(\_L))} \to \comp{\Spec(\ceu(\_L))_S}  \]

		\item If $S$ satisfies Assumptions \ref{ass:very good stack}, then any Lie algebroid over $S$ acts on $S$ along the identity $S \to S$. Then the infinitesimal quotient of $S$ is a weak infinitesimal quotient. Therefore infinitesimal quotients in the sens of (4) are weak infinitesimal quotients in the sens of (5). 
	\end{itemize} 
\end{Claim}

We can observe that there are essentially two kinds of infinitesimal quotients. The ones based on the \CE algebra, which are useful to relate these quotient with BV-like algebras as we described before. 

The other ones are infinitesimal quotient as discussed in Sections \ref{sec:on-the-derived-geometry-of-lie-algebroid} and \ref{sec:equivariant-symplectic-geometry} which we claim to be a more geometric approach to infinitesimal quotients. 

We believe that in order to reconcile the two pictures, we need to remember the \emph{graded mixed} structure on the \CE algebra (at least when the Lie algebroid is perfect). The correct notion of infinitesimal quotients and the tools to manipulate them still need to be improved on.  \\

We have said multiple times in Section \ref{sec:shifted-moment-maps-and-derived-symplectic-reduction} that the equivariant derived critical locus $\RCrit(\eq{f})$ is the symplectic reduction of the derived critical locus. As $\RCrit(\eq{f})$ is supposed to represent the BV-construction (following Construction \ref{cons:BV CG heurestic}). In particular $\Crit(f)$ itself should be viewed as the simplest derived critical locus. Then we argue that adding ghosts and anti-ghosts fields geometrically corresponds to taking a generalized symplectic reduction, motivating the following definitions: 

\begin{Def}\label{def:BV complex}
	A \defi{BV construction for $f$} is the structure of generalized derived symplectic reduction on $S$, an almost derived critical locus, along the map: 
	\[ S \to \RCrit(f)\]
\end{Def}

\begin{Def}\label{def:BV complex infinitesimal}
	An \defi{infinitesimal BV construction for $f$} is the structure of generalized infinitesimal derived symplectic reduction on $S$, an almost derived critical locus of almost finite presentation, along the map: 
	\[ S \to \RCrit(f)\]
	
	Again, this definition depends on the choice of a notion of infinitesimal quotient as in Definition \ref{conv:notion of quotient on S}.
\end{Def}

\begin{RQ}
	Given an almost derived critical locus $S$ on $f$, the existence of a BV construction (infinitesimal or not) for that $S$ is not obvious. A large class of examples of BV constructions will be given in Section \ref{sec:examples-of-bv-constructions}.\\
	
	Given an almost derived critical locus $S$ on $f$ and assuming that it admits a generalized symplectic reduction (infinitesimal or not) there is a priori no reason to have a unique such reduction associated to $S$ and $f$. There is in fact a full topological space of such structures which needs not be contractible or even connected. 
\end{RQ}

\begin{RQ}
	It is worth noticing that any two BV constructions are related by a Lagrangian correspondence: 
	
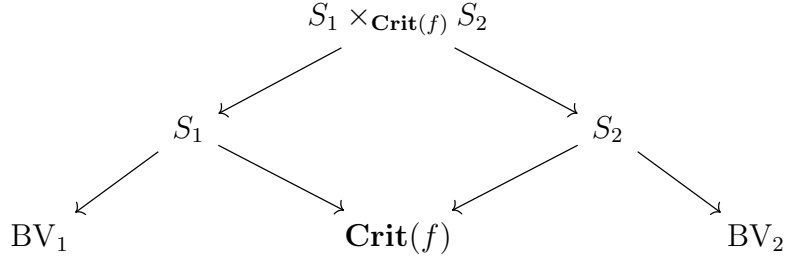
\[ \begin{tikzcd}
		& & S_1 \times_{\RCrit(f)} S_2 \arrow[dl] \arrow[dr] & & \\
		& S_1 \arrow[dl] \arrow[dr] & & S_2 \arrow[dl] \arrow[dr] & \\
		\BV_1 && \RCrit(f) && \BV_2
	\end{tikzcd} \]
	
	Viewed those Lagrangians as admissible morphisms, this defines a neither full or faithful sub-category of $\tx{Lag}_1(\star)_{/\Crit(f)}$.
\end{RQ}

\begin{RQ}
	Our abstract definition of BV construction (infinitesimal or not) might be ``too general'' a priori and have nothing to do with what a BV algebra usually ``looks like''. Although our objects are in fact much more general than the classical construction mainly due to the added flexibility of homotopy structures, we will see argue in Section \ref{sec:strictification} that our construction still ``looks like'' the classical BV construction. 
\end{RQ}

\begin{RQ}
	The construction from \cite{FK14} and \cite{CG21} are both infinitesimal BV constructions. Some ideas behind global (non-infinitesimal) BV construction for group actions are discussed in \cite{BSS21}. Our definition however works in a more general context including Segal groupoid actions.
\end{RQ}

We will conclude this section by showing that infinitesimal BV constructions are the formal completions of ``good\footnote{By good we mean that the Segal groupoid acting is such that each $\_G^n$ is a formal stack. These are the good Segal groupoids of Definition \ref{def:good integration}.}'' BV constructions following the same ideas as developed in Section \ref{sec:derivation-and-integration-of-lie-algebroids} and with Theorem \ref{th:formal completion of moment map and symplectic reductions}.

\begin{Th}\label{th:formal completion of BV constructions}
	Let $S$ be an almost derived critical locus satisfying Assumptions \ref{ass:very good stack}. Take a good BV construction on $S$, that is: 
	\begin{itemize}
		\item An action of a ``good'' Segal groupoid $\_G^\bullet$ on $S$. 
		\item A Lagrangian correspondence: 
		
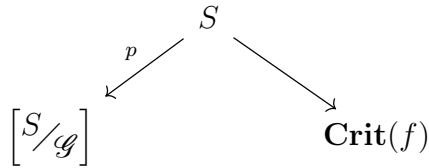
\[ \begin{tikzcd}
			& S \arrow[dr] \arrow[dl, "p"'] & \\
			\QS{S}{\_G} & & \RCrit(f)
		\end{tikzcd}\]
	\end{itemize}
	From Corollary \ref{cor:formal completion lie algebroid and relative tangent}, the infinitesimal quotient of $S$ by the Lie algebroid associated to $\_G$ is the formal completion of the projection $p$. 
	
	Then, $\QS{S}{\_L_\_G}$ is an infinitesimal BV construction.  
\end{Th}
\begin{proof}
	We use Corollary \ref{cor:formal completion lie algebroid and relative tangent} and Lemma \ref{lem:formal completion lagrangian correspondence remains a lagrangian correpspondence} to get a Lagrangian correspondence:
	\[ \begin{tikzcd}
		& S \arrow[dr] \arrow[dl] & \\
		\QS{S}{\_L_\_G} & & \RCrit(f)
	\end{tikzcd}\] 
	This defines the structure of infinitesimal BV construction.  
\end{proof}

We will see a more general version of this when discussing ``off-shell symmetries'' in the following section.

\subsubsection{Off-shell symmetries}\
\label{sec:symmetries}

\medskip

In this Section, we are going to discuss BV constructions that arise from an action (infinitesimal or not) on $X$ such that $f$ is equivariant with respect to that action.

\begin{Def}\label{def:off-shell symmetries}
	We say that an action on an almost derived critical locus $S$ is an action by \defi{off-shell} symmetries if: 
	\begin{itemize}
		\item It is \defi{off-shell}, that is, it comes from an action (in the sens of \ref{def:groupoid action}) of a groupoid $\_G$ over $X$ on $\pi_S : S \to X$. 
		\item $f$ is $\_G$-equivariant in the sens of Definition \ref{def:equivariant maps}. 
	\end{itemize}
\end{Def}

\begin{Def}\label{def:off-shell infinitesimal symmetries}
	Let $S$ be an almost derived critical locus of almost finite presentation. Then an infinitesimal quotient $\red{Y}$ is a weak infinitesimal quotient by \defi{off-shell} symmetries if: 
	\begin{itemize}
		\item There is an algebroid $\_L$ on $X$ such that $f$ is $\_L$-equivariant (Definition \ref{def:equivariant maps algebroids}).
		\item $\red{Y}$ is a weak infinitesimal quotient of $S$ along $S\to X$. 
	\end{itemize}
	If $S$ satisfies Assumptions \ref{ass:very good stack}, we instead require that the quotient comes from an infinitesimal action of $\_L$ on $S \to X$. 
\end{Def}

\begin{Def}\label{def:off-shell bv construction}
	An \defi{off-shell (infinitesimal) BV construction} is a (infinitesimal) BV construction whose symplectic reduction is obtained either by: 
	\begin{itemize}
		\item a quotient of $S$ by a groupoid of off-shell symmetries. 
		\item or a weak infinitesimal quotient coming from an algebroid of infinitesimal off-shell symmetries.
	\end{itemize}
\end{Def}

\begin{RQ}
	Not only there are plenty of examples of such off-shell symmetries, but we will see in Sections \ref{sec:lie-algebroid-moment-maps} and \ref{sec:bv-construction-for-groupoid-action} that for any Lie algebroid of off-shell symmetries (infinitesimal or not), we can construct an (infinitesimal or not) BV construction essentially given by the \emph{equivariant derived critical locus}. 
\end{RQ}

\begin{Prop}\label{prop:off-shell infinitesimal symmetries from off-shell symmetries}
	Let $\_G$ be a smooth groupoid of off-shell symmetries over $X$ satisfying Assumptions \ref{ass:very good stack}. Then its associated Lie algebroid $\_L$ is a Lie algebroid of infinitesimal off-shell symmetries. 
	
	Moreover, any smooth off-shell BV construction on $S$ induces an infinitesimal off-shell BV construction. 
\end{Prop}
\begin{proof}
	Similarly to the proof of Theorem \ref{th:formal completion of moment map and symplectic reductions}, we consider the following commutative diagram: 
	\[ \begin{tikzcd}
		S \arrow[r] \arrow[d] & X \arrow[d] \\
		\comp{\QS{S}{\_G}_S} \arrow[r] \arrow[d]& \QS{X}{\_L} \simeq \comp{\QS{X}{\_G}_X} \arrow[d] \\
		\QS{S}{\_G} \arrow[r] & \QS{X}{\_G}
	\end{tikzcd}\]
	
	Just like in the proof of Theorem \ref{th:formal completion of moment map and symplectic reductions}, all the squares of the diagram are pullback squares and $\comp{\QS{S}{\_G}_S}$ defines a weak infinitesimal quotient. 
	
	Moreover $f$ is $\_L$-equivariant because the following diagram commutes: 
	\[ \begin{tikzcd}
		X\arrow[d] \arrow[r, "f"] & \Aa^1 \\
		\QS{X}{\_L}\arrow[d] \arrow[ur, dashed] & \\
		\QS{X}{\_G} \arrow[uur]
	\end{tikzcd} \]
	
	Therefore, $\_L$ and $\comp{\QS{S}{\_G}_S}$ define an off-shell infinitesimal symmetry. Moreover if $\QS{S}{\_G}$ is part of a Lagrangian correspondence making it an off-shell BV construction, then by Lemma \ref{lem:formal completion lagrangian correspondence remains a lagrangian correpspondence}, $\comp{\QS{S}{\_G}_S}$ is also part of a Lagrangian correspondence, making it an off-shell infinitesimal BV construction.
\end{proof}

\subsection{Strictification and BV Charge}\label{sec:bv-charge}\

\medskip

The goal of this section is to compare the infinitesimal BV construction to that of \cite{FK14}. One of the main difficulty to compare our construction to the classical one is that we define an infinitesimal quotient stack while they only construct the associated \CE algebra. \\

Unfortunately, the tangent and cotangent complexes of the \CE algebras are complicated to compute mostly because it is unclear how to take a cofibrant resolution of this algebra even for perfect Lie algebroids. We expect that working with the graded mixed \CE algebras instead might simplify this issue. \\

In order to still be able to still discuss the comparison between our geometric description of the BV construction and the algebraic construction, we will need to restrict ourself to an \emph{underived setting}. In this section only, we will consider underived construction. In particular, tensor product, direct sums, Hom functors will \emph{not} be derived in this section.

We will adopt a framework similar to the one described in \cite[Section 3.2]{PS20}. We will consider perfect Lie algebroids concentrated in non-positive degrees over an almost derived critical locus $S = \Spec(R)$ with: \[R :=  \left(\left(\Sym_A \Tt_A[1] \oplus \_L_S [2]\right), \delta_S \right)\]   

We will restrict to non-positively graded perfect Lie algebroids, consider the \emph{underived} tangent complex, $\_T_B$ defined as the underived module of derivations of $B$. Since $A$ is cofibrant, we have that $\_T_A \simeq \Tt_A$. We will keep writing $\Tt_A$ bearing in mind that we use the model given by $\_T_A$. 

Since we take a perfect $\_L_\infty$-algebroid $\_L$, its \CE algebra is the completion of a semi-free algebra. Up to picking a connection, we have that: 
\[ \_T_{\ceu(\_L)} \simeq {\ceu}(\_L) \otimes_A \left(\Tt_A \oplus \_L[1] \right)  \]
together with a differential making it a module over $\ceu(\_L)$. We also have: 
\[ h^*\_T_{\ceu(\_L)} \simeq \Tt_A \oplus^\rho \_L[1]  \]
with $h: \ceu(\_L) \to A$. Note that this is exactly the same as $h^*\Tt_{\QS{X}{\_L}}$. \\

We will consider an underived analogue to the de Rham algebra defined as: 
\[ \DRs(B) := \left( \cSym_B \Omega_B^1[-1], D = \dr + \delta_B \right)\]

where $\Omega_B^1$ is the module of Kähler differentials whose dual is $\_T_B$. 

Just like the derived version, this algebra is filtered complete and $\dr$ increases the weight by $1$ while $\delta_B$ preserves the weight. Then (closed) $p$-forms of degree $n$ on $\ce(\_L)$ are \emph{defined} similary to Proposition \ref{prop:realization of de rham and (closed) forms} by: 
\[\_A^{p}(B, n) \simeq  \relg{\DRs(p)(B)[n+p]}\] 
\[\_A^{p,\tx{cl}}(B, n) \simeq  \relg{ F^p \cpl{\DRs}(B)[n+p]} \simeq \relg{ \prod_{i\geq p} \DRs(p+i)(B)[n+p] } \] 

We refer to \cite[Section 3.4]{PS20} for the more explicit description of these forms.

\subsubsection{Strictification}\
\label{sec:strictification}

\medskip

In this section we will consider an infinitesimal BV construction as the following data, fixing an almost derived critical locus $S$ of almost finite presentation and assuming that it admits an infinitesimal BV construction: 
\begin{itemize}
	\item $Y := \Spec(\ceu(\_L))$ where $\_L$ is a perfect Lie algebroid on $S$ concentrated in non-positive degree. 
	\item $\omega_Y$ a $(-1)$-shifted symplectic structure on $Y$. 
	\item $\gamma$ the data of the Lagrangian correspondence defining the generalized symplectic reduction. 
\end{itemize} 

Such a tuple will denoted $(L, S, \omega_Y, \gamma)$. Moreover, we are going to consider $S$ of the form $S = \Spec(R)$ with: 
\[R := \left( \Sym_A \left(\Tt_A[1] \oplus \_L_S [2]\right), \delta_S \right)\]

We will be interested in understanding these infinitesimal BV constructions only up to the following notion of equivalence: 

\begin{Def}\ \label{def:equivalences of BV constructions}
	An equivalence $(\_L, S, \omega_Y, \gamma) \to (\_L', S, \omega_Y', \gamma')$ is the following data: 
	\begin{itemize}
		\item An equivalence between the \CE algebras: 
		\[f: \ceu(\_L') \overset{\sim}{\to} \ceu(\_L)\]
		Note that this induces an equivalence between the ``infinitesimal quotient stacks'', $Y \overset{\sim}{\to} Y'$. 		
		\item A homotopy between closed $2$-forms:
		\[\omega_Y \sim f^* \omega_{Y'}\]
		\item A homotopy between the Lagrangian correspondences: \[\gamma \sim f^* \gamma'\]
	\end{itemize}
\end{Def}

For now let us chose $S$ that admits a generalized symplectic reduction and $(\_L, S, \omega_Y, \gamma)$ be an infinitesimal BV construction. 

\begin{Lem}\label{lem:equivalence lie algebroid of symmetry and tate terms}
	
	The non-degeneracy condition of the Lagrangian correspondence implies that there is a quasi-isomorphism:
	\[ \phi : R \otimes_{A} \_L_S^\vee[-2] \overset{\sim}{\to} \_L^\vee[-2] \]
\end{Lem}
\begin{proof}
	Consider the following commutative diagram, obtained by using a connection on $\RCrit(f)$ and on $R$, viewed as semi-linear stack on $X=\Spec(A)$ (for simplicity we will omit the notation $\nabla$ and keep in mind that the differential in the following diagram are non-trivial): 
	
	\begin{equation} \label{dia:equivalence lie algebroid of symmetry and tate terms}
		\adjustbox{scale=0.6,center}{
			\begin{tikzcd}[row sep=2cm, column sep=tiny]
				R \otimes_{A} \_L_S^\vee [-2] \arrow[dd, "\tx{Id}"] \arrow[dr, "\phi", dashed] \arrow[rr] & &\_T_{R} \arrow[dd] \arrow[dr, "\simeq"] \arrow[rr] & & R \otimes_{A} \left(\Tt_A \oplus \Ll_A [-1] \right) \arrow[dr, "\tx{Id}"] \arrow[dd] & \\
				& \_L^\vee[-2] \arrow[rr] \arrow[dd, "\tx{Id}", near start] & & R \otimes_{A} \left( \Tt_A \oplus \Ll_A[-1]\right) \oplus \_L^\vee[-2] \arrow[rr] \arrow[dd] & & R \otimes_{A} \left( \Tt_A \oplus \Ll_A[-1] \right) \arrow[dd] \\
				R \otimes_{A} \_L_S^\vee [-2] \arrow[rr] \arrow[dr, dashed] & & \_T_{R} \oplus \_L[1] \simeq h^*\_T_{\ceu(\_L)} \arrow[rr] \arrow[dr, "\omega_0^\flat"] & & R \otimes_{A} \left(\Tt_A \oplus \Ll_A[-1] \right) \oplus \_L[1] \arrow[dr, dashed] & \\
				& \_L^\vee[-2] \arrow[rr] & & \Omega^1_{R}[-1] \oplus \_L^\vee[-2] \simeq h^* \Omega^1_{\ceu(\_L)}[-1] \arrow[rr] & & \Omega_{R}^1[-1]
			\end{tikzcd} 
		}
	\end{equation}
	
	We have that: 
	\begin{itemize}
		\item All sequences from left to right are fibered.
		\item The rightmost square of the front and back faces are homotopy Cartesian. 
		\item The plain maps from the front to the back face are related through the symplectic structures and the natural quasi-isomorphism obtained from the Lagrangian correspondence. 
		\item Therefore dashed maps obtained by universal properties are weak equivalences.
	\end{itemize} 
\end{proof}

Using the results of Section \ref{sec:infty-morphism-and-homotopy-transfer-theorem}, we can transfer the $\_L_\infty$-algebroid structure from $\_L$ to $R \otimes \_L_S$ such that their \CE algebras are equivalent (via an $\infty$-quasi-isomorphism). Pulling back the symplectic structure and the Lagrangian correspondence along this equivalence we get: 

\begin{Prop}\ \label{prop:strictification}
	$(\_L, S, \omega_Y, \gamma)$ is equivalent in the sens of Definition \ref{def:equivalences of BV constructions} to: 
	\[ (\pi_S^* \_L_S, \phi^*\omega_Y, \phi^* \gamma)\]
	
	With $\pi_S : S \to X$.
	In particular, the \CE algebra of this equivalent $\_L_\infty$-algebroid is given by: 
	\[ \Sym_A \left( \Tt_A[1] \oplus \_L_S[2] \oplus \widehat{\_L_S^\vee[-1]} \right)\]
	where we use Notation \ref{not:partial completion} to denote the completion with respect to the \CE terms given by $\_L_S^\vee[-1]$.
\end{Prop} 
\begin{proof}
	Using the equivalence $\phi$ between perfect $\_L_\infty$-algebroids, this implies that the associated \CE algebras are equivalent since the \CE functor preserves weak equivalences between $\_L_\infty$-algebroid whose underlying module are cofibrant ($A$-cofibrant  $\_L_\infty$-algebroids). 
\end{proof}

\begin{RQ}\label{rq:non action algebroid}
	There is this curious observation that $\pi_S^* \_L_S$ is a $\_L_\infty$-algebroid on $S$ although $\_L_S$ is a priori not a an algebroid over $X$. If $\_L_S$ is itself a $\_L_\infty$-algebroid, we are going toward the idea of \emph{off-shell symmetries}. 
\end{RQ}

From now we will work with $(\pi_S^* \_L_S, \phi^*\omega_Y, \phi^* \gamma)$ and drop the $\phi^*$ notation. Now we will try to describe the symplectic structure we obtain.

\begin{Lem}\label{lem:pullback 2-form description}
	$h^*\omega_0$ is\footnote{Here $h^\star \omega_0$ is the map induced on the pullback $h^\star\omega_0 : h^* \Tt_{\ceu{(\_L)}} \to h^* \Ll_{\ceu{(\_L)}}[-1]$. It should not be confuse with the pullback of the $2$-form, denoted $h^*\omega_0$ given by the composition:
		\[ \Tt_R \to  h^\star \Tt_{\ceu{(\_L)}} \to h^* \Ll_{\ceu{(\_L)}}[-1] \to \Ll_R[-1] \]} part of the following commutative diagram where rows are split exact sequences:
	
	\begin{equation*}
		\adjustbox{scale=0.75,center}{\textbf{\textbf{}}
			\begin{tikzcd}[column sep = 1mm]
				R \otimes_A (\_L_S^\vee [-2] \oplus \_L_S[1])  \arrow[d, equals] \arrow[r] & h^* \_T_{\ceu(\_L)} \simeq R \otimes_A \left( \Tt_A \oplus \Ll_A[-1] \oplus \_L_S^\vee [-2] \oplus \_L_S[1] \right) \arrow[d, "h^* \omega_0 "] \arrow[r] & i^* \_T_{\RCrit(f)} \simeq R \otimes_A \left( \Tt_A \oplus \Ll_A[-1] \right) \arrow[d, "i^*\omega_{\RCrit(f)}"] \\
				R \otimes_A (\_L_S^\vee [-2] \oplus \_L_S[1]) \arrow[r] & h^*\Omega^1_{\ceu(\_L)}[-1] \simeq R \otimes_A \left( \Tt_A \oplus \Ll_A[-1] \oplus \_L_S^\vee [-2] \oplus \_L_S[1] \right) \arrow[r] & i^* \Omega^1_{\RCrit(f)}[-1] \simeq R \otimes_A \left( \Tt_A \oplus \Ll_A[-1] \right) 
			\end{tikzcd}
		}
	\end{equation*}
	
	In particular, $h^*\omega_0 \sim i^* \omega_{\RCrit(f)} + \omega_\phi$ with $\omega_\phi$ the part of $h^*\omega_0$ corresponding to the identity of $R \otimes_A (\_L_S^\vee [-2] \oplus \_L_S[1])$.  
\end{Lem}

\begin{proof}
	Using connections and Diagram \eqref{dia:Dia_1}, the bottom face gives us the diagram:

	\begin{equation*}
		\adjustbox{scale=0.75,center}{
			\begin{tikzcd}[column sep = 1mm]
				R \otimes_A \_L_S^\vee [-2]  \arrow[d, "\phi"] \arrow[r] & h^* \_T_{\ceu(\_L)} \simeq R \otimes_A \left( \Tt_A \oplus \Ll_A[-1] \oplus \_L_S^\vee [-2]  \right)\oplus \_L[1] \arrow[d, "h^* \omega_0"] \arrow[r] & i^* \_T_{\RCrit(f)}\oplus \_L[1] \simeq R \otimes_A \left( \Tt_A \oplus \Ll_A[-1] \right)\oplus \_L[1] \arrow[d] \\
				\_L^\vee[-2] \arrow[r] & h^*\Omega^1_{\ceu(\_L)}[-1] \simeq R \otimes_A \left( \Tt_A \oplus \Ll_A[-1]  \oplus \_L_S[1] \right)\oplus \_L^\vee [-2] \arrow[r] & \Omega^1_R[-1] \simeq R \otimes_A (\Tt_A \oplus \Ll_A[-1] \oplus \_L_S[1]) 
			\end{tikzcd}
		}
	\end{equation*}
	
	Using $\phi$ to replace $\_L$ by $R \otimes_A \_L_S$ via Lemma \ref{lem:equivalence lie algebroid of symmetry and tate terms}, we get: 
	
	\begin{equation*}
		\adjustbox{scale=0.75,center}{
			\begin{tikzcd}[column sep = 1mm]
				R \otimes_A \_L_S^\vee [-2]  \arrow[d, equals] \arrow[r] & h^* \_T_\ceu(\_L) \simeq R \otimes_A \left( \Tt_A \oplus \Ll_A[-1] \oplus \_L_S^\vee [-2]  \oplus \_L_S[1]\right) \arrow[d, "h^* \omega_0"] \arrow[r] & i^* \_T_{\RCrit(f)}\oplus R \otimes_A \_L_S[1] \simeq R \otimes_A \left( \Tt_A \oplus \Ll_A[-1] \oplus \_L_S[1]\right) \arrow[d] \\
				R \otimes_A \_L_S^\vee [-2]  \arrow[r] & h^*\Omega^1_\ceu(\_L)[-1] \simeq R \otimes_A \left( \Tt_A \oplus \Ll_A[-1]  \oplus \_L_S[1] \oplus \_L_S^\vee [-2]\right) \arrow[r] & \Omega^1_R[-1] \simeq R \otimes_A (\Tt_A \oplus \Ll_A[-1] \oplus \_L_S[1]) 
			\end{tikzcd}
		}
	\end{equation*}
	
	Taking the fiber of this whole diagram object-wise at the natural projection to the diagram:

	\begin{equation*}
		\adjustbox{scale=1,center}{
			\begin{tikzcd}
				R \otimes_A \_L_S  \arrow[d, equals] \arrow[r] & 0 \arrow[d, equals] \arrow[r] & R \otimes_A \_L_S[1] \arrow[d] \\
				R \otimes_A \_L_S   \arrow[r] & 0  \arrow[r] & R \otimes_A \_L_S[1]
			\end{tikzcd}
		}
	\end{equation*}
	
	we get exactly the diagram from the Lemma and it follows from Diagram \eqref{dia:Dia_1} that the vertical morphism are the identities.
\end{proof}

\begin{Def}\label{def:symplectic structure strict}
	We define a (necessarily strict) symplectic $2$-form $\omega^{\tx{st}}$ of degree $-1$ on $\ceu(\_L)^\sharp$ by the following map of graded space: 
	
	\[ \begin{split}
		\_T_{\ceu(\_L)}^\sharp =& \left( \ceu(\_L) \otimes_R i^* \Tt_{\RCrit(f)} \oplus \ceu(\_L) \otimes_{A} \left( \_L_S^\vee [-2] \oplus \_L_S[1] \right)\right)^\sharp \\
		\to & \left( \ceu(\_L) \otimes_{R} i^* \Ll_{\RCrit(f)}[-1] \oplus \ceu(\_L) \otimes_{A} \left( \_L_S^\vee [-2] \oplus \_L_S[1] \right)\right)^\sharp = \left(\Omega^1_{\ceu(\_L)}\right)^\sharp [-1]
	\end{split} \] 
	given by the identity plus $i^\star\omega_{\RCrit(f)}^\flat$. This is a $(-1)$-shifted symplectic form when we forget the differential on $A$ (denoted by the $^\sharp$ symbol).
\end{Def}
\begin{RQ}
	We define $\omega^{\tx{st}}$ as the identity plus the canonical symplectic structure from $\RCrit(f)$, and therefore its pullback, $h^\star \omega^{\tx{st}}$, is exactly given by $i^* \omega_{\RCrit(f)} + \omega_\phi$. This is compatible with the result of Lemma \ref{lem:pullback 2-form description}. However, nothing ensures us that $\omega^{\tx{st}}$ is compatible with the differential so it might not even be a $2$-form on $\ceu(\_L)$ itself.
\end{RQ}

\begin{Th}\label{th:strictification symplectic structure}
	The $(-1)$-shifted symplectic form $\omega$ on $Y$ can be described by\footnote{Note that neither summand is actually a $2$-form as they might not be compatible with the differential.}: 
	\[ \omega =  \omega^{\tx{st}} + \tilde{\omega} \]
	
	with $h^\star\omega^{\tx{st}} = i^* \omega_{\RCrit(f)} + \omega_\phi$ and $h^\star \tilde{\omega} = 0$. 
\end{Th}

\begin{proof}
	From the following Lemma \ref{lem:pullback 2-form description}, we know that $h^\star\omega_0 = h^\star \omega^{\tx{st}}$. Now, it follows that $\omega - \omega^{\tx{st}} = \tilde{\omega}$ is such that $h^\star\tilde{\omega} = 0$.
\end{proof}

In the end, we find that although our construction has similarities to the classical construction of BV, working up to homotopy and with almost derived critical loci adds so much more flexibility that we cannot completely recover a notion of ``uniqueness up to equivalence'' of the BV construction as can be found in \cite{FK14} for example.\\

Moreover, our construction does not provide a \emph{BV charge} and we do not have a good way to phrase this condition in a meaningful homotopy invariant way. However, even if we do not have a BV charge, the infinitesimal quotient we obtain is $(-1)$-shifted symplectic with a symplectic structure compatible with the differential on the algebra of functions. We do not think that a BV charge should exist in such degree of generality. \\

Nevertheless, it turns out that a rather large class of example of BV construction admits a BV charge. As we will see, a BV construction is Hamiltonian as soon as the symplectic structure is given by $\omega^{\tx{st}}$ (without the extra terms). \\

\subsubsection{Hamiltonian vector field and BV charge}\ \label{sec:hamiltonian-vector-field-and-bv-charge} 

\medskip

The goal of this section is to prove that if we have a BV construction $ (\pi_S^* \_L_S, \omega_Y, \gamma)$ such that $\omega_Y = \omega^{\tx{st}}$, then there is a BV charge and we can give a formula for the BV charge provided by the homological perturbation lemma. \\

The BV charge is nothing but a Hamiltonian vector field for the natural differential on the \CE algebra, $\ceu(\pi_S^*\_L_S)$, associated to the Lie algebroid of symmetries on $S$. To simplify the notations, in this section, we will denote $\_L := \pi_S^*\_L_S$ the $\_L_\infty$-algebroid over $R$. The key insight that will make this section work is the following proposition: 
\begin{Prop}\label{prop:iso polyvector field and de rham}
	
	There is an \emph{isomorphism} (we are considering the \emph{strict} construction of these graded mixed algebras) of graded mixed algebras: 
	\[ \Pol(\ceu(\_L), -1) \overset{\simeq}{\to} \gmc{\DRs}(\ceu(\_L)) \]
	such that:
	\begin{itemize}
		\item $\Pol(\ceu(\_L),-1)$ denotes the graded mixed complex of $(-1)$-shifted polyvector fields on $\ceu(\_L)$ with mixed differential given by $\left[ \Pi^{\tx{st}}, \right]_{\tx{SN}}$, where the bracket is the \defi{Schouten--Nijenhuis} bracket, and $\Pi^{\tx{st}}$ is the bivector field defining the (strict) Poisson structure dual to $\omega^{\tx{st}}$. 
		\item $ \gmc{\DRs}$ is the graded mixed version of $\DRs$ described at the beginning  of Section \ref{sec:bv-charge}.
		\item The map induces the identity of $\ceu(\_L)$ in weight $0$ and sends a derivation $X \in \_T_{\ceu(\_L)}$ to $\iota_X \omega^{\tx{st}} \in \Omega^1_{\ceu(\_L)}[-1]$. 
	\end{itemize}
\end{Prop}

\begin{Def}
	A vector field $X \in \_T_{\ceu(\_L)}$ of degree $1$ is called \defi{Hamiltonian} if: \[X = \left[\Pi^{\tx{st}}, Q \right]\] 
	for some $Q \in \ceu(\_L)$ (the part of weight 0). In particular this is equivalent, using Proposition \ref{prop:iso polyvector field and de rham}, to saying that $\iota_X \omega^{\tx{st}} = \dr Q$.  
\end{Def}

If we consider $\delta_{\ceu} \in \_T_{\ceu(\_L)}$ of degree $1$. Then finding a BV charge amounts to show that $\iota_{\delta_{\ceu}} \omega^{\tx{st}}$ is $\dr$-exact. 
First we will see with Theorem \ref{th:deformation retract ce strong} that we have a strong deformation retract: 
\[ \begin{tikzcd}
	\DRs(R)\arrow[r, shift left, "i_\infty"] & \arrow[l, shift left, "p"] \DRs(\ceu(\_L)) \arrow[loop right, "h_\infty"] 
\end{tikzcd}\]

Then $\iota_{\delta_{\ceu}} \omega^{\tx{st}}$ is exact \emph{for the total differential} if and only if $p(\iota_{\delta_{\ceu}} \omega^{\tx{st}})$ is exact for the total differential.
We will see in Corollary \ref{cor:projection of the constraction} that $p(\iota_{\delta_{\ceu}} \omega^{\tx{st}}) = \dr f$ and restricted to the weight $0$ part, given by $A$, the total differential is $D = \delta_A + \dr = \dr$. Therefore we have that: 
\[ \iota_{\delta_{\ceu}} \omega^{\tx{st}} = D(i_\infty(f) + h_\infty(\iota_{\delta_{\ceu}} \omega^{\tx{st}})) \] 

Moreover, in Lemma \ref{lem:properties infty morphism with de rham weight} we show that $i_\infty$ preserves the de Rham weight and $h_\infty$ decreases it by $1$. Since $\iota_{\delta_{\ceu}} \omega^{\tx{st}}$ has de Rham weight at most $1$ (since it is valued in $\Omega^1_{\ceu(\_L)}$), we have that: \[ Q := i_\infty(f) + h_\infty(\iota_{\delta_{\ceu}} \omega^{\tx{st}}) \in \ceu(\_L)\]
Moreover, $Q$ is of degree $0$ and therefore $Q \in \ceu(\_L)$. This implies that: 
\[DQ = \dr Q + \delta_\ce Q\] 
Since the de Rham weight of $\iota_{\delta_{\ceu}} \omega^{\tx{st}}$ is exactly $1$ and $\delta_\ce Q$ has de Rham weight $0$, this implies that $\delta_\ce Q = 0$ and we get: 
\[ \iota_{\delta_{\ceu}} \omega^{\tx{st}} = \dr Q\]

We have proven the following:

\begin{Th}\label{th:BV charge}
	A BV construction as above with strict symplectic structure has a BV charge given by the formula: \[ Q := i_\infty(f) + h_\infty(\iota_{\delta_{\ceu}} \omega^{\tx{st}}) \in \ceu{(\_L)}\]
\end{Th}

We now need to state and prove the claims we used to prove this theorem. First recall that the \CE  differential respects the \CE filtration and can therefore be decomposed in weight components:
\[ \delta_{\ceu} := \delta_0 + \delta_1 + \cdots\]

In what follows we will consider two different filtrations: \begin{itemize}
	\item The filtration coming from the de Rham algebra we will call the \defi{de Rham weight}. It is denoted as before by $F^p \DRs(\ceu(\_L))$
	\item The filtration induced by the \CE filtration that we will call the \defi{\CE weight}. It is denoted by $F_\_L^p \DRs(\ceu(\_L))$.
\end{itemize}  
We refer to Definition \ref{def:filtration vertical and de rham} for more details on those filtrations.

\begin{Lem} \label{lem:contraction of delta0}
	We have the following: 
	\begin{enumerate}
		\item 	$\delta_0$ restricted to $R$ is exactly the differential on $R$, $\delta_{R}$.
		\item   The restriction of $\delta_0$ to $\_L_S[2]$ is such that $(\iota_{\delta_{0}})_{\vert \_L_S[2]} \omega \in F_\_L^1 \DRs(A)$
		\item For $g \in A$, $\iota_{dg}$ is a derivation on $\Sym_{A} \Tt_A[1]$ and $\iota_{\iota_{dg}} \omega = \dr g$.
	\end{enumerate}
\end{Lem}
\begin{proof}
	The first claim is clear by definition of the \CE differential. \\
	
	For the second claim, $(\delta_{0})_{\vert \_L_S[2]}$ is zero on element generated by $A$ and $\Tt_A[1]$ and non-zero on $\_L_S[2]$. Such a derivation is a linear combination of elements of the form $a \frac{\partial}{\partial \pi}$ for $\pi \in \_L$ and $a \in R$. But $\iota_{a \frac{\partial}{\partial \pi}} \omega^{\tx{st}} = a \omega^{\tx{st}} ( \frac{\partial }{\partial \pi}, -)$. Since $\omega^{\tx{st}}$ pairs $\_L[2]$ and $\_L^\vee [-1] \in F^1 \DR(A)$, it shows that $ a \omega ( \frac{\partial }{\partial \pi}, -) \in F^1 \DR(A) \cap F_\_L^1 \DR(A)$. \\
	
	For the last claim, everything happens as if we are working on $\RCrit(f)$ with its canonical symplectic structure. By definition we have:
	\[ \iota_{\iota_{dg}}\omega_0 = \omega_0^\flat(\iota_{dg})\]
	
	We want to translate the problem in the Poisson setting. To do so we recall that $\omega_0^\flat$ and $\Pi^\flat$ are inverses of each other. Applying $\Pi^\flat$ to the equality $\iota_{\iota_{dg}} \omega_0 = dg$, we obtain: 
	\[\Pi^\flat \iota_{\iota_{dg}} \omega_0 = \Pi^\flat (dg)  \]
	
	But $\Pi^\flat \iota_{\iota_{dg}} \omega_0  = \Pi^\flat  \omega_0^\flat (\iota_{dg}) = \iota_{dg}$. Therefore we are left to prove that: 
	\[ \iota_{dg} = \Pi^\flat (dg) = \lbrace g, -\rbrace  \]
	
	But on $\RCrit(f)$, the Poisson structure is on $\Sym_{\_O_X} \Tt_X[1]$ and is induced by the pairing $\langle f, X \rangle = X(f)$ because it is the Poisson structure associated to the canonical symplectic structure. Therefore $\lbrace g, - \rbrace$ sends an element of $\_O_X$ to $0$ and sends $X \in \Tt_X[1]$ to $X(g) = dg.X = \iota_{dg}X$, which shows that $\iota_{dg} X = \lbrace g, X \rbrace$.  
\end{proof}

The lemma above describes $\iota_{\delta_0} \omega^{\tx{st}}$. We now turn to the description of the terms of higher filtration $\delta_{\geq 1}$. 

\begin{Lem}\label{lem:contraction higher ce differential}
	\[ \iota_{\delta_{\geq 1}} \omega^{\tx{st}} \in F_\_L^1 \DRs(\ceu(\_L)) \]
\end{Lem}
\begin{proof}
	Since $\delta_{\geq 1}$ is valued in $ F_\_L^1 \DRs(\ceu(\_L))$ then so is $\iota_{\delta_{\geq 1}} \omega^{\tx{st}}$.
\end{proof}

\begin{Cor}\label{cor:projection of the constraction}
	\[p(\iota_{\delta_{\ceu}} \omega^{\tx{st}}) = \dr f\]
\end{Cor}
\begin{proof}
	We use Lemma \ref{lem:contraction of delta0} and Lemma \ref{lem:contraction higher ce differential} in combination with Lemma \ref{lem:kernel p} saying that: \[F_\_L^1 \DRs(\ceu(\_L)) \subset \ker(p)\]
	This shows that:
	\[p(\iota_{\delta_{\ceu}} \omega^{\tx{st}}) = p(\dr f) = \dr f\]
	
\end{proof}

\subsubsection{Deformation retract of de Rham algebras}\label{sec:deformation-retract-of-de-rham-algebras}\

\medskip

The goal of this section is to compute the de Rham complexes of linear and \CE stacks. This is a purely technical section that is required for the proof of Theorem \ref{th:BV charge}.\\

\ttl{For linear derived stacks}\

\medskip

Let $X = \bf{Spec}(A)$ with $A$ cofibrant of almost finite presentation and $\_F \in \perf(X)$ concentrated in non-negative degrees. The linear stack associated to $\_F$ can be represented by $B = \Sym_A \_F^\vee \in \cdgacon$ endowed with the differential $\delta_B = \delta_A + \delta_{\_F^\vee}$. 

\begin{notation}\label{not:coordinate notation}
	In this section we will use the following notations:
	
	\[\beta_{1 \cdots k } := \beta_1 \wedge  \cdots \wedge \beta_k\]
	\[ \widecheck{\beta_{1 \cdots k}^i} =  \beta_1 \wedge  \cdots  \wedge \beta_{i-1} \wedge \beta_{i+1} \wedge \cdots \wedge \beta_k\]
	
\end{notation}

We will write: 
\[\gr{\DRs}(B) \simeq \gr{\Sym}_A \Ll_A[-1] \]
And the mixed structure of $\gmc{\DRs}$ is the de Rham differential. Moreover, using a connection on $\_F$, we have: 
\[  \Omega^1_B \simeq \pi^* (\_F \oplus^\nabla \Ll_A) \]

Then we get: 
\[\begin{split}
	\DRs(B) \simeq &\cSym_A \Ll_A[-1] \otimes_A \cSym_A \_F^\vee [-1] \otimes_A \Sym_A \_F^\vee \\
	\simeq & \Sym_A\left(\widehat{\Ll_A[-1]} \oplus \widehat{\_F^\vee [-1]} \oplus \_F^\vee \right)
\end{split}\] 

Where the differential on the right hand side is describe by the following:

\begin{Lem}\label{lem:differential DR complex linear stacks}
	The total differential is given by:
	\[D_B = \dr + \delta_{\_F^\vee} + \delta_A^v + \delta_{\_F^\vee}^v + \delta_\nabla\]
	
	with: 
	
	\begin{itemize}
		\item $\dr$ the de Rham differential. 
		\item $\delta_A^v$ the vertical differential on $\DRs(A)$ extended by zero on $\_F^\vee[-1]$ and $\_F^\vee$.
		\[\delta_A^v = \delta_A + \delta_{\Ll_A}\] 
		\item $\delta_{\_F^\vee}^v$ the differential on $\_F^\vee[-1]$ coming from the differential $\delta_{\_F^\vee}$ on $\_F^\vee$ extended $B$-linearly and by $0$ on $\Ll_A[-1]$ and $\_F^\vee$. 
		\item $\delta_{\_F^\vee} = \delta_B - \delta_A$ is the differential on $\_F^\vee$. 
		\item $\delta_\nabla$ the part of the differential induced by the connection, valued in: \[\cSym_A \Ll_A[-1] \otimes_A \Sym_A^{\geq 1} \_F^\vee\] 
	\end{itemize} 
\end{Lem}
\begin{proof}
	By definition we have $D_B = \dr + \delta_B^v$ the decomposition in de Rham and vertical differential. Moreover $\delta_B^v$ restricted to $B$ is $\delta_B$ itself and $\delta_B^v$ on $\Ll_B[-1]$ is the natural differential on $\Ll_B[-1]$. Using Proposition \ref{lem:tangent complex semi-linear stacks}, we have: \[\Ll_B[-1] \simeq B \otimes_A \left( \Ll_A[-1] \oplus^\nabla \_F^\vee[-1] \right)\] where the right hand side is the $A$-module with differential given by \[\delta_B + \delta_{\_F^\vee}^v + \delta_{\Ll_A} + \delta_\nabla\]
	
	Moreover $\delta_B + \delta_{\Ll_A} = \delta_{\_F^\vee} + \delta_A^v$ and $\delta_\nabla$ is the part of the differential sending $\_F^\vee[-1]$ to an element in $\left(\Ll_A[-1] \otimes_A  \_F^\vee\right)[-1]$ via the covariant derivative. 
\end{proof}

We are now ready to define the maps we need for the deformation retract. Before that we want to define two filtrations: 
\begin{itemize}
	\item The \defi{de Rham filtration}, denoted $F^p\DRs(B)$, given by the complete filtration associated with the graded mixed structure on $\DRs(B)$. This is the filtration associated with the ideal generated by $\Ll_A[1]$ and $\_F^\vee[-1]$ (see Example \ref{ex:filtered complexes}). 
	\item The \defi{$\_F$-filtration} denoted $F_\_F^p \DRs(B)$ induced by the natural filtration on $\Sym_B \_F^\vee$ of the symmetric algebra, generated by the terms $\_F^\vee$ and $\_F^\vee[-1]$ in $\DRs(B)$.  
\end{itemize}

\begin{Def}\label{def:deformation retract linear stack}
	There are maps: 
	\[ \begin{tikzcd}
		\DRs(A) \arrow[r, shift left, "i"] & \arrow[l, shift left, "p"] \DRs(B) \arrow[loop right, "h"] 
	\end{tikzcd}\]

	with: 
	\begin{itemize}
		\item $i$ is the map induced by the projection $\Aa_X (\_F) \rightarrow X$ and therefore respects the total differentials. This map preserves both the de Rham filtration and the $\_F$-filtration. 
		\item 	$p$ is the map induced by the $0$ section $X \rightarrow \Aa_X (\_F)$ and therefore respects the total differentials. It also preserves both filtrations.  
		\item 	$h$ is a degree $-1$ derivation defined on homogeneous elements, $f \in B$, $\alpha_i \in \Ll_A[-1], i=1\cdots p$, $d\xi_i \in \_F^\vee [-1], i=1 \cdots m$ and $\eta_i \in \_F^\vee, i=1 \cdots m'$ by:
		\[ h(f \alpha_{1 \cdots p} \wedge (d\xi)_{1  \cdots m} \wedge \eta_{1 \cdots m'}) = \frac{1}{m+m'} \sum_{i=1}^m \epsilon_h f  \alpha_{1 \cdots p}  \wedge \widecheck{(d\xi)_{1 \cdots m}^i} \wedge \xi^i \wedge \eta_{1\cdots m'}  \]
		
		With $\epsilon_h$ obeying the sign rule for $h$ a degree -1 derivation: 
		\[ \epsilon_h = (-1)^{n_h}, \qquad n_h = \sum_{j=1}^{i-1} ( \vert \xi_j \vert +1 ) + \sum_{j= i+1}^m (\vert \xi_j \vert +1 ) \vert \xi_i \vert  \] 
		
		In short, $h$ turns a $d\xi$ to a $\xi$. Therefore $h$ preserves the $\_F$-filtration and reduces the de Rham weight by $1$. 
	\end{itemize}
\end{Def}

\begin{Lem}\label{lem:kernel p}
	\[ \ker(p) = \prod_{p+q \geq 1} \cSym_B \Ll_B[-1] \otimes_B \cSym_B^{\geq p} \_F^\vee[-1] \otimes_B \Sym_B^{\geq q}\_F^\vee := F_\_F^1 \DRs(B) \]
\end{Lem}
\begin{proof}
	In other words, we need to show that the kernel of $p$ is the ideal of $\DRs(B)$ generated by $\_F^\vee$ and $\_F^\vee[-1]$. Since $p$ is the map induced by the $0$ section of the linear stack, it corresponds to the projection $\Sym_B \_F^\vee \to B$ and induces the projection:
	\[\DRs(B) \simeq \Sym_A \left(\widehat{\Ll_A[-1] \oplus \_F^\vee[-1]}  \oplus \_F^\vee\right) \to \cSym_A \Ll_A[-1] \simeq \DRs(A) \] 
	Therefore the kernel of this map is indeed the ideal generated by $\_F^\vee$ and $\_F^\vee[-1]$. 
\end{proof}

\begin{Th}\label{th:deformation retract linear stacks}
	$i$, $p$ and $h$ define a deformation retract:
	\[ \begin{tikzcd}
		\DRs(A) \arrow[r, shift left, "i"] & \arrow[l, shift left, "p"] \DRs(B) \arrow[loop right, "h"] 
	\end{tikzcd}\]
\end{Th}

\begin{proof}  
	It is straightfoward that $p \circ i = \id$. We need to show that $h$ is an homotopy between $\id$ and $i \circ p$ that is $\id - i\circ p = [d, h]$. On generator of $\DRs(B)$, denoted by $f \in B$, $\alpha_i \in \Ll_A[-1], i=1\cdots p$, $d\xi_i \in \_F^\vee [-1], i=1 \cdots m$ and $\eta_i \in \_F^\vee, i=1 \cdots m'$, we get: 
	\[( \id - i \circ p) (f \alpha_{1\cdots p}\otimes 1 \otimes 1 ) = 0 \]
	\[(\id - i\circ p) (d\xi_{1 \cdots m} \wedge \eta_{1 \cdots m'}) = d\xi_{1 \cdots m} \wedge \eta_{1 \cdots m'} \ \ \text{if } \ m+m' \neq 0 \]
	
	It is straightforward to check that $[d,h](f \alpha_{1 \cdots p}) = 0$ since $h$ sends $\cSym_B \Omega^1_B[-1]$ to zero. We only have to check this equation on terms of the form $d\xi_{1 \cdots m} \wedge \eta_{1 \cdots m'}$. 
	\[ \dr \circ h ( d\xi_{1 \cdots m} \wedge \eta_{1 \cdots m'}) = \dr \left( \frac{1}{m+m'} \sum_{i=1}^m \epsilon_h  \widecheck{d\xi_{1 \cdots m}^i} \wedge \xi^i \wedge \eta_{1\cdots m'}  \right)\]
	\[  = \frac{1}{m+m'} \left(\sum_{i=1}^m \epsilon_1 d\xi_{1 \cdots m} \eta_{1 \cdots m'} +  \sum_{k=1}^{m'} \sum_{i=1}^m  \epsilon_{dh} \widecheck{d\xi_{1 \cdots m}^i} d\eta_k \xi_i \widecheck{\eta_{1 \cdots m'}^k} \right)\]
	
	with: 
	\[ \epsilon_1 = (-1)^{n_1}, \qquad n_1 = n_h + \sum_{j\neq i} (\vert \xi_j \vert +1 ) + \sum_{j=i+1}^m (\vert \xi_j\vert +1) (\vert \xi_i \vert +1 )  \]
	\[ \begin{split}
		n_1 = &  \sum_{j=1}^{i-1} ( \vert \xi_j \vert +1 ) + \sum_{j= i+1}^m (\vert \xi_j \vert +1 ) \vert \xi_i \vert + \sum_{j\neq i} (\vert \xi_j \vert +1 ) + \sum_{j=i+1}^m (\vert \xi_j\vert +1) \vert \xi_i \vert + \sum_{j=i+1}^m (\vert \xi_j\vert +1)\\
		=& 2 \times \sum_{j\neq i} (\vert \xi_j\vert +1) + 2\times \sum_{j= i+1}^m (\vert \xi_j \vert +1 ) \vert \xi_i \vert \\
	\end{split} \]

	Since this number is even, $\epsilon_1 = 1$ and we have:
	\[ \sum_{i=1}^m \epsilon_1 d\xi_{1 \cdots m} \eta_{1 \cdots m'} = m \times d\xi_{1 \cdots m} \eta_{1 \cdots m'}\] 
	
	Now we consider $\epsilon_{dh}$: 
	\[ \epsilon_{dh} = (-1)^{n_{dh}}, \quad n_{dh} = n_h + \sum_{j\neq i}(\vert \xi_j \vert +1 ) + \vert \xi_i \vert + \sum_{j=1}^{k-1} \vert \eta_j \vert + \sum_{j=1}^{k-1} \vert \eta_j \vert ( \vert \eta_k \vert + 1) + \vert \xi_i \vert (\vert \eta_k \vert +1) \] 
	
	we get
	\[ n_{dh} = \sum_{j=1}^{i-1} ( \vert \xi_j \vert +1 ) + \sum_{j= i+1}^m (\vert \xi_j \vert +1 ) \vert \xi_i \vert + \sum_{j=1}^m (\vert \xi_j \vert +1 ) - 1 + \sum_{j=1}^{k-1} \vert \eta_j \vert (\vert \eta_k \vert +2) + \vert \xi_i \vert (\vert \eta_k \vert + 1) \]

	Now for $h \circ d$:
	\[ h \circ d ( d\xi_{1 \cdots m} \wedge \eta_{1 \cdots m'}) = h \left( \sum_{k=1}^{m'} \epsilon_d  d\xi_{1 \cdots m} \wedge d\eta_k \widecheck{\eta_{1\cdots m'}^k} \right)\]
	
	with: 
	\[ \epsilon_d = (-1)^{n_d}, \qquad n_d = \sum_{j=1}^m (\vert \xi_j \vert + 1) + \sum_{j= 1}^{k-1} \vert \eta_j \vert + \sum_{j=1}^{k-1} \vert \eta_j \vert ( \vert \eta_k \vert +1 ) \]
	
	Therefore we get: 
	\[ h \circ d ( d\xi_{1 \cdots m} \wedge \eta_{1 \cdots m'}) =\frac{1}{m+m'} \left( \sum_{k=1}^{m'} \epsilon_2 d\xi_{1 \cdots m} \wedge \eta_{1 \cdots m'} + \sum_{i=1}^m \sum_{k=1}^{m'} \epsilon_{hd} \widecheck{d\xi_{1 \cdots m}^i} d\eta_k \xi_i \widecheck{\eta_{1 \cdots m'}^k} \right)\]
	
	We have: 
	\[ \epsilon_2 = (-1)^{n_2}, \qquad n_2 = n_d + \sum_{j=1}^m (\vert \xi_j \vert +1) + \sum_{j=1}^{k-1} \vert \eta_j \vert \vert \eta_k \vert \]
	\[ \begin{split}
		n_2 =  &\sum_{j=1}^m (\vert \xi_j \vert + 1) + \sum_{j= 1}^{k-1} \vert \eta_j \vert + \sum_{j=1}^{k-1} \vert \eta_j \vert ( \vert \eta_k \vert +1 ) + \sum_{j=1}^m (\vert \xi_j \vert +1) + \sum_{j=1}^{k-1} \vert \eta_j \vert \vert \eta_k \vert \\
		=&2\times \sum_{j=1}^m ( \vert \xi_j\vert +1) + 2 \times \sum_{j=1}^{k-1} \vert \eta_j \vert ( \vert \eta_k \vert +1)
	\end{split} \]
	
	since $n_2$ is even, $\epsilon_2 = 1$ and:\[ \sum_{k=1}^{m'} \epsilon_2 d\xi_{1 \cdots m} \wedge \eta_{1 \cdots m'} = m' \times d\xi_{1 \cdots m} \wedge \eta_{1 \cdots m'}\]
	
	Finally we have: 
	\[ \epsilon_{hd} = (-1)^{n_{hd}}, \qquad n_{hd} = n_d + \sum_{j=1}^{i-1}(\vert \xi_j \vert +1) + \sum_{j=i+1}^m (\vert \xi_j \vert +1) \vert \xi_i \vert + \vert \xi_i \vert (\vert  \eta_k \vert +1) \]	
	\[ \begin{split}
		n_{hd} =& \sum_{j=1}^m (\vert \xi_j \vert + 1) + \sum_{j= 1}^{k-1} \vert \eta_j \vert + \sum_{j=1}^{k-1} \vert \eta_j \vert ( \vert \eta_k \vert +1 ) + \sum_{j=1}^{i-1}(\vert \xi_j \vert +1)\\ 
		& \qquad \qquad \qquad \qquad + \sum_{j=i+1}^m (\vert \xi_j \vert +1) \vert \xi_i \vert + \vert \xi_i \vert (\vert  \eta_k \vert +1)\\
		= &\sum_{j=1}^m (\vert \xi_j \vert + 1) + \sum_{j=1}^{k-1} \vert \eta_j \vert ( \vert \eta_k \vert +2 ) + \sum_{j=1}^{i-1}(\vert \xi_j \vert +1) \\
		&\qquad \qquad \qquad \qquad + \sum_{j=i+1}^m (\vert \xi_j \vert +1) \vert \xi_i \vert + \vert \xi_i \vert (\vert  \eta_k \vert +1)
	\end{split}\]
	
	Now to show the result, we need to show that $n_{hd} = n_{dh} +1 \mod 2$ but a straightfoward calculation show that:
	\[ n_{hd} - n_{dh} =  1 \]

	This shows that: 
	\[ [d,h](d\xi_{1 \cdots m} \wedge \eta_{1 \cdots m'}) = \frac{1}{m + m'} \left( (m+m')\times d\xi_{1 \cdots m} \wedge \eta_{1 \cdots m'} \right) = d\xi_{1 \cdots m} \wedge \eta_{1 \cdots m'}\]

\end{proof}

\begin{Th}\label{th:strong deformation retract linear stack}
	$(p,i,h)$ forms a strong deformation retract. In other words we have: \[ph = hi = h^2 = 0\]  
\end{Th}
\begin{proof}
	Using the same notations as before, we have:	
	\[ ph(d\xi_{1\cdots m} \eta_{1 \cdots m'}) = p \left( \frac{1}{m+m'} \sum_{i=1}^m \epsilon_h  \widecheck{d\xi_{1 \cdots m}^i} \wedge \xi^i \wedge \eta_{1\cdots m'} \right) = 0 \]
	
	For $\alpha$ a $p$-form of degree $n$ on $B$:
	\[ hi( \alpha) = h(\alpha) = 0\]
	
	\[h^2 (d\xi_{1\cdots m} \eta_{1 \cdots m'}) = h  \left( \frac{1}{m+m'} \sum_{i=1}^m \epsilon_h \widecheck{d\xi_{1 \cdots m}^i} \wedge \xi^i \wedge \eta_{1\cdots m'} \right) \]
	\[ =\frac{1}{(m+m')^2} \sum_{i=1}^m \epsilon_h \left(\sum_{k=1}^{i-1} \epsilon_1 \widecheck{d\xi_{1\cdots m}^{i,k}} \wedge \xi_i \wedge \xi_j \wedge \eta_{1 \cdots m'} +  \sum_{k=i+1}^{m'} \epsilon_2 \widecheck{d\xi_{1\cdots m}^{i,k}} \wedge \xi_i \wedge \xi_j \wedge \eta_{1 \cdots m'}\right)  \] 
	
	with \[\epsilon_1 = (-1)^{n_1}, \qquad n_1 = \sum_{j = 1}^{k-1} (\vert \xi_j \vert +1) + \sum_{j=k+1}^{m} (\vert \xi_j \vert +1)\vert \xi_k \vert - \vert \xi_k \vert  \]
	
	\[\epsilon_2 = (-1)^{n_2}, \qquad n_2 = \sum_{j = 1}^{k-1} (\vert \xi_j \vert +1) - \vert \xi_i \vert -1 + \sum_{j=k+1}^{m} (\vert \xi_j \vert +1)\vert \xi_k \vert + \vert \xi_k \vert \vert \xi_i \vert  \]
	
	Note that here $n_1 = n_1(i,k)$ and $n_2 = n_2 (i,k)$
	
	\[ \epsilon_h = (-1)^{n_h}, \qquad n_h = \sum_{j=1}^{i-1} ( \vert \xi_j \vert +1 ) + \sum_{j= i+1}^m (\vert \xi_j \vert +1 ) \vert \xi_i \vert  \]
	
	\[\begin{split}
		n_1 + n_h \mod 2=&  \sum_{j=k}^{i-1} (\vert \xi_j \vert +1) +  \sum_{j= i+1}^m (\vert \xi_j \vert +1 ) \vert \xi_i \vert + \sum_{j=k+1}^{m} (\vert \xi_j \vert +1)\vert \xi_k \vert - \vert \xi_k \vert \\
		=	&\sum_{j=k+1}^{i-1} (\vert \xi_j \vert +1) +  \sum_{j= i+1}^m (\vert \xi_j \vert +1 ) \vert \xi_i \vert + \sum_{j=k+1}^{m} (\vert \xi_j \vert +1)\vert \xi_k \vert+1
	\end{split} \]

	\[ \begin{split}
		n_2 + n_h \mod 2 =& \sum_{j=i}^{k-1} (\vert \xi_j \vert +1)- \vert \xi_i \vert -1 + \sum_{j=k+1}^{m} (\vert \xi_j \vert +1)\vert \xi_k \vert + \vert \xi_k \vert \vert \xi_i \vert \\ 
		& \qquad \qquad \qquad \qquad   + \sum_{j= i+1}^m (\vert \xi_j \vert +1 ) \vert \xi_i \vert\\
		=&\sum_{j=i+1}^{k-1} (\vert \xi_j \vert +1)+ \sum_{j=k+1}^{m} (\vert \xi_j \vert +1)\vert \xi_k \vert + \vert \xi_k \vert \vert \xi_i \vert  + \sum_{j= i+1}^m (\vert \xi_j \vert +1 ) \vert \xi_i \vert 
	\end{split}  \]
	
	The term in $(i,k)$ for $k<i$ is then: 
	
	\[ (-1)^{(n_1 + n_h)(i,k)} \widecheck{d\xi_{1\cdots m}^{i,k}} \wedge \xi_i \wedge \xi_j \wedge \eta_{1 \cdots m'} = (-1)^{\vert \xi_i \vert\vert \xi_k \vert }\times (-1)^{n_1 + n_h} \widecheck{d\xi_{1\cdots m}^{i,k}} \wedge \xi_k \wedge \xi_i \wedge \eta_{1 \cdots m'}\]
	
	The term for $(k,i)$ is then given by: 
	\[ (-1)^{(n_2 + n_h)(k,i)} \widecheck{d\xi_{1\cdots m}^{k,i}} \wedge \xi_k \wedge \xi_i \wedge \eta_{1 \cdots m'} \]
	
	But it turns out that:  \[(n_1 + n_h)(i,k) + \vert \xi_i \vert \vert \xi_k \vert = (n_2 + n_h)(k,i) +1\] 
	
	This proves that  $(-1)^{(n_1 + n_h)(i,k)} \widecheck{d\xi_{1\cdots m}^{i,k}} \wedge \xi_i \wedge \xi_k \wedge \eta_{1 \cdots m'} =  -(-1)^{(n_2 + n_h)(k,i)} \widecheck{d\xi_{1\cdots m}^{k,i}} \wedge \xi_k \wedge \xi_i \wedge \eta_{1 \cdots m'} $ for all $k < i$ and therefore the sum over all $i,k$ must be $0$. 
\end{proof}

\medskip

\ttl{For \CE algebras}

\medskip

Take this time $\_L$ a perfect $\_L_\infty$-algebroid concentrated in non-positive degree over $X = \bf{Spec}(A)$ with $A$ cofibrant of almost finite presentation. We denote: 
\[B := \ceu(\_L) \simeq \cSym_A \_L^\vee [-1]\]
with differential $\delta_B:= \delta_{\ce} = \delta_0 + \delta_1 + \cdots$ the weight decomposition of the \CE differential for the standard filtration on $\ceu(\_L)$. Under the conditions described at the beginning of Section \ref{sec:bv-charge}, we have: 
\[ \DRs(B) \simeq \cSym_A \left( \_L^\vee[-1] \oplus \Ll_A[-1] \oplus \_L^\vee[-2] \right)\]
together with a differential described in Lemma \ref{lem:differential de rham complex of ce}. 

\begin{Def}\ \label{def:filtration vertical and de rham}
	
	\begin{itemize}
		\item We define the \CE filtration\footnote{This is the analogue of the $\_F$-filtration from before.} of $\DRs(B)$ the filtration induced by the \CE filtration on $\ceu(\_L)$. : 
		\[ F_\_L^p \DRs(B) := \prod_{n \geq 0 } \cSym_A \Ll_X[-1] \otimes_A \cSym_A^{\geq n} \_L^\vee [-2] \otimes_A \widehat{\Sym_A}^{\geq p-n} \_L^\vee[-1]\]
		
		In particular we have again that $F_\_L^i \DRs(B) \subset \ker(p)$ for all $i \geq 1$. 
		
		\item  $\rel{\DR} (B)$ is also naturally endowed with the usual de Rham filtration induced by the de Rham weight: 
		\[ F_{dR}^p \DRs^{\geq p}(B) := \prod_{i \geq p} \DRs(B)(i) \]
	\end{itemize}
\end{Def}

\begin{Lem}\label{lem:differential de rham complex of ce}
	The differential is given by:
	\[D_B = \dr + (\delta_{\ceu} - \delta_A)+ \delta_A^v + \rho^* + \delta_{\_L^\vee[-2]}^v + \delta_+ + \delta_\nabla\]
	
	with: 
	
	\begin{itemize}
		\item $\dr$ the de Rham differential. It has \CE weight 0 and de Rham weight $1$.
		\item $\delta_A^v$ the vertical differential on $\cSym_A \Ll_A[-1] = \DRs(A)$ extended by zero on $\_L^\vee[-2]$ and $\_L^\vee[-1]$. It has de Rham and \CE weight 0.
		\item $\delta_{\ceu}- \delta_A = \delta_{\_L^\vee[-1]} + \delta_1 + \delta_2 + \cdots$ is the differential on $B$ (minus $\delta_A$) extended  by $0$ on $\Ll_B[-1]$ and $\_L^\vee[-2]$. $\delta_{\_L^\vee[-1]}$ has de Rham weight 0 and $\delta_i$ has \CE weight $i$.   
		\item $\delta_{\_L^\vee[-2]}^v $ is the differential on $\_L^\vee[-2]$ induced by $\delta_{\_L^\vee[-1]} = \delta_0 - \delta_A$. It has de Rham and \CE weight $0$.
		\item $\rho^*: \Ll_A[-1] \rightarrow \_L^\vee[-2]$ is $A$-linear and extended by $0$ on $\_L^\vee[-2]$ and $\_L^\vee[-1]$. It has de Rham weight $0$ and \CE weight $1$.    
		\item $\delta_+$ is the part of the differential induced by $(\delta_1 - \rho^* \circ \dr) + \delta_2 + \cdots$ on $\_L^\vee[-2]$. In particular it is valued in $F_\_L^2 B$ and increases the \CE weight by $1$ or more. 
		\item $\delta_\nabla$ the part of the differential induced by the connection, valued in: \[\cSym_A \Ll_A[-1] \otimes_A \cSym_A^{\geq 1}\_L^\vee[-1]\]
		and having de Rham and \CE weight $0$.  
	\end{itemize}
	We recapitulate these properties by:
	\begin{center}
		\begin{tabular}{|c|c|c|c|c|c|c|c|}
			\hline
			&  $\dr$& $(\delta_{\ceu} - \delta_A)$  & $\delta_A^v$  & $\rho^*$   & $\delta_{\_L^\vee[-2]}^v$  &   $\delta_+ $ & $\delta_\nabla$\\
			\hline
			CE weight& $0$ & $\geq 1$ & $0$ & $1$ & $0$ &  $\geq 1$& $0$\\
			\hline
			DR weight& $1$ & $0$ & $0$ & $0$ & $0$ & $0$ & $0$\\
			\hline
			Valued in & $F^1 $   &  $F_\_L^1 $   & $F^1 $    &  $F^1 \cap F_\_L^1$   &  $F^1  \cap F_\_L^1 $   &  $F^1  \cap F_\_L^2 $ & $F^1  \cap F_\_L^1 $   \\
			\hline
		\end{tabular}
	\end{center}
	
\end{Lem}

\begin{proof}
	By definition we have $D_B = \dr + \delta_B^v$ the decomposition in de Rham and vertical differentials. Moreover $\delta_B^v$ restricted to $B$ is $\delta_B$ itself and $\delta_B^v$ on $\Ll_B[-1]$ is the natural differential on $\Ll_B[-1]$. Using the description of $\Tt_{\ceu(\_L)}$ given at the beginning of Section \ref{sec:bv-charge}, we have: 
	\[\Ll_B[-1] \simeq B \otimes_A \left( \Ll_A[-1] \oplus\_L^\vee[-2] \right)\]
	where the right hand side is the $B$-module with differential given by \[\delta_A + \delta_{\_L^\vee[-1]}^v + \delta_{\Ll_B} + \delta_\nabla + \tilde{\delta}_1 + \tilde{\delta}_2 + \cdots\] where $\tilde{\delta}_i$ is the part of the differential on $\Ll_B[-1]$ induced by $\delta_i$ (it increases the \CE weight by $i$). \\
	
	Note that $\widetilde{\delta_1} = \widetilde{\rho^* \circ d} + \hat{\delta}_1$ and $\tilde{\rho^* \circ d}$ is exactly $\rho^* : \Ll_A[-1] \rightarrow \_L^\vee[-1]$. 
	
	We get $\tilde{\delta}_1 = \rho^* + \hat{\delta}_1$ where $\hat{\delta}_1$ is valued in $F^2\DRs(B)$. 
	
	Then we can define: \[\delta_+ = \hat{\delta}_1 + \tilde{\delta}_2 + \cdots\]
	It increases the \CE weight by at least $1$ and is valued in $F^2\DRs(B)$.
	
	Moreover we have:
	\[\delta_B + \delta_{\Ll_A} = \delta_{\_L^\vee[-1]} + \delta_A^v + \delta_1 + \delta_2 + \cdots\]
	with $\delta_1 + \delta_2 + \cdots = \delta_B - \delta_A = \delta_\ce - \delta_A$ and $\delta_\nabla$ is the part of the differential sending $\_F^\vee[-1]$ to an element in $\Ll_A[-1] \otimes_A  \_F^\vee$ via the covariant derivative. 
\end{proof}

\begin{Lem}\label{lem:strong deformation retract associated graded}
	
	In this situation the augmentation $\ceu(\_L)\to A$ only have a section when ``forgetting'' the terms \CE differential that increases the \CE weight. Doing this gives $\gr{\ceu}(\_L)$  and it gives the completion of the algebra of function of a semi-linear stack.  With the same formulas as in Definition \ref{def:deformation retract linear stack}, we have the maps: 
	\[ \begin{tikzcd}
		\DRs(A)  \arrow[r, shift left, "i"] & \arrow[l, shift left, "p"]  (\DRs(B)^\sharp, \delta')\arrow[loop right, "h"] 
	\end{tikzcd}\]
	
	where the underlying algebra $\DRs(B)^\sharp$ is endowed  with a differential given by the parts of the differential described in Lemma \ref{lem:differential de rham complex of ce} that preserve the \CE weight: 
	\[ \delta' :=  \dr + \delta_{\_F^\vee[-1]} + \delta_A^v + \delta_{\_L^\vee[-2]}^v + \delta_\nabla \]
	
	In particular, it forms a strong deformation retract thanks to Theorems \ref{th:deformation retract linear stacks} and \ref{th:strong deformation retract linear stack}.
	
\end{Lem}

\begin{Lem}\ \label{lem:filtration properties}
	$h$ respects the \CE filtration.
\end{Lem}
\begin{proof}
	We can see from the formula defining $h$ that it sends an element of \CE weight $m+m'$ given by $f\alpha_{1 \cdots n } d\xi_{1 \cdots m} \eta_{1 \cdots m'}$ to a sum of elements of \CE weight also $m+m'$ (or $0$ if $m = 0$ which is also of \CE weight $m+m'$).
\end{proof}

\begin{Th}\label{th:deformation retract ce strong}
	We obtain a new strong deformation retract: 
	\[ \begin{tikzcd}
		\DRs(A)\arrow[r, shift left, "i_\infty"] & \arrow[l, shift left, "p"] \DRs(B) \arrow[loop right, "h_\infty"] 
	\end{tikzcd}\]
	
	with: 
	\[ i_\infty = i + h \circ \sum (\Delta h)^k \circ  \Delta \circ i\]
	\[ h_\infty = h + h \circ \sum (\Delta h)^k \circ  \Delta \circ h\]
\end{Th}
\begin{RQ}
	Since $p$ is induced by the projection $B \to A$ (which respects the differentials), $p$ respects the total differentials on the nose and therefore we will get $p_\infty = p$. However, as there is no natural map $B \to A$ respecting the differentials, $i$ needs to be deformed in order to respect the differentials. 
\end{RQ}
\begin{proof}
	This new strong deformation retract is obtained by the homological perturbation Lemma for the strong deformation retract of Lemma \ref{lem:strong deformation retract associated graded}. We obtain: 
	\[ \begin{tikzcd}
		\DRs(A)\arrow[r, shift left, "i_\infty"] & \arrow[l, shift left, "p_\infty"] \DRs(B)  \arrow[loop right, "h_\infty"] 
	\end{tikzcd}\]
	
	With $\DRs(A)$ endowed with the differential $D_{B, \infty}$. 
	
	We have that $D_B =\delta' + \Delta$ with $\Delta$ that can be described using Lemma \ref{lem:differential de rham complex of ce}. 
	Then the formulas for the homological perturbation Lemma give us (see \cite{Cr04}):
	\[p_\infty = p + p \circ \sum (\Delta h)^k \circ \Delta\circ h \]
	\[ i_\infty = i + h \circ \sum (\Delta h)^k \circ  \Delta \circ i\]
	\[ h_\infty = h + h \circ \sum (\Delta h)^k \circ  \Delta \circ h\]
	\[ D_{B,\infty} = D_B + p \circ \sum (\Delta h)^k \circ  \Delta \circ i\]
	
	We just need to show that $D_{A,\infty} = D_A$ and $p_\infty = p$. It is enough to prove that $p \circ \sum (\Delta h)^k \circ  \Delta = 0$, which follows from the fact that $\Delta$ increases the \CE weight, and therefore is valued in $F_\_L^1 B \subset \ker(p)$  and the fact that $h$ preserves the \CE weight (Lemma \ref{lem:filtration properties}). 
\end{proof}

\begin{Cor}
	The de Rham cohomology of $\ceu(\_L)$ is isomorphic to the de Rham cohomology of the base $A$. 
\end{Cor}

\begin{Cor}\label{Cor_ExactnessFromStrongDeformationRetract}
	A closed element $a \in \DRs(\ceu(\_L))$ is exact if and only if $p(f)$ is exact in $\DRs(A)$. Moreover, if $p(f)= D_B b$, we have: 
	\[  a =  D_A (i_\infty(b) + h_\infty(a))  \] 
\end{Cor}
\begin{proof}
	If $a$ is exact then so is $p(a)$ since $p$ respects the differentials. Moreover if $p(a) = D_B b$ then we have: 
	\[ \begin{split}
		a & = i_\infty \circ p_\infty (a) + D_B \circ h_\infty (a) - h_\infty \circ D_A (a) \\
		&  = i_\infty \circ D_B b + D_A\circ h_\infty (a) - h_\infty \circ D_A (a) \\
		& =  D_A (i_\infty(b) + h_\infty(a)) 
	\end{split} \]
\end{proof}

Finally we need to understand the behavior of all the maps with respect to the de Rham weight. 

\begin{Lem}\label{lem:properties infty morphism with de rham weight}
	We have the following: 
	\begin{itemize}
		\item $i$ preserves the de Rham weight. 
		\item $p$ preserves the de Rham weight. 
		\item $h$ reduces the de Rham weight by $1$. 
		\item $\Delta$ increases the de Rham weight by at most $1$. 
		\item $i_\infty$ preserves the de Rham weight. 
		\item $h_\infty$ decreases the de Rham weight by $1$.
	\end{itemize}
\end{Lem}

\subsection{Examples of BV Constructions}\ \label{sec:examples-of-bv-constructions}

\medskip

In this section we will show that our framework encompasses and generalizes many of the already existing constructions (within the restriction of our framework). In particular we recover the two classes of constructions given by: 
\begin{itemize}
	\item The BV construction on the Koszul--Tate resolution of $f$. 
	\item The equivariant derived critical locus for a functional with off-shell symmetries (infinitesimal or not). 
\end{itemize}

For the second type of construction, we will see that this recovers the BV formalism for group actions as discussed in \cite{BSS21}, the infinitesimal BV formalism for off-shell actions of Lie algebroids and in the most general case, the BV formalism for off-shell actions of groupoids.  \\

We will also see that for off-shell actions, the heuristics described in Constructions \ref{cons:BV CG heurestic} and \ref{cons:BV FK heuristic} coincide.

\subsubsection{BV construction for the Koszul--Tate resolution}\ \label{sec:for-the-koszul--tate-resolution}

\medskip

We are going to explain how the construction in \cite{FK14} fits in our framework (see Construction \ref{cons:BV FK}). More precisely, their construction is a BV construction for the almost derived critical locus given by a Koszul--Tate resolution of the \emph{strict} critical locus.

In this section we will take $S := \iKT(f)$ the affine stack given by a Koszul--Tate resolution of $f$ and assume that the module $\_L_{\KT}$ is perfect. Moreover, for this example, we need to put ourselves in the setting described at the beginning of  Section \ref{sec:bv-charge} and consider underived tangent and cotangent complexes.   \\

The second step of the construction ensures that the $(-1)$-shifted symplectic structure we consider is in fact the \emph{strict} symplectic structure, $\omega^{\tx{st}}$ as described in Definition \ref{def:symplectic structure strict} on the graded algebra (without differential):
\[\Sym_A \left( \Tt_A[1] \oplus \_L_{\KT}[2] \oplus \widehat{\_L_{\KT}^\vee[-1]}\right) \]

Then Step (4) amounts to finding a $\_L_\infty$-algebroid structure on $\pi_S^*\_L_{\KT}$ whose associated \CE algebra has a differential not only compatible with the symplectic structure but also Hamiltonian\footnote{Using the discussion in Section \ref{sec:bv-charge}, it turns out that the main input of the construction is not that there is a BV charge, but rather that there exists a Lie algebroid compatible with the \emph{strict} symplectic structure.}.\\

Now from the construction of the BV charge, we now that the cdga $\BV_{\FK}$ is the \CE algebra associated with a $\_L_\infty$-algebroid structure on $ \pi^*\_L_{\tx{KT}}$ over $\iKT(f)$. We will consider the infinitesimal quotient from Definition \ref{conv:notion of quotient on S} given by the formal completion of the spectrum of this \CE algebra. Then the associated infinitesimal BV construction is:
\[ \QS{\iKT(f)}{\pi_s^*\_L_{\KT}} := \bf{BV}_{\FK} := \widehat{\Spec(\ceu(\_L_{\KT}))_{\KT(f)}}  \]

We then want to prove that this is a generalized infinitesimal symplectic reduction, which is exactly the content of the following theorem. For this theorem, we need again to be able to compute $h^* \_T_{\bf{BV}_{\FK}}$ which is the pullback of the tangent complex of a \CE algebra of a Lie algebroid over an affine base. From the setting described in Section \ref{sec:bv-charge} we have:
\[ h^*\_T_{\QS{\iKT(f)}{\pi^*\_L_{\KT}}} \simeq \_T_{\KT(f)} \oplus^\rho \pi^*\_L_{\KT}[1] \]

Note that in this context, all the notions (except for (5) which is a weaker notion) of infinitesimal quotient of Definition \ref{conv:notion of quotient on S} are equivalent. Then we have:

\begin{Th}\label{th:BV FK is part a lagrangian correspondence}
	There is a Lagrangian correspondence\footnote{In this setting mixing derived and underived constructions a Lagrangian correspondence in defined using the symplectic geometry as defined in \cite{PS20}. In particular the non-degeneracy condition tells that the \emph{strict} tangent and cotangent complexes are part of a \emph{homotopy} Cartesian square.}: 
	\[ \begin{tikzcd}
		& \iKT(f) \arrow[dl] \arrow[dr] & \\
		\iBV_{\FK} & & \RCrit(f)  
	\end{tikzcd} \]
\end{Th}
\begin{proof}
	The pullback of the canonical symplectic structures are equal in $\iKT(f)$ and therefore we can chose the $0$ isotropic correspondence structure. To conclude that this is a Lagrangian correspondence, we observe that the following diagram is homotopy Cartesian:
	\[ \begin{tikzcd}
		\_T_{\KT(f)} \arrow[r] \arrow[d] & p^*\_T_{\RCrit(f)} \simeq p^*\Omega^1_{\RCrit(f)} [-1] \arrow[d] \\
		h^*\_T_{\BV_{\FK}} \simeq h^*\Omega^1_{\BV_{\FK}} [-1] \arrow[r] & \Omega^1_{\KT(f)} [-1]
	\end{tikzcd} \] 
	
	From our assumptions we have that:
	\[p^* \_T_{\BV_{\FK}} \simeq \_T_{\KT(f)} \oplus^\rho \pi^*\_L_{\KT}[1]\] 
	And a connection on the semi-linear presentation of the stack ${\KT(f)} \to \RCrit(f)$ gives an equivalence:
	\[ \_T_{\KT(f)} \cong  i^*\_T_{\RCrit(f)} \oplus^\nabla\pi^*\_L_{\KT}^\vee[-2] \]
	
	Plugging these equivalences in the diagram proves that it is Cartesian.
\end{proof}

This implies that $\iBV_{\FK}$ is an infinitesimal $BV$ construction in the sens of Definition \ref{def:BV complex infinitesimal} for the inclusion $\iKT(f) \to \RCrit(f)$ of a  Koszul--Tate resolution of the strict critical locus. In particular, it is a BV construction where \emph{all} the symmetries of $f$ are taken into account.  

\subsubsection{Group actions and moment maps} \label{sec:with-moment-maps}\ 

\medskip

This section motivates a notion of BV construction for group actions instead of the infinitesimal actions of their Lie algebras. The idea is that such a construction would remember the global action of the group and therefore global phenomena induced by the group action. This idea is discussed in more details in \cite{BSS21} where they study the $G$-equivariant derived critical locus as a model for such a refined notion of BV construction. We argue that this is an instance of BV construction in our sens, and it is a particular example of the more general construction in Section \ref{sec:bv-construction-for-groupoid-action}.\\

We are interested in the BV construction where the action on $S$ is given by an action of a \emph{group} $G$. Again, given an arbitrary almost derived critical locus, nothing ensures the existence of such a BV construction. However, we will see that given a group action of off-shell symmetries (Definition \ref{def:off-shell symmetries}), we can always construct an almost derived critical locus $S$ whose generalized symplectic reduction is given by a quotient of $S$ by an action of $G$. \\

Let $G$ be an affine algebraic group acting  on $X$ such that $f: X \to \Aa_k^1$ is $G$-invariant. Then there is moment map $\mu: T^*X \to \G_g^*$ (Definition \ref{def:moment map structure}) whose symplectic reduction is $T^*\QS{X}{G}$ (Example \ref{ex:symplectic reduction cotangent moment map}). \\

From Example \ref{ex:moment map derived critical locus}, we know that by taking the Lagrangian intersection of $\mu$, we have a $(-1)$-shifted moment map $\mu_{-1} : \RCrit(f) \to \G_g^*[-1]$ whose symplectic reduction is $\RCrit(\eq{f})$. We will show that $\RCrit(\eq{f})$ is a BV construction for the inclusion $Z(\mu) \to \RCrit(f)$. All we are lacking is to show that $Z(\mu_{-1})$ is an almost derived critical locus: 

\begin{Lem} \label{lem:almost derived critical locus fiber moment map group}
	$Z(\mu_{-1})$ is an almost derived critical locus. 
\end{Lem}
\begin{proof}
	Consider the following diagram where all squares are pullbacks: 
	\[\begin{tikzcd}
		Z(\mu) \arrow[r] \arrow[d] & X \arrow[r] \arrow[d] & \star \arrow[d] \\
		T^*X \arrow[r, "\mu \times \pi"] & \G_g^* \times  X \arrow[r] & \G_g^* 
	\end{tikzcd}\]
	
	The left most square is a pullback of linear stacks over $X$ therefore we have: 
	\[ Z(\mu) \simeq \Aa_X(\Ll_X \times_{\G_g^* \otimes \_O_X} \_O_X) \simeq \Aa_X (\Ll_X \oplus^\mu \G_g^*[-1] \otimes \_O_X)\] 
	
	In particular, since $X = \Spec(A)$, it is affine and its algebra of functions is $\Sym_A \left(A \otimes \G_g[1] \oplus \Tt_X\right)$.\\
	
	Now consider the commutative diagram: 
	\[  \begin{tikzcd}
		X \arrow[d, equals] \arrow[r, "df"] & T^*X \arrow[d] & \arrow[l, "s_0"'] X \arrow[d, equals] \\
		X \arrow[r, "s_0"] & \G_g^* \times X & \arrow[l, "s_0"'] X \\
		X \arrow[u,equals] \arrow[r, equals] & X \arrow[u] & \arrow[l, equals] \arrow[u, equals] X
	\end{tikzcd} \] 
	
	This diagram is commutative because $f$ is $G$-equivariant and taking the limit of this diagram either first horizontally and then vertically or vice versa, we get an equivalence: 
	\[Z(\mu_{-1}) := \RCrit(f)\times_{\G_g^*[-1] \times X} X \simeq X \times_{Z(\mu)} X\]
	
	The right hand side is clearly affine with algebra of functions: 
	\[ A \otimes_{\Sym_A \Tt_A \oplus^\mu A \otimes\G_g[1] } A \simeq \Sym_A \left( \Tt_A[1] \oplus^\mu A \otimes \G_g[2]\right)\]
	and the differential restricted to $\Tt_A[1]$ is the contraction $\iota_{df}$. This shows that $Z(\mu_{-1})$ is an almost derived critical locus. 
\end{proof}

This lemma together with Example \ref{ex:shifted moment map algebroid derived critical locus} shows the following: 

\begin{Th}\label{th:equivariant critical locus is group like bv construction} Let $X$ be an Artin stack and $G$ an affine group of off-shell symmetries acting on $X$.
	Then $\Crit(\eq{f})$ is a generalized symplectic reduction of $Z(\mu_{-1}) \to \RCrit(f)$ and is therefore an off-shell BV construction on $f$ (with $S = Z(\mu_{-1})$).
\end{Th} 

We have seen with Proposition \ref{prop:formal completion and infinitesimal quotient} and Example \ref{ex:group action} that $\QS{Z(\mu_{-1})}{X \times \G_g}$ is the formal completion of the map $Z(\mu) \to \QS{Z(\mu_{-1})}{G}$. 

\begin{Cor}\label{prop:bv construction are formal completion of group-like bv construction}
	
	If $Y := \QS{S}{G}$ is a BV construction obtained by an action of off-shell symmetries (that is a BV construction as in Theorem \ref{th:equivariant critical locus is group like bv construction}) such that $G$ acts on $X$. Then the formal completion \[\widehat{Y} \simeq \widehat{Y_S} \simeq \QS{S}{X \times \G_g}\]
	is an infinitesimal BV construction in the sens of Definition \ref{def:BV complex infinitesimal}. 
\end{Cor}
\begin{proof}
	This is a direct consequence of Theorem \ref{th:formal completion of BV constructions}. 
\end{proof}

\begin{RQ}
	One of the intersecting feature of the BV construction for a $G$-action on $X$ of off-shell symmetries is that it can be understood both from the point of view of Construction \ref{cons:BV FK heuristic} and of Construction \ref{cons:BV CG heurestic}. For the first point of view, it is given, by definition, by a quotient of an almost derived critical locus (fitting in a Lagrangian correspondence). For the second point of view, we have shown that this is also the derived critical locus of the quotient map $\eq{f} : \QS{X}{G} \to \Aa^1$ (i.e. the equivariant derived critical locus): \[ \RCrit(\eq{f}) \simeq \QS{X}{G} \times_{T^*\QS{X}{G}} \QS{X}{G} \simeq \QS{Z(\mu_{-1})}{G}\]
\end{RQ}

\subsubsection{Lie algebroid action and moment maps}\ 
\label{sec:lie-algebroid-moment-maps}

\medskip

In the previous section we had the BV construction from a group action of off-shell symmetries. In this section we will show that a similar construction can be made from an infinitesimal action of a Lie algebroid of off-shell symmetries. \\

We know from Example \ref{ex:shifted moment map algebroid derived critical locus}, that if $\_L$ is a Lie algebroid of off-shell infinitesimal symmetries on $X$ (Definition \ref{def:off-shell infinitesimal symmetries}), then $\RCrit(f) \to L^*[-1]$ is a $(-1)$-shifted moment map with infinitesimal symplectic reduction given by the $\_L$-equivariant derived critical locus $\RCrit(\eq{f})$. \\

Just like in Lemma \ref{lem:almost derived critical locus fiber moment map group}, we can show that $Z(\mu_{-1})$ is an almost derived critical locus if $\_L$ is of almost finite presentation  and concentrated in non-positive degrees (see Proposition \ref{prop:quotient coadjoint action} knowing that $X$ is a smooth algebraic variety). 

This proves the following: 
\begin{Th}\label{th:bv construction lie algebroid of off-shell symmetries}
	Given a Lie algebroid of off-shell symmetries $\_L$, then $\RCrit(\eq{f})$ is an infinitesimal BV construction (for $S = Z(\rho_{-1}^*)$).  
\end{Th}

Again, this infinitesimal BV construction from a Lie algebroid of off-shell symmetries on $X$ can be understood both from the point of view of Construction \ref{cons:BV FK heuristic} and of Construction \ref{cons:BV CG heurestic}. For the first point of view, it is by definition a weak infinitesimal quotient of an almost derived critical locus $Z(\rho_{-1}^*)$. For the second point of view, this is the derived critical locus of the quotient map  $\eq{f} : \QS{X}{\_L} \to \Aa^1$. 

\subsubsection{BV construction for groupoid action}\label{sec:bv-construction-for-groupoid-action}\

\medskip

This section provides the most general BV construction for off-shell symmetries. It is very similar to both Sections \ref{sec:with-moment-maps} and \ref{sec:lie-algebroid-moment-maps}.

\begin{Th}\label{th:bv construction moment map groupoid}
	Let $\_G^\bullet$ be a smooth\footnote{In fact this can be generalized to any $\_G^\bullet$ such that the canonical $2$-form $T^*\QS{X}{\_G}$ is symplectic.} Segal groupoid of off-shell symmetries on $X$ and take a moment map, $\mu: T^*X \to L^*$,  for this action according to Definition \ref{def:moment map lie groupoid}. Then there is a moment map obtained from Corollary \ref{cor:moment map derived critical locus groupoid}: \[\mu_{-1} : \RCrit(f) \to L^*[-1]\] 
	whose symplectic reduction is $\RCrit(\eq{f})$, making it an off-shell BV construction:
	
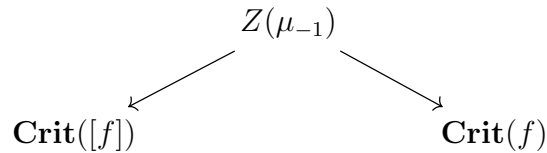
\[\begin{tikzcd}
		& Z(\mu_{-1}) \arrow[dl] \arrow[dr] & \\
		\RCrit(\eq{f})&& \RCrit(f)
	\end{tikzcd}
	\]
\end{Th}
\begin{proof}
	This is essentially a consequence of Example \ref{ex:shifted moment map algebroid derived critical locus} and the fact that $Z(\mu_{-1})$ is again an almost derived critical locus.
\end{proof}

Recall from Proposition \ref{prop:off-shell infinitesimal symmetries from off-shell symmetries} that any \emph{good} (in the sens of Definition \ref{def:good integration} Segal groupoid of off-shell symmetries over $X$ induces a Lie algebroid of off-shell infinitesimal symmetries with weak infinitesimal quotient the formal completion:
\[ \comp{\QS{Z(\mu_{-1})}{\_G}_{Z(\mu_{-1})}} \simeq \comp{\RCrit(\eq{f}_\_G)_{Z(\mu_{-1})}}\] 

Morover, from Proposition \ref{prop:off-shell infinitesimal symmetries from off-shell symmetries}, it is an off-shell infinitesimal BV construction: 
\[ \begin{tikzcd}
	& Z(\mu_{-1})  \arrow[dl] \arrow[dr] &\\
	\comp{\RCrit(\eq{f}_\_G)_{Z(\mu_{-1})}} && \RCrit(f)
\end{tikzcd}\]

\begin{Prop}
	Let $X$ be an affine stack satisfying Assumptions \ref{ass:very good stack}. Take $\_L$ a Lie algebroid over $X$ that integrates well to a smooth\footnote{In fact this can be generalized to any good integration $\_G^\bullet$ such that the canonical $2$-form $T^*\QS{X}{\_G}$ is symplectic.} groupoid $\_G$. Then we have an equivalence: \[\RCrit(\eq{f}_\_L) \simeq \comp{\RCrit(\eq{f}_\_G)_{Z(\mu_{-1})}}\] 
\end{Prop}
\begin{proof}
	Recall from Proposition \ref{prop:formal completion and cotangent} that we have an equivalence: 
	\[ T^*\QS{X}{\_L} \simeq \comp{\left(T^*\QS{X}{\_G}\right)_{Z(\mu)}}\] 
	
	Then we have the equivalences: 
	\[ \begin{split}
		\comp{\RCrit(\eq{f}_\_G)_{Z(\mu_{-1})}} \simeq&  \RCrit(\eq{f}_\_G) \times_{\RCrit(\eq{f}_\_G)_{\tx{DR}}} Z(\mu_{-1})_{\tx{DR}} \\
		\simeq & \left( \QS{X}{\_G} \times_{T^*\QS{X}{\_G}} \QS{X}{\_G} \right) \times_{\left( \QS{X}{\_G}_{\tx{DR}} \times_{\left(T^*\QS{X}{\_G}\right)_{\tx{DR}}} \QS{X}{\_G}_{\tx{DR}} \right)} \left(X_{\tx{DR}} \times_{Z(\mu)_{\tx{DR}}} X_{\tx{DR}}\right) \\
		\simeq & \left( \QS{X}{\_G} \times_{\QS{X}{\_G}_{\tx{DR}}} X_{\tx{DR}}  \right) \times_{ \left(T^*\QS{X}{\_G}\times_{\left(T^*\QS{X}{\_G}\right)_{\tx{DR}}} Z(\mu)_{\tx{DR}} \right)} \left( \QS{X}{\_G} \times_{\QS{X}{\_G}_{\tx{DR}}} X_{\tx{DR}}  \right) \\
		\simeq & \comp{\QS{X}{\_G}_X} \times_{\comp{\left(T^*\QS{X}{\_G}\right)_{Z(\mu)}}} \comp{\QS{X}{\_G}_X} \\
		\simeq & \QS{X}{\_G} \times_{ T^*\QS{X}{\_L}} \QS{X}{\_G} \\
		\simeq & \RCrit(\eq{f}_\_L)
	\end{split}\]
	
	where the fifth equivalence follows from Proposition \ref{prop:formal completion and cotangent} and Corollary \ref{cor:formal completion lie algebroid and relative tangent}. 
\end{proof}

This proposition implies that if $\_L$ integrates to a smooth Segal groupoid, then the infinitesimal BV construction obtained via Proposition \ref{prop:off-shell infinitesimal symmetries from off-shell symmetries} out of the BV construction given by $\RCrit(\eq{f}_\_G)$ is equivalent to the infinitesimal BV construction given by $\RCrit(\eq{f}_\_L)$ from Theorem \ref{th:bv construction lie algebroid of off-shell symmetries}.

\newpage
\appendix

\section{Model Categorical Setup}\label{sec:model-categorical-setup}\

\subsection{Good Model Structures}\label{sec:good-model-structures}\

\medskip

In this section, are are going to define the notion of ``good'' model structure. First, recall that $\Mod_k$ is equipped with its projective model structure where weak equivalences are quasi-isomorphims and fibrations are degree-wise surjective morphisms. Moreover, $\Mod_k$ satisfies the monoid axiom\footnote{The monoid axiom in a symmetric monoidal model category $\_M$ (\cite[Definition 2.1]{SS00}) is an axiom on the model structure given in \cite[Definition 2.2]{SS00} which gives some results (\cite[Theorem 3.1]{SS00}) on the model structure on modules and algebras over a commutative monoid in $\_M$.} as well as some other properties ensuring that we have a model structure on commutative monoids in $\Mod_k$ (namely commutative differential graded $k$-algebras) and on modules over those. We formalize these properties to define the notion of good model structure.   

\begin{Def}[{\cite[Section 1.1]{CPTVV}}] 
	\label{def:good model categories}
	
	Consider $\_M$ a closed symmetric monoidal $\Mod_k$-enriched (with tensor and cotensor) combinatorial model category. In other word, $\_M$ is a symmetric monoidal
	$\Mod_k$-model algebra as in \cite[Definition 4.2.20]{Ho99}. Moreover, \cite[Proposition A.1.1]{CPTVV} shows that $\_M$ is \emph{stable}. Then $\_M$ will be called a \defi{good model category} if it satisfies: 
	\begin{itemize}
		\item The unit $1$ is a cofibrant object in $\_M$. 
		\item For any cofibration $j : X \to Y$ in $\_M$, any object $A \in \_M$, and for any morphism $u : A \otimes X \to C$ in $\_M$ the strict push-out square in $\_M$: 
		\[ \begin{tikzcd}
			C \arrow[r] &  D \\
			A \otimes X \arrow[u] \arrow[r] & A \otimes Y \arrow[u]
		\end{tikzcd} \]
		is a homotopy pushout square. 
		\item For any cofibrant object $X \in \_M$, the functor $X \otimes - : \_M \to \_M$ preserves weak equivalences (i.e. $X$ is $\otimes$-flat). 
		\item $\_M$ is a tractable model category, that is, there are generating sets of cofibrations $I$, and trivial cofibrations $J$ in $\_M$ with cofibrant domains.
		\item Equivalences are stable under filtered colimits and finite products in $\_M$.
	\end{itemize} 
\end{Def}

\begin{notation}
	Given a symmetric monoidal category $\_M$, we denote by $\cdga_\_M$ the category of unital commutative monoids in $\_M$ and we will call them \defi{algebras} in $\_M$. 
	For the categories of graded algebras, filtered algebras, complete filtered algebras, graded mixed algebras and weak graded mixed algebras (see Appendix \ref{sec:filtered-and-complete-filtered-objects} and \ref{sec:graded-mixed-objects}), we will write $\gr{\cdga}$, $\filt{\cdga}$, $\cpl{\cdga}$, $\gmc{\cdga}$ and $\gmch{\cdga}$ respectively. 
\end{notation}

\begin{Def}
	\label{def:module over a commuative monoid}
	A \defi{(left)-module $N \in \_M$ over an algebra $A \in \cdga_\_M$} is a (left) action $\mu : A\otimes M \to M$ satisfying the usual axioms of a module. 
	
	The category of $A$-module is denoted $\Mod_A^\_M$.
	For the categories of graded modules, filtered modules, complete filtered modules, graded mixed modules, weak graded mixed modules, we will use the simpler notations $\gr{\Mod_A}$, $\filt{\Mod_A}$, $\cpl{\Mod_A}$, $\gmc{\Mod_A}$ and $\gmch{\Mod_A}$ respectively. 
\end{Def}

\begin{Prop}
	\label{prop:module over a commutative monoid good model category}
	Take $\_M$ a good model category then: 
	\begin{itemize}
		\item $\_M$ satisfies the monoid axiom and therefore the category $\cdga_\_M$ has a canonical model structure (defined via the free-forget adjunction).
		\item For any $A \in \cdga_\_M$, the category of $A$-modules is endowed with the structure of a symmetric monoidal combinatorial
		model category, for which the equivalences and fibrations are defined in $\_M$. In particular, fibrations and weak-equivalences are detected in $\_M$. Moreover, $\Mod_A^\_M$ carries a model structure making it a good model category (see the introduction of \cite[Section 1.3.1]{CPTVV}).
		\item For any equivalence $A \to A'$ in $\cdga_\_M$, the extension-restriction Quillen adjunction: 
		\[ \begin{tikzcd}
			\Mod_A^\_M \arrow[r, shift left] & \arrow[l, shift left] \Mod_{A'}^\_M
		\end{tikzcd} \]
		is a Quillen equivalence.
	\end{itemize}
\end{Prop}

\begin{RQ}\ \label{rq:model structure on A-modules}
	Since $\Mod_A^\_M$ is a good model category we can take again commutative monoids and module in $\Mod_A^\_M$ which essentially gives us $A$-algebras in $\_M$ and their modules.  
\end{RQ}

\begin{RQ}
	\label{rq:module stabilization and square zero extension}
	In \cite[Theorem 1.5.14]{Lu07}, it is shown that there is an equivalence: 
	\[ \Mod_A^\_M \simeq \Stab \left(\cdga_{/A}^\_M\right)\]
	Moreover, \cite[Remark 1.5.17]{Lu07} tells us that if the tensor product is exact in each variable, then taking $N \in \Mod_A^\_M$ then its image under the composition:
	\[ \begin{tikzcd}
		\Mod_A^\_M \arrow[r, "\simeq"]&  \Stab \left(\cdga_{/A}^\_M\right) \arrow[r, "\Omega^\infty"]& \cdga_{/A}^\_M  
	\end{tikzcd}\]
	is given by the square zero extension $A \boxplus N \to A$ defined in \ref{def:square zero extensions in M}. 
\end{RQ}

Given $F : \_M \to \_M'$ a symmetric monoidal functor between two good model categories. Then it naturally induces functors:
\[ F: \cdga_\_M \to \cdga_{\_M'} \qquad F: \Mod_A^\_M \to \Mod_{F(A)}^{\_M'}\]

\begin{Ex} \label{ex:good model structure}
	We have the following good model categories:
	\begin{itemize}
		\item $\Mod_k$ is a good model category and we will denote: \[\cdga := \cdga_{\Mod_k} \qquad \Mod_A := \Mod_A^{\Mod_k}\]
		for all $A\in \cdga$. Thanks to Proposition \ref{prop:module over a commutative monoid good model category}, for all $A \in \cdga$ we have that $ \Mod_A$ is a good model category. 
		\item Given $\_M$ a good model category, the category of graded mixed objects in $\_M$, denoted $\gmc{\_M}$ is again a good model category for its injective model structure (see Appendix \ref{sec:graded-mixed-objects}). 
		
		In particular for $\_M = \Mod_k$ we get graded mixed complexes, algebras and modules denoted by $\gmc{\Mod_k}$, $\gmc{\cdga}$ and $\gmc{\Mod_A}$ for $A \in \gmc{\cdga}$. 
	\end{itemize}
\end{Ex}

\begin{assumption}\label{ass:condition on good model category}
	In what follows, we will use two extra conditions on our good model structure: 
	\begin{enumerate}
		\item The tensor product is exact in each variables.
		\item The tensor product commutes with coproducts.
	\end{enumerate} 
	
	All the examples of good model categories we consider satisfy these two additional conditions.
\end{assumption}

\subsection{Cotangent Complex in $\_M$ }\
\label{sec:cotangent-complex-in-m}

\medskip

Let $\_M$  be a good model category satisfying Assumptions \ref{ass:condition on good model category}.  We are going to recall the construction of the tangent and cotangent complexes in such a good model category $\_M$. We will also discuss some functoriality properties with respect to the choice of $\_M$.  

\begin{Def}
	\label{def:square zero extensions in M}
	
	Given any $B$-module $N$, we can construct the \defi{trivial square zero extension} (see \cite[Section 1.2.1]{TV08}), ${B \boxplus N}$, given by the direct sum as a $B$-module, and with the product defined by the following map: 
	
	\[ \begin{tikzcd}
		(FA\oplus FM) \otimes (FA \oplus FM) \arrow[d, "\sim"] \\
		(FA \otimes FA) \oplus \left((FA \otimes FM) \oplus (FM \otimes FA)\right) \oplus FM \otimes FM \arrow[d, "\mu_A \oplus (\mu_M + \mu_M) \oplus 0"] \\
		FA \oplus FM
	\end{tikzcd} \]
	
	This defines a functor $ B \boxplus - $: 
	\[\begin{tikzcd}[row sep = tiny, column sep = small]
		\Mod_B \arrow[r] & \cdga_{A//B}^\_M \\
		M  \arrow[r, mapsto] & B \boxplus M
	\end{tikzcd} \]
\end{Def} 

\begin{Ex}
	If $\_M = \Mod_k$ then $A\boxplus M$ is the commutative algebra with product: 
	\[ (a,m)\dot (a',m') = (aa', am' + a'm)\]
\end{Ex}

The main point of this definition is to get the following description of derivations:

\begin{Lem}\label{lem:square zero extension and derivations in M}
	
	There is an isomorphism of $B$-modules\footnote{This can in fact be enriched either in $\Mod_B^\_M$ or in $\Mod_B$.}: 
	\[ \Hom_{\cdga_{A//B}^\_M} \left( B, B \boxplus M \right) \cong \Der_A^\_M (B,M) \]
\end{Lem}

\begin{Def} \label{def:cotagent complex in M}
	Given $f: A \to B$, the \defi{cotangent complex of $B$ relative to $A$} is defined as the $B$-module $\Llr{B}{A}^\_M$ representing derived $A$-linear derivations: 
	\[ \Hom_{B} \left(\Llr{B}{A}^\_M, M \right) \cong \Rr \Der_A (B, M) \]
\end{Def}

\begin{Cons}
	\label{cons:cotangent complex in M}
	
	Thanks to \cite[Lemma 1.3.4 and Remark 1.3.6]{CPTVV}, the functor: \[B \boxplus - : \Mod_B \to \cdga_{A//B}^\_M\] has a derived left adjoint $\Ll^\_M$ and by definition we have: 
	\[\Llr{B}{A}^\_M := \Ll^\_M(A \to B \overset{\id}{\to} B)\] 
\end{Cons}

\begin{Def}
	\label{def:tangent complex in M}
	
	In general, the relative tangent complex of $f: A \to B$ is the $B$-module of derived $A$-linear derivations of $B$: 
	\[ \Ttr{B}{A}^\_M := \Rr \Der_A^\_M(B,B)\]
	
	Using the defining property of the cotangent complex we get:
	\[ \Ttr{B}{A}^\_M \simeq \Hom_{B}^\_M \left( \Llr{B}{A}^\_M, B \right) \] 
\end{Def}

\begin{Prop}
	\label{prop:naturality in B cotangent complex in M}
	
	As explained in Proposition \ref{prop:module over a commutative monoid good model category}, given a morphism $f: B\to B'$ of $A$-algebras, we get an adjunction:
	\[ \begin{tikzcd}
		f^* : \Mod_B^\_M \arrow[r, shift left] & \arrow[l, shift left] \Mod_{B'}^\_M : f_*
	\end{tikzcd} \]
	
	Then we get a map of $B'$-modules: 
	\[ f^* \Llr{B}{A}^\_M \to \Llr{B'}{A}^\_M\]
	
	Moreover the cofiber of this map is $\Llr{B'}{B}^\_M$.
\end{Prop}
\begin{proof}

	Any derivation $\phi : B' \to M$ restricts to a derivation $\phi : B \to f_*M$ inducing a map\footnote{This is induced by the restriction map: \[\Homsub{\cdga_{/B'}^\_M}(B', B' \boxplus M) \to \Homsub{\cdga_{/B'}^\_M}(B, B' \boxplus M) \simeq \Homsub{\cdga_{/B}^\_M}(B, B \boxplus f_*M)\]} $\Der_{B'}(B', M) \to \Der_B(B, f_*M)$. This induces a morphism:
	\[ \Hom_{B'}(\Llr{B' }{A}, M) \to \Hom_B(\Llr{B}{A}, f_*M)  \simeq \Hom_{B'}(f^* \Llr{B}{A}, M) \]
	
	and therefore induces a morphism $f^* \Llr{B}{A} \to \Llr{B' }{A}$. 
	
	To compute the cofiber, we compute the fiber after applying $\Hom_{B'}(-, M)$:
	\[ \tx{fiber}\left( \Homsub{\cdga_{A//B'}^\_M} \left(B', B' \boxplus M \right)  \to  \Homsub{\cdga_{A//B}^\_M} \left(B, B \boxplus f_*M \right) \right)\]
	
	In other words, this is the fiber of the map that sends a derivation $\phi: B' \to M$ to its restriction $\phi : B \to f_* M$. But this is exactly the same as the fiber: 
	\[ \tx{fiber}\left( \Homsub{\cdga_{A//B'}^\_M} \left(B', B' \boxplus M \right)  \to  \Homsub{\cdga_{A//B'}^\_M} \left(B, B' \boxplus M \right) \right)\]
	
	which is given by derivations $\phi : B' \to M$ such that $\phi_{\vert B} \simeq 0$, and therefore $\phi$ is valued a  $B$-linear derivation. This proves that this fiber is equivalent to: \[\Homsub{\cdga_{B//B'}^\_M} \left( B', B' \boxplus M\right) \simeq \Hom_B'\left( \Llr{B'}{B}, M \right)\]  
\end{proof}

\begin{Cor} 
	\label{cor:cotangent fiber sequence in M}
	
	Given $f: A \rightarrow B$ and $g: B \rightarrow C$, we have a homotopy fiber sequence: 
	\[ \begin{tikzcd}
		\Ll^\_M_{B/A} \otimes_B C \arrow[r] & \Ll^\_M_{C/A} \arrow[r] &  \Ll^\_M_{C/B}
	\end{tikzcd}\]
	
	In particular for $A = k$ we write $ \Ll^\_M_{B/k} =  \Ll^\_M_{B} $ and we get: 
	\[ \begin{tikzcd}
		\Ll^\_M_{B} \otimes_B C \arrow[r] & \Ll^\_M_{C} \arrow[r] &  \Ll^\_M_{C/B}
	\end{tikzcd}\]			
\end{Cor}

\begin{Prop}
	\label{prop:pushout cotangent complexes in M}
	
	Consider the following pushout diagram: 
	\[ \begin{tikzcd}
		A \arrow[r,"f"] \arrow[d,"g"] & B \arrow[d, "i"] \\
		C \arrow[r, "j"] & D
	\end{tikzcd}\]
	then we have a pushout diagram: 
	\[ \begin{tikzcd}
		i^*f^*\Ll^\_M_A \simeq j^*g^* \Ll^\_M_A \arrow[r] \arrow[d] & i^*\Ll^\_M_B \arrow[d] \\
		j^*\Ll^\_M_C \arrow[r] &\Ll^\_M_{D}
	\end{tikzcd}\]
\end{Prop}

\newpage

\section{Filtrations and Completions} \label{sec:filtered-and-complete-filtered-objects}\

\subsection{Filtered and Complete Filtered objects} \label{sec:filtered-and-complete-filtered-complexes}\

\medskip

Essentially, a filtered object valued in $\_M$ a good model category is a sequence of maps: 
\[ \cdots \to F^{1}V \to F^0V \to F^{-1}V \to \cdots\]

which we view as a decomposition of $\colim F^pV$ in ``sub-objects\footnote{We say ``sub-objects'' because these maps are often considered as cofibrations (up to take a cofibrant replacement in the category of filtered objects).}''. 

\begin{Def}\ 
	\label{def:filtered objects}
	Let $(\Zz, \geq)$ denote the category associated to the ordered set $\Zz$ in which there is a unique morphism $i \to j$ if and only if $i \geq j$. The category of \defi{filtered objects} in a good model category $\_M$ (Definition \ref{def:good model categories}), is the functor category: 
	\[ \_M^{\tx{filt}} := \tx{Fun}\left( (\Zz, \geq), \_M \right)  \]  
	
	An object in $\_M^\tx{filt}$ will be denoted by $F^\bullet V$, or simply $FV$, with $F^pV \in \_M$ being the evaluation of the functor at $p \in \Zz$. 
	
	A filtered object is called \defi{non-negatively filtered} if $F^{-p}V = 0$ for all $p >0$.   
\end{Def}

\begin{Cons} \ 
	\label{cons:construction on filtered objects}
	\begin{itemize}
		\item Since $\_M$ is symmetric monoidal, then so is $\_M^{\tx{filt}}$ and the tensor product of $A$ and $B$ in $\_M^\tx{filt}$ is defined by: 
		\[ \begin{tikzcd}
			(\Zz,\geq) \arrow[r, "n \mapsto \coprod\limits_{i+j=n} i \times j"] & (\Zz, \geq) \times (\Zz, \geq) \arrow[r, "A \times B"] & \_M \times \_M\arrow[r, "-\otimes_\_M -"] & \_M
		\end{tikzcd}\]
		
		\item To each filtered complex $F^\bullet V$ , we can associated its colimit: 
		\[ V := \colimsub{p\in (\Zz, \geq)} F^p V \]
		This is left adjoint to the constant functor\footnote{Sending an objects $V$ to the filtered object with constant filtration given by $F^pV = V$ for all $p \in \Zz$.} : 
		\[ \begin{tikzcd}
			\colim : \_M^{\tx{filt}} \arrow[r, shift left] & \arrow[l, shift left]  \_M : 	\kappa
		\end{tikzcd} \]
		\item To each filtered objects $FV$ we can define it \defi{associated graded $\tx{Gr}(FV) \in \_M^{\tx{gr}}$}, defined as: 
		\[ \tx{Gr}^p (FV) := \faktor{F^pV}{F^{p+1}V}\]
		
		This defines a functor: 
		\[ \_M^{\tx{filt}} \to \_M^{\tx{gr}} \]
		called the associated graded to a filtered object.  
	\end{itemize}
\end{Cons}

\begin{Prop}\ \label{prop:model structure on filtered objects}
	$\filt{\_M}$ is a closed symmetric monoidal $\Mod_k$ enriched combinatorial model category. Moreover, from \cite[Remark 2.18]{CCN21}, $1$ is cofibrant.
\end{Prop}

\begin{RQ}\label{rq:enrichement of filtered objects}
	$\filt{\_M}$ is enriched over $\filt{\Mod_k}$ by setting: 
	\[ F^p\filt{\Homsub{\filt{\_M}}}(FV, FW) :=\iHomsub{\_M}(F^pV, F^pW)  \]
	
	The $\Mod_k$-enrichment is given by: 
	\[ \iHomsub{\filt{\_M}}(FV, FW) := \colim \filt{\Homsub{\filt{\_M}}}(FV, FW) :=\colim_p \iHomsub{\_M}(F^pV, F^pW) \]
\end{RQ}

\begin{RQ}\label{rq:filtered complex}
	The usual notion of \emph{filtered complex} is a sequence of monomorphisms of cochain complexes:
	
	\[ \cdots \to F^{1}V \to F^0V \to F^{-1}V \to \cdots\]
	It turns out that these are exactly the cofibrant objects\footnote{More generally, cofibrant objects in $\filt{\_M}$ are weight wise cofibrant objects with cofibrations between them: $F^{p+1}V \to F^pV$.} in $\filt{\Mod_k}$. Up to taking a fibrant replacement, we can always assume that our filtered complexes are of that form.	 
\end{RQ}

\begin{Lem}
	\label{lem:underlying complex commute with tensor product}
	
	The colimit functor is symmetric monoidal: 
	\[ \colim (FA \otimes FB) \simeq \colim (FA) \otimes \colim (FB) \]
\end{Lem}
\begin{proof}
	The goal is to compute: 
	\[ \colimsub{n} \bigoplus_{p+q=n} F^pA \otimes F^qB\]
	
	The direct sums are finite direct sums and therefore they are coproducts. We need to compute the following colimit over the diagram  $\Z \times \Z$ with a unique morphism $(i,j) \to (i', j')$ if and only if $\vert i' - i \vert + \vert j' - j \vert \leq 1$.  
	\[ \colimsub{p, q \in \Zz}F^pA \otimes F^qB \] 
	
	Since tensor product of modules commute with direct sums, we get: 
	\[ \colimsub{p, q \in \Zz}F^pA \otimes F^qB \simeq \colimsub{p\in \Zz}F^pA \otimes  \colimsub{ q \in \Zz} F^qB \simeq A \otimes B\] 
\end{proof}

A filtered object $FV$ can be though of as a collection of sub-object $F^pV$ of its colimit $V$. As $p$ goes to infinity, $F^pV$ becomes smaller and smaller. We want to consider those filtration where $F^pV$ converges to zero and is also ``complete'':   

\begin{Def} 
	\label{def:complete filtered obects}
	
	A filtered objects $FV \in \_M^{\tx{filt}}$ is called \defi{complete} if there is an equivalence (with $V := \colim FV$): 
	\[ V \overset{\sim}{\to} \lim\limits_{p \in(\Zz, \leq )} \faktor{V}{F^pV} \]
	
	The full sub-category of complete objects is denote by $\_M^{\tx{cpl}}$. The inclusion $\_M^{\tx{cpl}} \to \_M^{\tx{filt}}$ has a left adjoint called the \defi{completion functor}: 
	\[ \begin{tikzcd}
		\widehat{(-)} : \_M^{\tx{filt}} \arrow[r, shift left] & \arrow[l, shift left] \_M^{\tx{cpl}} : \iota
	\end{tikzcd}\]
	
	Given a filtered complex $FV$, its completion is given by the filtered object: 
	\[ \widehat{FV} = \lim\limits_{p \in \Zz} \faktor{V}{F^pV} \qquad F^p \widehat{V} := \lim\limits_{p \geq q} \faktor{F^pV}{F^qV}\] 
\end{Def}

\begin{Cons}\
	\label{cons:construction for complete filtered objects}
	\begin{itemize}
		\item We have a left adjoint functor (\cite[Definition 3.2]{Mou21}): 
		\[ \begin{tikzcd}
			\Mod^{\tx{gr}} \arrow[r, "\tot"] & \Mod^{\tx{cpl}} 
		\end{tikzcd}\] where $\tot$ sends a graded complex $(V(p))_{p \in \Zz}$ to the complete filtered complex given by: 
		\[ \tot(V) = \widehat{\bigoplus_{p \in \Zz}}V(p) \qquad F^p V := \prod_{q \leq 0} V(q-p)\]
		
		\item The functor $\tx{Gr}$ defined on filtered objects restricts to a functor on complete filtered objects. 
		\item We can define a functor: 
		\[ \colim : \cpl{\_M} \to \filt{\_M} \overset{\colim}{\to} \_M \]
		
		This is a composition of a left adjoint with a right adjoint. 
		
		\item Similarly to Remark \ref{rq:enrichement of filtered objects}, complete filtered objects in $\_M$ are enriched over filtered complexes, and even in complete filtered complexes thanks to the proof of \cite[Proposition 2.6]{CCN21}. Moreover it is also enriched over $\Mod_k$ using the $\colim$ functor. 
		
		\item The tensor product \emph{does not} restrict to the category of complete filtered objects. However, we can define a completed tensor product $\widehat{\otimes}$ on $\_M^{\tx{cpl}}$ given by: \[\widehat{FV} \widehat{\otimes} \widehat{FW} := \widehat{\iota(\widehat{FV}) \otimes \iota(\widehat{FW})}\] 
		
		This makes $\widehat{\otimes}$ a closed symetric monoidal structure on $\_M^{\tx{cpl}}$ and $\widehat{(-)}$ becomes symmetric monoidal functor (see \cite[Proposition 2.6]{CCN21}). In other words, the natural map: \[\widehat{FA \otimes FB} \to \widehat{FA} \widehat{\otimes} \widehat{FB}\] is an equivalence.
	\end{itemize}
\end{Cons}

\begin{Prop}
	\label{prop:model structure on complete filtered objects}
	For $\_M$ a good model category, there is a model structure on $\_M^{\tx{cpl}}$ given by:
	\begin{itemize}
		\item weak equivalences being morphisms $f: FA \to FB$ such that the map: \[\tx{Gr}(f) : \tx{Gr}(FA) \to \tx{Gr}(FB)\] is a graded weak equivalences\footnote{These are weak equivalences in $M^{\tx{gr}}$, that is, weak equivalences in each weight.}.
		\item The fibrations are morphisms inducing a fibration\footnote{A fibration $f : V \to W$ of graded objects in $M^{\tx{gr}}$ is a weight-wise fibration, that is each $f_p : V(p) \to W(p)$ is a fibration.} of graded objects: \[\tx{Gr}(f) : \tx{Gr}(FA) \to \tx{Gr}(FB)\] 
		\item A cofibration $f: FA \to FB$ is a an admissible cofibration. In other words, each map $F^p A \to F^pB$ is a cofibration, and the maps: 
		\[ F^pA \coprod_{F^{p+1}A} F^{p+1}B \to F^{p}B\]
		are also cofibrations (see \cite[Lemma 2.17]{CCN21}). 	
	\end{itemize}
	
	This is a closed symmetric monoidal $\Mod_k$-enriched, combinatorial model structure. 
\end{Prop}

\begin{Prop}[{\cite[Proposition 2.6]{CCN21}}]
	\label{prop:completion is quillen}
	
	The adjunction:
	\[ \begin{tikzcd}
		\widehat{(-)} : \_M^{\tx{filt}} \arrow[r, shift left] & \arrow[l, shift left] \_M^{\tx{cpl}} : \iota
	\end{tikzcd}\]
	is a Quillen adjunction between symmetric monoidal model categories.
\end{Prop}

\begin{Def}  
	\label{def:direct sum (complete) filtered complex}
	
	Given $FA$ and $FB$ two (complete) filtered complexes, then we can define their \defi{direct sum} to be the (complete\footnote{The direct sum of a finite number of complete filtered complexes can be shown to also be complete.}) filtered complex given by: 
	\[F^p (FA \oplus FB) = F^pA \oplus F^pB\] 
\end{Def}
We are giving a few consequences of this result:

\begin{Lem}
	\label{lem:underlying complex of a finite direct sum of filtered complexes}
	The underlying complex of a direct sum of filtered complexes is the finite sum of the filtered complexes: 
	\[ \colim (FV \oplus FW) \simeq \colim (FV) \oplus \colim (FW) \simeq V \oplus W\]
\end{Lem}
\begin{proof}
	Since the direct sum of cochain complexes is their coproduct, it commutes with colimits and we have: 
	\[ \colimsub{p \in \Zz}F^p(FV \oplus FW) = \colimsub{p \in \Zz}F^pV \oplus F^p W \simeq \colimsub{p \in \Zz}F^pV \oplus \colimsub{p \in \Zz}F^pW \simeq V \oplus W\]
\end{proof}
\begin{Lem}\ 
	\label{lem:completion and finite direct sum of filtered complexes}
	The completion of a finite direct sum of filtered complexes is the finite direct sum of their completions: 
	\[ \widehat{FV \oplus FW} \simeq \widehat{FV} \oplus \widehat{FW}\]
\end{Lem}
\begin{proof}
	Using Lemma \ref{lem:underlying complex of a finite direct sum of filtered complexes}, we must compute the following limit: 
	\[ \lim_p \faktor{V \oplus W}{F^pV \oplus F^p W} \simeq \lim_p \faktor{V}{F^pW} \oplus \faktor{V}{F^pW} \]
	Finite direct sums of complexes coincide with the finites products of the same complexes and therefore commute with limits: 
	
	\[ \widehat{FV \oplus FW} \simeq \lim_p \faktor{V}{F^pV} \oplus \faktor{W}{F^pW}  \simeq \widehat{FV} \oplus \widehat{FW} \]
	
	It is not hard to see using the same manipulations that the complete filtration on $\widehat{FV \oplus FW}$ coincide with the natural complete filtration on $\widehat{FV} \oplus \widehat{FW}$.
\end{proof}

\subsection{Filtered and Filtered Complete Algebras} \label{sec:filtered-and-filtered-complete-algebras}\

\medskip

This is the goal of the section to present the necessary elements on filtered algebras, that is commutative monoids in filtered objects. 

\begin{Prop}
	\label{prop:model structure of filtered (complete) algebras and module over them}
	The category of filtered (complete) algebras admits a model structure whose fibrations and weak-equivalences are detected by the forgetful functor $\cdga^{\tx{filt}} \to \Mod_k^{\tx{filt}}$. Moreover, there is a free-forget Quillen adjunction: 
	
	\[ \begin{tikzcd}
		\Mod_k^{\tx{filt}} \arrow[r, shift left] & \arrow[l, shift left] \cdga^{\tx{filt}}
	\end{tikzcd}\]	
\end{Prop}

\begin{Prop}
	\label{prop:completion and colimit for filtered algebras}
	The completion-inclusion adjunction for filtered complexes can be lifted to a Quillen adjunction of filtered algebras: 
	
	\[ \begin{tikzcd}
		\widehat{(-)} : \cdga^{\tx{filt}} \arrow[r, shift left] & \arrow[l, shift left] \cdga^{\tx{cpl}} 
	\end{tikzcd}\]
	
	Moreover, the following diagram of left adjoints commutes\footnote{Because the diagram of right adjoints easily commutes.}: 
	\[ \begin{tikzcd}
		\Mod_k^{\tx{filt}} \arrow[r, "\widehat{(-)}"] \arrow[d]&  \Mod_k^{\tx{cpl}} \arrow[d] \\
		\cdga^{\tx{filt}} \arrow[r, "\widehat{(-)}"]&  \cdga^{\tx{cpl}} 
	\end{tikzcd} \]  
\end{Prop}

\begin{Prop}
	\label{prop:colimit-contant adjunction for filtered algebra}
	The colimit-constant adjunction for filtered complexes can be lifted to an adjunction of filtered algebras: 
	
	\[ \begin{tikzcd}
		\colim : \cdga^{\tx{filt}} \arrow[r, shift left] & \arrow[l, shift left] \cdga : \kappa 
	\end{tikzcd}\]
	
	Moreover, the following diagram of left adjoints commutes\footnote{Because the diagram of right adjoints easily commutes.}: 
	\[ \begin{tikzcd}
		\Mod_k^{\tx{filt}} \arrow[r, "\colim"] \arrow[d]&  \Mod_k \arrow[d] \\
		\cdga^{\tx{filt}} \arrow[r, "\colim"]&  \cdga
	\end{tikzcd} \]  
\end{Prop}

In particular the last two propositions imply that the colimit functor from complete filtered algebras to algebras commutes with the forgetful functor to (filtered) modules.  

\begin{Ex}\ \label{ex:filtered complexes}
	
	\begin{itemize}
		\item Take $F$ an $A$-module. We can consider the non-negatively graded filtration on $\Sym_A F$ given by 
		\[ F^p \Sym_A F := \Sym_A^{\geq p } F := \bigoplus\limits_{n \geq p} \Sym_A^n F\] 
		These are the polynomials in $F$ with coefficients in $A$.
		This can be extended for $F$ a quasi-coherent sheaf on a derived stack $X$ by replacing $A$ by $\_O_X$. 
		\item  Given $A$ a commutative algebra and $I$ an ideal of $A$, we consider the filtration given by $F^pA = I^p$ and its completion is defined as: \[ \widehat{A_{I}}:= \lim_{n \in \Zz} \faktor{A}{I^n} \]
		This corresponds the the topological completion for the $I$-adic topology on $A$. Given a morphism $B \to A$, we can consider $\G_M$ the augmentation ideal defining a filtration on $B$. This gives the completion of $B \to A$ denoted $\comp{B_A}$.
		\item $k[t]$ can be viewed as the symmetric algebra on a single generator, $t$, with coefficients in $k$. From that point of view, its completions is the algebra of formal power series $k[\![t]\!]$. We obtain the same result by considering the $I$-adic completion for $I = (t)$. 
		
		More generally, the symmetric completion is equivalent to the completion of the symmetric algebra $\Sym_A F$ by the ideal generated by $F$, $I = (F)$. We denote that completion by: 
		\[ \cSym_A F\]
	\end{itemize}
\end{Ex}

\begin{notation}\label{not:partial completion} 
	
	Consider a $A$-module $F = F_1 \oplus F_2$. We want to consider the completion along $F_2$ only. In other words, we consider the completion along the ideal generated by $F_2$ in $\Sym_A F$. This completion will be denoted by: 
	\[\Sym_A(F_1 \oplus \widehat{F_2})\]
	
	Moreover we have that: 
	\[ \Sym_A(F_1 \oplus \widehat{F_2}) \simeq \Sym_A F_1 \otimes_A \cSym_A F_2\]    
\end{notation}

\begin{notation}\label{not:graded symmetric algebra}
	If we consider a $A$-module $F$. Then $\Sym_A F$ can be given the structure of a graded algebra. We denote by $\gr{\Sym_A}F$ the (non-negatively) graded algebra where the weight $p$ part is given by:
	\[ \Sym_A^p F\]
	
	This is the natural associated graded to the natural filtration on the symmetric algebra described in Example \ref{ex:filtered complexes}. 
\end{notation}

\newpage

\section{Graded Mixed Objects}\label{sec:graded-mixed-objects}\

\subsection{Graded Mixed Complexes}\label{sec:graded-mixed-complexes}\

\medskip

We are going to give the relevant results and constructions on graded mixed complexes. We follow the grading convention of \cite{CPTVV} and recall that our good model categories (Definition \ref{def:good model categories}) are exactly the categories introduce in \cite[Section 1.1]{CPTVV}. Essentially, given a good model category $\_M$, we look at graded objects on $\_M$, introducing a ``weight grading'' and then add a ``mixed differential'' that will increase that weight.

\begin{Def}\label{def:graded complex}
	We consider the category of \defi{graded object in $\_M$} denoted by $\gr{\_M}$. Its objects are families $(V(p))_{p \in \Zz}$ of elements in $\_M$ and its morphism are the weight preserving morphisms.
\end{Def}

\begin{RQ}
	Graded objects in $\_M$ can be viewed as the functor category $\tx{Fun}(\Zz, \_M)$ where the only morphisms in $\Zz$ are the identities. As such it comes with the projective model structure where fibration, cofibration and weak-equivalences are defined object-wise.
\end{RQ}

\begin{RQ}
	$\gr{\_M}$ is clearly enriched in $\gr{\Mod_k}$ (since $\_M$ is enriched in $\Mod_k$) by taking the weight-wise enrichment  $\_M$. 
\end{RQ}

\begin{notation}\label{not:graded objects}
	Take $E \in \gr{\_M}$.
	\begin{itemize}
		\item The \defi{weight $p$ part of $E$} is denoted $E(p)$
		\item We denote by $E((q))$ the graded object whose weight grading is shifted by $q$ such that: 
		\[ E((q))(p) := E(p-q)\]
		
		In particular, if $V$ is concentrated in weight $0$, then $V((-p))$ is concentrated in weight $p$.
		\item A morphism $f: E \to E'$ is the data of a collection of morphisms in each weight denoted:
		\[ f_p : E(p) \to E'(p)\]
	\end{itemize}
\end{notation}

\begin{Def} 
	\label{def:graded mixed complex}
	A \defi{graded mixed complex} $E$ in $\_M$ is given by: 
	\begin{itemize}
		\item A graded object $E \in \gr{\_M}$. 
		\item A morphism called the \defi{mixed differential}: 
		\[  \epsilon  : E \to E((-1))[-1]  \]
		such that $\epsilon^2 = 0$. This decomposes into maps: 
		\[\epsilon_p : E(p) \to E(p+1)[-1]\] increasing the weight by $1$. 
	\end{itemize}
	We denote by $\gmc{\_M}$ the category of graded mixed complexes whose morphisms are weight-preserving morphisms that respect the mixed differentials.  
\end{Def}

From the assumptions on $\_M$, $\gmc{\_M}$ is enriched in $\gmc{\Mod_k}$ (denoted by $\gmc{\iHomsub{\gmc{\_M}}}(-,-)$) but also in $\Mod_k$ by defining: 
\[ \iHomsub{\gmc{\_M}} (E,F)  := \ker \left(\epsilon :\gmc{\iHomsub{\gmc{\_M}}}(E,F)(0) \to \gmc{\iHomsub{\gmc{\_M}}}(E,F)(1)[1] \right) \] 

\begin{Lem}
	\label{lem:graded mixed complex as graded modules}
	
	Denote $\one[\epsilon] := \one \boxplus \one((1))[1]$ the square zero extension of $\one$ by $\one((1))[1]$ viewed as a module over $\one$ of weight $-1$ and degree $-1$. Then a graded mixed complexes is exactly a graded $\one[\epsilon]$-module: 
	\[ \gmc{\_M} \simeq \Mod_{\one[\epsilon]}^{\tx{gr}}\]
	
	Moreover the projection $\one[\epsilon] \to \one$ induces an adjunction: 
	\[ \begin{tikzcd}
		\gmc{\_M}\simeq \Mod_{\one[\epsilon]} \arrow[r, shift left] & \arrow[l, shift left] \gr{\Mod}_{\one} \simeq \gr{\_M}
	\end{tikzcd}\]	
\end{Lem}
\begin{proof}
	The projection $\one[\epsilon] \to \one$ induces a forgetful functor\footnote{This functor forgets the mixed differential.}:
	\[ \Mod_{\one[\epsilon]}^{\tx{gr}} \to \Mod_{\one}^{\tx{gr}} \simeq \_M^{\tx{gr}} \]
	
	We get this way the collection of objects $(E(p))_{p \in \Zz}$. The structure of $\one[\epsilon] \simeq \one \boxplus \one((1))[1]$-module on $E$ is given by a product $ \one((1))[1] \otimes E \to E$ which is given by a map $\epsilon : E \to E((-1))[-1]$. The square-zero product is equivalent to saying that $\epsilon^2 =0$. 
\end{proof} 

\begin{Prop}
	There is a model structure on $\gmc{\_M}$ whose weak equivalences and cofibrations are defined through the forgetful functor: \[\_U_\epsilon : \gmc{\_M} \to \_M^{\tx{gr}}\]
\end{Prop}

\begin{Def}
	\label{def:realization functor graded mixed complexes}
	
	The hom-tensor adjunction (applied to the unit $\one_{\_M}$ and using the $\Mod_k$ enrichment) induces the following adjunctions: 
	\[ \begin{tikzcd}
		\Mod_k \arrow[r, shift left] & \arrow[l, shift left] \_M : \rel{-}
	\end{tikzcd}\]
	\[ \begin{tikzcd}
		\gmc{\Mod_k} \arrow[r, shift left] & \arrow[l, shift left] \gmc{\_M} : \rel{-}
	\end{tikzcd}\]
	where $\rel{-}$ are called \defi{realization functors}. In particular, for $\_M = \gmc{\Mod_k}$ the first adjunction defines a realization functor: 
	\[\rel{-}: \gmc{\Mod_k} \to \Mod_k \]
\end{Def}

\begin{Prop}[{\cite[Proposition 1.5.1]{CPTVV}}] 
	\label{prop:realization of graded mixed complexes}
	
	When $\_M = \Mod_k$, then we have:
	\[ \rel{-} :=  \iHom (k, -)  \] 
	This is the derived internal Hom which is computed by taking a cofibrant resolution $\tilde{k} \to k$ in the projective model structure\footnote{The projective model structure is the model structure on graded mixed complexes induced by the projective model structure on $\Mod_k$.}. In particular we get: 
	
	\[\rel{E} := \iHom(\tilde{k}, E) \simeq \prod_{p \geq 0} E(p) \]
	where the right hand side is endowed with the total differential $D_E := d_E + \epsilon$. 
\end{Prop}

\begin{RQ}\label{rq:realisation as hom between filtered objects}
	It will turn out that graded mixed complexes can be seen as complete filtered objects (see Section \ref{sec:complete-filtered-commutative-algebras-and-weak-graded-mixed-complexes}). Under this fully-faithful inclusion, denoted by $\tot$, we have that: 
	\[ \rel{E} := \iHomsub{\cpl{\Mod}}(\tot(k), \tot(E))\]
	where $\tot(k)$ is in fact a fibrant cofibrant complete filtered object such that $F^pk = k$ if $p \geq 0$ and $0$ otherwise. Then we can also use Definition \ref{def:filtration of graded mixed complexes}, Corollary \ref{cor:colimit of filtered realization} and Remark \ref{rq:enrichement of filtered objects} to prove the proposition.
	
	This avoids the difficulty of taking a cofibrant replacement for $k$.
\end{RQ}

\begin{RQ}\label{rq:weight shifted realization}
	Similarly, if $\tilde{k}((-q)) \to k((-q))$ is a cofibrant resolution we we can define: 
	\[F^p \rel{E} := \iHom (k((-q)), E) \simeq \prod_{p \geq q} E(p) \]
	where the right hand side is again endowed with the total differential $D_E = d_E + \epsilon$. As the notation suggests this define a filtration on $\rel{E}$ which coincide with the natural filtration by taking the Hom functor from Remark \ref{rq:realisation as hom between filtered objects} but enriched in complete filtered objects. 
\end{RQ}

\begin{Def} \label{def:filtration of graded mixed complexes}
	
	First observe that the derived hom gives us for all $q \in \Zz$: 
	\[\Hom(k((-q)), k((-q-1))) \simeq k\]
	Picking $1 \in k$ corresponds to taking a collection of maps $k((-q)) \to k((-q-1))$. 
	They induce, for all $p \in \Zz$, the following maps: 
	\[ F^{p+1} \rel{E} \to F^{p}\rel{E}\]
	defining a filtered object whose colimit is $\rel{E}$ itself. This filtration is complete and we denote by $\tot(E)$ this filtered object, which defines a functor: 
	\[ \tot : \gmc{\Mod_k} \to \cpl{\Mod_k}\]
	This generalizes the functor of Construction \ref{cons:construction for complete filtered objects}.
\end{Def}

\begin{Cor}
	\label{cor:colimit of filtered realization}
	
	If $E \in \gmc{\Mod_k}$ is non-negatively weighted, then we have:
	\[ \colim \tot(E) \simeq \rel{E}\]
\end{Cor}

\begin{RQ}
	In general $\colim \circ \tot$ is called the \defi{Tate realization} of $E$ (see \cite[Definition 1.5.2]{CPTVV}). 
\end{RQ}

We have seen in Example \ref{ex:good model structure} that if $\_M$ is a good model category then $\gmc{\_M}$ is also a good model category and therefore its category of commutative monoids has a natural model structure, and modules over those also form good model categories (see Section \ref{sec:good-model-structures}).

\begin{notation}\label{not:graded mixed algebras and module over algebras}
	We denote by $\gmc{\cdga_\_M}$ the category of commutative monoid in $\gmc{\_M}$. 
	
	Given $A \in \gmc{\cdga_\_M}$, we denote by $\gmc{\Mod_A}$ the category of commutative monoid in $\gmc{\_M}$.
\end{notation}

\subsection{Weak Graded Mixed and Complete Filtered Objects} \label{sec:complete-filtered-commutative-algebras-and-weak-graded-mixed-complexes}\

\medskip

We observed thanks to Definition \ref{def:filtration of graded mixed complexes} that graded mixed complexes give complete filtered object. But the functor $\tot$ is not an equivalence and its essential image is given by the filtered objects whose differential has a weight decomposition concentrated in weight $0$ and $1$. 

If the following we will define a notion of \emph{weak} graded  mixed complex whose mixed differential increases the weight by an arbitrary degree, so that the complete filtered realization of those will recover all complete filtered objects.

\begin{Def}[{\cite[Definition 2.23]{CCN21}}]
	\label{def:weak graded mixed complex}
	
	A \defi{weak graded mixed complex in $\_M$} is a weight graded complex $E \in \gr{\_M}$ equipped with maps in $\_M$ of weight $k$ and degree $1$ for all $k \geq 1$: \[\delta_k : E \to E((-k))[-1]\]
	such that:
	\[\left(d + \sum_{k\geq 1} \delta_k\right)^2 =0\]
	with $d$ the differential on $E$. 
	
	This condition can be rephrased as: 
	\[ d \circ \delta_k + \delta_k \circ d  + \sum\limits_{i+j=k} \delta_i \circ \delta_j = 0 \]
	
	We denote by $\gmch{\_M}$ the category of graded mixed complexes whose morphisms are morphisms of graded objects that commute with $\delta_k$.
\end{Def}

We have seen in Lemma \ref{lem:graded mixed complex as graded modules} that graded mixed complexes are modules over $\one[\epsilon]$. It turns out that weak graded mixed complexes are module over a \emph{cofibrant resolution} of $\one[\epsilon]$. 

\begin{Lem}
	\label{lem:cofibrant resolution 1[epsilon]}
	
	Consider $\one[\epsilon_i \ \vert \ i \geq 1]$ the semi-free commutative monoid in $\_M$ generated by $\epsilon_i$ in degree $1$ and weight $i$ with differential given, for all $k \geq 1$, by: 
	\[ \delta(\epsilon_k) = -\sum_{i+j=k} \epsilon_i \epsilon_j\] 	
	Then this is a quasi-free resolution of $\one[\epsilon]$. 
\end{Lem}

\begin{Lem}
	\label{lem:weak graded mixed complex as module}
	
	The category of weak graded mixed complexes is equivalent to the category of graded modules over $\one[\epsilon_i \ \vert \ i \geq 1]$.  
\end{Lem}
\begin{proof}
	The proof is essentially an adaptation of the proof of Lemma \ref{lem:graded mixed complex as graded modules}.
\end{proof}

\begin{RQ}
	This description of weak graded mixed complexes makes it a good model category as it is a category of modules over the good model category $\gr{\_M}$. 
\end{RQ}

\begin{Cor}
	\label{cor:equivalence weak and strict graded mixed complexes}
	
	The inclusion of strict graded mixed complexes in weak graded mixed complexes is an equivalence: 
	
	\[ \gmc{\_M} \simeq \Mod_{\one[\epsilon]}^{\tx{gr}} \to \Mod_{\one[\epsilon_i \ \vert \ i\geq 1]}^{\tx{gr}}  \simeq \gmch{\_M} \]
	
	This equivalence is part of the extension-restriction of scalars adjunction associated to the weak equivalence $\one[\epsilon_i \ \vert \ i\geq 1] \to \one[\epsilon]$. Therefore this is a Quillen equivalence thanks to Proposition \ref{prop:module over a commutative monoid good model category}.
\end{Cor}

In a similar way to homotopy algebraic structures (such as $L_\infty$-algebras or algebroids for example), weak graded mixed complexes come equipped with a notion of $\infty$-morphism between weak graded mixed objects.

\begin{Def}
	\label{def:infty morphism of weak graded mixed complexes}
	
	An \defi{$\infty$-morphism between weak graded mixed complexes $E$ and $E'$} is the data of a collection of maps $\phi_k :E \to E'$ of degree $0$ and weight $k$ such that for all $k \geq 0$ we have: 
	\[ \sum_{i+j = k} (\epsilon_i \phi_j + \phi_i \epsilon_j) = 0  \]
	
	where $\epsilon_0 = d$ by convention. We denote by $\gmch{\_M_\infty}$ the category of weak graded mixed complexes with the $\infty$-morphisms.
\end{Def} 

\begin{RQ}
	From \cite[Remark 2.26]{CCN21}, $\gmch{\_M_\infty}$ has almost a model structure and the inclusion $\gmch{\_M} \to \gmch{\_M_\infty}$ induces an equivalence between the $\infty$-categories obtained after localizing at the weak-equivalences. 
\end{RQ}

The totalization functor of Construction \ref{cons:construction for complete filtered objects} and Definition \ref{def:filtration of graded mixed complexes} can be extended to an equivalence from weak graded mixed complexes. 

\begin{Prop}[{\cite[Proposition 2.27]{CCN21}}]
	\label{prop:totalization weak graded mixed complex}
	
	The functor:\[ \tot : \gmch{\Mod_k} \to \Mod_k^{\tx{cpl}}\]
	sending $E$ to the complete filtered cdga whose underlying filtered graded module is given by Construction \ref{cons:construction for complete filtered objects} to which we add the total differential $D = d + \sum_{k \geq 1} \epsilon_k$, induces an equivalence of $\infty$-categories.
\end{Prop} 

\begin{RQ}\label{rq:realization for weak graded mixed complexes}
	The realization functor of Definition \ref{def:realization functor graded mixed complexes} extends to weak graded mixed complexes. Moreover, if $\_M = \Mod_k$, we get an analogue of Proposition \ref{prop:realization of graded mixed complexes} expressing the realization as: 
	\[ \rel{E} \simeq \prod_{p \geq 0} E(p)\]
	
	together with the total differential $D_E = d_E + \epsilon_1 + \epsilon_2 + \cdots$. 
\end{RQ}
\begin{Lem}
	\label{lem:totalization and colimit}
	
	If $E$ is a non-negatively weighted weak graded mixed object, then: 
	\[ \colim \tot(E) \simeq \rel{E}\]
\end{Lem}
\begin{proof}
	This is a direct consequence of Remark \ref{rq:realization for weak graded mixed complexes}.
\end{proof}

\begin{Lem}
	\label{lem:associated graded and forgetting mixed structure}
	
	We have a commutative diagram: 
	\[ \begin{tikzcd}
		\gmc{\Mod_k} \arrow[r, "\tot"] \arrow[dr,"\tx{Forget}"'] & \cpl{\Mod_k} \arrow[d, "\tx{Gr}"] \\
		& \gr{\Mod_k} 
	\end{tikzcd}\]
	where $\tx{Forget}$ is the functor that forgets the mixed differentials $\epsilon_i$ for $i \geq 1$ induced by the scalar extension along $\one[\epsilon_i \ \vert \ i \geq 1] \to \one$. It is similar to the adjunction in Lemma \ref{lem:graded mixed complex as graded modules}. 
\end{Lem}

\newpage

\section{Formal Moduli Problem and Koszul Duality Context}\label{sec:formal-moduli-problem-and-koszul-duality-context} 

\medskip

This section aims to describe the notion of formal moduli problems. In Section \ref{sec:lie-and-linfty-algebroids}, we talk about the formal moduli problems under a base corresponding to a Lie algebroid. It is similar to the well known correspondence between Lie algebras and (commutative) formal moduli problems. We are going to present the general theory behind such correspondences and the essential definitions and tools involved to understand them. 

\subsection{Formal Moduli Problems}\label{sec:formal-moduli-problems}\

\medskip

This section is about axiomatizing the notion of deformation functors. The idea is that deformations can be understood as functors sending a class of ``small parameter objects'' called ``Artinian'', to a space of deformations along such objects. For example such a deformation functor assigns for each choice of Artinian local algebra $A$ an $\infty$-groupoid of deformations of $X$ along $\Spec(A)$. As the category of local Artinian algebras and the notion of ``infinitesimal'' object will play an important role, we will start, following \cite{Lu11}, \cite{CG18} and \cite{CCN20}, by defining a general framework in which we can speak of deformations along such objects.  

\subsubsection{Artinian and small algebras in a deformation context}\ \label{sec:Artinian-and-small-algebras-in-a-deformation-context} 

\medskip

The general idea behind Artinian algebras is to have a class of algebras considered ``small'' so that deforming along those algebras amounts to consider \emph{infinitesimal} deformations, that is, looking at the formal neighborhood of what we want to deform. We will always assume that our ambient category $\_A$, in which we will define the notion of Artinian object, has a terminal object. This leads to the general framework of deformation context: 

\begin{Def}[{\cite[Definition 1.1.3]{Lu11}}] \label{def:deformation context}
	
	A \defi{deformation context} is a pair $\left( \_A, \lbrace E_\alpha \rbrace_\alpha \right)$  where $\_A$ is a presentable\footnote{{$\_A$ is presentable if it is generated by a small set of ``small objects'' under homotopy colimits.}} $\infty$-category and $\lbrace E_\alpha \rbrace_\alpha$ is a set of objects in $\mathbf{Stab}(\_A)$\footnote{Recall that objects in $\mathbf{Stab}(\_A)$ correspond to spectrum objects on $\_A_*$, i.e., a sequence  $E=(\cdots, a_{2}, a_{1},a_0, a_{-1}, \cdots)$  of pointed objects of $\_A$, such that $a_i$ is equivalent to the loop space $\Omega a_{i+1}$. For $n\geq 0$, we denote $\Omega^{\infty-n}E = a_{n}$.}, the stabilization of $\_A$.  
	
	%already in nomenclature Stab
\end{Def}

Intuitively, $\lbrace E_\alpha \rbrace_\alpha$ correspond to the first-order objects we want to deform along. 	
To compare with the classical picture, we take $\_A$ to be the category of algebras. Then first order deformation are deformations along $k[\epsilon]$, and we take it as the underlying object of the spectrum $E= (\cdots,k[\epsilon_2],k[\epsilon_1],k[\epsilon_0],\cdots)$, where $k[\epsilon_i]$ is the $2$-dimensional $k$-augmented algebra with $\deg{ \epsilon_i}=-i$ and $\epsilon_i^2 = 0$.

In our situation, the family $\lbrace E_\alpha \rbrace_\alpha$ will always contains only a single spectrum object. Here are some other examples of deformation contexts. 

\newpage

\begin{Ex}\label{ex:deformation contextes}\
	
	\begin{itemize}
		
		\item For $\_A = \Mod_A$ (with $A \in \cdga$), we can defined the deformation context $\left(  \Mod_A, \lbrace E \rbrace \right)$ with only one spectrum object (\cite[Example 2.2]{CG18}): \[E = \left( A \boxplus A[n] \right)_{n \in \Zz}\]  
		
		\item Given $A \in \cdga$, we consider $\_A =\cdga_{/A}$ and and we can define the following deformation context\footnote{Note that $\bf{Stab}(\cdga_{/A}) \simeq \bf{Stab}(\cdga_{A//A}) \simeq \Mod_A$.}: 
		\[\left( \cdga_{/A}, \lbrace E = \left( A \boxplus A[n] \right)_{n \in \Zz} \rbrace \right)\]
		
		where $A \boxplus A[n]$ is the square zero extension of $A$ by $A[n]$ (\cite[Remark 2.5]{CG18}). We will call this deformation context $\dc$.

	\end{itemize}
\end{Ex}

From a deformation context, we can define the notion of ``small'' objects and morphisms using $\Omega^{\infty-n}E_\alpha$ as building blocks.

\begin{Def}[{\cite[Definition 1.1.14]{Lu11}}] \label{def:small objects and morphisms}
	Given a deformation context $\left( \_A, \lbrace E_\alpha \rbrace_{\alpha \in T} \right)$ we say that:
	\begin{itemize}
		\item A morphism $f : B \rightarrow B'$ is called \defi{elementary} if it is given by a pullback of the form: 
		\[ \begin{tikzcd}
			B \arrow[r] \arrow[d, "f"] & \star \arrow[d] \\
			B' \arrow[r] & \Omega^{\infty -n} ( E_\alpha)
		\end{tikzcd}\]
		
		for some $\alpha \in T$ and $n \geq 1$. 
		\item A morphism $f: B \rightarrow B'$ is called \defi{small} if it is a finite composition of elementary morphisms.
		\item An object $B \in \_A$ is called \defi{Artinian}, if the morphism $\_A \rightarrow \star$ is small\footnote{Artinian objects can be though as ``small'' objects. In fact they are called ``small'' in \cite{Lu11}.}. We denote by $\mathbf{Art}_\_A$ the full sub-category of $\_A$ given by small objects. 
	\end{itemize}
	
\end{Def}

\begin{Ex}
	Going back to the example of $A$-augmented algebras, we  have $\Omega^{\infty -n} E = A \boxplus A[n]$ so that in order to compute the homotopy pullback of $\star \to \Omega^{\infty -n} E$, one can replace the point $\star = A$ by the algebra $A\boxplus (A[n]\oplus A[n-1])$ (notice that $n\geq 1$), with differential induced by the identity $A[n] \to A[n-1]$. 
	
	The strict pullback exhibits therefore $B$ as a square zero extension of $B'$ along $A[n-1]$. 
\end{Ex}

We will now specialize these definitions to the deformation contexts we are interested in, namely $\dc$. 

\begin{Prop}
	\label{prop:artinian algebra properties}
	The category $\small$ is the smallest full sub-category of $\cdga_{/A}$ such that:  
	\begin{itemize}
		\item $ A \boxplus A[n] \in \small$ for all $n\geq 0$
		\item For any $B \in \small$ and any map $B \rightarrow A \boxplus A[n]$ with $n\geq 1$, the homotopy pullback $B \times_{A \boxplus A[n]} A$ is also Artinian:
		
		\[\begin{tikzcd}
			B \times_{A\boxplus A[n]} A \arrow[r] \arrow[d] & A \arrow[d] \\
			B \arrow[r] &  A \boxplus A[n]
		\end{tikzcd}\] 
	\end{itemize}   
\end{Prop}

\begin{proof} 
	By definition, we have that $A \boxplus A[n]$ is Artinian since the map $A \boxplus A[n] \rightarrow A$ is elementary, given by the pullback:
	\[ \begin{tikzcd}
		A \boxplus A[n] \arrow[r] \arrow[d] & A \arrow[d] \\
		A \arrow[r] & A \boxplus A[n+1]
	\end{tikzcd}\] 
	
	Moreover, the map $B \times_{A \boxplus A[n]} A \rightarrow B$  is elementary, therefore if $B$ is Artinian, then the composition $B \times_{A \boxplus A[n]} A \rightarrow B \rightarrow 0$ is small as well. Therefore $\dc$ satisfies the properties given in the proposition. It is the smallest such category because any Artinian object can be obtained by definition from some $A \boxplus A[n]$ in finitely many pullbacks along $A \rightarrow A \boxplus A[n]$ for $n\geq 1$. 
\end{proof}

\subsubsection{Formal moduli problems}\label{Sec_FMP}\

\medskip

The main idea is that a ``formal moduli problem'' is a functor sending $B$ to an $\infty$-groupoid of deformations. For example, given an object $X$, we will be interested in studying the functor that sends an Artinian object $A$ to the ``space of all deformations'' of $X$ along $A$.

\begin{Def}[{\cite[Definition 1.1.14]{Lu11}}] \label{def:FMP} 
	
	Given a deformation context $\left( \_A, \lbrace E_\alpha \rbrace_{\alpha \in T} \right)$, a functor between $\infty$-categories: \[F : \mathbf{Art}_{\_A} \rightarrow \igpd\]  is called a \defi{formal moduli problem} if it satisfies the following conditions: 
	\begin{itemize}
		\item Deformations along the point (terminal object) are trivial:
		\[ F(\star) \simeq \star \]
		
		\item Any pullback in $\mathbf{Art}_{\_A}$ along a small morphism $\phi : A \rightarrow B$ is sent to a pullback in $\igpd$: 
		
		\[ F \left( \begin{tikzcd}
			A' \arrow[r] \arrow[d] & A \arrow[d, "\phi"] \\
			B' \arrow[r] & B
		\end{tikzcd} \right) = \begin{tikzcd}
			F(A') \arrow[r] \arrow[d] & F(A) \arrow[d, "F(\phi)"] \\
			F(B') \arrow[r] & F(B)
		\end{tikzcd} \]
	\end{itemize}
	
	The $\infty$-category of such functors will be denoted $\FMP\left( \_A, \lbrace E_\alpha \rbrace_{\alpha \in T} \right)$ or $\FMP_\_A$ for short if the choice of the collection of spectrum objects $E_\alpha$ is clear. For the deformation context we are interested in, we will use the notation\footnote{This should be though of as the commutative formal moduli problems \emph{over} $A$ or \emph{under} $\Spec(A)$.}: \[\bf{FMP}_A := \FMP_{\dc}\]  
\end{Def}

\begin{Def}\label{def:formal spectrum} 
	
	For any $B \in \cdga_{/A}$ the following functor is a formal moduli problem:
	\[  \begin{tikzcd}[row sep = 1mm, column sep = tiny]
		\Spf_A (B) : \small \arrow[r] & \igpd \\
		\qquad \qquad C \arrow[r, mapsto] & \Mapsub{\cdga_{/A}} \left( B,C \right) 
	\end{tikzcd} \]
	
	Moreover we get a functor $\cdga_{/A} \rightarrow \FMP_A$ sending $B$ to $\Spf_A (B)$.
	
	$\Spf_A (B)$ is called the \defi{formal spectrum} of $B$ over $A$. 
\end{Def}

\begin{Prop}	
	Following \cite[Lemma 1.1.20]{Lu11}, a morphism $f: B \rightarrow C$ between Artinian algebras, $B,C \in \small$ is small if and only if it induces a surjection of commutative rings $H^0(B) \rightarrow H^0(C)$. Therefore the second condition of Definition \ref{def:FMP} can be rephrased as follows: 
	
	Any pullback in $\small$ such that $H^0 (B) \rightarrow H^0(C)$ or $H^0 (B') \rightarrow H^0(C)$ is surjective is sent to a pullback in $\igpd$: 
	
	\[ F \left( \begin{tikzcd}
		A' \arrow[r] \arrow[d] & A \arrow[d, "\phi"] \\
		B' \arrow[r] & B
	\end{tikzcd} \right) = \begin{tikzcd}
		F(A') \arrow[r] \arrow[d] & F(A) \arrow[d, "F(\phi)"] \\
		F(B') \arrow[r] & F(B)
	\end{tikzcd} \]
	It recovers the classical Schlessinger condition (see \cite{Sch68}).
\end{Prop}

It is an old heuristic that (commutative) deformation problems (over $k$) are classified by Lie algebras and it is due to a kind of duality (Koszul duality context) that is in fact related to the Koszul duality of the Lie and commutative operads (see \cite{CCN20}). This heuristic is formalized by the following theorem:  
\begin{Th}[See \cite{Lu11} and \cite{Pr10}]
	\label{th:Lurie-Pridham}
	
	There is an equivalence of $\infty$-categories: 
	\[ \begin{tikzcd}[row sep = tiny, ampersand replacement=\&]
		\mathbf{MC} \colon	\mathbf{Alg}_{\mathbf{Lie}} \arrow[r,shift left, "\sim"] \& \arrow[l,shift left] \FMP \colon \Tt[-1]  
	\end{tikzcd}\]
\end{Th}

\begin{War}\label{war:FMP equivalence is not always maurer-cartan} 
	For $\G_g$ a Lie algebra, the one thinks of $\mathbf{MC}$ as the functor assigning to an Artinian algebra $A$ the space space of Maurer--Cartan elements\footnote{The space of Maurer--Cartan elements of a Lie algebra $\G_g$ is the space of elements $x \in \G_g$ satisfying the Maurer--Cartan equation: \[ dx + \frac{1}{2}[x,x] = 0\]}. Unfortunately, this space does not generally define a formal moduli problem, but is not far from doing so. Indeed, on finite dimensional $\G_g$, the Maurer--Cartan functor gives the correct functor. The functor $\mathbf{MC}$ generalizes the Maurer--Cartan functor in order to circumvent this problem. 
\end{War}

\subsection{Formal Moduli Problems and Koszul Duality } \label{sec:formal-moduli-problems-and-koszul-duality-}\

\medskip

The goal of this section is to make sensee of the generalization of Theorem \ref{th:Lurie-Pridham} to an arbitrary Koszul duality context while discussing the specific example of formal moduli problem under $\Spec(A)$ and Lie algebroids.

\subsubsection{Tangent complex of a formal moduli problem} \label{sec:tangent-complex-of-a-formal-moduli-problem}\ 

\medskip

The goal of this section is to explain that the algebra controlling a given formal moduli problem is in some sens the ``tangent'' of this formal moduli problem. This explains why first-order deformations and obstructions to lifting deformations are controlled by the tangent complex.\\

To motivate the following definition, we will start by explaining the example of the formal spectrum (Definition \ref{def:formal spectrum}) of an $A$-augmented commutative algebra $B$. We have that
\[\Spf_A (B)(C) = \Mapsub{\cdga_{/A}} \left( B,C \right)\]

In the case of first order deformations, $C =  A[\epsilon]$, this mapping space corresponds to the tangent complex of $\tx{Spec}(B)$ at the point $f: \Spec(A) \rightarrow \Spec (B)$.
It is shown in \cite[Proposition 1.4.1.6]{TV08} that the geometric realization of connective truncation of $\Tt_{B,f} \simeq f^* \Tt_B$ is given by: 
\begin{equation}\label{eq:tangent functor and tangent complex}
	\Spf_A(B)(A\boxplus A[n]) \simeq \Mapsub{\cdga_{/A}}(B, A \boxplus A[n])
	\simeq 
	\Mapsub{A} (A, f^* \Tt_{B}[n]) \simeq \rel{f^*\Tt_{B}[n]}
\end{equation} 

Indeed the map $f:B \to A$ induces a map: 
\[ \Mapsub{\cdga_{/A}}\left(A, A\boxplus M\right) \to \Mapsub{\cdga_{/A}}\left(B, A\boxplus M\right) \simeq \Mapsub{\cdga_{/A}}\left(B, B\boxplus f_*M \right)\]

The right hand side is represented by $f^* \Ll_B$ and with $M = A$ we recover the claim.

This motivates the definition of the ``tangent'' as the evaluation of the formal moduli problem functor at all ``first-order elements''\footnote{Recall that such ``first-order'' objects correspond to the elements in $\lbrace \Omega^{\infty -n} (E_\alpha) \rbrace_{\alpha \in T, n \in \Nn}$, for example the square zero extensions $A\boxplus A[n]$ in the case of $A$-augmented algebras.} of a given deformation context.

\begin{Def}[Tangent Functor, {\cite[Section 1.2]{Lu11}}] \label{def:tangent functor}
	
	Given a deformation context $\left( \_A, \lbrace E_\alpha \rbrace_\alpha \right)$, we define \defi{the tangent complex} of a formal moduli problem $F$ at $\alpha$ to be the spectrum object: \[F(E_\alpha) \in \mathbf{Sp}\coloneqq \mathbf{Stab}(\mathbf{sSet})\]
	which we denote\footnote{This notation is by analogy to the tangent complex but is a priori only a (collection of) spectrum objects. However, we will see that $\Tt_F$ can be ``represented'' by an actual tangent complex.} $\Tt_F$ in case there is only a single $\alpha\in T$.

\end{Def}   

This spectrum object verifies $\Omega^{\infty -n}F(E_\alpha) \simeq F(\Omega^{\infty -n}E_\alpha)$ for all $n\geq 0$ (see \cite[Remark 1.2.7]{Lu11}). Notice that the tangent complex $\Tt_F$ is not actually a cochain complex but rather only a spectrum object a priori. These two categories are related by the composition:

\begin{equation}\label{eq:doldkan}
	\Mod_A \stackrel{\tx{Forget}}{\longrightarrow} \Mod_{\mathbb Z} \stackrel{\tx{DK}}{\simeq} \mathbf{Stab}(\mathbf{sAb}) \stackrel{\tx{Forget}}{\longrightarrow}  \mathbf{Stab}(\mathbf{sSet})=\mathbf{Sp}
\end{equation}

\begin{Prop}[{\cite[Lemma 2.15]{CCN20}}] \label{prop:tangent space and tangent complex} 
	A formal moduli problem $F\in \FMP_A$ has a unique pre-image in $\Mod_A$ under the functor \eqref{eq:doldkan}, which we also denote by $\Tt_F$.	
	
	Moreover for any $n \geq 0$, we have an equivalence: 
	\[ \Mapsub{\Mod_A} \left( A[-n], \Tt_F \right) \simeq F(A\boxplus A[n])\]
	
\end{Prop}

It turns out that the associated spectrum object to $f^*\Tt_{B}$ (given by the composition \eqref{eq:doldkan}) coincides with the spectrum object given by the collection of $\Spf_A(B)(A\boxplus A[n])$ so that the cotangent complex $f^*\Tt_{B}$ is the representative of tangent functor $\Tt_{\Spf_A(B)}$, thanks to the equivalence \eqref{eq:tangent functor and tangent complex}.

\subsubsection{Koszul Duality Context}\label{sec:koszul-duality-context}\

\medskip

Going back again to the setting of a general deformation context $\_A$, we are looking to identify an $\infty$-category of algebraic objects $\_B$ to $\FMP_{\_A}$.

The formal spectrum $\Spf $ from in Definition \ref{def:formal spectrum} is a functorial way to construct formal moduli problems out of objects of $\_A$. 
Assuming the existence of the desired equivalence, we could construct a functor:

$$\mathfrak D \colon \_A^\mathrm{op} \stackrel{\Spf }{\longrightarrow}\FMP_{\_A} \stackrel{\sim}\to \_B.$$

This functor should be interpreted as a ``weak duality'' functor, which is not an equivalence, but at least its restriction to Artinian objects should behave as an equivalence into a subcategory of $\_B$ given by ``good'' objects.

In this section, we introduce the notion of Koszul duality context as the appropriate axiomatic framework that enables us to obtain the desired equivalence, which we give with Theorem \ref{th:FMP equivalence tkoszul duality context}.

\begin{Def}[Dual Deformation Context, {\cite[Definition 2.11]{CG18}}] \label{def:dual deformation context}
	
	A pair $\left( \_B, \lbrace F_\alpha \rbrace_{\alpha \in T} \right)$ is called a \defi{dual deformation context } if $\_B$ is a presentable $\infty$-category and $\_F_\alpha \in \mathbf{Stab}(\_B^{\tx{op}})$.\\
	
	We say that an object (resp. morphism) of $\_B$ is \defi{good} if it is Artinian when considered in $\left( \_B^{\tx{op}}, \lbrace F_\alpha \rbrace_{\alpha \in T} \right)$\footnote{Here $\left( \_B^{\tx{op}}, \lbrace F_\alpha \rbrace_{\alpha \in T} \right)$ is generally not a deformation context since $ \_B^{\tx{op}}$ will in general not be presentable. However the definition of Artinian object still makes sensee in $\_B^{\tx{op}}$.}. We denote by $\_B^{\tx{gd}}$ the full sub-category of good objects of $\_B$. 
\end{Def}  

\begin{Ex}\ \label{ex:dual deformation contexts}
	
	\begin{itemize}
		\item If $A \in \cdga$. Then $\left( \Mod_A, \left( A[n] \right)_{n \in \Zz} \right)$ is a deformation context and since taking the opposite category exchanges the suspension and desuspension functor, we get the dual deformation context $\left(\Mod_A, \left( A[-n] \right)_{n \in \Zz} \right)$.
		
		When $A$ is bounded and concentrated in non-positive degrees, \cite[Lemma 2.16]{CG18} tells us that an element $M \in \Mod_A$ in this dual deformation context is good if it is perfect and cohomologically concentrated in positive degrees (see also \cite[Remark 2.17]{CG18})
		
		\item There is a dual deformation context: 
		\[ \left( \algbd_A, \lbrace\tx{Free}(0 : A[-n] \to \Tt_A)\rbrace_{n \in \Zz} \right)\]
		
		Note that every good Lie algebroid have underlying cofibrant $A$-module and is finitely generated (when forgetting the differential) over $\Tt_A$ (see \cite[Lemma 6.10]{Nu19a}). 
		
	\end{itemize}
	
\end{Ex}

\begin{Def}[Koszul Duality Context, {\cite[Definition 2.18]{CG18}}]\ \label{def:koszul duality context}
	
	A \defi{weak Koszul duality context}  is the data of: 
	\begin{itemize}
		\item A deformation context $\left( \_A, \lbrace E_\alpha \rbrace_{\alpha \in T} \right)$ 
		\item A dual deformation context $\left( \_B, \lbrace F_\alpha \rbrace_{\alpha \in T} \right)$ 
		\item An adjunction: 
		
		\[ \begin{tikzcd}
			\G_D : \_A \arrow[r, shift left] & \arrow[l, shift left] \_B^{\tx{op}} : \G_D'
		\end{tikzcd}\]
	\end{itemize}
	
	such that for all $n \geq 0$, there is an equivalence $\Omega^{\infty - n} E_\alpha \simeq \G_D' \left( \Omega^{\infty -n} F_\alpha \right)$. \\
	
	It is called a \defi{Koszul Duality} context if the following hold:
	
	\begin{enumerate}
		\item for every object $B \in \_B^{\tx{gd}}$, the counit morphism $\G_D \G_D' A \rightarrow A$ is an equivalence. 
		\item For each $\alpha$, the functor
		\[ \varTheta_\alpha :  \_B \rightarrow \mathbf{Sp} \]
		sending $B \in \_B$ to the spectrum object given by: \[\left( \Map_{\_B}(\Omega^{\infty -n} E_\alpha, B) \right)_{n \in \Zz} \in \mathbf{Sp}\]
		is conservative and preserves sifted colimits. 
	\end{enumerate} 
\end{Def}

\begin{RQ}
	A weak Koszul duality context together with condition (1) gives us a weak deformation theory according to the terminology of Lurie, \cite[Definition 1.3.1]{Lu11}. A Koszul duality context is an example of a \defi{deformation theory} according to \cite[Definition 1.3.9]{Lu11}. 
\end{RQ}

\begin{Ex}
	Let $A \in \cdgacon$ bounded. We have a Koszul duality context given by the dualization (\cite[Example 2.20]{CG18}):
	\[ \begin{tikzcd}
		(-)^\vee : \left(\Mod_A, \left( A[n] \right)_{n \in \Zz} \right) \arrow[r, shift left] & \arrow[l, shift left] \left( \Mod_A^\mathrm{op}, \left( A[-n] \right)_{n \in \Zz} \right) : (-)^\vee
	\end{tikzcd} \]
\end{Ex}

\begin{Ex}[{\cite[Theorem 3.9]{CG18} and \cite{Nu19a}} ]\
	\label{ex:koszul duality context lie algebroid}
	
	With the (dual) deformation context of Examples \ref{ex:deformation contextes} and \ref{ex:dual deformation contexts}, if $A$ is cofibrant non-positively graded, almost finitely presented and eventually coconnective (in other words satisfies Assumption \ref{ass:very good stack}), then we have a Koszul duality context:
	\[\begin{tikzcd}
		\ce : \algbd_A  \arrow[r, shift left] & \arrow[l, shift left] \left( \cdga_{/A}\right)^{\tx{op}} : \G_D
	\end{tikzcd}\]
\end{Ex}

\begin{Prop}[{\cite[Proposition 2.22]{CG18}}] \
	
	\label{prop:properties koszul duality context}
	Given a Koszul duality context (with the same notations as in Definition \ref{def:koszul duality context}), we have the following:
	
	\begin{itemize}
		\item $\G_D \left( \Omega^{\infty -n} E \right) \simeq \Omega^{\infty -n} F$ for all $n \geq 0$. 
		\item For every Artinian object $A \in \mathbf{Art}_\_A$, the unit map $A \rightarrow \G_D' \G_D A$ is an equivalence. 
		\item The adjunction $\G_D \dashv \G_D'$ induces an equivalence:
		
		\[ \begin{tikzcd}
			\G_D : \mathbf{Art}_\_A \arrow[r, shift left] & \arrow[l, shift left]\left( \_B^{\tx{gd}}\right)^{\tx{op}} : \G_D' 
		\end{tikzcd}\]
		
		\item If $M \in \mathbf{Art}_\_A$ and $f\colon A \rightarrow B$ is a small morphism in $\_A$ then $\G_D$ sends the pullback diagram: 
		\[ \begin{tikzcd}
			P \arrow[r] \arrow[d] & A \arrow[d, "f"] \\
			M \arrow[r] & B 
		\end{tikzcd}\]
		
		to a pullback diagram. 
		
	\end{itemize}
\end{Prop}

\begin{Cor}\label{cor:koszul duality context and FMP}
	Given a Koszul duality context, we can construct the following functor: 
	\[ \begin{tikzcd}
		\psi : \_B \arrow[r] &  \mathbf{Fun}\left(\_B^\tx{op}, \igpd \right) \arrow[r, "\circ \G_D"] & \mathbf{Fun}\left( \_A, \igpd \right)
	\end{tikzcd}  \]
	
	Then $\psi$ factor through $\FMP(\_A, \lbrace E_\alpha \rbrace)$ and we get a map:
	\[\Psi : \_B \rightarrow \FMP(\_A, \lbrace E_\alpha \rbrace)\]
\end{Cor}

It turns out that the definition of Koszul duality context ensures that the map $\Psi : \_B \rightarrow \FMP(\_A, \lbrace E_\alpha \rbrace)$ is in fact an equivalence.

\begin{Th}[{\cite[Theorem 2.33]{CG18}} or {\cite[Theorem 1.3.12]{Lu11}}] \label{th:FMP equivalence tkoszul duality context}
	
	Given a Koszul duality context as above, the functor $\Psi$ is an equivalence:
	\[\Psi : \_B \rightarrow \FMP(\_A, \lbrace E_\alpha \rbrace)\] 
\end{Th}

\begin{Prop}[{\cite[Proposition 2.36]{CG18}}] \label{prop:FMP equivalence and tangent}
	We have the following commutative diagram: 
	\[ \begin{tikzcd}
		\_B \arrow[rr, "\Psi"] \arrow[rd, "\varTheta"'] & & \FMP\left( \_A, \lbrace E_\alpha \rbrace \right) \arrow[dl, "T"] \\
		& \mathbf{Sp} & 
	\end{tikzcd}\]
	
\end{Prop}

\begin{RQ}\label{rq:structure tangent functor}
	This diagram shows that $TF$ is equivalent to $\varTheta (\Tt_F)$ with $\Tt_F \in \_B$ so that this tangent functor has more structure that being just a spectrum object. Moreover when $\_A = \dc$ and $F = \Spf(B)$ then $\Tt_F = f^*\Tt_B$ with $f: B \to A$ (see Lemma \ref{lem:tangent formal spectrum}).
\end{RQ}

 \nocite{*}
 
 \newpage
\bibliographystyle{alpha}
\bibliography{biblio2}

\makeatletter
\providecommand\@dotsep{5}
\makeatother
%\listoftodos\relax

\end{document}